\documentclass[a4paper,11pt]{article}

\usepackage{graphicx, rotating}
\usepackage{subfigure}
\usepackage[justification=centering]{caption}

\usepackage[utf8]{inputenc}
\usepackage[toc]{appendix}
\usepackage{twoopt}

\usepackage{amssymb}
\usepackage{amsthm}
\usepackage{amsmath,a4wide}
\numberwithin{equation}{section}
\usepackage{hyperref}
\usepackage{mathtools}
\usepackage{amsfonts}

\usepackage{authblk}

\usepackage{wrapfig}
\usepackage{graphicx, rotating}
\usepackage{subfigure}
\usepackage[justification=centering]{caption}

\usepackage{marvosym}
\usepackage{verbatim}

\usepackage{pdfsync}
\usepackage{caption}
\usepackage{hyperref}
\usepackage{enumerate}
\mathtoolsset{showonlyrefs}

\theoremstyle{plain}
\newtheorem{theorem}{Theorem}[section]

\theoremstyle{plain}
\newtheorem{mtheorem}[theorem]{Metatheorem}

\theoremstyle{plain}
\newtheorem{corollary}[theorem]{Corollary}

\theoremstyle{plain}
\newtheorem{lemma}[theorem]{Lemma}

\newtheorem*{lemma*}{Lemma}

\theoremstyle{plain}
\newtheorem{proposition}[theorem]{Proposition}

\newtheorem{definition}[theorem]{Definition}

\theoremstyle{remark}
\newtheorem*{remark}{\textbf{Remark}}

\theoremstyle{remark}

\newtheorem{example}[theorem]{Example}
\theoremstyle{plain}

\theoremstyle{definition}

\providecommand{\norm}[1]{\lVert #1 \rVert}
\newcommand{\R}{\mathbb{R}}
\newcommand{\C}{\mathbb{C}}
\newcommand{\T}{\mathbb{T}}
\newcommand{\Rd}{\mathbb{R}^d}
\newcommand{\Z}{\mathbb{Z}}
\newcommand{\N}{\mathbb{N}}
\newcommand{\Lt}[1][d]{L^2(\R^{#1})}
\newcommand{\G}{\mathcal{G}}
\newcommand{\F}{\mathcal{F}}

\newcommand{\indicator}{\raisebox{2pt}{$\chi$}}
\renewcommand{\l}{\lambda}
\renewcommand{\L}{\Lambda}

\newcommand{\vol}{\textnormal{vol}}
\renewcommand{\H}{\mathbb{H}}

\newcommand{\spec}{\text{SPEC}}

\DeclareMathOperator*{\esssup}{ess\,sup}
\DeclareMathOperator*{\essinf}{ess\,inf}
\DeclareMathOperator*{\supp}{supp}
\DeclareMathOperator*{\sinc}{sinc}

\newcommandtwoopt{\xarrow}[2][0.5cm][0]{\mathrel{\rotatebox[origin=c]{#2}{$\xrightarrow{\rule{#1}{0pt}}$}}}
\usepackage{xcolor}

\newcommand{\NO}[1]{\textcolor{darkgray}{#1}}

\title{
	\textbf{Time-Frequency Analysis} \\ \textit{Lecture Notes}
}

\author{
	Markus Faulhuber%\footnote{\href{mailto:markus.faulhuber@univie.ac.at}{markus.faulhuber@univie.ac.at}}
}

\affil{NuHAG, Faculty of Mathematics, University of Vienna}

\date{Summer Term 2021\\ Version: \today}

\begin{document}

\maketitle

\noindent
These lecture notes accompanied the course \textit{Time-Frequency Analysis} given at the Faculty of Mathematics of the University of Vienna in the summer term 2021. The material is suitable for an advanced undergraduate course in mathematics or a mathematics class for PhD students. Besides standard linear algebra and calculus only some basics from functional analysis are needed. A course in Fourier analysis may be of advantage, but is not needed. The course contained 4 academic units per week. The appendices and Section \ref{sec:frame_set} were not presented in class.

\bigskip

\noindent
I gratefully acknowledge the feedback from my students throughout the course and for pointing out several typos in the lecture notes. In particular, the comments of L.\ Köhldorfer, F.\ Moscatelli, I.\ Shafkulovska, and E.\ Stefanescu helped to improve these lecture notes.

\newpage
\tableofcontents

\newpage

\vfill

\textit{Hitherto communication theory was based on two alternative methods of signal analysis. One is the description of the signal as a function of time; the other is Fourier analysis. Both are idealizations, as the first method operates with sharply defined instants of time, the second with infinite wave-trains of rigorously defined frequencies. But our everyday experiences --- especially our auditory sensations --- insist on a description in terms of \textsf{both} time and frequency.}

\bigskip
-- \cite{Gab46} D.~Gabor. Theory of Communication, \textit{Journal of the Institution of the Electrical Engineers}, 93(26):429--457 (1946).

\bigskip

\begin{wrapfigure}{l}{.2\textwidth}
	\vspace{-.25cm}
	\includegraphics[width=.2\textwidth]{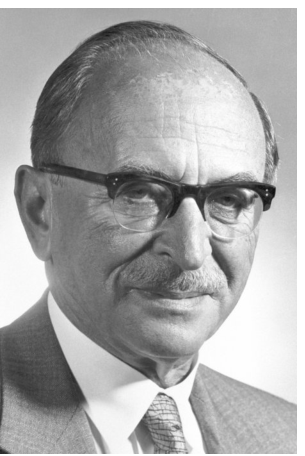}
\end{wrapfigure}
\noindent
\textbf{Short biography}. Dennis Gabor (originally D\'enes G\'abor, June 5, 1900 -- February 8, 1979) was born in Budapest, Austria-Hungary at that time. He was an electrical engineer and physicist. After gaining his doctorate from the Charlottenburg Technical University in Berlin, now known as the Technical University of Berlin, he worked for Siemens AG. In 1933 he left Nazi Germany, as he was considered to be Jewish. Britain invited him to work at the development department of the British Thomson-Houston company (BTH) and he also became a British citizen. In 1947 he invented holography, for which he was awarded the Nobel Prize in Physics in 1971:

``\textit{For his invention and development of the holographic method.}"

\url{https://www.nobelprize.org/prizes/physics/1971/gabor/biographical/}

\newpage
\textbf{Information.}
The main sources for these lecture notes are the textbooks by Folland \cite{Fol89}, de Gosson \cite{Gos11} and Gröchenig \cite{Gro01}. Whenever a lack of references occurs, one of these sources may be consulted.

\section*{Some Engineering Terminology}
\addcontentsline{toc}{section}{Some Engineering Terminology}
In this section we are going to introduce some important notions from sampling theory. We mainly sum up the introduction from \cite{Zay93} and quote some passages verbatim. One aim of this section is to introduce some engineering terms and `translate' them into mathematical terms.

The most important term in communication engineering and signal analysis is the term \textit{signal}. Actually, there are several types of signals which we need to distinguish, namely continuous (analog) signals and discrete (digital) signals. Whenever we only speak of a signal, we mean the first type. Mathematically, a signal is nothing but a (continuous) function, which is why we may use the words signal and function interchangeably. From a physical point of view, a signal $f(t)$ may represent the voltage difference at time $t$ between two points in an electrical circuit or the pressure of a sound wave. A digital signal, on the other hand, is just a sequence of numbers $(c_k)_{k \in \Z}$.

A prototype task in communication might be to convert an analog signal into a digital signal, transmit the digital signal to a receiver and convert it back to an analog signal (maybe even the original one). By processing a signal $f$ we mean operating on it in some fashion and usually also require that this operation is reversible. Mathematically, this means we apply some transformation (linear or not) to the function $f$ which we require to be invertible.

Lastly, we want to transmit information in a ``cheap" fashion, meaning that we want to extract as little as possible data from a signal in order to reconstruct it at the receiver. This is where sampling theory begins.

By sampling a signal $f$, i.e., evaluating $f$ at discrete points, we apply a (linear) transformation to convert an analog signal into a digital signal. Ways of converting the discrete signal back to the continuous signal usually go under the name \textit{sampling theorem}. We will study a specific sampling theorem (actually \textit{the} sampling theorem par excellence) later on. There, it will be of importance that the signal is \textit{band-limited}. This simply means that the so-called \textit{spectrum} of the signal has finite support. On the other hand, a signal which only lasts for a finite duration is called \textit{time-limited}. The classical uncertainty principle tells us that it is impossible for a signal to be both, time- and band-limited. This would pose great problems for digital communication, but (luckily) signals can be almost time- and band-limited at the same time.

Throughout the manuscript, we will also assign more physical and engineering terminology to certain mathematical concepts.

\section*{Notation and Recap}\label{sec_Intro}
\addcontentsline{toc}{section}{Notation and Recap}

One of the most essential tools in time-frequency analysis is the Fourier transform of a function $f$. The values $f(x)$ represent the temporal behavior of a signal and its Fourier transform $\widehat{f}(\omega)$ describes the amplitudes of the occurring frequencies.

The variables $x$ and $\omega$ will usually be in $\Rd$. For the special case $d = 1$, $x$ is a point in time and $\omega$ is a specific frequency. More generally, $x$ can be chosen from a locally compact Abelian group $G$ and $\omega$ is then a character from the dual group $\widehat{G}$. In the case of $\Rd$, the dual group $\widehat{\Rd}$ is isomorphic to $\Rd$ and we will not distinguish between them. Another prominent example of a locally compact Abelian group is the $d$-dimensional torus $\T^d$ which has dual group $\Z^d$ (and vice versa).

We start with introducing some notation and a collection of essential results.

\subsection*{Notation}
\addcontentsline{toc}{subsection}{Notation}
Vectors in $\Rd$ are always considered to be column vectors;
\begin{equation}
	x =
	\begin{pmatrix}
		x_1\\
		\vdots\\
		x_d
	\end{pmatrix}
\end{equation}
However, for convenience reasons we will usually write
\begin{equation}
	x = (x_1, \ \ldots, \ x_d).
\end{equation}
Another notation we will use is the following
\begin{equation}
	x^2 = x \cdot x = x^T x.
\end{equation}
Hence, $\cdot$ denotes the standard Euclidean inner product on $\Rd$ and $x^T$ is, of course, the transpose of $x$. The Euclidean norm is denoted by $|.|$ and given by
\begin{equation}
	|x| = \sqrt{x \cdot x} \ .
\end{equation}
If $M \in \R^{d \times d}$ is a square matrix, we will also use the notation
\begin{equation}
	M x^2 = x \cdot Mx
	\quad \text{ and } \quad
	M^{-T} = (M^{-1})^T = (M^T)^{-1}.
\end{equation}
In the latter case, $M$ needs to be invertible (of course).

The complex vector $z \in \C^d$ consists of a real and imaginary part denoted by
\begin{equation}
	z = \Re(z) + i \ \Im(z).
\end{equation}
We will use the same notation in $\C^d$ as in $\Rd$;
\begin{equation}
	z^2 = z \cdot z = z^T z, \quad z \in \C^d.
\end{equation}
However, note that this is not the inner product on $\C^d$. The norm of an element in $\C^d$ is given by
\begin{equation}
	|z|^2 = \overline{z} \cdot z = z^*z, \quad z^* = \overline{z}^T.
\end{equation}

The integral
\begin{equation}
	\int_{\Rd} f(x) \, dx  = \int_\R \dots \int_\R f(x_1, \ \ldots, \, x_d) \, dx_1 \dots dx_d
\end{equation}
is the Lebesgue integral of $f$ on $\Rd$ and the measure of a measurable set $\Omega \subset \Rd$ is given by
\begin{equation}
	|\Omega| = \int_{\Rd} \indicator_\Omega(x) \, dx,
\end{equation}
where $\indicator_\Omega$ is the indicator function on the set $\Omega$. For a number $1 \leq p < \infty$, the $L^p$-norm, or simply $p$-norm, of a function $f$ is denoted by
\begin{equation}
	\norm{f}_p = \left(\int_{\Rd} |f(x)|^p dx\right)^{1/p}.
\end{equation}

The Banach space of all functions with finite $p$-norm is denoted by $L^p ( \Rd )$. Of course, a function is then only defined up to an equivalence class and $f$ is actually a representative of the class. The usual adjustment for the case $p = \infty$ is
\begin{equation}
	\norm{f}_\infty = \esssup_{x \in \Rd} |f(x)| .
\end{equation}
At this point we re-call Hölder's inequality. For $f \in L^p(\Rd)$ and $g \in L^q(\Rd)$, with $\frac{1}{p} + \frac{1}{q} = 1$, we have
\begin{equation}\label{eq_Holder}
	\norm{f \, g}_1 \leq \norm{f}_p \norm{g}_q.
\end{equation}
For $p = 2$, we actually get a Hilbert space with inner product
\begin{equation}
	\langle f, g \rangle = \int_{\Rd} f(x) \, \overline{g(x)} \, dx, \qquad \langle f, f \rangle = \norm{f}_2^2. 
\end{equation}
In particular, Hölder's inequality for the case $p=q=2$ now implies the Cauchy-Schwarz inequality by using the triangle inequality.
\begin{equation}
	|\langle f, g \rangle| = |\int_{\Rd} f(t) \overline{g(t)} \, dt | \leq \int_{\Rd} |f(t) g(t)| \, dt = \norm{f g}_1 \leq \norm{f}_2 \norm{g}_2.
\end{equation}
A signal $f \in \Lt$ is said to have finite energy, as $\norm{f}_2^2$ is considered as the energy of a signal. Thus, the Hilbert space $\Lt$ is also called the space of finite-energy signals.

\subsection*{The Fourier Transform}
\addcontentsline{toc}{subsection}{The Fourier Transform}
The Fourier transform is usually first defined for functions in $L^1 ( \Rd ) \cap L^2 ( \Rd )$ and then extended to all of $\Lt$ by a density argument. We will skip the technical details here and define the Fourier transform right away for functions in $\Lt$. We stress the fact that there are several different normalizations for the Fourier transform and that the choice of normalization affects all further normalizations which we need to make in the sequel.
\begin{definition}
	The Fourier transform of $f \in \Lt$ is given by
	\begin{equation}
		\mathcal{F} f (\omega) =  \widehat{f}(\omega) = \int_{\Rd} f(x) \, e^{-2 \pi i \omega \cdot x} \, dx.
	\end{equation}
\end{definition}
In engineering terms, the function $\widehat{f}$ is called the amplitude spectrum of the signal $f$. The notation $\mathcal{F} f$ reflects the fact that the Fourier transform is a linear operator acting on a function space, whereas the notation $\widehat{f}$ refers to the fact that the Fourier transform of an $\Lt$ function is again a function (in $\Lt$).

A signal $f$ is completely described by its amplitude spectrum $\widehat{f}$ (and vice versa) and can be written as a continuous superposition of its spectral values;
\begin{equation}
	f(x) = \int_{\Rd} \widehat{f}(\omega) e^{2 \pi i x \cdot \omega} \, d\omega
\end{equation}
We identify the above formula as the inverse Fourier transform $\F^{-1}$. For the formula to hold point-wise, $\widehat{f}$ needs to be integrable and $f$ needs to be continuous and integrable (continuity is necessary for the point-wise statement and integrability to assure that we can define $\widehat{f}$). The inversion formula holds on $\Lt$ (a.e.) by using the inversion formula on the dense subspace $L^1(\Rd) \cap \Lt$ and a density argument.

For an integrable function $f \in L^1(\Rd)$, we note the fact
\begin{equation}\label{eq_F_infty_1}
	\norm{\widehat{f}}_\infty \leq \norm{f}_1 ,
\end{equation}
which is refined in the following lemma.
\begin{lemma}[Riemann-Lebesgue]
	If $f \in L^1( \Rd )$, then $\widehat{f}$ is uniformly continuous and
	\begin{equation}
		\lim_{|\omega| \to \infty} |\widehat{f}(\omega)| \to 0.
	\end{equation}
\end{lemma}

The next result shows that the Fourier transform is actually a unitary operator on $\Lt$.
\begin{theorem}[Plancherel]\label{thm_Plancherel}
	For $f \in \Lt$ we have
	\begin{equation}
		\norm{f}_2 = \norm{\widehat{f}}_2 .
	\end{equation}
\end{theorem}

More generally, we have Parseval's formula.
\begin{theorem}[Parseval]\label{thm_Parseval}
	For $f, \, g \in \Lt$, we have
	\begin{equation}\label{eq_Parseval}
		\langle f, g \rangle = \langle \widehat{f}, \widehat{g} \rangle.
	\end{equation}
\end{theorem}

As a next result, we will give the extension of equation \eqref{eq_F_infty_1}.
\begin{theorem}[Hausdorff-Young]
	Let $1 \leq p \leq 2$ and let $q$ be such that $\frac{1}{p} + \frac{1}{q} = 1$. Then
	\begin{equation}
		\mathcal{F} \colon  L^p ( \Rd ) \to L^q ( \Rd )
		\qquad
		\textnormal{ and }
		\qquad
		\norm{\widehat{f}}_q \leq \norm{f}_p .
	\end{equation}
\end{theorem}

\begin{definition}
	The convolution of two functions $f$ and $g$ is defined as
	\begin{equation}
		f * g \, (x) = \int_{\Rd} f(y) g(x-y) \, dy
	\end{equation}
\end{definition}
%\begin{equation}
%	\F(f*g) = \F f \, \F g \ .
%\end{equation}
We recall that convolution is taken to point-wise multiplication under the Fourier transform. 
%UE
\begin{align}
	\F(f*g)(\omega) & = \int_{\Rd} \left( \int_{\Rd} f(y)g(x-y) \, dy \right) e^{-2 \pi i \omega \cdot x} \, dx\\
	& = \int_{\Rd} \left( \int_{\Rd} f(y)g(x-y) \, dy \right) e^{-2 \pi i \omega \cdot y} e^{-2 \pi i \omega \cdot (x-y)} \, dx\\
	& \hspace*{-.25cm}\stackrel{\textnormal{Fubini}}{=} \int_{\Rd} f(y) e^{-2 \pi i \omega \cdot y} \left( \int_{\Rd} g(x-y) e^{-2 \pi i \omega \cdot (x-y)} \, dx \right) \, dy\\
	& = (\F f \, \F g)(\omega).\label{eq_conv_FT}
\end{align}
Also, we will use Young's convolution inequality.
\begin{theorem}[Young's Inequality]\label{thm_Young_conv}
	Let $1 \leq p,q \leq r \leq \infty$ with $\frac{1}{p} + \frac{1}{q} = \frac{1}{r} + 1$ and suppose $f \in L^p(\Rd)$ and $g \in L^q(\Rd)$. Then
	\begin{equation}
		\norm{f * g}_r \leq \norm{f}_p \norm{g}_q .
	\end{equation}
\end{theorem}
In particular, if $f, \ g\in L^1(\Rd)$, then the convolution satisfies
\begin{equation}\label{eq_conv_L1}
	\norm{f*g}_1 \leq \norm{f}_1 \norm{g}_1 \ ,
\end{equation}
We remark that \eqref{eq_conv_L1} shows that $L^1(\Rd)$ is a Banach algebra under convolution and, by employing the Lemma of Riemann-Lebesgue, that the Fourier transform maps $L^1(\Rd)$ into a (dense) subalgebra of $C_0(\Rd)$.

The following theorem and its proof are taken from \cite[App.~A]{Fol89} and states that the Fourier transform of a Gaussian is another Gaussian.
\begin{theorem}[Fourier Transform of a Gaussian]\label{thm_FT_Gauss}
	Let $M$ be a $d \times d$ matrix with real entries and the properties
	\begin{equation}
		M^T = M
		\quad \text{ and } \quad
		M > 0,
	\end{equation}
	where the second assertion means that $M$ is positive definite. Then, for any $\omega \in \Rd$ we have
	\begin{equation}
		\int_{\Rd} e^{-\pi M x^2} e^{-2 \pi i \omega \cdot x} \, dx = \det(M)^{-1/2} e^{- \pi M^{-1} \omega^2}
	\end{equation}
\end{theorem}
\begin{proof}
	The proof proceeds in 3 steps. First, we prove the case for $d = 1$, then for the case that $M$ is a diagonal matrix and, lastly, the general case stated in the theorem.
	
	\bigskip
	\noindent
	\textit{Step 1}. Let $d = 1$ so $M$ is a real, positive scalar. Let
	\begin{equation}
		I(\omega) = \int_{\R} e^{- \pi M x^2} e^{-2 \pi i \omega x} \, dx.
	\end{equation}
	Then, we are allowed to differentiate under the integral and we have
	\begin{align}
		I'(\omega) & = \int_{\R} (-2 \pi i x) e^{-\pi M x^2} e^{-2 \pi i \omega x} \, dx = \frac{i}{M} \int_{\R} \dfrac{d}{dx} (e^{-\pi M x^2}) e^{-2 \pi i \omega x} \, dx\\
		& \stackrel{\text{Int.~by parts}}{=} - \frac{i}{M} \int_{\R} e^{- \pi M x^2} \dfrac{d}{dx} (e^{-2 \pi i \omega x}) \, dx = -\frac{2 \pi \omega}{M} I(\omega).
	\end{align}
	Therefore, we have that $\frac{I'(\omega)}{I(\omega)} = - \frac{2 \pi \omega}{M}$, or $\frac{d}{d \omega} \log(I(\omega)) = - \frac{2 \pi \omega}{M}$. Therefore,
	\begin{equation}
		\log(I(\omega)) = -\frac{\pi \omega^2}{M} + C,
	\end{equation}
	or, equivalently
	\begin{equation}
		I(\omega) = e^{- \frac{\pi \omega^2}{M} + C} = C_1 \, e^{-\frac{\pi \omega^2}{M}}.
	\end{equation}
	Now,
	\begin{equation}
		C_1 = I(0) = \int_{\R} e^{-\pi M x^2} \, dx = M^{-1/2},
	\end{equation}
	which shows the result for the first case.
	
	\bigskip
	\noindent
	\textit{Step 2}. Suppose $M$ is a $d \times d$ diagonal matrix. Then the integral under consideration factors;
	\begin{equation}
		\int_{\Rd} e^{-\pi M x^2} e^{-2 \pi i \omega \cdot x} \, dx = \left(\int_{\R} e^{-\pi M_{1,1} x_1^2} e^{-2 \pi i \omega_1 x_1} \, dx_1\right) \ldots \left(\int_{\R} e^{-\pi M_{d,d} x_d^2} e^{-2 \pi i \omega_d x_d}, \, dx_d\right).
	\end{equation}
	Hence, the result follows by employing the result from the first step.
	
	\bigskip
	\noindent
	\textit{Step 3}. Now, suppose $M$ is not a diagonal matrix. Since $M^T = M$ there is a rotation $R$ such that $R^T M R = D$ is diagonal. Setting $x = R y$ and using the fact that $R^T = R^{-1}$, by step 2 we obtain
	\begin{align}
		\int_{\Rd} e^{- \pi M x^2} e^{-2 \pi i \omega \cdot x} \, dx & = \underbrace{|\det(R)|}_{=1}\int_{\Rd} e^{- \pi (Ry) \cdot M (Ry)} e^{- 2 \pi i \omega \cdot R y} \, dy = \int_{\Rd} e^{- \pi D y^2} e^{-2 \pi i (R^T \omega) \cdot y} \, dy\\
		& = \det(D)^{-1/2} e^{-\pi  (R^{-1} \omega) \cdot D^{-1} (R^{-1} \omega)} = \det(M)^{-1/2} e^{-\pi M^{-1} \omega^2}.
	\end{align}
\end{proof}

Another prominent function in this manuscript is the indicator function of the interval $[-\frac{1}{2}, \frac{1}{2}]$. Its Fourier transform is the $\sinc$ function, also called cardinal sine.
\begin{align}
	\F \indicator_{[-\frac{1}{2}, \frac{1}{2}]} (\omega) & = \int_\R \indicator_{[-\frac{1}{2}, \frac{1}{2}]}(x) e^{-2 \pi i \omega x} \, dx = \int_{-1/2}^{1/2} e^{-2 \pi i \omega x} \, dx\\
	& = \frac{e^{-2 \pi i \omega x}}{-2 \pi i \omega} \Big|_{x=-1/2}^{1/2} = - \frac{e^{-\pi i \omega} - e^{\pi i \omega}}{2 i \pi \omega} = -\frac{\sin(-\pi \omega)}{\pi \omega} = \frac{\sin(\pi \omega)}{\pi \omega} = \sinc(\omega).
\end{align}
We note that the value at 0 equals 1 (by using l'Hospital's rule) and that the $\sinc$-function is not Lebesgue-integrable. However, it is a finite energy signal (since $\indicator_{[-\frac{1}{2}, \frac{1}{2}]} \in \Lt[]$). Thus, $\sinc$ possesses a Fourier transform as an $\Lt[]$-function. Since it is an even function, its Fourier transform and its inverse Fourier transform \footnote{Note that the Fourier transform of the $\sinc$-function cannot be computed by real integration methods, but that it can be computed using complex methods (using the homotopy version of Cauchy's integral theorem and the residue theorem). The advantage of the complex method is that the singularity at $x=0$ (which is a removable singularity) can be avoided by choosing a ``roundabout" contour near 0.} coincide;
\begin{equation}
	\F \sinc(\omega) = \F^{-1} \sinc(\omega) = \F^{-1} (\F \indicator_{[-\frac{1}{2}, \frac{1}{2}]})(\omega) = \indicator_{[-\frac{1}{2}, \frac{1}{2}]}(\omega).
\end{equation}

\subsection*{Fourier Series}
\addcontentsline{toc}{subsection}{Fourier Series}
By
\begin{equation}
	e_k(t) = e^{2 \pi i k \cdot t}, \quad t \in \T^d
\end{equation}
we denote the complex exponential with frequency $k \in \Z^d$. We say that $e_k$\footnote{More generally, if $\omega \in \Rd$, we say that $e_\omega(t) = e^{2 \pi i \omega \cdot t}$ is a pure frequency of frequency $\omega$. However, note that, in general, $e_\omega$ need not be periodic with period 1.} is a pure (integer) frequency. Recall that the set of pure integer frequencies, i.e.,
\begin{equation}
	\{e^{2 \pi i k \cdot t} \mid k \in \Z^d\}
\end{equation}
is an orthonormal basis for $L^2(\T^d)$ or $L^2([0,1]^d)$. It is easy to see that the system is orthonormal by a simple calculation (which we perform only for $d = 1$ for notational convenience, the $d$-dimensional case follows by iterated integration.);
\begin{equation}
	\langle e_k, e_l \rangle_{L^2(\T)} = \int_\T e^{2 \pi i k t} e^{-2 \pi i l t} \, dt = \int_\T e^{2 \pi i (k-l) t} \, dt =
	\begin{cases}
		1, & k = l\\
		0, & \text{ else}
	\end{cases}
\end{equation}
By the Weierstrass approximation theorem, the trigonometric polynomials $\sum_{k=-N}^N c_k e_k(t)$ are dense in $C(\T^d)$ and, therefore, in $L^2(\T^d)$. Therefore, the set of pure integer frequencies is an orthonormal basis for $L^2(\T^d)$. As a consequence, a function $f \in L^2(\T^d)$ can be expanded as
\begin{equation}\label{eq_Fourier_expansion}
	f(t) = \sum_{k \in \Z^d} c_k e^{2 \pi i k \cdot t},
\end{equation}
with coefficients
\begin{equation}
	c_k = \langle f, e_k \rangle_{L^2(\T^d)}.
\end{equation}
Plancherel's theorem holds in this case as well;
\begin{equation}
	\norm{(c_k)}_{\ell^2(\Z^d)} = \norm{f}_{L^2(\T^d)}.
\end{equation}
Also, for $f \in L^1(\T^d) \supset L^2(\T^d)$\footnote{Note that we do not have such an inclusion for $L^p(\Rd)$-spaces because the Lebesgue measure on $\Rd$ is not finite.} the Lemma of Riemann-Lebesgue holds;
\begin{equation}
	\langle f, e_k \rangle_{L^2(\T^d)} \to 0, \quad |k| \to \infty.
\end{equation}
Furthermore, using the multi-index convention $|k|=k_1+ \ldots +k_d$ now, if we denote the space of trigonometric polynomials of order not larger than $N$ by
\begin{equation}
	T_N = \{ p_N(t) = \sum_{|k|=-N}^N c_k e^{2 \pi i k \cdot t} \mid t \in \T^d, \, c_k \in \C, \, N \in \N\},
\end{equation}
then the truncated Fourier series is the best approximating element in $T_N$ for any $f \in L^2(\T^d)$;
\begin{equation}
	\norm{f - \sum_{|k|=-N}^N \langle f, e_k \rangle e_k}_{L^2(\T^d)} \leq \norm{f - p_N}_{L^2(\T^d)}, \quad \forall p_N \in T_N.
\end{equation}

Equation \eqref{eq_Fourier_expansion} admits the following interpretation. A periodic signal $f \in L^2(\T^d)$ can be decomposed into pure waves of integer frequencies with amplitudes $(c_k)_{k \in \Z^d}$. The amplitudes $(c_k)_{k \in \Z^d}$ are derived from taking linear measurements of the signal $f$ with respect to the orthonormal system $\{e_k\}$.

A goal of time-frequency analysis, following the ideas in \cite{Gab46}, is the decomposition of non-periodic signals into ``simple building blocks" similar to \eqref{eq_Fourier_expansion}, such that the occurring coefficients can be interpreted as the amplitudes of the building blocks. In particular, the coefficients should contain some information on how ``strong" a certain frequency is at a certain point in time.

\newpage
\section{Basic Concepts in Time-Frequency Analysis}\label{sec_TFA}
We consider functions from the Hilbert space $\Lt$ and define the following operators;
\begin{equation}
	T_x f(t) = f(t-x)
	\qquad \textnormal{ and } \qquad
	M_\omega f(t) = f(t) \, e^{2 \pi i \omega \cdot t}.
\end{equation}
These operators (illustrated in Figure \ref{fig_TF-shifts}) are usually called translation operator and modulation operator, respectively. However, in the field of time-frequency analysis we say $T_x$ is a time-shift by $x$ and $M_\omega$ is called a frequency-shift by $\omega$. These operators are unitary on $\Lt$, i.e.,
\begin{equation}
	\langle T_x f, T_x g \rangle = \langle f, g \rangle
	\qquad \textnormal{ and } \qquad
	\langle M_\omega f, M_\omega g \rangle = \langle f, g \rangle.
\end{equation}
\begin{figure}[ht]
	\subfigure[A Gauss function.]{
	\includegraphics[width=.45\textwidth]{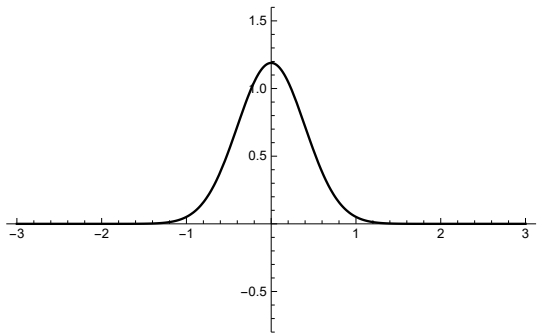}
	}
	\hfill
	\subfigure[A time-shifted Gauss function.]{
	\includegraphics[width=.45\textwidth]{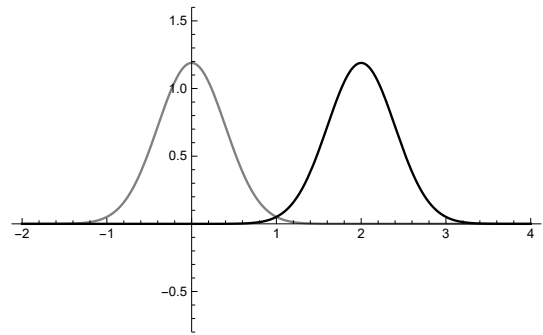}
	}
	
	\bigskip

	\subfigure[A frequency shifted Gauss function (real part) with its Gaussian envelop.]{
	\includegraphics[width=.45\textwidth]{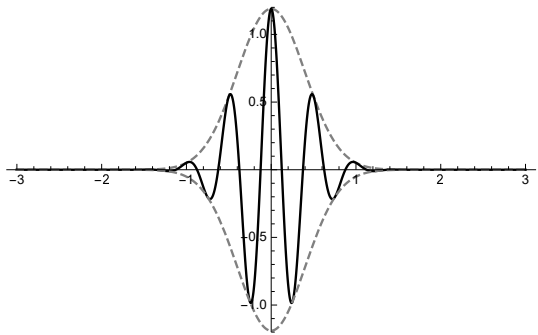}
	}
	\hfill
	\subfigure[A time-frequency shifted Gauss function (real part) with its Gaussian envelop.]{
	\includegraphics[width=.45\textwidth]{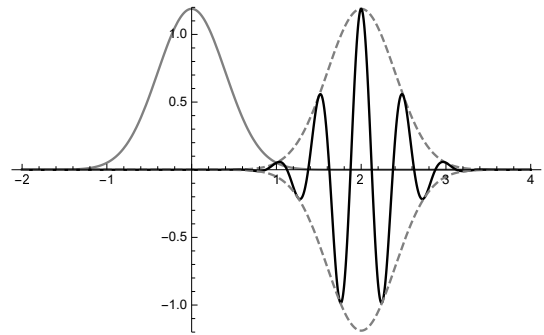}
	}
	\caption{\footnotesize{The standard Gauss function $g_0(t) = 2^{1/4}e^{-\pi t^2}$ and its time-, frequency- and time-frequency-shifted copies (real part) $T_2 g_0$, $M_2 g_0$, $M_2 T_2 g_0$, respectively.}\label{fig_TF-shifts}}
\end{figure}
In general, they do not commute. We have the canonical commutation relation
\begin{equation}\label{eq_comm_rel}
	M_\omega T_x = e^{2 \pi i \omega \cdot x} \, T_x M_\omega,
\end{equation}
which is at the heart of time-frequency analysis. Equation \eqref{eq_comm_rel} shows that time-shifts and frequency-shifts do not commute in general. However, they commute up to a phase factor. A phase factor is a complex number of modulus 1 ($|c| = 1$, $c \in \C$). The commutation relation for time- and frequency-shifts \eqref{eq_comm_rel} will accompany us throughout the course. Sometimes the non-commutativity may simply be ignored (e.g., because we are interested in absolute values), other times the bookkeeping of the exponential factors might be annoying but important and again other times the non-commutativity will be crucial to obtain (beautiful) results right at the heart of time-frequency analysis.

Equation \eqref{eq_comm_rel} follows from a direct computation;
\begin{align}
	M_\omega T_x f(t) = f(t-x) \, e^{2 \pi i \omega \cdot t}
	= f(t-x) \, e^{2 \pi i \omega \cdot (t-x)} e^{2 \pi i \omega \cdot x}
	= e^{2 \pi i \omega \cdot x} \ T_x M_\omega f(t).
\end{align}
It is immediate from \eqref{eq_comm_rel} that time-shifts and frequency shifts commute if and only if $\omega \cdot x \in \Z$.

Furthermore, for any $1 \leq p \leq \infty$, time- and frequency-shifts (and hence time-frequency shifts) are isometries on $L^p(\Rd)$, i.e.,
\begin{equation}
	\norm{M_\omega T_x f}_p = \norm{f}_p \ .
\end{equation}

%UE
Next, we study the behavior of $T_x$ and $M_\omega$ under the Fourier transform. By a direct calculation, we will see that
\begin{equation}\label{eq_TFshifts_FT}
	\F T_x = M_{-x} \F
	\qquad \textnormal{ and } \qquad
	\F M_\omega = T_\omega \F.
\end{equation}
%UE
We start with the first assertion.
\begin{align}
	\mathcal{F}(T_x f)(\xi) = \int_{\Rd} f(t-x) \, e^{-2 \pi i \xi \cdot t} \, dt
	= \int_{\Rd} f(t) \, e^{-2 \pi i \xi \cdot (t+x)} \, dt
	= M_{-x} (\mathcal{F} f)(\xi).
\end{align}
The second formula follows analogously.
\begin{align}
	\F(M_\omega f)(\xi) = \int_{\Rd} f(t) \, e^{2 \pi i \omega \cdot t} e^{-2 \pi i \xi \cdot t} \, dt
	= \int_{\Rd} f(t) \, e^{-2 \pi i (\xi - \omega) \cdot t} \, dt
	= T_\omega (\F f)(\xi).
\end{align}
This shows that the use of the word frequency-shift is justified for the operator $M_\omega$. 

By combining \eqref{eq_comm_rel} and \eqref{eq_TFshifts_FT} we get
\begin{equation}
	\F M_\omega T_x = T_\omega M_{-x} \F = e^{2 \pi i \omega \cdot x} M_{-x} T_\omega \F
\end{equation}
The composition of $T_x$ and $M_\omega$ is called a time-frequency shift, which we denote by
\begin{equation}
	\pi(\l) = M_\omega T_x, \qquad \l = (x, \omega)  \in \R^{2d} \, .
\end{equation}

In the sequel, we will also make use of other, familiar operators. We start with the isotropic dilation operator, which we denote by $D_a$. It acts on a function in the following way;
\begin{equation}
	D_a f(t) = a^{-d/2} f( a^{-1} t), \qquad a > 0, \ t \in \Rd .
\end{equation}
The factor $a^{-d/2}$ makes the operator unitary. We are also interested in the operator's behavior under the Fourier transform. This will be helpful to get an intuitive understanding for uncertainty principles, introduced later on.
\begin{equation}
	\F D_a = D_{\frac{1}{a}} \F \ .
\end{equation}
%UE
This can easily be seen by the following calculation. First, note that $\Phi_a : \Rd \to \Rd$, $t \mapsto a^{-1} t$ is a linear diffeomorphism on $\Rd$ for $a > 0$. Hence,
\begin{align}
	\F (D_a f)(\omega) & = \int_{\Rd} a^{-d/2} f(a^{-1} t) e^{-2 \pi i \omega \cdot t} \, dt\\
	& = \int_{\Rd} a^{-d/2} f \left(\Phi_a^{-1}(a^{-1} t) \right) \underbrace{\left|\det\left(D\Phi_a^{-1}\right) \right|}_{= a^{d}} e^{-2 \pi i \omega \cdot \Phi_a^{-1}(t)} \, dt\\
	& = \int_{\Rd} a^{d/2} f(t) e^{-2 \pi i a \omega \cdot t} \, dt = D_{\frac{1}{a}} (\F f) (\omega).
\end{align}

We will also encounter a more general (anisotropic) dilation operator, which acts on functions in a similar way.

Furthermore, we note that the Fourier transform takes differentiation to multiplication with a monomial and vice versa. By $\partial_k$, we denote the partial derivative in the $k$-th component. Then
\begin{equation}\label{eq_FT_derivative}
	\F \left(\partial_k^n f \right)(\omega) = (2 \pi i \omega_k)^n \, \F f(\omega)
	\qquad \textnormal{ and }
	\qquad
	\F \Big( (-2 \pi i t_k)^n f \Big)(\omega) = \partial_k^n (\F f) (\omega).
\end{equation}
%UE
It suffices to carry out the proof for $n = 1$. First, assume that $f \in \mathcal{S}(\Rd)$ (the Schwartz space of rapidly decreasing functions). Then, by using integration by parts, we see that
\begin{align}
	\F \left(\partial_k f (t) \right)(\omega) & = \int_{\Rd} \partial_k f(t) e^{-2 \pi i \omega \cdot t} \, dt = \underbrace{f(t) e^{-2 \pi i \omega \cdot t} \Big|_{t = - \infty}^{\infty}}_{= 0} - \int_{\Rd} f(t) (-2 \pi i \omega_k) e^{-2 \pi i \omega \cdot t} \, dt\\
	& = (2 \pi i \omega_k) \F f(\omega).
\end{align}
The result for higher order derivatives follows by induction and the result for $f \in \Lt$ follows by a density argument. For the second assertion we use the differential quotient and denote the unit vector in the $k$-th direction by $\mathbf{e}_k$.
\begin{align}
	\lim_{\delta \to 0} \frac{\widehat{f}(\omega + \delta \mathbf{e}_k) - \widehat{f}(\omega)}{\delta}
	& = \int_{\Rd} f(t) e^{-2 \pi i \omega \cdot t} \lim_{\delta \to 0} \frac{e^{-2 \pi i \delta \mathbf{e}_k \cdot t} - 1}{\delta} \, dt\\
	& = \int_{\Rd} f(t) e^{-2 \pi i \omega \cdot t} (-2 \pi i t_k) \, dt.
\end{align}
The exchange of integration and taking the limit is justified by the dominated convergence theorem and for taking the limit we used l'Hospital's rule. Again, the result for $n > 1$ follows by induction and the result for $\Lt$ by a density argument.

\section{The Short-Time Fourier Transform}\label{sec_STFT}
The short-time Fourier transform is the fundamental tool in time-frequency analysis. In a distributional sense, it can be seen as a generalization of the Fourier transform, as we will see later. It aims to overcome the drawback that $f(t)$ only gives temporal information of a signal $f$ while under the Fourier transform we obtain only information about its frequency distribution $\widehat{f}(\omega)$. We will start with an illustrative example of a multi-component signal.
\begin{example}
	Let $f$ be the following multi-component signal.
	\begin{align}
		f(t) = 
		\begin{cases}
			e^{4 \pi i t} + e^{8 \pi i t}, & -3 \leq t \leq -1\\
			e^{2 \pi i t}, & -1 \leq t \leq 1\\
			e^{2 \pi i t^2}, & 1 \leq t \leq 4\\
			0, & \text{else}
		\end{cases}.
	\end{align}
	\begin{figure}[ht]
		\subfigure[Real part of the multi-component signal.]
		{
			\includegraphics[width=.45\textwidth]{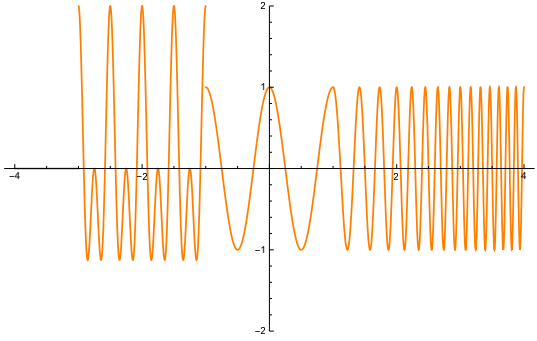}
		}
		\hfill
		\subfigure[Imaginary part of the multi-component signal.]
		{
			\includegraphics[width=.45\textwidth]{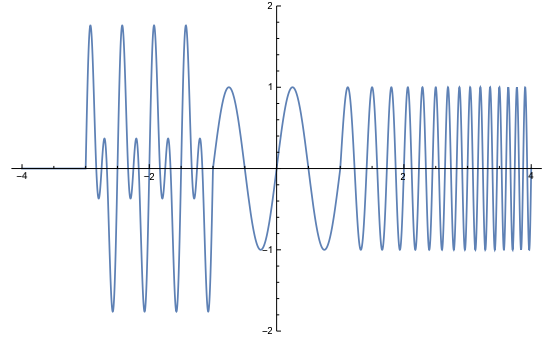}
		}
		\caption{\footnotesize{A multi-component signal consisting of a superposition of two pure frequencies with 2 and 4 Hz, followed by a pure frequency of 1 Hz and a linear chirp, i.e., the frequency increases linearly.}}
	\end{figure}
	\begin{figure}[ht]
		\centering
		\subfigure[Absolute value of the positive frequencies of $\widehat{f}(\omega)$.]{
			\includegraphics[width=.45\textwidth]{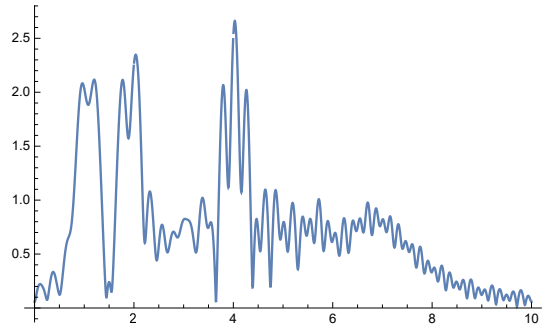}
		}
		\caption{\footnotesize{The positive part of the Fourier transform of the multi-component signal. There are peeks around the major occurring frequencies at 1,2 and 4 Hz. Also, the contributions of the chirp are apparent throughout the frequency band $[2,8]$ Hz.}}
	\end{figure}
	
	In order to obtain local information of the signal $f$, we can use a window function $g$. We will use the standard box function $b_0$, which is the generator of all $b$-splines, and the standard Gaussian window $g_0$.
	\begin{align}
		b_0(t) =
		\begin{cases}
			1, & -\frac{1}{2} \leq t \leq \frac{1}{2}\\
			0, & \text{else}
		\end{cases}
		\qquad
		\text{ and }
		\qquad
		g_0(t) = 2^{1/4} e^{- \pi t^2}.
	\end{align}
	We note that both windows are normalized $\norm{b_0}_2 = 1$ and $\norm{g_0}_2 = 1$.
	\begin{figure}[ht]
		\subfigure[Real part of the signal localized around 0 with the box function $b_0$.]
		{
			\includegraphics[width=.45\textwidth]{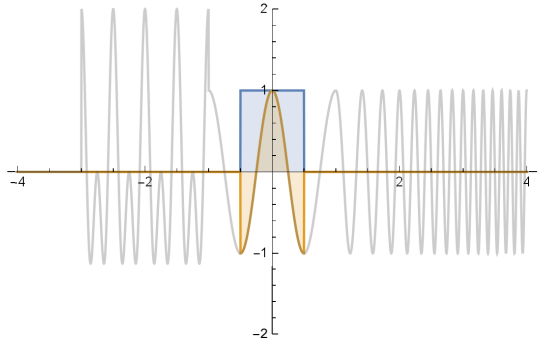}
		}
		\hfill
		\subfigure[Real part of the signal localized around 0 with the standard Gaussian $g_0$.]
		{
			\includegraphics[width=.45\textwidth]{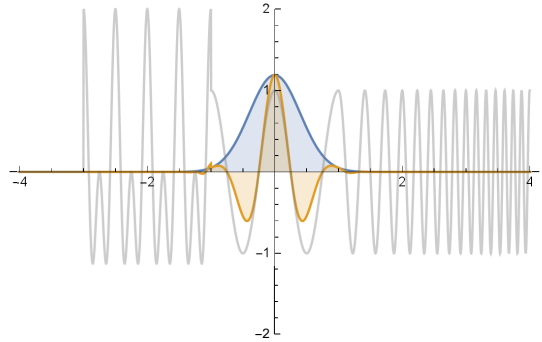}
		}
		\caption{\footnotesize{Localization of the multi-component signal with the box function and the standard Gaussian.}}\label{fig_signal_window}
	\end{figure}
	\begin{figure}[ht]
		\subfigure[Absolute value of the Fourier transform of the localized signal using $b_0$.]
		{
			\includegraphics[width=.45\textwidth]{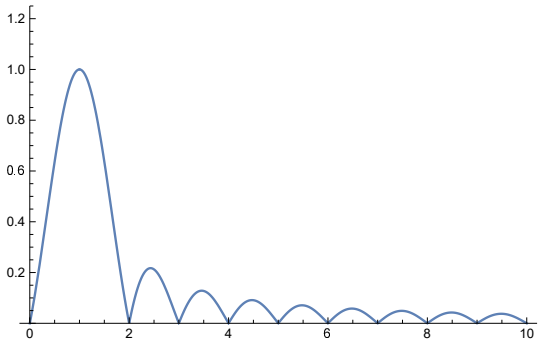}
		}
		\hfill
		\subfigure[Absolute value of the Fourier transform of the localized signal using $g_0$.]
		{
			\includegraphics[width=.45\textwidth]{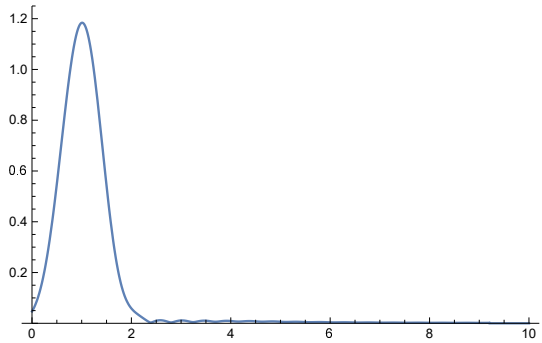}
		}
		\caption{\footnotesize{The Fourier transform of the windowed signals reveals a local behavior as well.}}\label{fig_FT}
	\end{figure}
	However, a joint time-frequency distribution should yield a picture in the time-frequency plane which reveals the local behavior of the signal. We illustrate such a representation with 2 spectrograms. The precise definition of a spectrogram will be given in Section \ref{sec_spec}. As can already be seen from Figure \ref{fig_signal_window} and Figure \ref{fig_FT}, the choice of the window crucially affects the localization properties of the signal. This fact will also become apparent in Figure \ref{fig_spec}.
	\begin{figure}
		\subfigure[A spectrogram using the window $b_0$.]
		{
			\includegraphics[width=.45\textwidth]{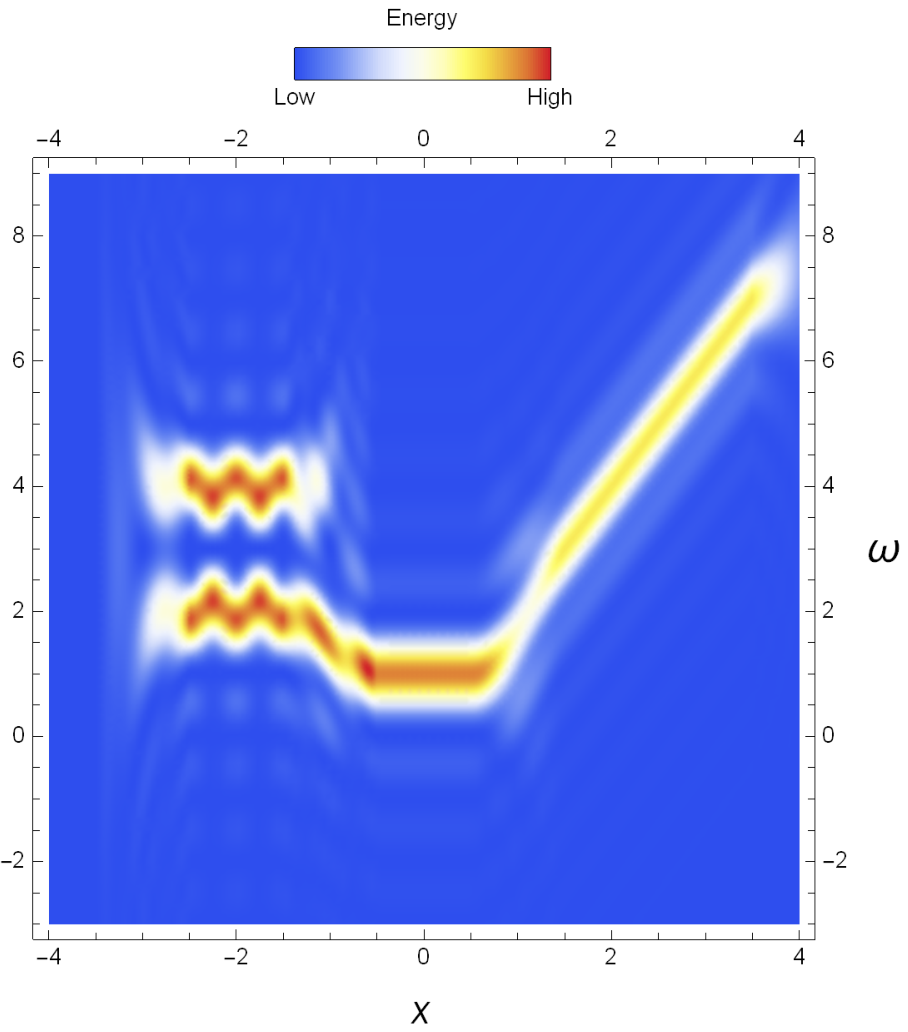}
		}
		\hfill
		\subfigure[A spectrogram using the window $g_0$.]
		{
			\includegraphics[width=.45\textwidth]{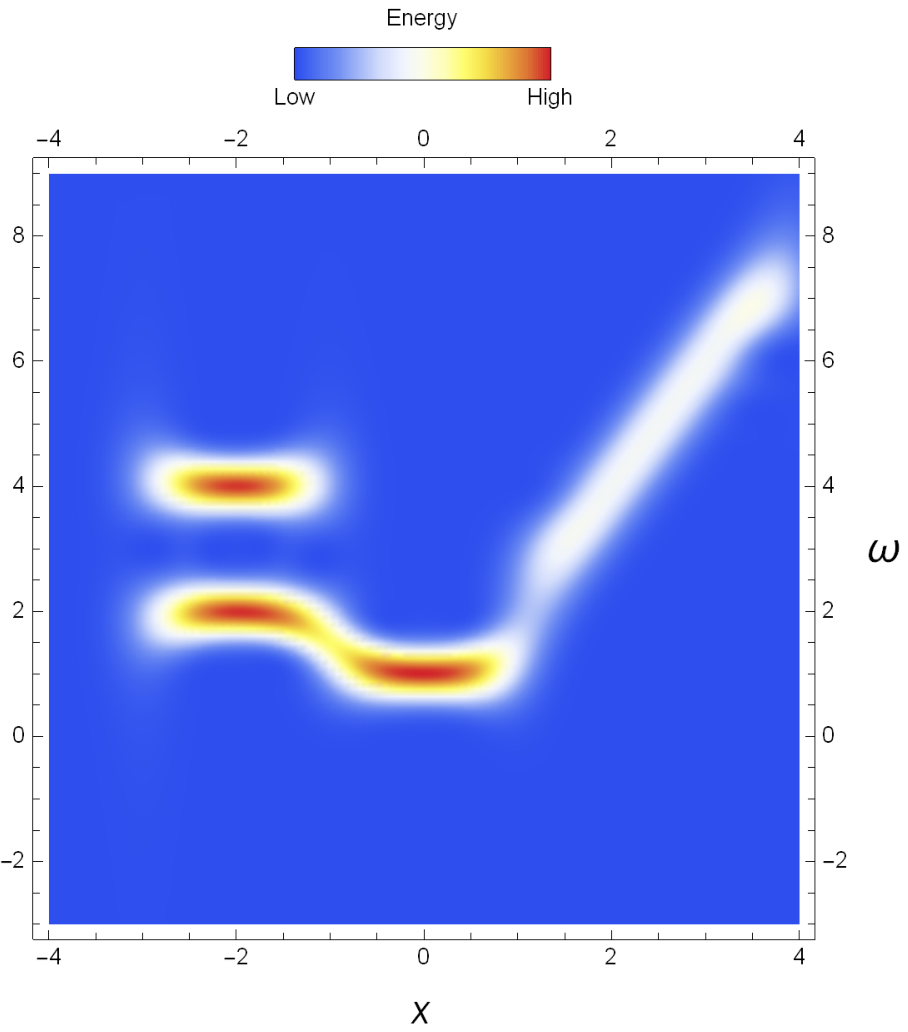}
		}
		\caption{\footnotesize{The spectrograms of the multi-component signal with different localization windows. The $x$-direction indicates the position to which the window is shifted and the $\omega$-direction shows the values of the energy distribution of the local Fourier transfroms.}}\label{fig_spec}
	\end{figure}
	We note that the use of the box function $b_0$ does not yield a satisfying localization of the signal in the Fourier domain, as smoothness of a function affects the decay of its Fourier transform. Therefore, the localized Fourier transform with window $b_0$ decays only like $\frac{1}{\omega}$. This can also be seen in Figure \ref{fig_spec} (a).
	\flushright{$\diamond$}
\end{example}

After this motivating example, we come to the definition of the short-time Fourier transform (STFT) and discuss some of its properties. We mainly follow the textbook by Gröchenig \cite[Chap.~3]{Gro01}.
\begin{definition}
	For a (non-zero) function $g$, called a window function, the short-time Fourier transform of a function $f$ with respect to the window $g$ is given by
	\begin{equation}
		V_g f(x, \omega) = \int_{\Rd} f(t) \overline{g(t-x)} e^{-2 \pi i \omega \cdot t} \, dt = \langle f, \pi(\l) g \rangle, \qquad \l = (x, \omega) \in \R^{2d}.
	\end{equation}
\end{definition}

For the moment, we only give this formal definition without further specification of $f$ and $g$. The STFT is also called sliding-window Fourier transform or voice transform. The notation $V_g f$ comes from the latter name. The name sliding window Fourier transform comes from the interpretation of $V_gf(x,.)$ ($x$ fixed) as a local Fourier transform of $f$ around the point $x$. As $x$ varies, the window $g$ slides along the $x$-axis to all possible positions. Written as an inner product, the STFT is often interpreted as a linear measurement, taken at the point $\l \in \R^{2d}$.
\begin{example}\label{ex_STFT_g0}
	We will compute the STFT of the standard Gaussian
	\begin{equation}
		g_0(t) = 2^{d/4} e^{- \pi t^2},
	\end{equation}
	so $\norm{g_0}_2 = 1$, with itself as a window.
	\begin{align}
		V_{g_0} g_0 (x, \omega) & = 2^{d/2} \int_{\Rd} e^{-\pi t^2} e^{-\pi (t-x)^2} e^{-2 \pi i \omega \cdot t} \, dt\\
		& = 2^{d/2} \int_{\Rd} e^{-\pi (t+\frac{x}{2})^2} e^{-\pi (t-\frac{x}{2})^2} e^{-2 \pi i \omega \cdot (t+\frac{x}{2})} \, dt\\
		& = 2^{d/2} e^{-\pi i \omega \cdot x} \int_{\Rd} e^{-2 \pi (t^2 + \frac{x^2}{4})} e^{-2 \pi i \omega \cdot t} \, dt\\
		& = e^{-\pi i \omega \cdot x} e^{- \frac{\pi}{2} x^2} 2^{d/2} \int_{\Rd} e^{-2 \pi t^2} e^{-2 \pi i \omega \cdot t} \, dt\\
		& = e^{-\pi i \omega \cdot x} e^{-\frac{\pi}{2} x^2} \underbrace{\int_{\Rd} e^{-\pi t^2} e^{-2 \pi i \frac{\omega}{2^{1/2}} \cdot t} \, dt}_{\F g_0\left(\frac{\omega}{2^{1/2}}\right)}\\
		& = e^{-\pi i \omega \cdot x} e^{-\frac{\pi}{2} (x^2 + \omega^2)}.
	\end{align}
	Thus, the STFT of a Gaussian is again a Gaussian (in the time-frequency plane with a precisely determined phase factor).
	\flushright{$\diamond$}
\end{example}

We will now study the basic properties of the STFT. Besides writing the STFT as an inner product of $f$ and $M_\omega T_x g$, we can also write it in various other ways, which can be useful at times.
\begin{align}\label{eq_STFT_notation}
	V_g f(x, \omega) & = \F(f \, T_x \overline{g})(\omega)\\
	& = \langle \widehat{f}, T_\omega M_{-x} \widehat{g} \rangle\\
	& = e^{-2 \pi i x \cdot \omega} \F (\widehat{f} \, T_\omega \overline{\widehat{g}})(-x)\\
	& = e^{-2 \pi i x \cdot \omega} (f * M_\omega \overline{g}^\vee)(x)\\
	& = (\widehat{f} * M_{-x} \widehat{\overline{g}}^\vee)(\omega),
\end{align}
where $g^\vee(x) = g(-x)$ is the flip operation.

So far, we have not discussed when the STFT is well-defined. Clearly, all of the above formulas make sense for $\Lt$ functions because if $f,g \in \Lt$, then the product $f \, T_x \overline{g} \in L^1(\Rd)$ by the Cauchy-Schwarz inequality and $V_gf(x,\omega) = \F (f \, T_x \overline{g})(\omega)$ is defined point-wise. For most of the time, we will be happy with the assumption that $f,g \in \Lt$ (because then we may use Parseval's theorem and Plancherel's theorem). However, by Hölder's inequality \eqref{eq_Holder} $f \, T_x \overline{g} \in L^1(\Rd)$ whenever $f \in L^p(\Rd)$ and $g \in L^q(\Rd)$ with $\frac{1}{p} + \frac{1}{q} = 1$, so the STFT is again defined point-wise in this case.
%The notation $V_g f(x, \omega) = \langle f , M_\omega T_x g\rangle$ is also useful for extending the STFT to situations when the integral is no longer defined.

The following result can partially be seen as a ``Riemann-Lebesgue-like" statement for the STFT.
\begin{lemma}
	Let $f,g \in \Lt$, then $V_g f$ is uniformly continuous on $\R^{2d}$.
\end{lemma}
\begin{proof}
	This follows mainly from the continuity of the operator groups $\{T_x\}$ and $\{M_\omega\}$, i.e.,
	\begin{equation}
		\lim_{x \to 0} \norm{T_x f - f}_2 = 0
	\end{equation}
	and
	\begin{equation}
		\lim_{\omega \to 0} \norm{M_\omega f - f}_2 = \lim_{\omega \to 0} \norm{T_\omega \widehat{f} - \widehat{f}}_2 = 0.
	\end{equation}
	The statement follows by combining the two results.
\end{proof}
Using the Lemma  of Riemann-Lebesgue, we also conclude that $\lim_{|\omega| \to \infty} V_gf(x, \omega) = 0$ for any $x$ (because $f \, T_x \overline{g} \in L^1(\Rd)$ for $f,g \in \Lt$).

The following property of the STFT shows a fundamental concept in time-frequency analysis. The usual interpretation is that, up to a phase factor, the Fourier transform rotates the time-frequency plane (by $90^\circ$ in the case $d = 1$). Also, it shows that if we localize a signal $f$ with a window $g$, we localize its spectrum $\widehat{f}$ with the Fourier transform $\widehat{g}$ of the original window.
\begin{proposition}[Fundamental Identity of Time-Frequency Analysis]\label{pro_fitf}
	For $f,g \in \Lt$, the following holds;
	\begin{equation}\label{eq_fitf}
		V_g f(x, \omega) = e^{-2 \pi i x \cdot \omega} V_{\widehat{g}} \widehat{f} (\omega, -x).
	\end{equation}
\end{proposition}
\begin{proof}
%	To be done in the exercise session.
	%UE
	The proof is a straight-forward computation.
	\begin{align}
		V_g f(x, \omega) & = \langle f, M_\omega T_x g \rangle
		= \langle f, e^{2 \pi i \omega \cdot x} T_x M_\omega g \rangle\\
		& = e^{-2 \pi i \omega \cdot x} \langle f, T_x M_\omega g \rangle
		= e^{-2 \pi i \omega \cdot x} \langle \F f, \F T_x M_\omega g \rangle\\
		& = e^{-2 \pi i \omega \cdot x} \langle \F f, M_{-x} T_\omega \F g \rangle
		= e^{-2 \pi i \omega \cdot x} V_{\widehat{g}} \widehat{f} (\omega, -x).
	\end{align}
\end{proof}
The identity is sometimes also referred to as the \textit{basic identity of time-frequency analysis}. Also, we conclude that $\lim_{|x| \to \infty} V_g f(x, \omega) = 0$ for any $\omega$ because $|V_g f(x,\omega)| = |V_{\widehat{g}} \widehat{f} (\omega, -x)|$ and we may apply the Lemma of Riemann-Lebesgue to $\widehat{f} \, T_\omega \widehat{g}$, for $\widehat{f}, \widehat{g} \in \Lt$.

The next formula is known as the covariance principle. Its interpretation is that a time-frequency shift by $(\xi, \eta) \in \R^{2d}$ simply translates the STFT in the time-frequency plane (again, up to a phase factor).
\begin{proposition}[Covariance Principle]\label{pro_covar}
	For $f,g \in \Lt$, we have
	\begin{equation}\label{eq_covar}
		V_g (M_\eta T_\xi  f)(x, \omega) = e^{-2 \pi i \xi \cdot (\omega - \eta)} V_g f(x - \xi, \omega - \eta), \qquad x, \xi, \omega, \eta \in \Rd.
	\end{equation}
\end{proposition}
\begin{proof}
%	To be done in the exercise session.
%	UE
	The proof is a straight-forward computation.
	\begin{align}
		V_g(M_\eta T_\xi  f)(x, \omega) & = \langle M_\eta T_\xi f, M_\omega T_x g \rangle\\
		& = \langle f, T_{-\xi} M_{-\eta} M_\omega T_x g \rangle\\
		& = \langle f, e^{2 \pi i \xi \cdot (-\eta + \omega)} M_{\omega-\eta} T_{x-\xi} g \rangle\\
		& = e^{-2 \pi i \xi \cdot (\omega - \eta)} V_g f(x-\xi,\omega-\eta)
	\end{align}
\end{proof}
We remark that the above proof can be adapted and that the result stays true whenever $V_gf$ is defined. The only $\Lt$ specific property we used this time was the inner product $\langle . , . \rangle$, so whenever this bracket is well-defined (e.g., by duality) the proof holds (e.g., for $\mathcal{S}(\Rd)$ and $\mathcal{S}'(\Rd)$ or $L^p(\Rd)$ and $L^q(\Rd)$ with $\frac{1}{p} + \frac{1}{q} = 1$).

The STFT also enjoys a property which is quite similar to Parseval's identity \eqref{eq_Parseval}.
\begin{theorem}[Orthogonality Relations]\label{thm_ortho}
	Let $f_1, f_2, g_1, g_2 \in \Lt$, then $V_{g_k} f_k \in \Lt[2d]$ for $k \in \{1,2\}$ and
	\begin{equation}\label{eq_OR}
		\langle V_{g_1} f_1, V_{g_2} f_2 \rangle_{\Lt[2d]} = \langle f_1, f_2 \rangle \overline{\langle g_1, g_2 \rangle}.
	\end{equation}
\end{theorem}
\begin{proof}
	For technical reasons, we first assume that $g_k \in L^1(\Rd) \cap L^\infty(\Rd) \subset \Lt$ (dense). Therefore, $f_k \, T_x g_k \in \Lt$ for all $x \in \Rd$ and Parseval's formula \eqref{eq_Parseval} applies to the integral in $\omega$.
	\begin{align}
		\langle V_{g_1} f_1, V_{g_2} f_2 \rangle_{\Lt[2d]} & = \int_{\Rd} \int_{\Rd} V_{g_1} f_1 (x, \omega) \overline{V_{g_2} f_2 (x, \omega)} \, d\omega \, dx\\
		& = \int_{\Rd} \left( \int_{\Rd} \F( f_1 \, T_x \overline{g_1})(\omega) \, \overline{\F(f_2 \, T_x \overline{g_2})(\omega)} \, d\omega \right) \, dx\\
		& = \int_{\Rd} \left( \int_{\Rd} f_1(t) \overline{f_2(t)} \; \overline{g_1(t-x)} g_2(t-x) \, dt \right) \, dx \, .
	\end{align}
	As $f_1, f_2 \in L^2(\Rd, dt)$, the product $f_1 \, \overline{f_2} \in L^1(\Rd, dt)$ and, by assumption, the product $\overline{g_1} g_2 \in L^1(\Rd, dx)$. We may therefore exchange the order of integration. We have
	\begin{align}
		\langle V_{g_1} f_1, V_{g_2} f_2 \rangle_{\Lt[2d]} & = \int_{\Rd} f_1(t) \overline{f_2(t)} \left( \int_{\Rd} \overline{g_1(t-x)} g_2(t-x) \, dx \right) \, dt\\
		& = \langle f_1, f_2 \rangle \overline{\langle g_1, g_2 \rangle}.
	\end{align}
	By a density argument the result extends to $g_k \in \Lt$, $k=1,2$.
\end{proof}
A shorter proof can be obtained by introducing the following unitary operators for functions of $2d$ variables and the tensor product of two functions;
\begin{equation}
	\mathcal{T}_a F(x,t) = F(t,t-x)
	\quad \text{ and } \quad
	\F_2 F(x,\omega) = \int_{\Rd} F(x,t) e^{-2 \pi i t \cdot \omega} \, dt.
\end{equation}
The operators $\mathcal{T}_a$ is an asymmetric coordinate change and the operator $\F_2$ is the partial Fourier transform in the second ($d$) variable(s). The tensor product of two functions is denoted by
\begin{equation}
	(f \otimes g) (x,\omega) = f(x) g(\omega)
\end{equation}

\noindent
\textit{Second Proof of the Orthogonality Relations}. If $f,g \in \Lt$, then
\begin{align}
	\langle V_{g_1} {f_1}, V_{g_2} f_2 \rangle_{\Lt[2d]} & = \langle \F_2 \mathcal{T}_a(f_1 \otimes \overline{g_1}), \F_2 \mathcal{T}_a(f_2 \otimes \overline{g_2}) \rangle_{\Lt[2d]}\\
	& = \langle f_1 \otimes \overline{g_1}, f_2 \otimes \overline{g_2} \rangle_{\Lt[2d]}\\
	& = \langle f_1, f_2\rangle \overline{\langle g_1, g_2 \rangle}.
\end{align}
\begin{flushright}
	$\square$
\end{flushright}

As a corollary, we obtain a Plancherel-like result for the STFT.
\begin{corollary}\label{cor_Vgf_L2}
	For $f,g \in \Lt$, we have
	\begin{equation}
		\norm{V_g f}_2 = \norm{f}_2 \norm{g}_2 \, .
	\end{equation}
	In particular, if $\norm{g}_2 = 1$ then
	\begin{equation}\label{eq_norm_f_Vgf}
		\norm{f}_2 = \norm{V_g f}_2, \qquad \forall f \in \Lt .
	\end{equation}
	Hence, in this case the STFT is an isometry from $\Lt$ into $\Lt[2d]$.
\end{corollary}

We note that \eqref{eq_norm_f_Vgf} tells us that $f$ is completely determined by $V_g f$. Furthermore, we have the following implication
\begin{equation}
	\langle f , M_\omega T_x g \rangle = 0, \, \forall (x, \omega) \in \R^{2d}
	\qquad  \Longrightarrow \qquad
	f = 0.
\end{equation}
This means that the set of time-frequency shifts of the window $g$ is complete in $\Lt$ and, hence, spans a dense subspace of $\Lt$.

\begin{remark}
	In particular, the last result tells us that
	\begin{equation}
		X = \text{span } \{ M_\omega T_x g_0 \mid (x, \omega) \in \R^{2d} \}
	\end{equation}
	is dense in $\Lt$, where $g_0(t) = 2^{d/4} e^{- \pi t^2}$ is the $d$-dimensional standard Gaussian. We used Plancherel's theorem in order to prove this result (for general $g \in \Lt$), but there are other proofs as well (e.g., using Fourier series \cite[Chap.~1.5]{Gro01}). Now, we will prove that Plancherel's theorem is also a consequence of the fact that $X$ is dense in $\Lt$.
	
	For $f \in X$ we have
	\begin{equation}
		f(t) = \sum_{k=1}^n c_k M_{\omega_k} T_{x_k} g_0 .
	\end{equation}
	When computing $\norm{f}_2 = \langle f, f \rangle$, we need to compute inner products of time-frequency shifted Gaussians. We recall that
	\begin{align}
		\langle g_0, M_\omega T_x g_0 \rangle = e^{-\pi i x \cdot \omega} e^{-\frac{\pi}{2} (x^2 + \omega^2)}
	\end{align}
	We compute
	\begin{align}
		\langle M_\eta T_\xi g_0 , M_\omega T_x g_0 \rangle
		& = \langle g_0 , T_{-\xi} M_{\omega - \eta} T_x g_0 \rangle \\
		& = \langle g_0, e^{2 \pi i \xi \cdot (\omega-\eta)} M_{\omega-\eta} T_{x-\xi} g_0 \rangle \\
		& = e^{-2 \pi i \xi \cdot (\omega-\eta)} e^{-\pi i (x-\xi) \cdot (\omega - \eta)} e^{-\frac{\pi}{2} ((x-\xi)^2+(\omega-\eta)^2)} \\
		& = e^{-\pi i (x+\xi) \cdot (\omega - \eta)} e^{-\frac{\pi}{2} ((x-\xi)^2+(\omega-\eta)^2)}
	\end{align}
	Now, we perform the same computation for the Fourier transforms of the time-frequency shifted Gaussians.
	\begin{align}
		\langle \F (M_\eta T_\xi g_0), \F (M_\omega T_x g_0) \rangle
		& = \langle e^{2 \pi i \eta \cdot \xi} M_{-\xi} T_\eta \F g_0, e^{2 \pi i \omega \cdot x} M_{-x} T_{\omega} \F g_0 \rangle \\
		& = e^{2 \pi i \eta \cdot \xi} e^{-2 \pi i \omega \cdot x} \langle g_0, T_{-\eta} M_{\xi-x} T_\omega g_0 \rangle \\
		& = e^{2 \pi i \eta \cdot \xi} e^{-2 \pi i \omega \cdot x} \langle g_0, e^{2 \pi i \eta \cdot(\xi-x)} M_{\xi-x} T_{\omega-\eta} g_0 \rangle \\
		& = e^{2 \pi i \eta \cdot \xi} e^{-2 \pi i \omega \cdot x} e^{-2 \pi i \eta \cdot(\xi-x)} e^{-\pi i (\xi-x) \cdot (\omega - \eta)} e^{-\frac{\pi}{2}((\xi-x)^2+(\omega-\eta)^2)} \\
		& = e^{-\pi i (\omega - \eta) \cdot(x + \xi)} e^{-\frac{\pi}{2} ((x-\xi)^2 + (\omega-\eta)^2)} .
	\end{align}
	In the last step we used the fact that $(x-\xi)^2 = (\xi-x)^2$. Hence, these computations show that
	\begin{equation}
		\langle M_\eta T_\xi g_0, M_\omega T_x g_0 \rangle = \langle \F( M_\eta T_\xi g_0), \F (M_\omega T_x g_0) \rangle
	\end{equation}
	Therefore
	\begin{align}
		\langle f, f \rangle & = \langle \sum_{k=1}^n c_k M_{\omega_k} T_{x_k} g_0, \sum_{l=1}^n c_l M_{\omega_l} T_{x_l} g_0 \rangle \\
		& = \langle \sum_{k=1}^n c_k \F (M_{\omega_k} T_{x_k} g_0), \sum_{l=1}^n c_l \F (M_{\omega_l} T_{x_l} g_0) \rangle \\
		& = \langle \widehat{f}, \widehat{f} \rangle ,
	\end{align}
	by linearity of the inner product. Summing up, this shows that Plancherel's theorem is equivalent to the completeness of the set of time-frequency shifted standard Gaussians.
	\flushright{$\diamond$}
\end{remark}

After knowing that the information of $f$ is completely contained in $V_gf$, the question that pops up is: How can we recover $f$ from $V_g f$?

Before we can answer this question, we need a quick excursion in vector-valued integrals. In time-frequency analysis, superpositions of time-frequency shifts, such as
\begin{equation}
	f = \iint_{\R^{2d}} F(x, \omega) (M_\omega T_x g) \, d(x, \omega) .
\end{equation}
are of utmost importance. Note that $g$ is time-frequency shifted by $(x,\omega)$, but it also depends on its argument. Hence, the integral gives us back a vector (i.e., a function) in the Hilbert space $\Lt$, which depends on the same argument as $g$. Therefore, we may also speak of a function-valued integral. More generally, integrals may also be matrix- or operator-valued. Later on, we will use a weak formulation for such integrals. Furthermore, consider the following example. If $F \in \Lt[2d]$, then the conjugate-linear functional
\begin{equation}
	\ell(h) = \iint_{\R^{2d}} F(x,\omega) \overline{\langle h, M_\omega T_x g \rangle} \, d(x,\omega)
\end{equation}
is a bounded functional on $\Lt$. This can be seen by applying the Cauchy-Schwarz inequality and the Plancherel-like result from Corollary \ref{cor_Vgf_L2};
\begin{equation}
	|\ell(h)| \leq \norm{F}_{\Lt[2d]} \norm{V_g h}_2 = \norm{F}_{\Lt[2d]} \norm{g}_2 \norm{h}_2.
\end{equation}
This means that $\ell$ defines a unique function $f = \iint_{\R^{2d}} F(x,\omega) M_\omega T_x g \, d(x,\omega) \in \Lt$ with norm $\norm{f}_2 \leq \norm{F}_{\Lt[2d]} \norm{g}_2$ and satisfying $\ell(h) = \langle f, h \rangle$.

With this notion at hands, we will be able to prove the inversion formula for the STFT.
\begin{theorem}[Inversion formula for the STFT]
	Let $g \in \Lt$ and choose $\widetilde{g} \in \Lt$ such that $\langle \widetilde{g}, g \rangle \neq 0$. Then we have for all $f \in \Lt$
	\begin{equation}\label{eq_inv_STFT}
		f = \frac{1}{\langle \widetilde{g}, g \rangle} \iint_{\R^{2d}} V_g f(x, \omega) M_\omega T_x \widetilde{g} \, d(x, \omega).
	\end{equation}
\end{theorem}
\begin{proof}
	By Corollary \ref{cor_Vgf_L2} $V_g f \in \Lt[2d]$. Therefore, the following vector-valued integral is a well defined function $\widetilde{f} \in \Lt$;
	\begin{equation}
		\widetilde{f} = \frac{1}{\langle \widetilde{g}, g \rangle} \iint_{\R^{2d}} V_g f(x, \omega) M_\omega T_x \widetilde{g} \, d(x, \omega).
	\end{equation}
	By using the orthogonality relations, we see that
	\begin{align}
		\langle \widetilde{f}, h \rangle & =
		\frac{1}{\langle \widetilde{g}, g \rangle} \iint_{\R^{2d}} V_g f(x, \omega) \underbrace{\langle M_\omega T_x \widetilde{g}, h \rangle}_{\overline{V_{\widetilde{g}} h}} \, d(x, \omega)\\
		& = \frac{1}{\langle \widetilde{g}, g \rangle} \langle V_g f, V_{\widetilde{g}} h \rangle_{\Lt[2d]} = \frac{1}{\langle \widetilde{g}, g \rangle} \langle f, h \rangle \overline{\langle g, \widetilde{g} \rangle}\\
		& = \langle f, h \rangle.
	\end{align}
	This formula holds for any $h \in \Lt$ and we conclude that $\widetilde{f} = f$.
\end{proof}
Note that, as $g \neq 0$, we may actually choose $\widetilde{g} = g$ in the inversion formula. Hence, in this case the inversion formula reads
\begin{equation}
	f = \frac{1}{\norm{g}_2^2} \iint_{\R^{2d}} V_g f(x, \omega) M_\omega T_x g \, d(x, \omega) = \frac{1}{\norm{g}_2^2} \iint_{\R^{2d}} V_g f(\l) \, \pi(\l) g \, d\l, \qquad
	 \l = (x, \omega) \in \R^{2d}.
\end{equation}
Of course, this formula simplifies a bit if we assume that $\norm{g}_2 = 1$,
\begin{equation}
	f = \iint_{\R^{2d}} V_g f(\l) \, \pi(\l)g \, d\l.
\end{equation}
Hence, a function $f \in \Lt$ can be written as a continuous superposition of a time-frequency shifted window with weights obtained from its STFT. In this sense, \eqref{eq_inv_STFT} is similar to the inversion formula of the Fourier transform. The main difference is that the elementary building blocks, i.e., the complex exponential $e^{2 \pi i x \cdot \omega}$, are not in $\Lt$ whereas the elementary functions $M_\omega T_x \widetilde{g}$ are usually picked to be particularly nice $\Lt$-functions.

The time-frequency analysis of a signal now usually consists of three steps.

\noindent
\textit{Analysis.} Given a signal $f$, its STFT $V_g f$ is computed. We will interpret it as a joint time-frequency distribution of $f$. The window $g$ crucially influences the analysis and one usually demands that $g$ and $\widehat{g}$ decay sufficiently fast.

\noindent
\textit{Processing.} The STFT $V_g f(x,\omega)$ is then transformed into a new function $F(x,\omega)$. Some typical processing steps are truncation of $V_g f$ to a region where something interesting seems to happen or where $|V_g f|$ is above a given threshold. In an engineering language such processing steps are called feature extraction, signal segmentation and signal compression, depending on the purpose. Furthermore, the STFT is also used as a pre-processing tool in the area of machine learning.

\noindent
\textit{Synthesis.} The processed signal is then reconstructed by using the modified inversion formula
\begin{equation}
	\widetilde{f} = \frac{1}{\langle \widetilde{g}, g \rangle} \iint_{\R^{2d}} F(x,\omega) M_\omega T_x \widetilde{g} \, d(x,\omega).
\end{equation}
Note that signal $\widetilde{f}$ is reconstructed from $F(x,\omega)$ and not from $V_g f(x,\omega)$. Hence, $\widetilde{f}$ may naturally differ from $f$. Also we remark again that distinct windows might be used for the analysis and the synthesis.

It is also customary to write the inversion formula as a superposition of rank-one operators. Consider a Hilbert space $\mathcal{H}$ and let $u \otimes \overline{v}$ denote the rank-one operator defined by
\begin{equation}
	(u \otimes \overline{v})(h) = \langle h, v \rangle_\mathcal{H} u, \quad u,v,h \in \mathcal{H}.
\end{equation}
Then, the inversion formula \eqref{eq_inv_STFT} is the following continuous resolution of the identity operator
\begin{equation}
	I_{\Lt} = \frac{1}{\langle \widetilde{g}, g \rangle} \iint_{\R^{2d}} M_\omega T_x \widetilde g \otimes \overline{M_\omega T_x g} \, d(x, \omega).
\end{equation}

\section{Quadratic Time-Frequency Representations}
Again, for most of this chapter we will follow the textbook by Gröchenig \cite{Gro01}.

Quadratic representations are often interpreted as joint (probability) densities of a function $f \in \Lt$ and its Fourier transform $\widehat{f} \in \Lt$ on $\R^{2d}$. The space $\R^{2d}$ might be called time-frequency plane or phase space, depending on the specific situation.

Such joint time-frequency representations were investigated by E.~Wigner \cite{Wig32} in the context of quantum mechanics with the goal of finding a joint probability distribution for the position and momentum variables. Now, quadratic representations are also popular in engineering and we will call them (quadratic) time-frequency representations. Mathematically, we are looking for a sesquilinear form $G(f,g)(x,\omega)$; that is $G$ is linear in the first argument $f$ and conjugate linear in the second argument $g$. Then, there are two ways to make $G$ ``quadratic" in $f$. We either consider $Cf=|G(f,g)|^2$ with $g$ fixed or $Cf = G(f,f)$. In both cases the quadratic form satisfies
\begin{equation}\label{eq_quadratic_TFrep}
	C(c_1 f + c_2 h) = |c_1|^2 \, C f + |c_2|^2 \, C h + c_1 \overline{c_2} \, G(f,h) + \overline{c_1} c_2 \, G(h,f), \quad c_1, c_2 \in \C.
\end{equation}
The non-linearity in \eqref{eq_quadratic_TFrep} causes problems because the superposition of two signals $f$ and $h$ introduces the cross-terms $G(f,h)$ and $G(h,f)$. These are often hard to separate from the main components of interest $Cf$ and $Ch$. Also, the analysis and interpretation of these cross-terms is difficult.

On the positive side, a quadratic time-frequency representation of the form $G(f,f)$ does not depend on a window $g$ and should display the time-frequency content of $f$ in a pure, unobstructed form.

\subsection{The Spectrogram}\label{sec_spec}
In this section we are going to briefly discuss the spectrogram and some of its properties. We have seen examples already at the beginning of Section \ref{sec_STFT}.
\begin{definition}
	Let $g \in \Lt$ be a window with the property $\norm{g}_2 = 1$. Then, the spectrogram of $f$ with respect to the window $g$ is given by
	\begin{equation}
		\spec_g f(x, \omega) = |V_g f(x, \omega)|^2.
	\end{equation}
\end{definition}
As already apparent from its definition, the spectrogram $\spec_g f$ inherits properties from the STFT. In particular, it is covariant and preserves the energy and, furthermore, is non-negative (by definition);
\begin{itemize}
	\item $\spec_g f(x, \omega) \geq 0, \qquad \forall (x, \omega) \in \R^{2d}$,
	\item $\spec_g \left(M_\eta T_\xi f \right)(x, \omega) = \spec_g f(x-\xi, \omega-\eta)$ by Proposition \ref{pro_covar}, which we also called the Covariance Principle,
	\item $\norm{\spec_g f}_{L^1(\R^{2d})} = \norm{f}_{\Lt}^2$ by Corollary \ref{cor_Vgf_L2}.
\end{itemize}
The last point demonstrates the fact that, up to normalization of $f$, the spectrogram may serve as a probability density on $\R^{2d}$ of the joint time-frequency content of a signal $f$.

\subsection{The Rihaczek Distribution}
We will now study the Rihaczek distribution as described in \cite{Gro03_FeiStr}, which is a rather simple time-frequency representation. It is more or less simply the tensor product of $f$ and its Fourier transform $\widehat{f}$.
\begin{equation}\label{eq_Riha}
	R f(x, \omega) = f(x) \overline{\widehat{f}(\omega)} e^{-2 \pi i x \cdot \omega}.
\end{equation}

The Rihaczek distribution is intimately connected to the STFT by the Fourier transform (on $\R^{2d}$)
\begin{equation}\label{eq_Riha_Vff}
	R f(x, \omega) = \widehat{V_f f}(\omega, -x).
\end{equation}
This is confirmed by a direct computation.
\begin{align}
	\widehat{V_f f}(\omega, -x) & = \int_{\Rd} \left( \int_{\Rd} V_f f(\xi, \eta) e^{-2 \pi i (\xi \cdot \omega - x \cdot \eta)} \, d\eta \right) \, d \xi\\
	& = \int_{\Rd} \left( \int_{\Rd} \left(\int_{\Rd} f(t) \overline{f(t-\xi)} e^{-2 \pi i \eta \cdot t} \, dt \right) e^{-2 \pi i (\xi \cdot \omega - x \cdot \eta)} \, d\eta \right) \, d \xi\\
	& = \int_{\Rd} \underbrace{\left( \int_{\Rd} \left(\F \left(f \, T_\xi \overline{f}\right) (\eta) \right) \,  e^{2 \pi i x \cdot \eta} \, d\eta \right)}_{= \F^{-1} \left( \F \left(f \, T_\xi \overline{f}\right) \right)(x) } e^{- 2\pi i \xi \cdot \omega} \, d \xi\\
	& = \int_{\Rd} f(x) \overline{f(x-\xi)} e^{-2 \pi i \xi \cdot \omega} \, d\xi
	= f(x) \int_{\Rd} \overline{f(\xi)} e^{-2 \pi i (-\xi + x) \cdot \omega} \, d\xi\\
	& = f(x) \overline{\int_{\Rd} f(\xi) e^{2 \pi i (-\xi + x) \cdot \omega} \, d\xi}
	= f(x) \overline{\widehat{f}(\omega)} e^{-2 \pi i x \cdot \omega}\\
	= R f(x, \omega).
\end{align}

We can also state \eqref{eq_Riha_Vff} in a more general way.
\begin{equation}
	R(f,g)(x, \omega) = f(x) \overline{\widehat{g}(\omega)} e^{-2 \pi i x \cdot \omega} = \widehat{V_g f} (\omega, -x).
\end{equation}
We call $R(f,g)$ the cross-Rihaczek distribution of $f$ and $g$. In this context, we should $R f = R(f, f)$ actually call the auto-Rihaczek distribution of $f$. We note that, up to the complex exponential, the STFT factors under the Fourier transform. We will mainly use $Rf$ when presenting the classical uncertainty principle of Heisenberg, Pauli and Weyl.

\subsection{The Ambiguity Function}
Another time-frequency representation, which is up to a complex exponential the same as the STFT, is the cross-ambiguity function of two functions $f$ and $g$.
\begin{definition}\label{def_Afg}
	For $f$ and $g$ in $\Lt$, their cross-ambiguity function is defined as
	\begin{equation}
		A(f,g)(x,\omega) = \int_{\Rd} f(t+\tfrac{x}{2}) \overline{g(t-\tfrac{x}{2})} e^{-2 \pi i \omega \cdot t} \, dt = e^{\pi i x \cdot \omega} V_g f(x, \omega).
	\end{equation}
	For the case $f = g$, we write $A f = A(f,f)$ and call it the (auto-)ambiguity function of $f$.
\end{definition}
By definition, it follows immediately that
\begin{equation}
	\overline{A f(-x,-\omega)} = A f(x, \omega).
\end{equation}
Particularly, $Af(0,0)$ is a real value. Of course, most properties of the STFT carry over to the cross-ambiguity function. In particular the orthogonality relations in Theorem \ref{thm_ortho} imply
\begin{equation}
	\norm{Af}_{\Lt[2d]} = \norm{f}_{\Lt}^2.
\end{equation}

However, the ambiguity function $Af$ is a quadratic representation of a function $f$, whereas the STFT with window $g$, i.e., $V_gf$, is a linear transformation of $f$. Therefore, $f$ can only be recovered up to a phase factor by $Af$. In particular we have $A(c f) = A f$ for any $|c| = 1$. Such a $c$ is called a phase factor. If $f(0) \neq 0$, we have the following inversion formula
\begin{equation}\label{eq_ambi_inv}
	f(x) = \frac{1}{\, \overline{f(0)} \,} \int_{\Rd} Af(x, \omega) e^{\pi i x \cdot \omega} \, d \omega.
\end{equation}
This can be obtained as follows. First, we note that for fixed $x \in \Rd$, $Af(x, .)$ is the Fourier transform of $t \mapsto f(t+\tfrac{x}{2}) \overline{f(t-\tfrac{x}{2})}$. Hence, by the Fourier inversion formula we have
\begin{equation}
	f(t+\tfrac{x}{2})\overline{f(t-\tfrac{x}{2})} = \int_{\Rd} Af(x, \omega)e^{2 \pi i t \cdot \omega} \, d \omega .
\end{equation}
By setting $t = \tfrac{x}{2}$, we obtain the desired result. Furthermore, it is easy to check that for a phase factor $|c| = 1$, equation \eqref{eq_ambi_inv} holds for $c f$ as well.

Of course, it is not quite ``fair" to compare the inversion formula for the auto-ambiguity function with the inversion formula for the STFT. If we introduce an ``auto-STFT", i.e., $V_f f$, then we can also only recover $f$ from $V_f f$ up to a phase factor. The key difference is that for the STFT, we completely know the window function $g$. Therefore, for the cross-ambiguity function, we may completely recover $f$ from $A(f,g)$ by knowing $g$.

The ambiguity function occurs naturally in radar applications and, therefore, is often called the radar ambiguity function. We will briefly discuss a simplified model as described in \cite[Chap.~4.2]{Gro01}. Suppose that the distance and speed of an unidentified flying object (UFO) are to be determined. For this purpose a transmitter/receiver station sends a test signal $f$ towards the UFO. The signal $f$ is reflected and a part of the reflected signal is received as an echo $e$. The test signal might be a ``pulse" of short duration of the form $f(x) = \sigma(x) e^{2 \pi i \omega_0 \cdot x}$., where $\sigma$ is a nice envelope signal of slow variation. This means that $\supp(\widehat{\sigma}) \subset [-W,W]$ with $W$ small compared to the carrier frequency $\omega_0$, so $\supp(\widehat{f}) \subset[\omega_0-W,\omega_0+W]$. Let $r$ be the distance between the transmitter/receiver station and the UFO, $v$ the relative speed between the station and th UFO and $c$ the speed of light. Then, the echo $e$ is received with a time lag $\Delta t = \frac{2r}{c}$. Moreover, each frequency $\omega$ in the band $[\omega_0-W,\omega_0+W]$ of $f$ undergoes a Doppler shift $\Delta \omega = -\frac{2\omega v}{c}$. Since the bandwidth $2W$ is small compared to $\omega_0$, the frequency shifts can be approximated by $\Delta \omega \approx -\frac{2 \omega_0 v}{c}$, independently of the exact form of $f$. We ignore any further distortion of the signal $f$ in this example, therefore the echo $e$ has the form
\begin{equation}
	M_{\Delta\omega}T_{\Delta t} f.
\end{equation}
At the receiver, the echo is then compared to time-frequency shifts of the original signal $f$ by taking the correlation
\begin{equation}
	|\langle e, M_\omega T_x f\rangle| = |V_f f(x-\Delta t, \omega-\Delta \omega)| = |Af(x-\Delta t, \omega-\Delta \omega)|.
\end{equation}
The values of $\Delta t$ and $\Delta \omega$, and hence the distance $r$ and the velocity $v$ of the UFO, can be determined by means of the following lemma.

\begin{lemma}\label{lem_ambi_max}
	Let $f \in \Lt$, $f \neq 0$, then
	\begin{equation}
		|Af(x, \omega)| < Af(0,0) = \norm{f}_2^2,
	\end{equation}
	for all $(x, \omega) \neq (0,0)$.
\end{lemma}
\begin{proof}
	The inequality
	\begin{equation}
		|Af(x, \omega)| = |\langle f, M_\omega T_x f \rangle| \leq \norm{f}_2^2
	\end{equation}
	is just the Cauchy-Schwarz inequality. Equality holds if and only if $M_\omega T_x f = c f$ (i.e., we have linear dependence) for some $(x, \omega) \in \R^{2d}$ and $|c| = 1$. Now, assume $x \neq 0$ and further assume that $T_x |f| = |c M_\omega T_x f| = |c f| = |f|$, then $|f|$ would be periodic with period $x$. However, as $f \in \Lt$, it cannot be periodic unless $f = 0$. If $x = 0$ and $\omega \neq 0$, then, analogously, $T_\omega| \widehat{f}| = |\widehat{f}|$ would be $\omega$-periodic. Therefore, $|Af(x, \omega)|$ is maximal if and only if $(x, \omega) = (0,0)$.
	
	The value of $Af(0,0)$ yields the energy of $f$, i.e.,
	\begin{equation}
		Af(0,0) = \int_{\Rd} f(t) \overline{f(t)} \, dt = \langle f, f \rangle = \norm{f}_2^2.
	\end{equation}
\end{proof}
Now, to find the lag $(\Delta t, \Delta \omega)$, one determines experimentally the values $(x_0,\omega_0)$ where the correlation function $|\langle e, M_\omega T_x f \rangle|$ takes its maximum. According to the lemma above, the position of this maximum is the desired lag $(\Delta t, \Delta \omega)$.

%UE
We give now two examples of auto-ambiguity functions, which are used to measure time-frequency concentrations.
\begin{example}\label{ex_ambi}
	\begin{enumerate}[(a)]
	\item	The ambiguity function of the box function is given by
		\begin{equation}
			A b_0(x, \omega) =
			\begin{cases}
				\frac{\sin(\pi \omega (1-|x|))}{\pi	\omega}, & |x| \leq 1\\
				0, & \text{else}
			\end{cases}
		\end{equation}
		\begin{figure}[ht!]
			\centering
			\includegraphics[width=.45\textwidth]{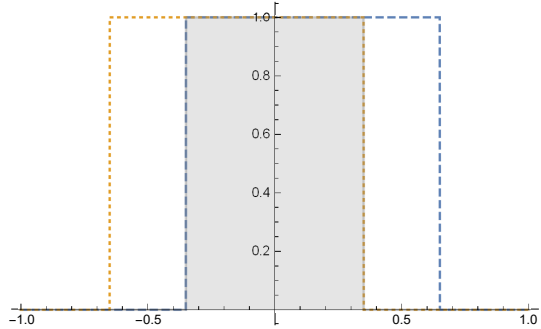}
			\caption{\footnotesize{Illustration of the function $b_0(t+\tfrac{x}{2})b_0(t-\tfrac{x}{2})$ for some $x$ with $|x|\leq 1$.}}
		\end{figure}
		
		\noindent
		We show this by a directly computing
		\begin{equation}
			A b_0 (x, \omega) = \int_\R  b_0(t+\tfrac{x}{2})b_0(t-\tfrac{x}{2}) e^{-2 \pi i \omega \cdot t} \, dt.
		\end{equation}
		We distinct the cases $x \in [0,1]$ and $x \in [-1,0]$. We start with the case $x \in [0,1]$.
		\begin{align}
			\int_\R  b_0(t+\tfrac{x}{2})b_0(t-\tfrac{x}{2}) e^{-2 \pi i \omega t} \, dt
			& = \int_{\tfrac{-1+x}{2}}^{\tfrac{1-x}{2}} e^{-2 \pi i \omega t} \, dt\\
			& = \int_{\tfrac{-1+x}{2}}^{\tfrac{1-x}{2}} \cos(2 \pi \omega t) \, dt
			= \frac{\sin(2 \pi \omega t)}{2 \pi \omega} \Bigg|_{t = \tfrac{-1+x}{2}}^{\tfrac{1-x}{2}}\\
			& = \frac{\sin(2 \pi \omega \frac{1-x}{2})}{2 \pi \omega} - \frac{\sin(2 \pi \omega \frac{-1+x}{2})}{2 \pi \omega}\\
			& = \frac{\sin(\pi \omega (1-x))}{\pi \omega}.
		\end{align}
		The case $x \in [-1,0]$ follows easily by the symmetry $A b_0(x,\omega) = A b_0(-x, \omega)$ and the result follows.
	
		\item Next, we compute the ambiguity function of the 1-dimensional standard Gaussian $g_0(t) = 2^{1/4} e^{- \pi t^2}$, $t \in \R$ (the case for of the $d$-dimensional standard Gaussian follows just as easily.) As an in-between step, we will use the dilation operator $D_a g(t) = a^{-1/2} g(a^{-1} t)$ and its behavior under the Fourier transform and Theorem \ref{thm_FT_Gauss}, which shows that $\F g_0 = g_0$.
		\begin{align}\label{eq_ambi_Gauss}
			A g_0(x, \omega)& = 2^{1/2} \int_\R e^{-\pi(t+\tfrac{x}{2})^2} e^{-\pi(t-\tfrac{x}{2})^2} e^{-2 \pi i \omega \cdot t} \, dt\\
			& = e^{- \pi \tfrac{x^2}{2}} \int_\R \underbrace{2^{1/2} e^{-2 \pi t^2}}_{D_{1/\sqrt{2}} g_0} e^{-2 \pi i \omega \cdot t} \, dt\\
			& = e^{- \pi \tfrac{x^2}{2}} \F \left(D_{1/\sqrt{2}} g_0 \right)(\omega) = e^{- \pi \tfrac{x^2}{2}} D_{\sqrt{2}} \F g_0(\omega) = e^{- \pi \tfrac{x^2}{2}} D_{\sqrt{2}} g_0(\omega)\\
			& = e^{- \tfrac{\pi}{2} (x^2 + \omega^2)}.
		\end{align}
		Note, that we also have the relation $A(f,g)(x,\omega) = e^{\pi i x \cdot \omega} V_g f(x, \omega)$. Hence, using the result from Example \ref{ex_STFT_g0} that $V_{g_0} g_0(x,\omega) = e^{-\pi i x \cdot \omega} e^{-\frac{\pi}{2} (x^2+\omega^2)}$ we would have obtained the result without any further computations.
		
		\begin{figure}[!ht]
			\subfigure[Density-plots of the absolute values of the auto-ambiguity functions of the box function $b_0$ and the standard Gaussian $g_0$.]
			{
				\includegraphics[width=.425\textwidth]{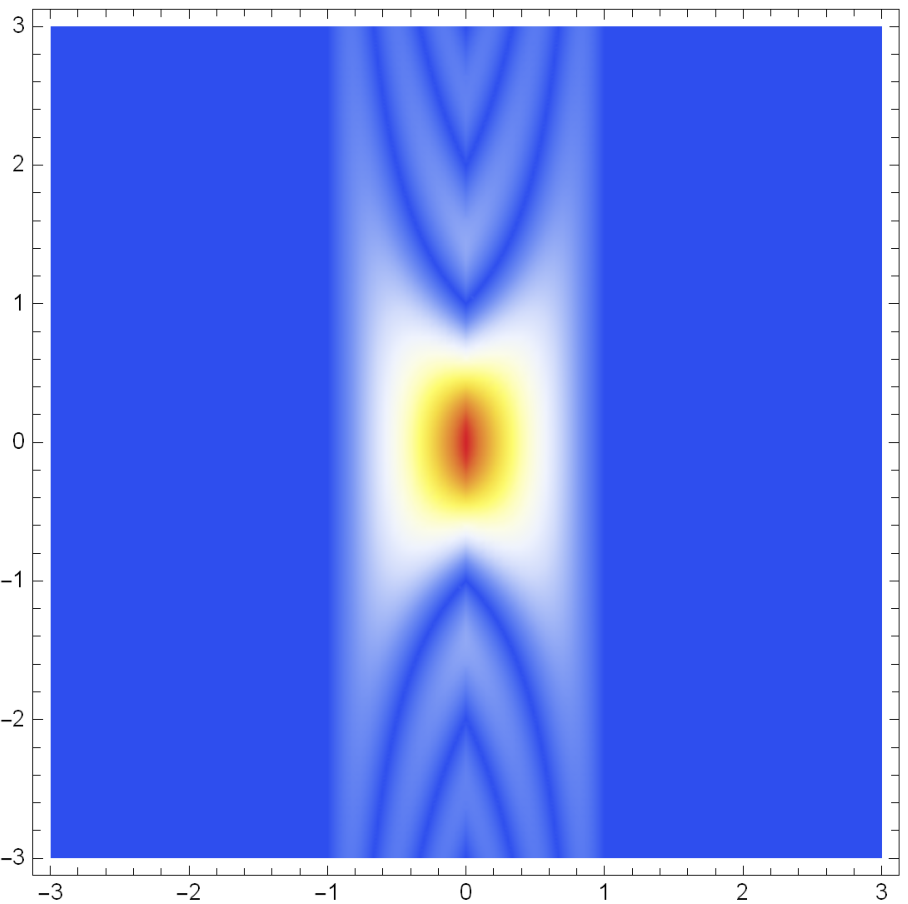}
				\hspace{0.5cm}
				\includegraphics[width=.425\textwidth]{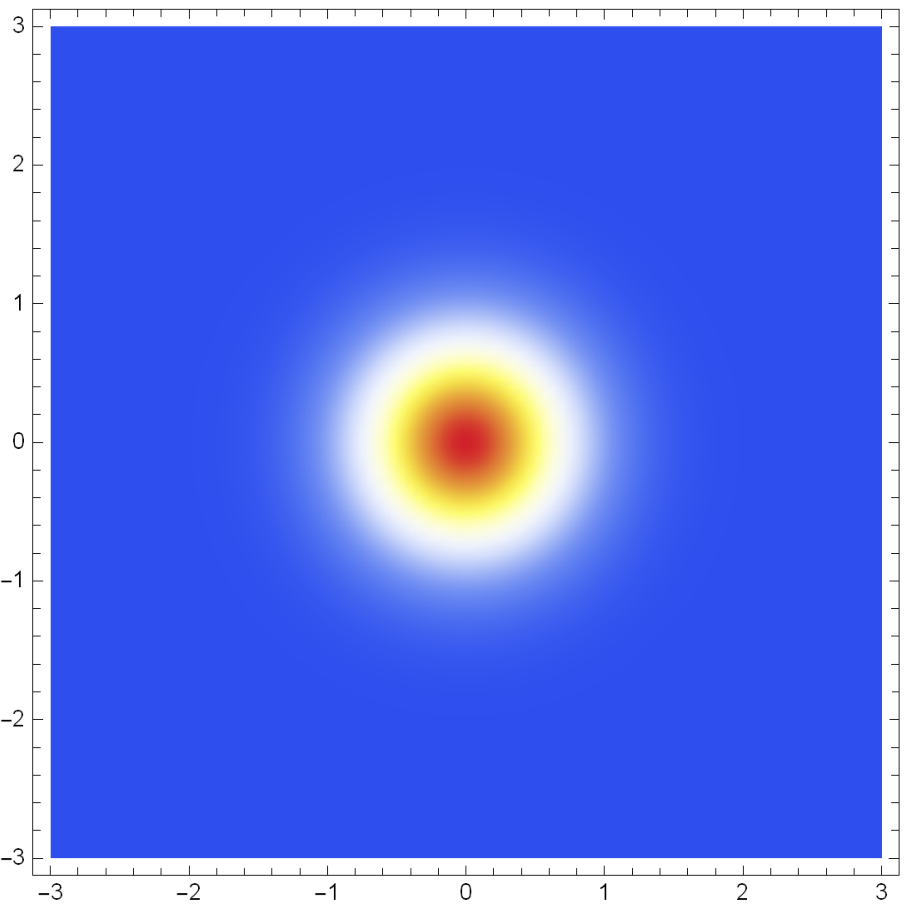}
			}
			\\
			\subfigure[3d-plots of the absolute values of the auto-ambiguity functions of the box function $b_0$ and the standard Gaussian $g_0$.]
			{
				\includegraphics[width=.425\textwidth]{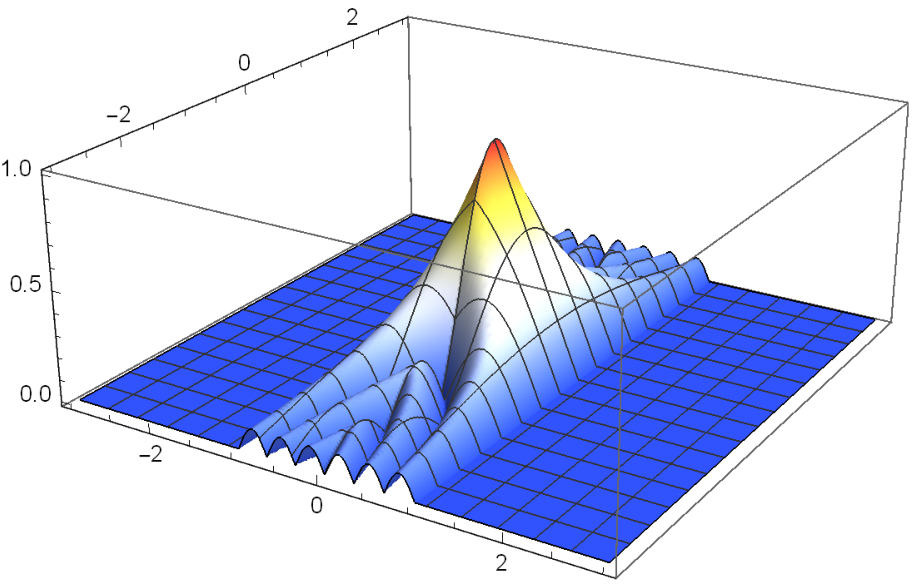}
				\hspace{0.5cm}
				\includegraphics[width=.425\textwidth]{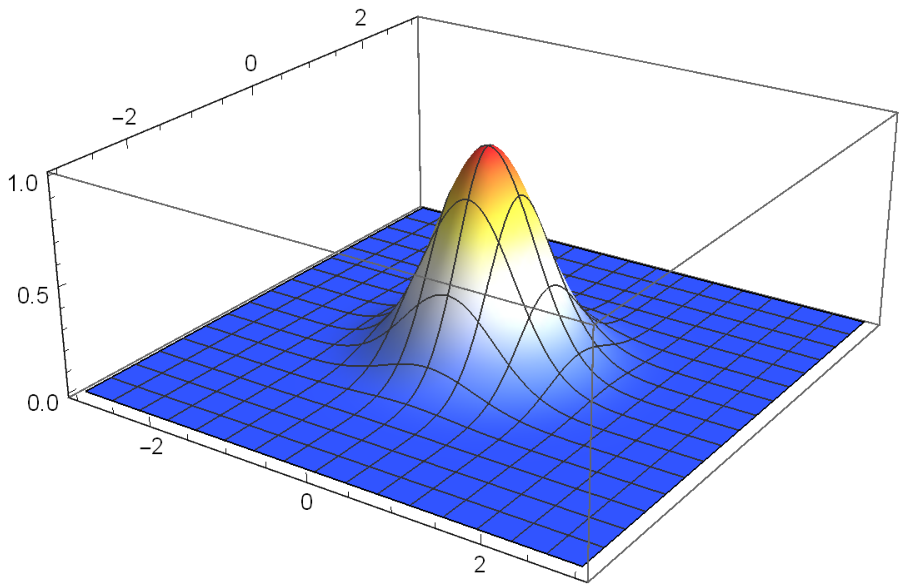}
			}
			\caption{\footnotesize{Time-frequency concentration of the box-function $b_0$ and the standard Gaussian $g_0$. }}
		\end{figure}
	\end{enumerate}
	\flushright{$\diamond$}
\end{example}
\begin{remark}
	We see that, for $x = 0$, we obtain the Fourier transform of the box function
	\begin{equation}
		A b_0(0,\omega) = \F b_0(\omega) = \sinc(\omega),
	\end{equation}
	which demonstrates, again, that the frequency resolution of the box function is not suitable for measuring time-frequency content. On the time-side, the ambiguity function measures the auto-correlation of a signal. Since the box function is an even function (i.e., $g^\vee = g$), for $\omega = 0$, we obtain the convolution of the box function with itself by using l'Hospital's rule
	\begin{equation}
		Ab_0(x,0) = \lim_{\omega \to 0} \frac{\sin(\pi \omega) (1-|x|)}{\pi \omega} = 1-|x|.
	\end{equation}
	
	For the second example, we see that for $x = 0$, we obtain the auto-correlation of the Gaussian spectrum
	\begin{equation}
		Ag_0(0,\omega) = e^{-\frac{\pi}{2} \omega^2},
	\end{equation}
	which is a dilated Gaussian, and not the standard Gaussian. The same is true in the case $\omega = 0$;
	\begin{equation}
		Ag_0(x,0) = e^{-\frac{\pi}{2} x^2}.
	\end{equation}
	This is, as in the case of the box function, the convolution of the Gaussian with itself.
	
	What causes the difference in the spectral behavior of the box function compared to the Gaussian, is the fact that $Ab_0(0,\omega)$ does actually not give the Fourier transform of $b_0$, but the auto-correlation of the spectrum $\F b_0$. But, the auto-correlation corresponds to a convolution in this case. Now, the $\sinc$-function is invariant under convolution, as can easily be seen by the following computation;
	\begin{equation}
		(\sinc * \sinc) (\omega) = (\F b_0*\F b_0)(\omega) = \F(\underbrace{b_0 \, b_0}_{=b_0})(\omega) = \F b_0 (\omega) = \sinc(\omega).
	\end{equation}
	Note that, in general, we have
	\begin{equation}
		Af(0,\omega) = \F(f \, \overline{f})(\omega) = (\F f * \F \overline{f}) (\omega)
		\quad \text{ and } \quad
		Af(x,0) = (f * \overline{f}^\vee) (x).
	\end{equation}
	\begin{flushright}
		$\diamond$
	\end{flushright}
\end{remark}

Since, we have the relation $A(f,g)(x,\omega) = e^{\pi i x \cdot \omega} V_gf(x,\omega)$, it is clear that the ambiguity functions enjoys similar properties as the STFT. In particular, the orthogonality relations and the Parseval-like identity hold;
\begin{align}
	\langle A(f_1,g_1), A(f_2,g_2) \rangle_{\Lt[2d]} & = \langle e^{\pi i x \cdot \omega} V_{g_1} f_1, e^{\pi i x \cdot \omega} V_{g_2} f_2 \rangle_{\Lt[2d]}\\
	& = \langle V_{g_1} f_1, V_{g_2} f_2 \rangle_{\Lt[2d]} = \langle f_1, f_2\rangle \overline{\langle g_1, g_2 \rangle}.
\end{align}
Therefore, we also have
\begin{equation}
	\norm{A(f,g)}_{\Lt[2d]} = \norm{f}_2 \norm{g}_2.
\end{equation}
As the ambiguity function has a symmetry in the time-frequency shifts, $A(f,g)(x,\omega) = \langle \pi(-\frac{\l}{2}) f, \pi( \frac{\l}{2}) g \rangle$, the analogue to Proposition \ref{pro_fitf} (Fundamental identity of TFA) becomes more symmetric as well;
\begin{align}
	A(f,g)(x,\omega) & = e^{\pi i x \cdot \omega} V_gf(x,\omega) = e^{\pi i x \cdot \omega} e^{-2 \pi i x \cdot \omega} V_{\widehat{g}}\widehat{f} (\omega, -x)
\\
& = e^{-\pi i x \cdot \omega} V_{\widehat{g}}\widehat{f}(\omega, -x) = A(\widehat{f}, \widehat{g})(\omega, -x).
\end{align}
Lastly, we discuss a covariance result for the ambiguity function. For $f,g \in \Lt$, $x,\omega, t \in \Rd$ and $\xi_1,\xi_2, \eta_1, \eta_2 \in \Rd$, we have
\begin{align}\label{eq_cov_ambi}
	& \, A(M_{\eta_1} T_{\xi_1} f, M_{\eta_2} T_{\xi_2} g)(x,\omega)\\
	& = \int_{\Rd} (M_{\eta_1} T_{\xi_1} f)(t+\tfrac{x}{2})(M_{\eta_2} T_{\xi_2} f)(t-\tfrac{x}{2})e^{-2 \pi i \omega \cdot t} \, dt\\
	& = \int_{\Rd} f(t+\tfrac{x}{2} - \xi_1) \overline{g(t-\tfrac{x}{2}-\xi_2)} \, e^{2\pi i \eta_1 \cdot (t+\frac{x}{2})} e^{-2 \pi i \eta_2 \cdot (t-\frac{x}{2})} e^{-2 \pi i \omega \cdot t} \, dt \quad \substack{\phantom{=}}_{(t \mapsto t+\tfrac{\xi_1}{2}+\tfrac{\xi_2}{2})}\\
	& = e^{\pi i x \cdot (\eta_1+\eta_2)} \int_{\Rd} f(t+\tfrac{x - \xi_1 + \xi_2}{2}) \overline{g(t-\tfrac{x - \xi_1 + \xi_2}{2})} \,  e^{-2 \pi i (\omega-\eta_1+\eta_2) \cdot (t+\frac{\xi_1+\xi_2}{2})} \, dt\\
	& = e^{\pi i x \cdot (\eta_1+\eta_2)} e^{-\pi i \omega \cdot (\xi_1+\xi_2)} e^{\pi i (\eta_1-\eta_2) \cdot (\xi_1+\xi_2)} \int_{\Rd} f(t+\tfrac{x - \xi_1 + \xi_2}{2}) \overline{g(t-\tfrac{x - \xi_1 + \xi_2}{2})} \,  e^{-2 \pi i (\omega-\eta_1+\eta_2) \cdot t} \, dt\\
	& = e^{\pi i x \cdot (\eta_1+\eta_2)} e^{-\pi i \omega \cdot (\xi_1+\xi_2)} e^{\pi i (\eta_1-\eta_2) \cdot (\xi_1+\xi_2)} A(f,g)(x-\xi_1+\xi_2,\omega-\eta_1+\eta_2).
\end{align}
In particular, for $\xi_1=\xi_2=\xi$ and $\eta_1=\eta_2=\eta$ we obtain
\begin{equation}
	A(M_\eta T_\xi f, M_\eta T_\xi g)(x,\omega) = e^{2 \pi i (x \cdot \eta - \omega \cdot \xi)} A(f,g)(x,\omega).
\end{equation}
This shows that shifting, both, $f$ and $g$ in the time-frequency plane, only yields a phase factor.

\subsection{The Wigner Distribution}
The Wigner distribution was ``invented" by Eugene Wigner in 1932 \cite{Wig32} in the context of quantum mechanics. We quote from the introduction in \cite{Gosson_Wigner_2017}:

\textit{In this paper \textnormal{[\cite{Wig32}; author's note]}, Wigner introduced a probability quasi-distribution that allowed him to express quantum mechanical expectation values in the same form as the averages of classical statistical mechanics. There have been many speculations about how Wigner got his idea; his eponymous transform seems to be pulled out of thin air.}

We will now study some of the properties of this almost mysterious transform, including Hudson's theorem, which will essentially take us all the way to very recent mathematical research. Besides the textbook of Gröchenig \cite{Gro01}, the book by de Gosson \cite{Gosson_Wigner_2017} may be consulted as a further reference and for more information on the Wigner transform.

\begin{definition}\label{def_Wigner}
	The Wigner distribution of a function $f \in \Lt$ is given by
	\begin{equation}
		W f(x, \omega) = \int_{\Rd} f(x + \tfrac{t}{2}) \overline{f(x - \tfrac{t}{2})} e^{-2 \pi i \omega \cdot t} \, dt.
	\end{equation}
	The cross-Wigner transform of two functions $f, g \in \Lt$ is given by
	\begin{equation}
	W (f,g) (x, \omega) = \int_{\Rd} f(x + \tfrac{t}{2}) \overline{g(x - \tfrac{t}{2})} e^{-2 \pi i \omega \cdot t} \, dt.
	\end{equation}
\end{definition}
\noindent
We quote from the introduction in \cite{Gosson_Wigner_2017}:

\textit{Truly, Wigner's definition, in modern notation did not remind of anything one had seen before; a rapid glance suggests it is something of a mixture of a Fourier transform and a convolution. And, yet it worked! For instance, under some mild assumptions on the function $\psi$ \textnormal{[which may be a quantum state or wave function; author's note]} we recover the probability amplitudes $|\psi(x)|^2$ and $|\phi(p)|^2$ \textnormal{[$\phi = \widehat{\psi}$, author's note]} of quantum mechanics:
\begin{align}
	\int_{-\infty}^\infty W\psi(x,p) \, dp = |\psi(x)|^2,\\
	\int_{-\infty}^\infty W\psi(x,p) \, dx = |\phi(p)|^2;
\end{align}
assuming $\psi$ normalized to unity and integrating any of these equalities, we get
\begin{equation}
	\int_{-\infty}^\infty \int_{-\infty}^\infty W\psi(x,p) \, dp \, dx = 1,
\end{equation}
so that the Wigner transform $W\psi$ can be used as a mock probability distribution.
}

The fact that we (can) only call the Wigner transform a quasi-distribution stems from the fact that it fails to be non-negative in general. The only exceptions come from Gaussian functions, which is Hudson's Theorem. Before we prove that $Wf$ is a quasi-distribution, we need to describe the connections to the other time-frequency representations.

We recall that the reflection of a function $g$ is denoted by $g^\vee$;
\begin{equation}
	g^\vee(t) = g(-t), \qquad t \in \Rd .
\end{equation}

The following lemma shows the algebraic connection between $V_g f$ (and hence $A(f,g)$) and $W(f,g)$.
\begin{lemma}[Algebraic Relations]\label{lem_Wigner_ambi_algebraic}
	For $f,g \in \Lt$ we have
	\begin{equation}
		W(f,g)(x, \omega) = 2^d e^{4 \pi i x \cdot \omega} V_{g^\vee} f(2x, 2\omega) = 2^d A(f, g^\vee) (2x, 2 \omega).
	\end{equation}
\end{lemma}
\begin{proof}
	We will use the substitution $u = x + \tfrac{t}{2}$ in the definition of $W(f,g)$ and compute
	\begin{align}
		W(f,g)(x, \omega) & = \int_{\Rd} f(x+\tfrac{t}{2}) \overline{g(x-\tfrac{t}{2})} e^{-2 \pi i \omega \cdot t} \, dt\\
		& = 2^d \int_{\Rd} f(u) \overline{g(-(u-2x)} e^{-2 \pi i \omega \cdot (2u-2x)} \, du\\
		& = 2^d e^{4 \pi i x \cdot \omega} \langle f, M_{2\omega} T_{2x} g^\vee \rangle\\
		& = 2^d e^{4 \pi i x \cdot \omega} V_{g^\vee} f(2x, 2\omega).
	\end{align}
\end{proof}
We will now collect properties of the Wigner distribution from the STFT. Some of these properties readily follow by using Lemma \ref{lem_Wigner_ambi_algebraic}.
\begin{proposition}\label{pro_Wigner}
	For $f,g \in \Lt$ the cross-Wigner distribution has the following properties.
	\begin{enumerate}[(i)]
		\item $W(f,g) \in L^\infty(\R^{2d}) \cap C_0(\R^{2d})$, in particular $W(f,g)$ is uniformly continuous on $\R^{2d}$ and is bounded by
		\begin{equation}
			\norm{W(f,g)}_\infty \leq 2^d \norm{f}_2 \norm{g}_2 .
		\end{equation}
		\item $W(f,g) = \overline{W(g,f)}$. In particular, $W f$ is real-valued.
		\item For $\xi_1, \xi_2, \eta_1, \eta_2 \in \Rd$ we have
		\begin{align}
			& \, W(M_{\eta_1}T_{\xi_1}f, \, M_{\eta_2}T_{\xi_2}g)\\
			= & \, e^{\pi i (\xi_1-\xi_2) \cdot (\eta_1+\eta_2)} e^{2 \pi i x \cdot (\eta_1-\eta_2)} e^{-2 \pi i \omega \cdot (\xi_1-\xi_2)} \, W(f,g)(x-\tfrac{\xi_1+\xi_2}{2}, \omega -\tfrac{\eta_1+\eta_2}{2}).
		\end{align}
		In particular, $Wf$ is covariant, i.e.,
		\begin{equation}
			W(M_\eta T_\xi f)(x, \omega) = Wf(x-\xi, \omega-\eta).
		\end{equation}
		\item $W(\widehat{f}, \widehat{g})(x, \omega) = W(f,g)(-\omega, x)$
		\item Also, we have the Parseval-like identity, which goes under the name Moyal's formula;
		\begin{equation}\label{eq_Moyal}
			\langle W(f_1,g_1), W(f_2,g_2) \rangle_{\Lt[2d]} = \langle f_1,f_2 \rangle \overline{\langle g_1, g_2 \rangle}.
		\end{equation}
	\end{enumerate}
\end{proposition}
\begin{proof}
The proofs follow mainly by direct computation.
	\begin{enumerate}[(i)]
		\item The assertion that $W(f,g)$ is uniformly continuous and decaying to 0 follows from the respective properties of $V_gf$ by using Lemma \ref{lem_Wigner_ambi_algebraic}. To show the $L^\infty$-bound we compute
		\begin{align}
			|W(g,f)(x,\omega)| & = 2^d |V_{g^\vee} f(2x,2\omega)| = 2^d |\langle f, M_{2 \omega} T_{2x} g^\vee \rangle|\\
			& \leq 2^d \norm{f}_2 \norm{M_{2 \omega} T_{2x} g^\vee}_2 = 2^d \norm{f}_2 \norm{g}_2.
		\end{align}
		In the above calculations we used, again, Lemma \ref{lem_Wigner_ambi_algebraic} and the Cauchy-Schwarz inequality. Also we used that $M_\omega$ and $T_x$ are unitary and that flipping $g$, i.e., $g \mapsto g^\vee$, is unitary as well.
		\item By using the substitution $t \mapsto -t$, we compute
		\begin{align}
			W(f,g)(x, \omega) & = \int_{\Rd} f(x+\tfrac{t}{2}) \overline{g(x-\tfrac{t}{2})} e^{-2 \pi i \omega \cdot t} \, dt\\
			& = \overline{\int_{\Rd} \overline{f(x-\tfrac{t}{2})} g(x+\tfrac{t}{2}) e^{-2 \pi i \omega \cdot t} \, dt}\\
			& = \overline{W(g,f)(x, \omega)}.
		\end{align}
		\item We compute
		\begin{align}
			& \, W(M_{\eta_1} T_{\xi_1} f, M_{\eta_2} T_{\xi_2}g)(x,\omega)\\
			& = \int_{\Rd} (M_{\eta_1} T_{\xi_1} f)(x+\tfrac{t}{2}) \overline{(M_{\eta_2} T_{\xi_2}g)(x-\tfrac{t}{2})} e^{-2\pi i \omega \cdot t} \, dt\\
			& = \int_{\Rd} f(x-\xi_1+\tfrac{t}{2}) \overline{g(x-\xi_2-\tfrac{t}{2})} e^{2 \pi i \eta_1 \cdot (x+\frac{t}{2})} e^{-2\pi i \eta_2 \cdot(x-\frac{t}{2})} e^{-2 \pi i \omega \cdot t} \, dt \quad \substack{\phantom{=}}_{(t \mapsto t+\xi_1-\xi_2)}\\
		& = e^{2 \pi i x \cdot(\eta_1-\eta_2)}\int_{\Rd} f(x-\tfrac{\xi_1+\xi_2}{2} + \tfrac{t}{2}) \overline{g(x-\tfrac{\xi_1+\xi_2}{2}-\tfrac{t}{2})} e^{-2 \pi i (\omega-\frac{\eta_1+\eta_2}{2}) \cdot (t+\xi_1-\xi_2)} \, dt\\
		& = e^{2 \pi i x \cdot(\eta_1-\eta_2)} e^{-2 \pi i \omega \cdot(\xi_1-\xi_2)} \int_{\Rd} f(x-\tfrac{\xi_1+\xi_2}{2} + \tfrac{t}{2}) \overline{g(x-\tfrac{\xi_1+\xi_2}{2}-\tfrac{t}{2})} e^{-2 \pi i (\omega-\frac{\eta_1+\eta_2}{2}) \cdot t} \, dt\\
		& = e^{2 \pi i x \cdot(\eta_1-\eta_2)} e^{-2 \pi i \omega \cdot(\xi_1-\xi_2)} e^{\pi i (\xi_1-\xi_2)(\eta_1+\eta_2)} W(f,g)(x-\tfrac{\xi_1+\xi_2}{2},\omega-\tfrac{\eta_1+\eta_2}{2}).
		\end{align}
		\item We note that $\widehat{g^\vee} = \widehat{g}^\vee$. We compute
		\begin{align}
			W(\widehat{f}, \widehat{g})(x, \omega) & = 2^d e^{4 \pi i x \cdot \omega} \langle \widehat{f}, M_{2 \omega} T_{2x} \widehat{g}^\vee \rangle\\
			& = 2^d e^{4 \pi i x \cdot \omega} \langle f, T_{-2\omega} M_{2 x} g^\vee \rangle\\
			& = 2^d e^{-4 \pi i x \cdot \omega} \langle f, M_{2x} T_{-2 \omega} g^\vee \rangle\\
			& = W(f,g)(-\omega, x).
		\end{align}
		In the above computations we used Lemma \ref{lem_Wigner_ambi_algebraic}, Parseval's formula \eqref{eq_Parseval} and the commutation relations \eqref{eq_comm_rel}.
		\item We compute
		\begin{align}
			& \, \iint_{\R^{2d}} W(f_1,g_1)(x,\omega) \overline{W(f_2,g_2)(x, \omega)} \, dx d\omega\\
			= & \, 2^{2d} \iint_{\R^{2d}} V_{g_1^\vee} f_1(2x, 2\omega) \overline{V_{g_2^\vee} f_2 (2x, 2\omega)} \, dx d\omega\\
			= & \, \langle f_1, f_2 \rangle \overline{\langle g^\vee_1, g^\vee_2 \rangle}\\
			= & \, \langle f_1, f_2 \rangle \overline{\langle g_1, g_2 \rangle}
		\end{align}
		When changing from the integral formulas to the inner products, we made the substitution $(2x, 2 \omega) \mapsto (x, \omega)$, which is why the factor $2^{2d}$ disappeared. Also, we used the orthogonality relations \eqref{eq_OR} of the STFT.
	\end{enumerate}
\end{proof}

Before we proceed, we introduce some more notation. By $\mathcal{T}_s$ we denote the symmetric coordinate change defined by
\begin{equation}\label{eq_symcoord_shift}
	\mathcal{T}_s F(x,t) = F(x+\tfrac{t}{2}, x-\tfrac{t}{2}).
\end{equation}
We also note that its inverse is given by
\begin{equation}\label{eq_symcoord_shift_inv}
	\mathcal{T}_s^{-1} F(x,t) = F(\tfrac{x+t}{2},x-t).
\end{equation}
Furthermore, we re-call the notion of the partial Fourier transforms $\F_1$ and $\F_2$, which transform a function of two arguments only in the first and the second argument, respectively.
\begin{equation}
	\F_1 F(x,\omega) = \int_{\Rd} F(t,\omega) e^{-2 \pi i x \cdot t} \, dt.
	\qquad \text{ and } \qquad
	\F_2 F(x,\omega) = \int_{\Rd} F(x,t) e^{-2 \pi i \omega \cdot t} \, dt.
\end{equation}
Using these (unitary) operators, the cross-Wigner distribution can be written as
\begin{equation}\label{eq_Wigner_tensor}
	W(f,g)(x,\omega) = \F_2 \left(\mathcal{T}_s (f \otimes \overline{g})(x,t)\right)(x,\omega),
\end{equation}
where $(f \otimes g)(x, \omega) = f(x) g(\omega)$ is the already familiar tensor product. From this factorization, we easily obtain an inversion formula for the (auto-)Wigner distribution.
\begin{equation}
	f(x) \overline{f(y)} = \left( \mathcal{T}_s^{-1} \F_2^{-1} W f \right)(x,y) = \int_{\Rd} Wf \left(\tfrac{x+y}{2}, \omega \right) e^{2 \pi i (x-y) \cdot \omega} \, d\omega.
\end{equation}
Now, if $f(0) \neq 0$, then
\begin{equation}
	f(x) = \frac{1}{\overline{f(0)}} \int_{\Rd} Wf \left( \tfrac{x}{2}, \omega \right) e^{2 \pi i x \cdot \omega} \, d\omega.
\end{equation}
Of course, this inversion formula is true only up to a global phase factor $|c|=1$.

The next lemma is highly important and reveals a deep connection between the ambiguity function and the Wigner distribution via the Fourier transform.
\begin{lemma}\label{lem_Wigner_ambi_FT}
	For $f, g \in \Lt$, we have
	\begin{equation}
		W(f,g)(x, \omega) = \F A(f,g)(\omega, -x)
		\qquad \text{ and } \qquad
		A(f,g)(x, \omega) = \F W(f,g)(\omega, -x)
	\end{equation}
\end{lemma}
\begin{proof}
	We note that the Fourier transform on $\R^{2d}$ factors into the partial Fourier transforms on $\Rd$, i.e., $\F = \F_1 \F_2 = \F_2 \F_1$. By using \eqref{eq_Wigner_tensor} we have
	\begin{align}
		\F^{-1} W(f,g)(x, \omega) & = \F_1^{-1} \F_2^{-1} \F_2 \mathcal{T}_s(f \otimes \overline{g})(x,\omega)\\
		& = \int_{\Rd} \mathcal{T}_s(f \otimes \overline{g})(t,\omega) e^{2 \pi i t \cdot x} \, dt\\
		& = \int_{\Rd} f(t+\tfrac{\omega}{2}) \overline{g(t-\tfrac{\omega}{2})} e^{2 \pi i x \cdot t} \, dt\\
		& = A(f,g)(\omega,-x).
	\end{align}
	This proves the first assertion. The second one follows analogously.
	\begin{align}
		\F W(f,g)(\omega, -x) & = \F_1 \F_2 \F_2^{-1} \mathcal{T}_s(f \otimes \overline{g})(\omega,x)\\
		& = \int_{\Rd} \mathcal{T}_s(f \otimes \overline{g})(t,x) e^{-2 \pi i \omega \cdot t} \, dt\\
		& = \int_{\Rd} f(t+\tfrac{x}{2}) \overline{f(t-\tfrac{x}{2})} e^{-2 \pi i \omega \cdot t} \, dt\\
		& = A(f,g)(x, \omega).
	\end{align}
	The computations certainly hold for Schwartz functions (or on the dense subspace $L^1(\Rd) \cap \Lt$) and the formulas extend to all $f,g \in \Lt$ by continuity of $\F$ and $\mathcal{T}_s$.
\end{proof}
After this discussion on the connection with other time-frequency representations, we will discuss the support properties of the Wigner distribution. To simplify notation, we are only going to discuss this in the case $d=1$ (hence, for intervals and not for $d$-dimensional hyper cubes).
\begin{lemma}
	For $f \in \Lt[]$, if $\supp(f) \subset [a,b]=\mathcal{Q}_{[a,b]}$, then $Wf(x,\omega) = 0$ for $x \notin \mathcal{Q}_{[a,b]}$. Likewise, if $\supp(\widehat{f}) \subset \mathcal{Q}_{[c,d]}$, then $Wf(x,\omega) = 0$ for $\omega \notin \mathcal{Q}_{[c,d]}$.
\end{lemma}
\begin{proof}
	$Wf(x,\omega) \neq 0$ is possible only if $x + \frac{t}{2} \in \supp(f)$ and $x - \frac{t}{2} \in \supp(f)$. Then $x = \frac{1}{2}(x+\frac{t}{2})+\frac{1}{2}(x-\frac{t}{2}) \in \mathcal{Q}_{[a,b]}$. Thus $Wf(x,\omega) = 0$ for $x \notin \mathcal{Q}_{[a,b]}$.
	
	The second statement follows from the relation $Wf(x,\omega) = W \widehat{f}(\omega,-x)$.
\end{proof}
So, in contrast to the ambiguity function (and the STFT), the Wigner transform preserves the support properties of $f$ and $\widehat{f}$ exactly (so does the Rihaczek distribution). Also, note that the auto-Wigner distribution is taken as a joint (quasi-) probability distribution in phase space. The interpretation of the STFT usually depends on the auxiliary window and the ``sliding" of the window $g$ spreads out the support of $f$.

The next result is at the origin of interest in quantum mechanics and shows that the Wigner distribution yields the correct marginal densities, which then allow to interpret the Wigner distribution as a joint (quasi-) probability density function.
\begin{lemma}\label{lem_marginal_density}
	Let $f, \widehat{f} \in L^1(\Rd) \cap \Lt$. Then
	\begin{equation}
		\int_{\Rd} Wf(x,\omega) \, d\omega = |f(x)|^2,
	\end{equation}
	\begin{equation}
		\int_{\Rd} Wf(x,\omega) \, dx = |\widehat{f}(\omega)|^2.
	\end{equation}
	In particular, if $\norm{f}_2 = \norm{\widehat{f}}_2 = 1$, then
	\begin{equation}
		\iint_{\R^{2d}} Wf(x,\omega) \, d(x,\omega) = 1.
	\end{equation}
\end{lemma}
\begin{proof}
	We use the algebraic relation \ref{lem_Wigner_ambi_algebraic} between the Wigner transform and the STFT and write the STFT as a convolution of Fourier transforms (see \eqref{eq_STFT_notation});
	\begin{equation}\label{eq_marginal_density_aux1}
		|Wf(x,\omega)| = 2^d|V_{f^\vee} f(2x,2\omega)| = 2^d|\widehat{f} * M_{-2x} \widehat{\overline{f}}^\vee(2\omega)|.
	\end{equation}
	By Young's Theorem \ref{thm_Young_conv} we have the convolution inequality \eqref{eq_conv_L1} which tells us that the convolution of two $L^1(\Rd)$-functions is again an $L^1(\Rd)$-function. Hence, under the Assumption that $\widehat{f} \in L^1(\Rd)$ and \eqref{eq_marginal_density_aux1} we see that, for any fixed $x \in \Rd$, $Wf(x,\omega) \in L^1(\Rd, d\omega)$. Therefore, the Fourier inversion formula for the partial Fourier transform $\F_2$ is applicable and yields
	\begin{align}
		\int_{\Rd} Wf(x,\omega) \, d\omega & = \left(\F_2^{-1} \left( \F_2 \mathcal{T}_s(f \otimes \overline{f})\right)\right)(x,0) = \mathcal{T}_s(f \otimes \widehat{f})(x,0) = |f(x)|^2.
	\end{align}
	Now, using Proposition \ref{pro_Wigner}(iv) we compute
	\begin{equation}
		\int_{\Rd} Wf(x,\omega) \, dx = \int_{\Rd} W \widehat{f}(\omega,-x) \, dx = |\widehat{f}(\omega)|^2.
	\end{equation}
\end{proof}
In the language of quantum mechanics, this means that $\int_{\Rd} Wf(x,\omega) \, d\omega$ yields the probability density $|f(x)|^2$ for the position variable and $\int_{\Rd}WF(x,\omega) \, dx$ yields the density $|\widehat{f}|^2$ for the momentum variable.

Also, we remark that, for a suitable subspace of $\Lt$ such that $Wf \in L^1(\R^{2d})$ (which allows us to use Fubini's theorem), we have
\begin{equation}
	\norm{f}_2^2 = \int_{\Rd}\int_{\Rd} Wf(x,\omega) \, dx \, d\omega = \int_{\Rd}\int_{\Rd} Wf(x,\omega) \, d\omega \, dx = \norm{\widehat{f}}_2^2.
\end{equation}
Thus, Lemma \ref{lem_marginal_density} implies Plancherel's theorem (on this subspace, which is actually dense).

We will now discuss the positivity of the Wigner distribution and, hence, whether it can serve as an energy density or joint probability density. We mention right away that, in general the Wigner distribution fails to be positive. For example, let $f$ be an odd function, i.e., $f(t) = -f^\vee(t) = -f(-t)$, then
\begin{equation}
	Wf(0,0) = \int_{\Rd} f(\tfrac{t}{2}) \overline{f(-\tfrac{t}{2})} \, dt = -2^d \int_{\Rd} f(t) \overline{f^\vee(-t)} \, dt = -2^d \int_{\Rd} |f(t)|^2 \, dt = -2^d \norm{f}_2^2 < 0.
\end{equation}
Negative values cannot be interpreted as a probability or as an energy density. However, a single point in the time-frequency plane or phase space also has no physical meaning. Now, the hope is that, maybe there is a ``nice" subspace of ``physically realizable" functions for which the Wigner distribution is positive? The answer to this question is given by Hudson's theorem \cite{Hud74}.

\begin{theorem}[Hudson]\label{thm_Hudson}
	Assume that $f \in \Lt$. Then $Wf(x, \omega) > 0$ for all $(x,\omega) \in \R^{2d}$ if and only if $f$ is a generalized Gaussian, i.e.,
	\begin{equation}\label{eq_gen_Gauss}
		f(t) = e^{- \pi t \cdot A t + 2 \pi b \cdot t + c},
	\end{equation}
	where $A \in GL(\C,d)$ is an invertible $d \times d$ matrix with entries from $\C$ and positive definite real part ($\Re(A) > 0$) and with $\overline{A}^T = A^* = A$, $b \in \C^d$ and $c \in \C$.
\end{theorem}
We will postpone the proof to a later point as we need more preparation for the proof.

A bit more generally, Theorem \ref{thm_Hudson} can also be stated as follows \cite{GroJamMal19}, \cite{Lie90}.
\begin{theorem}[Hudson-Lieb]\label{thm_Hudson_Lieb}
	Let $f,g \in \Lt$ (non-zero). Then $W(f,g) \geq 0$ for all $(x,\omega) \in \R^{2d}$ if and only if $f = K \, g$, $K > 0$ and $g$ is a generalized Gaussian of the form \eqref{eq_gen_Gauss}. In this case, $W(f,g)(x, \omega) > 0$ for all $(x, \omega) \in \R^{2d}$.
\end{theorem}

We note that $W(f,g)$ is not necessarily real-valued. However, if it is non-negative and real-valued, then $f$ and $g$ are the same generalized Gaussian up to a positive constant and the cross-Wigner distribution is essentially an auto-Wigner distribution and actually positive everywhere. Hudson's original theorem dates back to 1974 and the extension of Lieb was established in 1990.

A related and hard problem, which still remains open in full generality, is to characterize pairs of functions $f$ and $g$ for which $W(f,g)(x, \omega) \neq 0$ for all $(x,\omega) \in \R^{2d}$. Only very recently examples of pairs $(f,g)$, which are not Gaussians and yield zero-free Wigner distributions, or likewise zero-free STFTs and ambiguity functions, have been found and described by Gröchenig, Jaming and Malinnikova in their article \cite{GroJamMal19} published in 2020. The prototype example they give is obtained from the one-sided exponential.
\begin{example}
	Let $f$ be a one-sided exponential, i.e.,
	\begin{align}
		f(t) = \indicator_{\R_+}(t) e^{-\pi t} =
		\begin{cases}
			e^{- \pi t}, & t > 0\\
			0, & else
		\end{cases}.
	\end{align}
	\begin{figure}[ht]
		\centering
		\includegraphics[width=.7\textwidth]{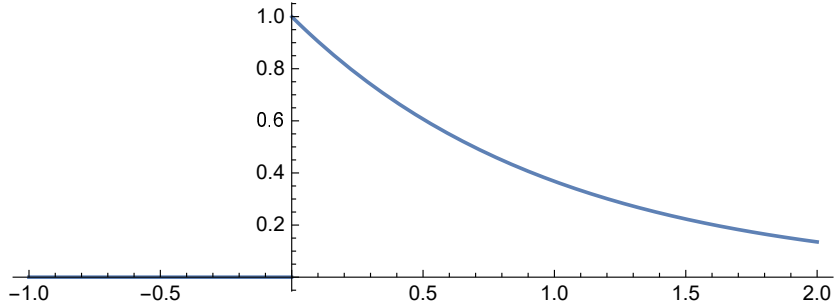}
		\caption{\footnotesize{The one-sided exponential function.}}
	\end{figure}
	Assume $x \geq 0$, then the ambiguity function is given by
	\begin{align}
		A f(x, \omega) & = \int_\R \indicator_{\R_+}(t-\tfrac{x}{2})\indicator_{\R_+}(t+\tfrac{x}{2}) e^{-\pi (-\frac{x}{2})} e^{-\pi (t+\frac{x}{2})} e^{-2 \pi i \omega t} \, dt\\
		& = \int_{\frac{x}{2}}^\infty e^{-2 \pi t} e^{-2 \pi i \omega t} \, dt
		= \int_{\frac{x}{2}}^\infty e^{-2 \pi i (\omega - i) t} \, dt
		= \frac{e^{-2 \pi i (\omega - i) t}}{-2 \pi i (\omega-i)} \Bigg|_{t=\frac{x}{2}}^\infty\\
		& = \frac{e^{-\pi i (\omega-i) x}}{2\pi i(\omega-i)}.
	\end{align}
	The part for $x < 0$ follows analogously and we get
	\begin{equation}
		Af(x,\omega) = \frac{e^{-\pi i(\omega-i)|x|}}{2 \pi i(\omega-i)}.
	\end{equation}
	In particular, $A f(x, \omega) \neq 0$ for all $(x, \omega) \in \R^2$. It follows immediately that $W(f, f^\vee) \neq 0$ by the algebraic relation in Lemma \ref{lem_Wigner_ambi_algebraic}.
	\begin{figure}[ht]
		\includegraphics[width=.45\textwidth]{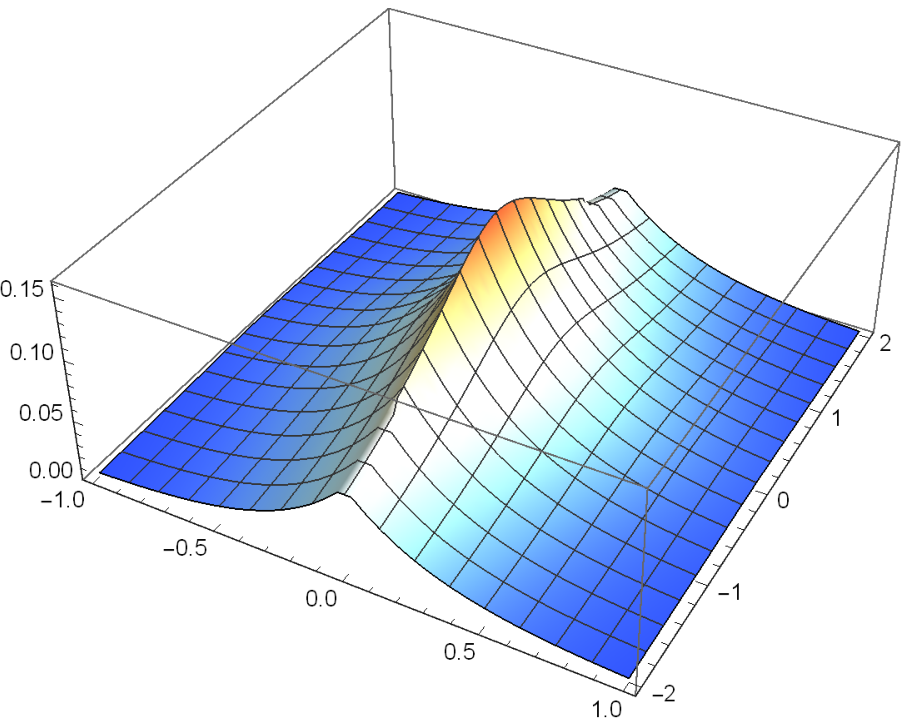}
		\hfill
		\includegraphics[width=.4\textwidth]{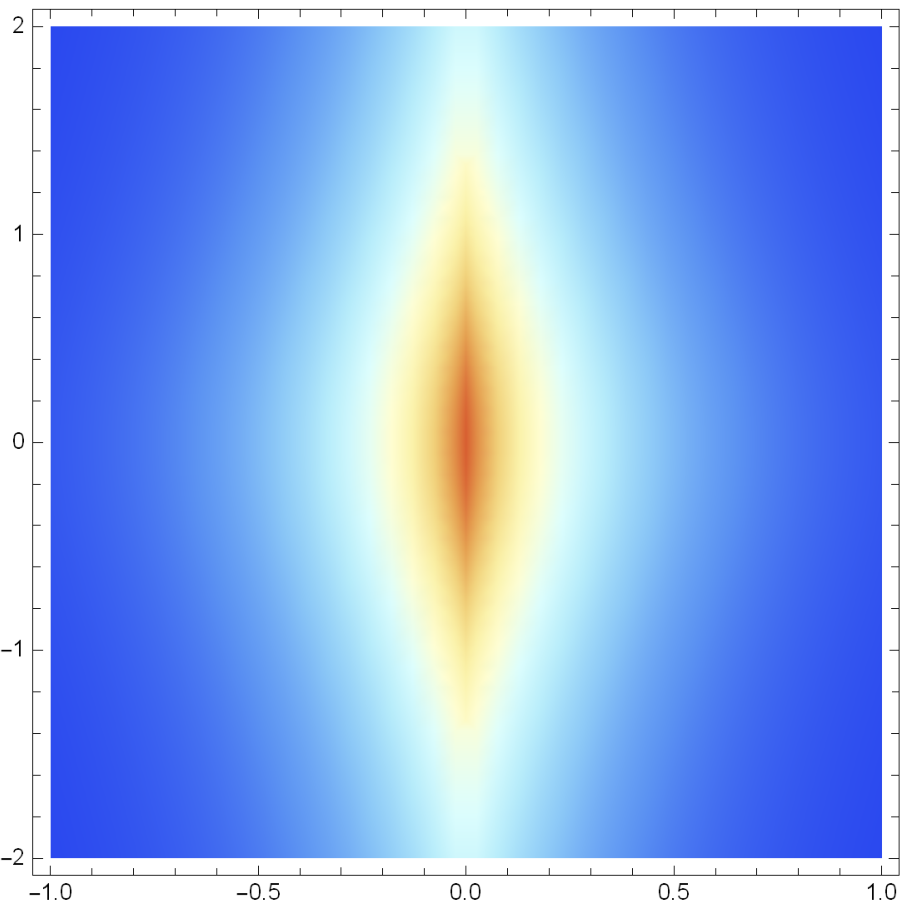}
		\caption{\footnotesize{3d plot and density plot of the absolute values of the ambiguity function of the 1-sided exponential. While it is well-concentrated in the time domain, the localization in the frequency domain is not satisfying, as the decay is only $\mathcal{O}(\frac{1}{\omega})$. This is caused by the discontinuity of the complex exponential (as smoothness on the time-side corresponds to fast decay in the frequency domain).}}
	\end{figure}
	
	On the other hand, it follows from the Hudson-Lieb Theorem \ref{thm_Hudson_Lieb} that if $f = \pm f^\vee$, i.e., $f$ is either an odd or even function, then $A(f, f^\vee) = \pm A f$ can only be positive if $f$ is a generalized Gaussian.
	\begin{itemize}
		\item If $f = f^\vee$, i.e., $f$ is even, then $W(f, f^\vee) = W f$. Also, $W f$ is real-valued and $Wf(0,0) = 2^d \norm{f}_2^2$. By Hudson's theorem the Wigner transform is only positive for Gaussians. Therefore, in this case the Wigner distribution must possess zeros if $f$ is not a generalized Gaussian. Now, by the algebraic relation given by Lemma \ref{lem_Wigner_ambi_algebraic} the same holds for the ambiguity function $Af$.
		\item If $f = - f^\vee$, i.e., $f$ is odd, then $W(f, f^\vee) = - Wf$ and, $W (f,f^\vee)(0,0) = 2^d \norm{f}_2^2$. Therefore, as $Wf$ is real-valued, the Wigner distribution $W(f,f^\vee)$ is real-valued and by the Hudson-Lieb theorem either possesses zeros or $f$ is a generalized Gaussian. Again, by Lemma \ref{lem_Wigner_ambi_algebraic} the same is true for the ambiguity function $Af$.
	\end{itemize}
	This shows that if the ambiguity function $Af$ and cross Wigner distribution $W(f,f^\vee)$ are zero-free, then $f$ does naturally not possess ``simple" symmetries, meaning that it cannot be an even or an odd function, unless it is a Gaussian.
	\flushright $\diamond$
\end{example}

Curiously, this prototype example was (almost) already given by Augustus Janssen more than 20 years earlier in 1996. In \cite{Jan96}, Janssen already computed the ambiguity function of the one-sided exponential (evaluated only at integer points, i.e., $(x, \omega) \in \Z^2$). However, it was not noticed that a zero-free ambiguity function was already given in \cite{Jan96} until the appearance of \cite{GroJamMal19}. The authors re-discovered this example independently as the one-sided exponential is also a prototype example of so-called totally positive functions of finite type. These functions have attracted considerable attraction in recent years in Gabor analysis and we will encounter them later on as well.

\begin{example}
	Next, we will compute the Wigner distribution of the point measure or delta distribution $\delta_0$, where $\langle \delta_x,f \rangle = \overline{f(x)}$. For $f \in \mathcal{S}(\Rd)$, by using \eqref{eq_Wigner_tensor}, we have
	\begin{align}
		\langle W(\delta_0, \delta_0), f \rangle & = \langle \F_2 \mathcal{T}_s (\delta_0 \otimes \overline{\delta_0}), f \rangle\\
		& = \langle \delta_0 \otimes \overline{\delta_0}, \mathcal{T}_s^{-1} \F_2^{-1} f \rangle\\
		& = \overline{\mathcal{T}_s^{-1} \F_2^{-1} f(0,0)}
	\end{align}
	Using the formula \eqref{eq_symcoord_shift_inv} for $\mathcal{T}_s^{-1}$, we obtain
	\begin{align}
		\overline{\mathcal{T}_s^{-1} \F_2^{-1} f(0,0)} & = \overline{\int_{\Rd} f \left( \tfrac{x+y}{2} , \omega \right) e^{2 \pi i (x-y) \cdot \omega} \, d\omega} \ \Bigg|_{x=y=0}\\
		& = \int_{\Rd} \overline{f(0,\omega)} \, d\omega\\
		& = \langle \delta_0 \otimes \mathbf{1}, f \rangle,
	\end{align}
	where $\mathbf{1}$ is the constant 1-function. Therefore, $W \delta_0 = \delta_0 \otimes \mathbf{1} \in \mathcal{S}'(\Rd)$.
	\flushright{$\diamond$}
\end{example}

\section{The Poisson Summation Formula}
The Poisson summation formula is a highly useful tool in time-frequency analysis, but also in many other branches of mathematics. It relates the periodization of a function $f$ to a Fourier series on the torus $\T^d$ with Fourier coefficients obtained from the Fourier transform on $\Lt$.

\begin{definition}\label{def_periodization}
	Given a function $f$, its $\alpha$-periodization is defined by
	\begin{equation}\label{eq_periodization}
		\mathcal{P}_\alpha f(t) = \sum_{k \in \Z^d} f(t + \alpha k), \qquad t \in \Rd, \alpha > 0.
	\end{equation}
\end{definition}
We note that if $f \in L^1(\Rd)$, then $\mathcal{P}_\alpha f \in L^1(\T^d)$. If $\alpha = 1$, we will omit the index and write
\begin{equation}
	\mathcal{P} f(t) = \mathcal{P}_1 f(t).
\end{equation}
Also, we have the following property.
\begin{lemma}\label{lem_periodization_trick}
	If $f \in L^1(\Rd)$, then for all $\alpha > 0$ we have
	\begin{equation}
		\int_{\Rd} f(t) \, dt = \int_{[0,\alpha]^d} \left( \mathcal{P}_\alpha f(t) \right) \, dt.
	\end{equation}
\end{lemma}
\begin{proof}
	We note that the cubes $\alpha k + [0,\alpha]^d$ form a (disjoint) partition of $\Rd$ (up to an overlap of measure 0).
	We compute
	\begin{align}
		\int_{\Rd} f(t) \, dt & = \sum_{k \in \Z^d} \int_{\alpha k + [0,\alpha]^d} f(t) \, dt\\
		& = \int_{[0,\alpha]^d} \left( \sum_{k \in \Z^d} f(t + \alpha k) \right) \, dt.
	\end{align}
	As $f \in L^1(\Rd)$, the exchange of summation and integration is justified by Fubini's theorem.
\end{proof}

We are now ready to prove the Poisson summation formula in its standard version.
\begin{proposition}[Poisson Summation Formula]\label{pro_PSF_Zd}
	Let $f$ be continuous and $\varepsilon > 0$, $C > 0$. Assume that $|f(t)| \leq C (1+|t|)^{-d-\varepsilon}$ and $|\widehat{f}(\omega)| \leq C (1+|\omega|)^{-d-\varepsilon}$. Then
	\begin{equation}\label{eq_PSF_Zd}
		\sum_{k \in \Z^d} f(t + k) = \sum_{l \in \Z^d} \widehat{f}(l) e^{2 \pi i l \cdot t}.
	\end{equation}
	The above identity holds point-wise for all $t \in \Rd$ and both sums converge absolutely for all $t \in \Rd$.
\end{proposition}
\begin{proof}
	The decay conditions imply that $f$ as well as $\widehat{f}$ are in $L^1(\Rd)$ and by assumption (and the Lemma of Riemann-Lebesgue) both are continuous. As $f \in L^1(\Rd)$ we have that $\sum_{k \in \Z^d}f(t+k) = \mathcal{P} f(t) \in L^1(\T^d)$. We want to check the Fourier coefficients of $\mathcal{P} f$.
	\begin{align}
		\F_{\T^d} (\mathcal{P} f)(l) & = \int_{\T^d} \mathcal{P} f(t) e^{-2 \pi i l \cdot t} \, dt\\
		& = \int_{\T^d} \left( \sum_{k \in \Z^d} f(t+k) e^{-2 \pi i l \cdot(t+k)} \right) \, dt\\
		& = \int_{\Rd} f(t) e^{-2 \pi i l \cdot t} \, dt = \F f(l) = \widehat{f}(l).
	\end{align}
	By assumption $\sum_{l \in \Z^d} | \widehat{f}(l)| < \infty$, we see that $\mathcal{P} f$ has an absolutely converging Fourier series
	\begin{equation}
		\mathcal{P} f(t) = \sum_{k \in \Z^d} f(t+k) = \sum_{l \in \Z^d} \widehat{f}(l) e^{2 \pi i l \cdot t}.
	\end{equation}
\end{proof}
\begin{remark}
	\begin{enumerate}[(i)]
		\item The conditions in Proposition \ref{pro_PSF_Zd} are stronger than needed for the point-wise equality to hold, but they ensure absolute convergence of both series as well as the point-wise equality. The conditions on $f$ and $\widehat{f}$ under which the formula holds point-wise have been worked out, e.g., by Gröchenig in \cite{Gro_Poisson_1996}.
	
		\item Note that we assume that $f \in L^1(\Rd)$, which, by the Lemma of Riemann-Lebesgue, implies that $\widehat{f}$ is continuous. On the Fourier side, we also assume that $\widehat{f} \in L^1(\Rd)$, which also implies that $\F^{-1} \widehat{f}$ is continuous. Many authors hence drop the condition on $f$ being continuous as they implicitly assume that $f(t) = \F^{-1} \widehat{f}(t)$.
		
		\item If we replace the absolute convergence in Proposition \ref{pro_PSF_Zd} by convergence in $\Lt$ and pointwise equality by equality almost everywhere, we obtain a weaker version of the Poisson summation formula.
		
		\textit{If $\sum_{k \in \Z^d} f(t+k) \in L^2(\T^d)$ and $\sum_{k \in \Z^d} |\widehat{f}(k)|^2<\infty$, then \eqref{eq_PSF_Zd} holds almost everywhere.}
		\item The Poisson summation formula can also be written for arbitrary lattices $\L \subset \Rd$ and their duals. For a lattice $\L =  A \Z^d$, $A \in GL(\R,d)$, the dual lattice is given by $\L^\perp = A^{-T} \Z^d$ and the Poisson summation formula reads
		\begin{equation}
			\sum_{\l \in \L} f(t + \l) = \vol(\L)^{-1} \sum_{\l^\perp \in \L^\perp} \widehat{f}(\l^\perp) e^{2 \pi i \l^\perp \cdot t}
		\end{equation}
		\item In analytic number theory the Poisson summation formula yields the following functional equation for the Jacobi theta function $\theta(\tau) = \sum_{k \in \Z} e^{\pi i \tau k^2}$, $\tau \in \H$;
		\begin{equation}
			\theta(\tau) = \sqrt{\tfrac{i}{\tau}} \ \theta( -\tfrac{1}{\tau} )
		\end{equation}
		This functional equation was already used by Bernhard Riemann to come up with the functional equation for the Riemann zeta function $\zeta(s)$.
	\end{enumerate}
\end{remark}

Using Proposition \ref{pro_PSF_Zd} for the function $M_\omega T_x f(t)$ at $t=0$, we obtain
\begin{align}
	\sum_{k \in \Z^d} f(k-x)e^{2\pi i \omega \cdot k} & = \sum_{k \in \Z^d} M_\omega T_x f(k)
	= \sum_{l \in \Z^d} \F(M_\omega T_x f)(l)\\
	& = \sum_{l \in \Z^d} e^{2 \pi i \omega \cdot x} M_{-x} T_\omega \F f(l)
	= \sum_{l \in \Z^d} \widehat{f}(l-\omega) e^{2 \pi i x \cdot (\omega - l)},
\end{align}
or, likewise, by exchanging the order of translation and modulation we obtain
\begin{equation}
	\sum_{k \in \Z^d} f(k-x) e^{2 \pi i \omega (k-x)} = \sum_{l \in \Z^d} \widehat{f}(l-\omega) e^{-2 \pi i x \cdot l}.
\end{equation}

\subsection{The Whittaker-Nyquist-Kotelnikov-Shannon Sampling Theorem}\label{sec_WNKS}
We are now going to state and prove a sampling theorem for band-limited functions. The information theoretical problem is as follows. Assuming that a continuous signal is band-limited, how many discrete measurements does one need in order to be able to completely determine the signal by its samples?

The solution to this question is often (only) referred to as the Nyquist-Shannon Sampling Theorem. However, it is also known as the Whittaker-Nyquist-Kotelnikov-Shannon (WNKS) Sampling Theorem \footnote{Often, only the contributions of the information theorists H.~Nyquist (1928) and C.~Shannon (1949) are mentioned, but the mathematical foundations were known already decades before to E.~Whittaker (around 1915). Probably the problem had been treated mathematically already even earlier. For a long time the work of V.~Kotelnikov from 1933 was not known outside Russia and it became more prominent only in the 1950s, after Shannon's work. Kotelinkov's work was probably the first work which treated the problem of sampling a continuous function in an information theoretical context \cite{Kot33} (see also \cite[Chap.~2]{Zay93})}.

\begin{definition}
	A function $f \in \Lt[]$ is called band-limited with band limit $\frac{\mathsf{B}}{2}$ and band width $\mathsf{B}$ if its Fourier transform $\widehat{f}$ is supported on the interval $[-\tfrac{\mathsf{B}}{2}, \tfrac{\mathsf{B}}{2}]$, i.e., $\widehat{f}(\omega) = 0$ for all $|\omega| > \frac{\mathsf{B}}{2}$.
\end{definition}
We will now state and prove the sampling theorem for the case $\mathsf{B} = 1$, but the general case follows in the same manner.
\begin{theorem}[WNKS Sampling Theorem]
	Let $f$ be a band-limited function with band limit $\frac{1}{2}$. In particular
	\begin{equation}\label{eq_WNKS_f}
		f(t) = \int_{-1/2}^{1/2} \widehat{f}(\omega) e^{2 \pi i t \cdot \omega} \, d\omega,
	\end{equation}
	with $\widehat{f} \in L^2([-\frac{1}{2}, \frac{1}{2}]) \left( \subset L^1([-\frac{1}{2}, \frac{1}{2}] \right)$ and, hence, $f$ is also continuous. Then $f$ can be reconstructed from its samples on $\Z$ by the formula
	\begin{equation}
		f(t) = \sum_{k \in \Z} f(k) \frac{\sin(\pi (t-k))}{\pi(t-k)}, \quad t \in \R.
	\end{equation}
	The series converges absolutely and uniformly on compact subsets of $\R$.
\end{theorem}
\begin{proof}
	Note that $\widehat{f}$ coincides with its periodization on $[-\frac{1}{2}, \frac{1}{2}]$, i.e.,
	\begin{equation}
		\widehat{f}(\omega) = (\mathcal{P}\widehat{f}(\omega)) \, \indicator_{[-\frac{1}{2}, \frac{1}{2}]}(\omega) = \sum_{l \in \Z} \widehat{f}(\omega + l) \indicator_{[-\frac{1}{2}, \frac{1}{2}]}(\omega).
	\end{equation}
	We apply the Poisson summation formula (in reverse) to obtain
	\begin{equation}
		\sum_{l \in \Z} \widehat{f}(\omega + l) = \sum_{k \in \Z} f(k) e^{-2 \pi i k \omega}.
	\end{equation}
	Next, we use the fact that $f$ is the inverse Fourier transform of $\widehat{f}$ to apply the Poisson summation formula;
	\begin{align}
		f(t) & = \F^{-1} (\widehat{f})(t) = \int_\R \widehat{f}(\omega) e^{2 \pi i \omega t} \, d\omega = \int_\R \sum_{l \in \Z} \widehat{f}(\omega + l) \indicator_{[-\frac{1}{2}, \frac{1}{2}]}(\omega) e^{2 \pi i \omega t} \, dt\\
		& = \int_\R \sum_{k \in \Z} f(k) e^{-2 \pi i k \omega}  \indicator_{[-\frac{1}{2}, \frac{1}{2}]} (\omega) e^{2 \pi i t \omega}\, d \omega = \sum_{k \in \Z} f(k) \underbrace{\int_{\R} \indicator_{[-\frac{1}{2}, \frac{1}{2}]}(\omega) e^{2 \pi i (t-k) \omega} \, d \omega}_{= \F^{-1} \left(\indicator_{[-\frac{1}{2},\frac{1}{2}]}\right)(t-k)}\\
		& = \sum_{k \in \Z} f(k) \frac{\sin(\pi (t-k))}{\pi (t-k)}
	\end{align}
	Our arguments all work in the $\Lt[]$-sense, however, by applying the Cauchy-Schwarz inequality (for the space $\ell^2(\Z)$) to our statement, we actually see that the convergence is absolutely and uniformly on compact subsets of $\R$.
\end{proof}
The value $\frac{1}{2}$, i.e., $\mathsf{B} = 1$ was of course chosen on purpose to obtain a statement and a proof in a simple form. More generally, if $\widehat{f}$ is supported on the interval $[-\frac{\mathsf{B}}{2}, \frac{\mathsf{B}}{2}]$, then
\begin{align}
	f(t) & = \sum_{k \in \Z} f(\tfrac{k}{\mathsf{B}}) \sinc(\mathsf{B}(t-\tfrac{k}{\mathsf{B}}))\\
%	= \sum_{k \in \Z} f(\tfrac{k}{\mathsf{B}}) \sinc(\tfrac{\mathsf{B} t - k}{\mathsf{B}})\\
	& = \sum_{k \in \Z} f(\mathsf{T} k) \sinc\left(\frac{t - \mathsf{T} k}{\mathsf{T}}\right),
\end{align}
where we set $\mathsf{T} = \frac{1}{\mathsf{B}}$. The value of $\mathsf{T}$ then gives the time between two successive samples. The rate $\frac{1}{\mathsf{T}}$ is the sampling rate and in the previous case is exactly $\mathsf{B}$, which is known as the Nyquist rate. This is the minimum rate at which we need to sample in order to reconstruct $f$ completely (see, e.g., \cite[Chap.~2]{Zay93}) by means of a sampling function, in this case the $\sinc$ function. In general, assuming a band-limited signal with band limit $\frac{\mathsf{B}}{2}$, we distinguish 3 cases:
\begin{enumerate}[(i)]
	\item $\mathsf{T} > \frac{1}{\mathsf{B}}$: undersampling, no guaranteed reconstruction
	\item $\mathsf{T} = \frac{1}{\mathsf{B}}$: critical sampling, reconstruction is possible with the WKNS sampling formula, but losing 1 sample leads to undersampling
	\item $\mathsf{T} < \frac{1}{\mathsf{B}}$: oversampling, reconstruction is possible, but the samples are linearly dependent. Note that the reconstruction formula $f(t) = \sum_{k \in \Z} f(\mathsf{T} k) \sinc\left(\frac{t - \mathsf{T} k}{\mathsf{T}}\right)$ still holds, because we extended the frequency-band $[-\frac{1}{2\mathsf{T}}, \frac{1}{2\mathsf{T}}] \supset [-\frac{\mathsf{B}}{2}, \frac{\mathsf{B}}{2}]$. Due to the linear dependence of the samples it is still possible to reconstruct if (some) samples are lost.
\end{enumerate}
\begin{figure}[ht]
		\subfigure[Undersampling of $\sinc(t)^2$. The sampling rate is only half the Nyquist rate, hence we can only recover $\sinc(t)$, but not $\sinc(t)^2$. Note that $\sinc(t)$ and $\sinc(t)^2$ coincide at the sampling points $k \in \Z$.]
		{
			\includegraphics[width=.45\textwidth]{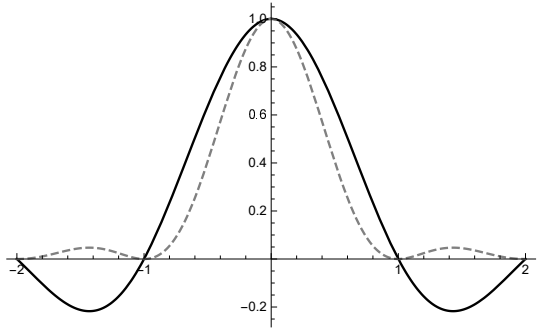}
		}
		\hfill
		\subfigure[Reconstruction of $\sinc(t)^2$ using the WKNS sampling theorem.]
		{
			\includegraphics[width=.45\textwidth]{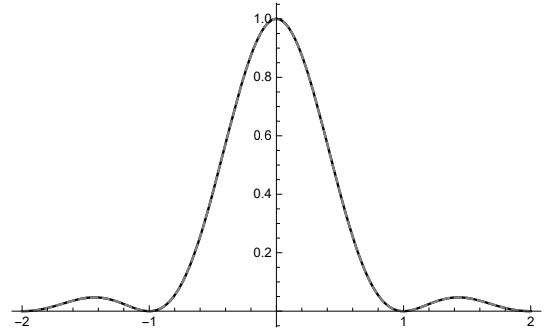}
		}
		\\
		\subfigure[Reconstruction of $\sinc(t)^2$ at the Nyquist rate with the lost sample $f(0)$.]
		{
			\includegraphics[width=.45\textwidth]{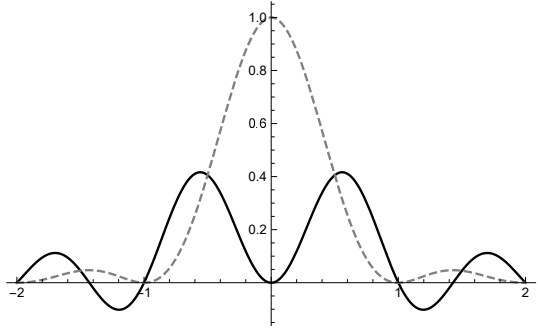}
		}
		\hfill
		\subfigure[Reconstruction of $\sinc(t)^2$ at the Nyquist rate with the lost sample $f(\tfrac{1}{2})$.]
		{
			\includegraphics[width=.45\textwidth]{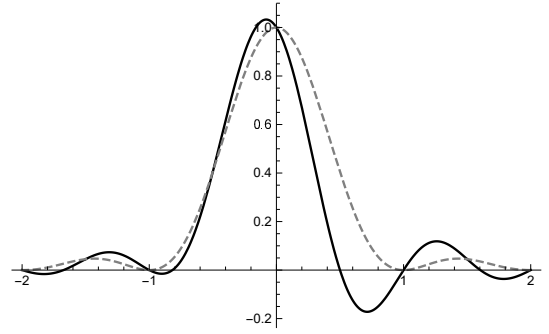}
		}
		\caption{\footnotesize{The function $\sinc(t)^2$ (gray, dashed) and some un-/successful reconstructions (black, solid) using the WKNS sampling theorem. Note that $\F(\sinc(t)^2)(\omega)$ is supported on $[-1,1]$ and therefore the Nyquist rate is $\mathsf{B} = 2$.}}
	\end{figure}

A further generalization of the WKNS sampling theorem is obtained by shifting the spectrum of the band-limited function $f$ by $\omega_0$, i.e.,
\begin{equation}
	\supp(T_{\omega_0}\widehat{f}) \subset [\omega_0 - \frac{\mathsf{B}}{2}, \omega_0 + \frac{\mathsf{B}}{2}]. 
\end{equation}
So, if $\supp(\widehat{f}) \subset [-\frac{\mathsf{B}}{2}, \frac{\mathsf{B}}{2}]$, then $T_{\omega_0} \widehat{f}$ is supported on $[\omega_0 - \frac{\mathsf{B}}{2}, \omega_0 + \frac{\mathsf{B}}{2}]$. Since, by \eqref{eq_TFshifts_FT} we have
\begin{equation}
	M_\omega = \F^{-1} T_\omega \F,
\end{equation}
it follows that $\F^{-1} T_{\omega_0} \widehat{f} = M_{\omega_0} \F^{-1} \widehat{f} = M_{\omega_0} f$.

Now, define the function $f_1(t) = M_{\omega_0} f(t)$, which has Fourier transform $\widehat{f_1}(\omega) = T_{\omega_0} \widehat{f}(\omega)$. We may apply the Poisson summation formula to $M_{-\omega_0}f_1(t) = f(t)$;
\begin{equation}
	f(t) = e^{-2 \pi i \omega_0 t} f_1(t) = \sum_{k \in \Z} e^{-2 \pi i \omega_0 \mathsf{T} k} f_1(\mathsf{T} k) \sinc\left( \tfrac{t-\mathsf{T}k}{\mathsf{T}}\right)
\end{equation}
Hence, the function $f_1$ which is band-limited within the interval $[\omega_0 - \frac{\mathsf{B}}{2}, \omega_0 + \frac{\mathsf{B}}{2}]$ can be written as
\begin{equation}
	f_1(t) = \sum_{k \in \Z} f_1(\mathsf{T} k) \sinc\left( \tfrac{t-\mathsf{T}k}{\mathsf{T}}\right) e^{2 \pi i \omega_0 (t-\mathsf{T} k)}
\end{equation}

Now, assume that $\supp(\widehat{f}) \subset \bigcup_{l=1}^N I_l$, where $I_l = \omega_l + [-\frac{B_l}{2}, \frac{B_l}{2}]$, $\omega_l \in \R$ and
\begin{equation}
	|I_m \cap I_n| = 0,
\end{equation}
for $l \neq m$. Then we can only apply the WKNS sampling theorem directly to a single frequency band $\omega_m + I_m$. To apply it to the whole signal we need to cut the spectrum into single frequency bands;
\begin{equation}
	\widehat{f}(\omega) =\sum_{l \in \Z} \widehat{f}(\omega) \indicator_{I_l}(\omega) = \sum_{l \in \Z} \widehat{f_l}(\omega).
\end{equation}
By setting $f_l(t) = \int_{I_l} \widehat{f_l}(\omega) e^{2 \pi i t \omega} \, d\omega$ we have a band-limited function with $\supp(\widehat{f_l}) \subset I_l$ which can be recovered by the sampling theorem.
\begin{equation}
	f_l(t) = \sum_{k \in \Z} f_l(\mathsf{T}_l k) \sinc \left( \frac{t-\mathsf{T}_l k}{\mathsf{T}_l}\right)e^{2 \pi i \omega_l (t- \mathsf{T}_l k)}.
\end{equation}
To recover the original signal, we need to ``glue" together the band-limited functions $f_l$. If
\begin{equation}
	f(t) = \int_\R \widehat{f}(\omega) e^{2 \pi i t \omega} \, d\omega,
\end{equation}
then we also have $f(t) = \sum_{l=1}^N \int_{I_l} \widehat{f_l}(\omega) e^{2 \pi i t \omega} \, d\omega = \sum_{l=1}^N f_l(t)$ and we get
\begin{equation}
	f(t) = \sum_{l=1}^N \sum_{k \in \Z} f_l(\mathsf{T}_l k) \sinc \left( \frac{t-\mathsf{T}_l k}{\mathsf{T}_l} \right) e^{2 \pi i \omega_l(t-\mathsf{T}_l k)}.
\end{equation}

Lastly, assume $f \in L^1(\R) \cap C(\R)$ is given as the Fourier transform of an $L^1(\R) (\cap \Lt[])$ function $\widehat{f}$ (so $f \in \Lt[]$ as well), so it need not be band-limited. Then, we may slice the Fourier domain into frequency bands $I_l = l + [-\frac{1}{2}, \frac{1}{2}]$ of band-width 1, so $\bigcup_{l \in \Z} I_l = \R$ (and $|I_m \cap I_n| = 0$). As
\begin{equation}
	\widehat{f}(\omega) = \sum_{l \in \Z} \widehat{f}(\omega) \indicator_{I_l}(\omega) = \sum_{l \in \Z} \widehat{f_l}(\omega),
\end{equation}
we may apply the WKNS sampling theorem to every frequency band $I_l$ and obtain
\begin{align}
	f(t) & = \sum_{l \in \Z} \sum_{k \in \Z} f_l(k) \sinc(t-k) e^{2 \pi i l(t-k)}\\
	& = \sum_{l \in \Z} \sum_{k \in \Z} f_l(k) \sinc(t-k) e^{2 \pi i l t} \underbrace{e^{-2 \pi i l k}}_{=1},
\end{align}
where $f_l(\omega) = \int_{I_l} \widehat{f_l}(\omega) e^{2 \pi i t \omega} \, d\omega$. So, we need to perform an inverse Fourier transform on every frequency band to obtain a part of the signal $f$ which is then band-limited on $I_l$. We note that a band-limited function actually has infinite support in the time domain, as we will see in the next section. So, even if the signal $f$ has a finite time duration, none of the pieces $f_l$ will have a finite time duration as they are all band-limited.

In the $\Lt[]$-sense we can actually write
\begin{equation}\label{eq_WKNS_non-bandlimited}
	f = \sum_{k,l \in \Z} f_l(k) M_l T_k \sinc = \sum_{k,l \in \Z} \langle f, M_l T_k \sinc \rangle M_l T_k \sinc = \sum_{k,l \in \Z} V_{\sinc} f(k,l) \, \pi(k,l) \sinc.
\end{equation}
The equality of the coefficients $f_l(k) = \langle f, M_l T_k \sinc \rangle$ follows because the system $\{M_l T_k \sinc \mid k,l \in \Z\}$ is actually orthonormal with dense span in $\Lt[]$. Thus, the coefficients $f_l(k)$ in the WKNS expansion are therefore uniquely given by the inner products of $f$ with the (integer) time-frequency shifted $\sinc$-function, or likewise, by sampling the STFT $V_{\sinc}f$ on (the integer lattice) $\Z^2$. However, this can also be confirmed by a direct computation.
\begin{align}
	f_l(t) & = \int_{\R} \widehat{f}(\omega) \indicator_{l+[-\frac{1}{2}, \frac{1}{2}]}(\omega) e^{2 \pi i t \omega} \, d\omega = \F^{-1}(\widehat{f} \, T_l \indicator_{[-\frac{1}{2}, \frac{1}{2}]})(t) = f(t) * (M_l \sinc)(t)\\
	& \stackrel{\eqref{eq_STFT_notation}}{=} e^{2 \pi i t l} V_{\sinc} f(t,l),
\end{align}
as $\overline{\sinc}^\vee(t) = \sinc(t)$. Hence, sampling $f_l$ on $\Z$ yields
\begin{equation}
	f_l(k) = \underbrace{e^{2 \pi i k l}}_{=1} V_{\sinc} f(k,l), \quad k \in \Z.
\end{equation}

%We note that the classical WKNS sampling theorem for band-limited functions on $[-\frac{1}{2}, \frac{1}{2}]$ can be obtained as a special case of \eqref{eq_WKNS_non-bandlimited}.
%\begin{equation}
%	f(t) = \sum_{k,l \in \Z} f_l(k) \sinc(t-k) e^{2 \pi i l t} = \sum_{k \in \Z} f_0(k) \sinc(t-k),
%\end{equation} 
%as $f_l(t) \equiv 0$ for all $l \in \Z \backslash\{0\}$. Now, note that $f_0(t) = \int_{-1/2}^{1/2} \widehat{f_0}(\omega) e^{2 \pi i t \omega} \, d\omega = f(t)$ by definition of the band-limited function $f$. On the other hand
%\begin{equation}
%	f = \sum_{k \in \Z} \langle f, M_0 T_k \sinc \rangle M_0 T_k \sinc = \sum_{k \in \Z} V_{\sinc} f(k,0) \, \pi(k,0) \sinc.
%\end{equation}

\begin{remark}
	Beating the Nyquist rate and still obtaining perfect reconstruction is possible, if further restrictions of the signal are known. One such restriction might be that the (digitized) signal is ``sparse", i.e., having few non-zero entries, in some domain (e.g., with respect to a certain basis). The concept of sparsity is excessively used in the area of compressed sensing. However, this would lead too far away from the intended topics of this course and will hence not be discussed any further.
\end{remark}

\section{Uncertainty Principles}
Classical uncertainty principles are statements about the pair $(f,\widehat{f})$, including inequalities on their supports or a vanishing theorem. We will introduce uncertainty principles in a non-rigorous and rather descriptive way as stated in \cite{Gro03_FeiStr}. They can be summed up in the following metatheorems. 
\newcounter{tmp}
\begingroup
\setcounter{tmp}{\value{theorem}}% store current value of theorem counter
\setcounter{theorem}{0} %assign desired value to theorem counter
\renewcommand\thetheorem{\Alph{theorem}}% locally redefine the representation of the theorem counter
\begin{mtheorem}\label{mthm_A}
	A function $f$ and its Fourier transform $\widehat{f}$ cannot both be simultaneously small.
\end{mtheorem}

\begin{mtheorem}\label{mthm_B}
	A function $f$ occupies an area of at least 1 in the time-frequency plane (or phase space).
\end{mtheorem}

\begin{mtheorem}\label{mthm_C}
	Every time-frequency representation comes with its own uncertainty principle.
\end{mtheorem}
\endgroup

\setcounter{theorem}{\thetmp} % restore value of theorem counter
The above metatheorems should be understood in a heuristic sense rather than in a precise mathematical way. At this point, we remark that the Nyquist sampling rate might be understood as a reverse uncertainty principle. We called the case $\mathsf{T} > \frac{1}{\mathsf{B}}$ undersampling. Formulated differently, we might put it as
\begin{equation}
	\mathsf{T} \mathsf{B} > 1.
\end{equation}
Hence, on average we have less than 1 sample per unit square in the time-frequency plane, which is too little information to recover a (band-limited) signal.

In order to obtain rigorous mathematical statements, we need precise definitions of the terms ``small" and ``to occupy an area". Metatheorem \ref{mthm_C} stresses the fact that we cannot beat the ``classical" uncertainty principles by choosing a different ``time-frequency representation". To be more precise, the idea is to take an uncertainty principle of type \ref{mthm_A} or \ref{mthm_B} for the pair $(f,\widehat{f})$ and replace it by a statement about the STFT or the Wigner distribution. The statement might be a (support) inequality or a vanishing result.

The size of a function is usually measured by some $L^p$-norm. In the case that $p = 2$, we also speak of the energy. Also, decay conditions might be used. In order to describe how much area a function occupies, we need the concept of the essential support of a function and, already established, the notion of the time-frequency plane.

We start with the very classical Heisenberg-Pauli-Weyl uncertainty principle.
\begin{theorem}[Heisenberg-Pauli-Weyl Uncertainty Principle]\label{thm_HPW_UP}
	Let $f \in \Lt[]$ and $a,b \in \R$. Then
	\begin{equation}
		\left( \int_\R (x-a)^2 |f(x)|^2 \, dx \right)^{1/2} \left( \int_\R (\omega-b)^2 |\widehat{f}(\omega)|^2 \, d\omega \right)^{1/2} \geq \frac{1}{4 \pi} \norm{f}^2_2 \, .
	\end{equation}
\end{theorem}
\begin{proof}
%	Also, without loss of generality, we may assume that
%	\begin{equation}
%		a = \frac{1}{\norm{f}_2^2} \int_{\R} x |f(x)|^2 \, dx
%		\qquad \textnormal{ and } \qquad
%		b = \frac{1}{\norm{\widehat{f}}_2^2}\int_{\R} \omega |\widehat{f}(\omega)|^2 \, d\omega.
%	\end{equation}
%	% UE
%	This can be seen by computing the critical points of
%	\begin{equation}
%		a \mapsto \int_\R (x-a)^2 |f(x)|^2 \, dx .
%	\end{equation}
%	We write
%	\begin{align}
%		\int_\R (x-a)^2 |f(x)|^2 \, dx & = \langle (x-a) f, (x-a) f \rangle\\
%		& = \norm{x f}_2^2 - 2 a \langle x f, f \rangle + a^2 \norm{f}^2
%	\end{align}
%	Finding the minimum of the quadratic polynomial in $a$ is not hard.
%	\begin{equation}
%		\dfrac{d}{d a} \left( \norm{x f}_2^2 - 2 a \langle x f, f \rangle + a^2 \norm{f}^2 \right) = 0.
%	\end{equation}
%	This yields the critical point $a \mapsto \frac{\langle x f, f \rangle}{\norm{f}_2^2}$. Computing the second derivative shows that we have a minimum and since we deal with a quadratic polynomial in one variable, this minimum is global.
%	
%	It is no restriction to assume that
%	\begin{equation}
%		\norm{f}_2 = \norm{\widehat{f}}_2 = 1.
%	\end{equation}
%	In this case, $|f(x)|^2$ and $|\widehat{f}(\omega)|^2$ can be interpreted as probability densities and, hence, $a$ and $b$ are the expectation values of random variables.
	
	We note that by looking at the function
	\begin{equation}
		f_{a,b}(x) = e^{-2 \pi i b \cdot x} f(x+a) = M_{-b} T_{-a} f(x),
	\end{equation}
	we may assume that $a = b = 0$. Further, we need some necessary technical assumptions. For the following calculations we actually need to assume that $x f(x) \in \Lt[]$ and $\omega \widehat{f}(\omega) \in \Lt[]$, so $f', \widehat{f}' \in \Lt[]$ as well. However, on the other hand, if either of the assumptions is not met, then the inequality is trivially true. An (unmotivated) integration by parts shows that
	\begin{equation}
		\int_\R \left(x f(x)\right) \overline{f'(x)} \, dx = x |f(x)|^2 \Bigg|_{x = - \infty}^\infty - \int_\R \left( |f(x)|^2 + x f'(x) \overline{f(x)} \right) \, dx.
	\end{equation}
	Assuming $\lim_{x \to \pm \infty} x |f(x)|^2 = 0$, which holds for any Schwarz function, the above equation can be re-arranged in the following way
	\begin{equation}
		\int_\R |f(x)|^2 \, dx = - \int_\R x f(x) \overline{f'(x)} \, dx - \int_\R x \overline{f(x)} f'(x) \, dx = - 2 \Re \left( \int_\R x f(x) \overline{f'(x)} \, dx \right),
	\end{equation}
	as $x \in \R$. We note the simple estimate
	\begin{equation}
		\left| - 2 \Re \left(\int_\R x f(x) \overline{f'(x)} \, dx\right) \right| \leq 2 \left| \int_\R x f(x) \overline{f'(x)} \, dx \right|.
	\end{equation}
	Now, we use the Cauchy-Schwarz inequality to obtain
	\begin{equation}
		\left| \int_\R x f(x) \overline{f'(x)} \, dx \right| = |\langle xf, f' \rangle| \leq \norm{xf}_2 \norm{f'}_2 = \left(\int_\R x^2 |f(x)|^2 \, dx \right)^{1/2} \left( \int_\R |f'(x)|^2 \, dx \right)^{1/2}.
	\end{equation}
	From the last three calculations it follows that
	\begin{equation}
		\int_\R |f(x)|^2 \, dx \leq 2 \left(\int_\R x^2 |f(x)|^2 \, dx \right)^{1/2} \left( \int_\R |f'(x)|^2 \, dx \right)^{1/2}.
	\end{equation}
	By combining Plancherel's theorem and the action of the Fourier transform on derivatives \eqref{eq_FT_derivative} we get
	\begin{equation}
		\left( \int_\R |f'(x)|^2 \, dx \right)^{1/2} = \left( \int_\R |(2 \pi i \omega) \widehat{f}(\omega)|^2 \, d\omega \right)^{1/2}.
	\end{equation}
	Combining the last lines, we obtain
	\begin{equation}
		\frac{\norm{f}_2^2}{4 \pi} \leq \norm{xf}_2 \norm{\omega\widehat{f}}_2
	\end{equation}
\end{proof}
We note that in the statement of Theorem \ref{thm_HPW_UP} we have a factor of $2 \pi$ due to our choice of the normalization of the Fourier transform. Often, the uncertainty principle is stated with a factor $\frac{1}{2}$ due to a different convention of normalizing the Fourier transform. In quantum mechanics, the statement usually also involves Planck's reduced constant $\hbar$. We will make a little excursion into the field of quantum mechanics now.

\textit{Excursion}. The following paragraphs are the author's interpretation of quantum mechanical results and it should be emphasized that the author does not possess a solid background on the subject.

In quantum mechanics, the (energy -- or density -- of the) Fourier transform is something which can be measured and, hence, depends on measure(able) units. One way to write it is the following.
\begin{equation}
	\F_\hbar f(p) = \left( \frac{1}{2 \pi \hbar} \right)^{d/2} \int_{\Rd} f(x) e^{- \frac{i}{\hbar} p \cdot x} \, dx
\end{equation}
Now, assuming $\norm{f}_2 = 1$, the expression $|f(x)|^2$ is the probability density of the chance to find a particle at a certain position, or rather in an interval. The measure unit is meter [m]. The expression $\hbar =  \frac{h}{2 \pi}$ is Planck's constant $h$ divided by $2 \pi$. It is named after physicist Max Planck\footnote{Max Planck was awarded the Nobel Prize in Physics ``in recognition of the services he rendered to the advancement of Physics by his discovery of energy quanta" in 1919 for the year 1918. During the selection process in 1918, the Nobel Committee for Physics decided that none of the year's nominations met the criteria as outlined in the will of Alfred Nobel. According to the Nobel Foundation's statutes, the Nobel Prize can in such a case be reserved until the following year, and this statute was then applied. Max Planck therefore received his Nobel Prize for 1918 one year later, in 1919.\\ \url{https://www.nobelprize.org/prizes/physics/1918/summary/}} in recognition of his discovery that the energy of harmonically oscillating systems can only change in integer multiples of some smallest portion proportional to the oscillating frequency. Planck's constant describes this proportionality factor; $\Delta E = \hbar \omega$ ($\Delta E$ being the change of energy, $\omega$ being the circle frequency). Hence, $h$ measure energy times time, [J s], which can also be expressed as [m$^2$ kg/s]. Lastly, we note that $p$, hence, has measure unit [m kg/s] which is velocity times mass, also known as momentum. Hence, $|\widehat{f}|^2$ (with the normalization induced by $\F_\hbar$) is the probability density that the particle has a certain momentum. For $d = 1$, Theorem \ref{thm_HPW_UP} is usually stated as
\begin{equation}
	\Delta_f x \Delta_f p \geq \frac{\hbar}{2}.
\end{equation}
The interpretation is that a particle's position and momentum cannot be measured to arbitrary precision at the same time. It is highly important to understand that this is not a restriction due to measurement errors, it is imposed by a mathematical theorem! Admittedly, expressed as a mathematical theorem the uncertainty principle is much more prosaic and loses some of its mystery. It is a simple inequality for standard deviations.

A model where it finds applications is in a simple model of the hydrogen atom. The hydrogen atom consists of a core with one proton and it has an electron which oscillates around it. In the simplest model, the electron only moves within one direction ($\R$) and satisfies the Schrödinger equation of the harmonic oscillator. Then $a = \int_\R x |f(x)|^2 \, dx$ is the expected position of the electron and $b = \int_\R p |\widehat{f}(p)|^2 \, dp$ is its expected momentum. Finally, $\Delta_f x$ and $\Delta_f p$ are the standard deviations of the probability densities $|f(x)|^2$ and $|\widehat{f}(p)|^2$.

If $f \in \Lt$ for general $d \in \N$ (the physically most relevant case being $d = 3$ or $d = 3N$ for $N$ particles), then the uncertainty principle holds for conjugate coordinates, i.e.,
\begin{equation}
	\Delta_f x_k \Delta_f p_k \geq \frac{\hbar}{2}, \qquad k = 1, \ldots, d.
\end{equation}

We note that in time-frequency analysis the pair $(x,p)$ of position and momentum is simply exchanged by the pair $(x,\omega)$. In quantum mechanics, the point $(x,p) \in \R^{2d}$ is said to be a point in phase-space, which is the same concept as the time-frequency plane. Finally, by setting $\hbar = \frac{1}{2 \pi}$ we obtain the Fourier transform as used in time-frequency analysis. Mathematically, there is hence no big difference between the fields of quantum mechanics and time-frequency analysis.
\begin{flushright}
	$\diamond$
\end{flushright}

The Heisenberg-Pauli-Weyl uncertainty principle can more generally be derived from an inequality about (non-commuting) self-adjoint operators on a Hilbert space. For this, let $A$ and $B$ be two linear operators on a Hilbert space $\mathcal{H}$, i.e.,
\begin{align}
	A: & \mathcal{H} \to \mathcal{H} & B: & \mathcal{H} \to \mathcal{H}\\
	& f \mapsto Af & & f \mapsto Bf
\end{align}
We denote the commutator of these operators by
\begin{equation}
	[A,B] = AB - BA.
\end{equation}
\begin{lemma}\label{lem_UCP_AB}
	Let $A$ and $B$ be (possibly unbounded) self-adjoint operators on $\mathcal{H}$. Then
	\begin{equation}
		\norm{(A-a)f}_\mathcal{H} \norm{(B-b)f}_\mathcal{H} \geq \tfrac{1}{2} | \langle [A,B] f, f \rangle_\mathcal{H}|,
	\end{equation}
	for all $a,b \in \R$ and for all $f$ in the domain of $AB$ and $BA$. Equality holds if and only if $(A-a)f = i c(B-b)f$ for some $c \in \R$.
\end{lemma}
\begin{proof}
	We rewrite the commutator and use the fact that $A$ and $B$ are self-adjoint to obtain
	\begin{align}
		\langle [A,B]f, f\rangle_\mathcal{H} & = \langle \left( (A-a)(B-b)-(B-b)(A-a) \right) f, f \rangle_\mathcal{H}\\
		& = \langle (B-b)f,(A-a)f \rangle_\mathcal{H} - \langle (A-a)f, (B-b)f \rangle_\mathcal{H}\\
		& = \langle (B-b)f,(A-a)f \rangle_\mathcal{H} - \overline{\langle (B-b)f, (A-a)f \rangle_\mathcal{H}}\\
		& = 2 \Im \left(\langle (B-b)f, (A-a)f \rangle_\mathcal{H} \right).
	\end{align}
	Now, we use the Cauchy-Schwarz inequality and get
	\begin{align}
		|\langle [A,B]f, f\rangle_\mathcal{H}| & \leq 2 |\langle (B-b)f, (A-a)f \rangle_\mathcal{H}|\label{eq_UCP_AB}\\
		& \leq 2 \norm{(A-a)f}_{\mathcal{H}} \norm{(B-b)f}_{\mathcal{H}}\label{eq_UCP_AB_CSI}
	\end{align}
	Equality in \eqref{eq_UCP_AB} holds if and only if $\langle(B-b)f, (A-a)f\rangle_\mathcal{H}$ is purely imaginary, and equality in \eqref{eq_UCP_AB_CSI} (Cauchy-Schwarz inequality) holds if and only if $(A-a)f = \l (B-b)f$, $\l \in \C$. Together, this implies that for equality to hold $\l = i c$ with $c \in \R$.
\end{proof}

To deduce Theorem \ref{thm_HPW_UP} from Lemma \ref{lem_UCP_AB} we need to introduce the position and momentum operators\footnote{The factor $\frac{1}{2 \pi i}$ for the momentum operator $P$ clearly comes from the normalization of the Fourier transform. In quantum mechanics, it is often stated as $P = -i \hbar \partial_x$.} on $\Lt[]$, given by
\begin{equation}
	X f(x) = x f(x)
	\quad \text{ and } \quad
	Pf(x) = \tfrac{1}{2 \pi i} f'(x).
\end{equation}
The position operator $X$ gives back the expectation value of the position of a particle in the quantum state $f$;
\begin{equation}
	\langle X f, f\rangle = \int_{\R} X f(x) \overline{f(x)} \, dx = \int_\R x |f(x)|^2 \, dx.
\end{equation}
The momentum operator gives back the expectation value of the momentum of the particle in the quantum state $f$;
\begin{equation}
	\langle P f, f \rangle = \langle \F P f, \F f \rangle \stackrel{\eqref{eq_FT_derivative}}{=} \langle X \F f, \F f \rangle = \int_\R \omega |\widehat{f}(\omega)|^2 \, d\omega.
\end{equation}
Having settled the notation, we note that the Schwartz space $\mathcal{S}(\R)$ is a common domain for $X, P, XP, PX$. The largest common domain for these operators is the subspace
\begin{equation}
	\{f \in \Lt[] \mid x f(x), f', xf' \in \Lt[]\}.
\end{equation}
Also, we note that these operators are the infinitesimal generators of the modulation  and translation groups, i.e.,
\begin{equation}
	\dfrac{d}{d\omega} M_\omega f \Big|_{\omega=0} = 2 \pi i X f
	\quad \text{ and } \quad
	\dfrac{d}{dx} T_x f \Big|_{x=0} = -2 \pi i P f.
\end{equation}
To apply Lemma \ref{lem_UCP_AB} we need to check that $X$ and $P$ are indeed self-adjoint on $\Lt[]$. We start with the position operator and compute
\begin{equation}
	\langle X f, f \rangle = \int_\R x f(x) \overline{f(x)} \, dx = \int_\R f(x) \overline{(x f(x))} \, dx = \langle f, X f \rangle,
\end{equation}
as $x = \overline{x}$ because $x \in \R$. Now, we compute for the momentum operator
\begin{equation}
	\langle P f, f \rangle = \langle \F P f, \F f \rangle = \langle X \F f, \F f \rangle = \langle \F f, X \F f \rangle = \langle f, \F^{-1} X \F f \rangle = \langle f, P f \rangle,
\end{equation}
where we used Parseval's identity and \eqref{eq_FT_derivative}\footnote{Alternatively, one may write the inner product as an integral and use integration by parts to establish the result for a suitable subspace of $\Lt[]$, e.g., $\mathcal{S}(\R)$.}.

\begin{proof}[Second proof of Theorem \ref{thm_HPW_UP}]
	Assume that $f$ is in the domain of $X,P,XP,PX$. Then the commutator of $X$ and $P$, applied to $f$, is given by
	\begin{equation}
		[X,P]f(x) = \tfrac{1}{2 \pi i} (x f'(x)-(xf)'(x)) = \tfrac{1}{2 \pi i} \left(x f'(x) - (f(x)+x f'(x))\right) = -\tfrac{1}{2 \pi i} f(x).
	\end{equation}
	Thus, by Lemma \ref{lem_UCP_AB} we have
	\begin{equation}\label{eq_UCP_XP}
		\tfrac{1}{4 \pi} \norm{f}_2^2 = \tfrac{1}{2} |\langle [X,P]f, f \rangle| \leq \norm{(X-a)f}_2^2 \norm{(P-b)f}_2^2.
	\end{equation}
	The first factor is
	\begin{equation}
		\norm{(X-a)f}_2^2 = \int_\R (x-a)^2 |f(x)|^2 \, dx
	\end{equation}
	and by Plancherl's theorem and \eqref{eq_FT_derivative} the second factor is
	\begin{equation}
		\norm{(P-b)f}_2^2 = \norm{\F((P-b)f)}_2^2 = \norm{(X-b)\F f}_2^2 = \int_ \R (\omega-b)^2 |\widehat{f}(\omega)|^2 \, d\omega.
	\end{equation}
\end{proof}
Equality in \eqref{eq_UCP_XP} holds if and only if $(P-b)f = i c(X-a)f$ for some $c \in \R$. This is the differential equation
\begin{equation}
	f'-2 \pi i b f = -2 \pi c(x-a) f.
\end{equation}
The solutions to this differential equation are exactly scalar multiples of the time-frequency shifted, dilated standard Gaussian $g_0$;
\begin{equation}
	M_b T_a e^{-\pi c t^2} = e^{-\pi c (t-a)^2}e^{2 \pi i b t}.
\end{equation}
As we assume $f \in \Lt[]$, we need $c > 0$ and can indeed write $M_b T_a D_{1/\sqrt{c}} g_0$ (up to a scalar).

A mathematical consequence which can be drawn from the uncertainty principle is the following result.
\begin{corollary}
	For $f \in \Lt[]$ we have
	\begin{equation}
		\norm{X f}_2^2 + \norm{P f}_2^2 \geq \tfrac{1}{2\pi} \norm{f}_2^2,
	\end{equation}
	with equality if and only if $f(t) = c e^{-\pi t^2}$.
\end{corollary}
\begin{proof}
	We apply the inequality
	\begin{equation}
		2 \alpha \beta \leq \alpha^2 + \beta^2
	\end{equation}
	to the uncertainty principle with $a=b=0$ and $\alpha = \norm{Xf}_2$ and $\beta = \norm{P f}_2$. For equality we need $\alpha = \beta$ and equality in the uncertainty principle. This means we need
	\begin{equation}
		Pf = \pm i X f
		\quad \Longleftrightarrow \quad
		f'(x) = \pm 2 \pi x f(x)
	\end{equation}
	and the only $\Lt[]$ solutions are the Gaussians $c \, e^{-\pi x^2}$, $c \in \C$.
\end{proof}

Theorem \ref{thm_HPW_UP} can also be formulated by means of the Rihaczek distribution.
\begin{equation}
	\left( \iint_{\R^2} |(x-a) (\omega-b) Rf(x,\omega)|^2 \, d(x,\omega) \right)^{1/2} \geq \frac{1}{4 \pi} \norm{f}_2^2.
\end{equation}
We already identified the Rihaczek distribution as a very crude time-frequency representation. This is due to the fact that it still represents the temporal and spectral behavior separately, i.e., it is (up to a phase factor) a simple tensor product of $f$ and $\widehat{f}$. Nonetheless, Metatheorem \ref{mthm_C} tells us that the inability to exactly measure the instantaneous time-frequency behavior of a function $f$ is not due to the crude representation of $Rf$.

The next results treats $f$ and $\widehat{f}$ separately as well. In this case, we measure the size of a function by its support. This leads to the qualitative statement that a function cannot simultaneously be time- and band-limited.

\begin{theorem}[Benedicks]
	Assume $f \in L^1(\Rd) \cap L^2(\Rd)$. If
	\begin{equation}
		|\supp(f)| \ |\supp(\widehat{f})| < \infty,
	\end{equation}
	then $f \equiv 0$.
\end{theorem}
\begin{proof}
	Let $A = \{ x \in \Rd \mid f(x) \neq 0 \}$ and $B = \{\omega \in \Rd \mid \widehat{f}(\omega) \neq 0 \}$. By using dilations of $f$, we may assume that $|A| < 1$. Then, by using Lemma \ref{lem_periodization_trick}, we obtain
	\begin{align}
		| \{ x \in \T^d \mid \exists k \in \Z^d \colon f(k+x) \neq 0 \} |
		& = |\{ x \in \T^d \mid \sum_{k \in \Z^d} \indicator_A(k+x) \geq 1 \}| \\
		& \leq \int_{\T^d} \sum_{k \in \Z^d} \indicator_A(k+x) \, dx
		= \int_{\Rd} \indicator_A(x) \, dx = |A| < 1 = | \T^d |.
	\end{align}
	Therefore, the set $N_x = \{x \in \T^d \mid f(k+x) = 0, \forall k \in \Z\}$ has positive measure.
	
	Next, we observe that
	\begin{equation}
		\int_{\T^d} \sum_{l \in \Z^d} \indicator_B(l+\omega) \, d\omega = \int_{\Rd} \indicator_B(\omega) \, d\omega = |B|.
	\end{equation}
	Assume $|B| < \infty$. This implies that $\sum_{l \in \Z^d} \indicator_B(l+\omega)$ is finite (almost everywhere) on $\T^d$. Hence, for (almost all) $\omega \in \T^d$, the set $\widehat{N_\omega} = \{l \in \Z^d \mid \widehat{f}(l+\omega) \neq 0 \}$ is finite.
	
	For fixed $\omega \in \T^d$ consider the periodization
	\begin{equation}
		\mathcal{P}(M_{-\omega} f)(x) = \sum_{k \in \Z^d} f(k+x)e^{-2 \pi i \omega \cdot (k+x)} = \sum_{k \in \Z^d} T_{-x} M_{-\omega}f(k),
	\end{equation}
	which is a well defined Fourier series for $f \in L^1(\Rd) \cap \Lt$. By the Poisson summation formula we have
	\begin{equation}
		\sum_{k \in \Z^d} (T_{-x} M_{-\omega} f)(k) = \sum_{l \in \Z^d} \F(T_{-x} M_{-\omega} f)(l) = \sum_{l \in \Z^d} M_{x} T_{-\omega} \widehat{f}(l).
	\end{equation}
	Hence, the $l$-th Fourier coefficient of the above periodization equals
	\begin{equation}
		\F(M_{-\omega} f)(l) = \widehat{f}(l+\omega), \qquad l \in \Z^d.
	\end{equation}
	Since the set $\widehat{N_\omega}$ is finite, this expression is non-zero for only finitely many $l \in \Z^d$, which means that $\mathcal{P}(M_{-\omega}f)$ equals a trigonometric polynomial $Q$. However, as $\mathcal{P}(M_{-\omega} f)$ vanishes for all $x \in N_x$, the trigonometric polynomial $Q$ vanishes on $N_x$ as well. But a trigonometric polynomial that vanishes on a set of positive measure must be identically zero, i.e., $Q \equiv 0$. Therefore, the Fourier coefficients $\widehat{f}(\omega+l)$ are 0 for all $l \in \Z^d$ and (almost all) $\omega \in \T^d$. Hence, $\widehat{f} \equiv 0$ and $f \equiv 0$.
\end{proof}
Benedicks' uncertainty principle holds more generally for $f \in L^p(\Rd)$. It also leads to a qualitative uncertainty principle for the Rihaczek distribution:

\begin{center}
	\textit{If $|\supp(Rf)| < \infty$, then $f \equiv 0$.}
\end{center}

This motivates the question, whether such uncertainty principles can be obtained for $V_gf$, $A(f,g)$ and $W(f,g)$ as well. The statement in Metatheorem \ref{mthm_C} already suggests that we will obtain similar uncertainty principles in these cases.

We will need the following results to prove qualitative uncertainty principles for the STFT, the ambiguity function and the Wigner distribution.
\begin{lemma}\label{lem_Vgf_product}
	Assume $f_1, f_2, g_1, g_2 \in \Lt$. Then
	\begin{equation}
		\F \left( V_{g_1} f_1 \, \overline{V_{g_2} f_2} \right)(x, \omega) = \left(V_{f_2}f_1 \, \overline{V_{g_2} g_1}\right)(-\omega, x).
	\end{equation}
\end{lemma}
\begin{proof}
	Recall that the Fourier transform of the STFT is given by the (cross-)Rihaczek distribution, i.e.,
	\begin{equation}
		\widehat{V_g f}(\xi, \eta) = e^{2 \pi i \xi \cdot \eta} f(-\eta) \overline{\widehat{g}(\xi)}.
	\end{equation}
	Furthermore, by using Parseval's identity we obtain
	\begin{equation}
	\F (f \, \overline{g})(\omega) = \int_{\Rd} f(t) \overline{g(t)} e^{-2 \pi i \omega \cdot t} \, dt = \langle f, M_\omega g \rangle = \langle \widehat{f}, T_\omega \widehat{g} \rangle.
	\end{equation}
	Since$V_{g_1} f_1$ and $V_{g_2} f_2$ are in $\Lt[2d]$ (because $f_1,f_2,g_1,g_2 \in \Lt$), their product is in $L^1(\R^{2d})$ and we can write
	\begin{align}
		\F \left( V_{g_1} f_1 \overline{V_{g_2} f_2} \right)(x, \omega)
		& = \langle \widehat{V_{g_1}f_1}, \, T_{(x,\omega)} \widehat{V_{g_2} f_2} \rangle_{\Lt[2d]}\\
		& = \iint_{\R^{2d}} \widehat{V_{g_1} f_1}(\xi, \eta) \overline{\widehat{V_{g_2} f_2}(\xi-x, \eta-\omega)} \, d(\xi,\omega)\\
		& = \iint_{\R^{2d}} \left(e^{2 \pi i \xi \cdot \eta} f_1(-\eta) \overline{\widehat{g_1}(\xi)} \right)\\
		& \qquad \qquad \overline{\left(e^{2 \pi i (\xi-x) \cdot (\eta-\omega)} f_2(-(\eta-\omega)) \overline{\widehat{g_2}(\xi-x)} \right)} \, d(\xi, \eta)\\
		& = \int_{\Rd} f_1(\eta) \overline{f_2(\eta+\omega)} e^{-2 \pi i x \cdot \eta} \, d\eta \, \overline{\int_{\Rd} \widehat{g_1}(\xi) \overline{\widehat{g_2}(\xi-x)} e^{-2 \pi i \omega \cdot (\xi-x))} \, d\xi}
	\end{align}
	The first integral is already the desired expression $V_{f_2} f_1(-\omega, x)$. The second integral can also be expressed by
	\begin{equation}
		\langle \widehat{g_1}, T_x M_\omega \widehat{g_2} \rangle = \langle g_1, M_x T_{-\omega} g_2 \rangle = V_{g_2} f_2 (-\omega, x).
	\end{equation}
\end{proof}
We obtain the following corollary.
\begin{corollary}\label{cor_Vgf_products}
	For $f, g \in \Lt$, the function
	\begin{equation}
		F(x,\omega) = e^{2 \pi i x \cdot \omega} V_g f(x, \omega) \, V_g f(-x,-\omega)
	\end{equation}
	satisfies
	\begin{equation}
		\widehat{F}(x,\omega) = F(-\omega, x).
	\end{equation}
	Furthermore, the function
	\begin{equation}
		F_{(\xi,\eta)} (x,\omega) = e^{2 \pi i x \cdot \omega} V_g \left(M_\eta T_\xi f\right)(x,\omega) \, V_g \left(M_\eta T_\xi f\right)(-x,-\omega)
	\end{equation}
	satisfies
	\begin{equation}
		\widehat{F_{(\xi, \eta)}}(x, \omega) = F_{(\xi, \eta)} (-\omega, x), \qquad
		\forall (\xi, \eta) \in \R^{2d}.
	\end{equation}
\end{corollary}
\begin{proof}
	It suffices to prove the first assertion, the second follows by replacing $f$ by $M_\eta T_\xi f$. We note that
	\begin{equation}
		V_g f(-x, -\omega) e^{2 \pi i x \cdot \omega} = \int_{\Rd} f(t) \overline{g(t+x)} e^{2 \pi i \omega \cdot (t+x)} \, dt = \int_{\Rd} f(t-x) \overline{g(t)} e^{2 \pi i \omega \cdot t} \, dt = \overline{V_f g(x, \omega)}.
	\end{equation}
	Therefore, we can write
	\begin{equation}
		F(x, \omega) = \left( V_g f \, \overline{V_f g} \right) (x,\omega).
	\end{equation}
	By using Lemma \ref{lem_Vgf_product} we see that
	\begin{equation}
		\widehat{F}(x,\omega) = \F \left( V_g f \, \overline{V_f g} \right)(x, \omega) = \left(V_g f \, \overline{V_f g}\right)(-\omega, x) = F(-\omega, x).
	\end{equation}
\end{proof}

\begin{theorem}
	Let $f,g \in \Lt$, then the following are equivalent:
	\begin{enumerate}[(i)]
		\item $|\supp(V_g f)| < \infty$\label{thm_UP_Vgf}
		\item $|\supp(A(f,g))| < \infty$\label{thm_UP_Afg}
		\item $|\supp(W(f,g))| < \infty$\label{thm_UP_Wfg}
		\item Either $f \equiv 0$ or $g \equiv 0$ (or both).\label{thm_UP_0}
	\end{enumerate}
\end{theorem}
\begin{proof}
	The equivalence of \eqref{thm_UP_Vgf}, \eqref{thm_UP_Afg} and \eqref{thm_UP_Wfg} follows from the algebraic relations in Lemma \ref{lem_Wigner_ambi_algebraic}. Hence, we only need to prove \eqref{thm_UP_Vgf} $\Rightarrow$ \eqref{thm_UP_0} as the other direction is trivially true.
	We use the function from Corollary \ref{cor_Vgf_products}
	\begin{equation}
		F_{(\xi, \eta)}(x, \omega) = e^{2 \pi i x \cdot \omega} \, V_g (M_\eta T_\xi f)(x, \omega) \, V_g (M_\eta T_\xi f)(-x, -\omega).
	\end{equation}
	By the covariance principle, given by formula \eqref{eq_covar} in Proposition \ref{pro_covar}, we see that
	\begin{equation}
		|V_g \left( M_\eta T_\xi f\right)(x,\omega)| = |V_g f(x-\xi, \omega-\eta)|.
	\end{equation}
	By assumption, the support of $V_g f$ has finite measure, which now implies that $F_{(\xi, \eta)}$ has finite measure. By Corollary \ref{cor_Vgf_products}, $\widehat{F_{(\xi, \eta)}}$ must have finite measure as well and we have a pair $(F, \widehat{F})$ where, both, $F$ and $\widehat{F}$ have finite measure. By Benedicks' theorem, this implies that $F_{(\xi, \eta)} \equiv 0$ for all $(\xi, \eta) \in \R^{2d}$. Using the covariance principle again and evaluating $F_{(\xi, \eta)}$ at $(x, \omega) = (0,0)$, we see that
	\begin{equation}
		|F_{(-\xi,-\eta)}(0,0)| = |V_g \left(M_{-\eta} T_{-\xi} f\right)(0,0)|^2 = |V_g f(\xi,\eta)|^2, \qquad \forall (\xi, \eta) \in \R^{2d}.
	\end{equation}
	Particularly, this shows that $V_g f \equiv 0$ and by the isometry property $\norm{V_g f}_2 = \norm{f}_2 \norm{g}_2$, obtained from the orthogonality relations \eqref{eq_OR}, we conclude that either $f \equiv 0$ or $g \equiv 0$ (or both).
\end{proof}
\begin{remark}
	This result was, e.g., established by Janssen \cite{Jan98}, who also proved the following result. Let $H$ be any half-space of $\R^{2d}$, then
	\begin{equation}
		|\{ (x,\omega) \in H \mid Wf(x,\omega) \neq 0 \}| \in \{0, \infty\}.
	\end{equation}
\end{remark}

We will now continue with essential support conditions. Hence, we need to define the essential support of a function.
\begin{definition}
	A function is said to be $\varepsilon$-concentrated on a (measurable) set $T \subset \Rd$, if
	\begin{equation}
		\left(\int_{\Rd \backslash T} |f(t)|^2 \, dt\right)^{1/2} \leq \varepsilon \norm{f}_2^2 .
	\end{equation}
\end{definition}
For small $\varepsilon$, the definition above tells us that most of the energy is concentrated in $T$, which may than be considered as the essential support of $f$. The first idea to modify Benedicks' theorem is to replace the support condition by an essential support condition using the $\varepsilon$-concentration. This leads to the uncertainty principle of Donoho-Stark
\begin{theorem}[Donoho-Stark]
	Let $f \in \Lt$ (not the zero-function) be $\varepsilon_T$-concentrated on $T \subset \Rd$ with $\widehat{f}$ being $\varepsilon_\Omega$-concentrated on $\Omega \subset \Rd$. Then
	\begin{equation}
		|T||\Omega| \geq (1 - \varepsilon_T - \varepsilon_\Omega)^2.
	\end{equation}
\end{theorem}
\begin{proof}
	Without loss of generality we may assume that $T$ and $\Omega$ have finite measure. We introduce the time-limiting and band-limiting operators
	\begin{equation}
		P_T f = \indicator_T f
		\quad \text{ and } \quad
		Q_\Omega f(t) = \F^{-1}(\indicator_\Omega \widehat{f})(x) = \int_\Omega \widehat{f}(\omega) e^{2 \pi i x \cdot \omega} \, d\omega.
	\end{equation}
	Both operators are orthogonal projections on $\Lt$. The range of $P_T$ is $L^2(T, dx)$ and the range of $Q_\Omega$ consists of all $\Lt$ functions with spectrum in $\Omega$, i.e., $\supp(\widehat{f}) \subset \Omega$. With this notation, $f$ is $\varepsilon_T$-concentrated if and only if
	\begin{equation}
		\norm{f - P_T f}_2 \leq \varepsilon_T \norm{f}_2 .
	\end{equation}
	Similarly, $\widehat{f}$ is $\varepsilon_\Omega$-concentrated if and only if
	\begin{equation}
		\norm{f - Q_\Omega f}_2 \leq \varepsilon_\Omega \norm{f}_2.
	\end{equation}
	Since $\norm{Q_\Omega}_{op} \leq 1$, we obtain
	\begin{equation}
		\norm{f - Q_\Omega P_T f}_2 \leq \norm{f - Q_\Omega f}_2 + \norm{Q_\Omega (f - P_T f)}_2 \leq (\varepsilon_\Omega + \varepsilon_T) \norm{f}_2,
	\end{equation}
	and consequently
	\begin{equation}\label{eq_DS_aux1}
		\norm{Q_\Omega P_T f}_2 \geq \norm{f}_2 - \norm{f - Q_\Omega P_T f}_2 \geq (1-\varepsilon_\Omega-\varepsilon_T) \norm{f}_2.
	\end{equation}
	
	Next, we compute the integral kernel and then the Hilbert-Schmidt norm of $Q_\Omega P_T$. We have
	\begin{align}
		Q_\Omega P_T f(x) & = \F^{-1} (\indicator_\Omega \F (P_T f))(x)\\
		& = \int_\Omega\left(\int_T f(t) e^{-2 \pi i \omega \cdot t} \, dt\right)e^{2 \pi i x \cdot \omega} \, d\omega
	\end{align}
	Since $T$ and $\Omega$ have finite support and $f \in L^2(T) \subset L^1(T)$, this double integral converges absolutely and by Fubini we may exchange the order of integration. This yields
	\begin{equation}
		Q_\Omega P_T f(x) = \int_{\Rd} k(x,t) f(t) \, dt,
	\end{equation}
	with integral kernel
	\begin{equation}
		k(x,t) = \indicator_T(t) \int_\Omega e^{2 \pi i(x-t)\cdot \omega} \, d\omega = \indicator_T(t) T_t(\F^{-1} \indicator_\Omega)(x).
	\end{equation}
	The Hilbert-Schmidt norm of $Q_\Omega P_T$ is given by
	\begin{equation}
		\norm{Q_\Omega P_T}_{H.S.}^2 = \iint_{\R^{2d}}|k(x,t)|^2 \, d(x,t).
	\end{equation}
	Since the translation operator $T_t$ and the (inverse) Fourier transform are unitary, we have for fixed $t$ that
	\begin{align}
		\int_{\Rd} |k(x,t)|^2 \, dx & = \indicator_T(t) \norm{T_t \F^{-1} \indicator_\Omega}_2^2\\
		& = \indicator_T(t) \norm{\indicator_\Omega}_2^2\\
		& = |\Omega| \indicator_T(t),
	\end{align}
	and, therefore,
	\begin{equation}\label{eq_DS_aux2}
		\iint_{\R^{2d}} |k(x,t)|^2 \, d(x,t) = |\Omega||T|.
	\end{equation}
	By combining \eqref{eq_DS_aux1} and \eqref{eq_DS_aux2} and the fact that the operator norm is dominated by the Hilbert-Schmidt norm, we obtain
	\begin{align}
		(1-\varepsilon_T-\varepsilon_\Omega)^2 \norm{f}_2^2 & \leq \norm{Q_\Omega P_T f}_2^2\\
		& \leq \norm{Q_\Omega P_T}_{op}^2 \norm{f}_2^2\\
		& \leq \norm{Q_\Omega P_T}_{H.S.}^2 \norm{f}_2^2\\
		& \leq |T||\Omega| \norm{f}_2^2.
	\end{align}
\end{proof}
By letting $\varepsilon$ tend to 0, we obtain a precise statement for Metatheorem \ref{mthm_B}. The essential support conditions can be transferred to the STFT, ambiguity function and Wigner distribution as well. The next result is called the ``Weak Uncertainty Principle for the STFT".
\begin{theorem}[Weak Uncertainty Principle for the STFT]
	Assume $f,g \in \Lt$ with $\norm{f}_2 = \norm{g}_2 = 1$. Let $U \subset \R^{2d}$ and $\varepsilon \geq 0$ such that
	\begin{equation}
		\iint_U |V_gf(x,\omega)|^2 \, d(x,\omega) \geq 1 - \varepsilon.
	\end{equation}
	Then, $|U| \geq 1 - \varepsilon$.
\end{theorem}
\begin{proof}
	The Cauchy-Schwarz inequality implies
	\begin{equation}
		|V_gf(x,\omega)| = | \langle f, M_\omega T_x g \rangle| \leq \norm{f}_2 \norm{g}_2 = 1.
	\end{equation}
	Hence,
	\begin{equation}
		1 - \varepsilon \leq \iint_U |V_g f(x,\omega)|^2 \, d(x,\omega) \leq \norm{V_gf}_\infty^2 |U| \leq |U|.
	\end{equation}
\end{proof}
It is obvious that the same inequality holds if we replace $V_gf$ by $A(f,g)$. By the relation in Lemma \ref{lem_Wigner_ambi_FT} the inequality holds for the Wigner distribution $W(f,g)$ as well.

We note that there are many more uncertainty principles and for an overview, the reader is referred to \cite{Gro03_FeiStr}.

\section{Gabor Systems and Frames}
The following section follows the textbooks of Gröchenig \cite{Gro01} and Christensen \cite{Christensen_2016}.

After our (extensive) preliminary work, we will now study Gabor systems and Gabor frames. The field might be called Gabor analysis and is considered to be a sub-field of (applied) harmonic analysis. It has been a fruitful domain for engineers for decades and has particular applications in wireless communications and is fundamental in the new 5G transmission system. Gabor systems are named after Dennis Gabor who studied such systems in his article \textit{Theory of Communication} \cite{Gab46} in 1946. However, he was not the first person to study these systems, they have been considered already by John von Neumann in the context of (reformulating) quantum mechanics \cite{Neumann_Quantenmechanik_1932} in the 1930s.

\medskip

From a sampling theoretic point of view, we have seen the WKNS sampling theorem for band-limited functions and an extension to non-band-limited function, which involved sampling the STFT with a $\sinc$-window. For general windows $g$, we have only dealt with continuous time-frequency representations, which is fine for many theoretical purposes. However, this is not always satisfactory for practical purposes and we would like to have something similar to the extension of the WKNS sampling theorem. For this purpose, we consider the inversion formula \eqref{eq_inv_STFT} for the STFT
\begin{equation}
	f = \frac{1}{\langle \widetilde{g}, g \rangle} \iint_{R^{2d}} V_gf(x, \omega) M_\omega T_x \widetilde{g} \, d(x,\omega).
\end{equation}
This is a continuous expansion of $f$ with respect to the uncountable system of functions
\begin{equation}
	\{M_\omega T_x \widetilde{g} \mid (x,\omega) \in \R^{2d} \}.
\end{equation}
The coefficients in this continuous expansions are given by the evaluations of $V_gf$ at $(x,\omega)$. We might call $g$ the analysis window and $\widetilde{g}$ the synthesis window.

However, $\Lt$ is a separable Hilbert space and we should be able to find a series expansion of $f \in \Lt$ with respect to a countable set of time-frequency shifts of $\widetilde{g}$. Assume that, in the time-frequency plane, $\widetilde{g}$ has its essential time-frequency support on a set $E \subset \R^{2d}$. Let $(x_1,\omega_1)$ and $(x_2, \omega_2)$ be two neighboring points in the time-frequency plane, then, for $i = 1,2$, $M_{\omega_i} T_{x_i} \widetilde{g}$ has essential time-frequency support $E + (x_i,\omega_i)$. The two supports will have a ``large" overlap for neighboring points, and also, with similar reasoning, the information carried by the coefficients $V_gf(x_i,\omega_i)$ is roughly the same for $i=1,2$. The representation of $f$ by the inversion formula is thus highly redundant and also the interpretation of $V_gf$ as the time-frequency content becomes rather vague.

The goal is, therefore, to obtain a discrete representation of $f$ by a countable set of time-frequency shifts of the synthesis window $\widetilde{g}$, such that, in the time-frequency plane, the essential supports of $M_\omega T_x \widetilde{g}$  have minimal overlap. A first attempt is to replace the integral by a Riemann sum over a sufficiently dense discrete set, representing $f$ as
\begin{equation}\label{eq_expansion_dual}
	f = \mathop{\sum \sum}_{k,l \in \Z^d} \langle f, M_{\beta l} T_{\alpha k} \, g \rangle M_{\beta l} T_{\alpha k} \, \widetilde{g},
\end{equation}
$\alpha$ and $\beta$ sufficiently small.

More generally (or modestly, as one prefers), we could also try to find an expansion of the form
\begin{equation}\label{eq_Gabor_expansion}
	f = \mathop{\sum \sum}_{k,l \in \Z^d} c_{k,l} M_{\beta l} T_{\alpha k} g,
\end{equation}
with coefficients $c_{k,l} = c_{k,l}(f)$ to be determined. For $\alpha, \beta$ chosen such that the essential time-frequency supports $(\alpha k, \beta l) + E$ of $M_{\beta l} T_{\alpha k} g$ are almost disjoint, the coefficients $c_{k,l}$ obtain more meaning as a measure of the time-frequency content of $f$ in the region $(\alpha k, \beta l) + E$. We note that \eqref{eq_Gabor_expansion} can be seen as a generalized Fourier series. Consider the function space $L^2([0,1]^d)$, then any element $f \in L^2([0,1]^d)$ has an expansion of the form
\begin{equation}
	f(t) = \sum_{l \in \Z^d} c_l \, e^{2 \pi i l \cdot t} = \sum_{l \in \Z^d} c_l \, M_l \indicator_{[0,1]^d}(t),
\end{equation}
where the $c_l$ are the Fourier coefficients. The set $\{M_l \indicator_{[0,1]^d} \}$ yields an orthonormal basis for $L^2([0,1]^d)$. Shifting the set by $T_k$, $k \in \Z^d$, we obtain an orthonormal basis for $\Lt$ and have an expansion of the form
\begin{equation}
	\mathop{\sum \sum}_{k,l \in \Z^d} c_{k,l} M_l T_k \indicator_{[0,1]^d}.
\end{equation}
As discussed in the previous sections, the properties of $\indicator_{[0,1]^d}$ are not really satisfactory for time-frequency analysis. Gabor (as well as already earlier von Neumann) considered the Hilbert space $\Lt[]$ and exchanged the indicator function by the Gaussian function and used time-frequency shifts $M_l T_k$ with respect to $\Z^2$, i.e., (k,l) $\in \Z^2$. Expansions of type \eqref{eq_Gabor_expansion} are now usually called Gabor expansions and the coefficients are called Gabor coefficients.

Yet, a third way of discretization is obtained from the interpretation of the spectrogram $|V_g f(x,\omega)|^2$ as an energy density of $f$ in a time-frequency cell centered at $(x,\omega)$. We then need to sample $V_gf$ densely enough so that the energy is preserved under the discretization. This means that, for positive constants $0 < A \leq B < \infty$, the following inequality needs to be fulfilled
\begin{equation}\label{eq_Frame_Zd}
	A \norm{f}_2^2 \leq \mathop{\sum \sum}_{k,l \in \Z^d} |V_gf(\alpha k, \beta l)|^2 \leq B \norm{f}_2^2, \qquad \forall f \in \Lt.
\end{equation}
Determining pairs $(\alpha, \beta) \in \R_+^2$ such that the above inequality holds, is in general a very challenging topic. The meaning of the right-hand side of inequality \eqref{eq_Frame_Zd} is that the sampling operation is continuous on $\Lt$, while the left-hand side of \eqref{eq_Frame_Zd} expresses that $f$ is uniquely determined by the samples of the STFT. We remark that $f$ also depends continuously on the samples $V_g f(\alpha k, \beta l)$. From this point of view, the discretization problem for the STFT is reminiscent of related problems and questions in the theory of band-limited functions, such as the Nyquist-Shannon sampling theorem.

We will see that the three approaches of discretization are equivalent and yield the same answers. However, expansions of type \eqref{eq_expansion_dual} and \eqref{eq_Gabor_expansion} raise new mathematical questions.
\begin{enumerate}[(a)]
	\item In general, the time-frequency atoms $M_{\beta l} T_{\alpha k} g$ are not orthogonal. Therefore, the question pops up in which way the series \eqref{eq_expansion_dual} and \eqref{eq_Gabor_expansion} converge.
	\item Given a window $g$, how can we determine a dual window $\widetilde{g}$ in \eqref{eq_expansion_dual}? More generally, how can we determine the Gabor coefficients in \eqref{eq_Gabor_expansion}?
	\item What are suitable lattice parameters $\alpha$ and $\beta$ such that $f$ is uniquely determined by samples of $V_gf$ on $\alpha \Z^d \times \beta \Z^d$.
	\item How does the uncertainty principle manifest itself in the context of discrete time-frequency representations?
	\item Can we construct orthonormal bases of the form $\{M_{\beta l} T_{\alpha k} g \mid k,l \in \Z^d \}$, where $g$ has ``nice" properties in the time-frequency plane?
\end{enumerate}

Even after these questions may have been answered, the applied signal analysit will still not be happy. The Gabor expansion \eqref{eq_Gabor_expansion} still lives within the infinite-dimensional space $\Lt$. From an engineering point of view, we therefore require another discretization step, which yields a finite-dimensional model of time-frequency analysis and gives numerical algorithms. We will not deal with these questions in the scope of the course. We refer the interested reader to \cite[Chap.~8]{FeiStr98} as a starting point for further reading.

\subsection{Frames}
We start with the general theory of frames. Frames extend the concept of bases, which leads to, in general, non-orthogonal, overcomplete systems, which means that the elements will also not be linearly independent. Motivated by \eqref{eq_Frame_Zd}, we start with the following definition.
\begin{definition}
	Let $\mathcal{H}$ be a (separable) Hilbert space. A set $\{e_\gamma \mid \gamma \in \Gamma\}$ in $\mathcal{H}$ is called a frame if there exist positive constants $0<A\leq B<\infty$ such that for all $f \in \mathcal{H}$ the following inequality holds
	\begin{equation}\label{eq_frame}
		A \norm{f}_\mathcal{H}^2 \leq \sum_{\gamma \in \Gamma} |\langle f, e_\gamma \rangle|^2 \leq B \norm{f}_\mathcal{H}^2.
	\end{equation}
	Any two constants $A,B$ satisfying \eqref{eq_frame} are called frame bounds (or frame constants). If $A = B$ the set $\{e_\gamma \mid \gamma \in \Gamma \}$ is said to constitute a tight frame.
\end{definition}
\begin{example}
	\begin{enumerate}[(i)]
		\item Any orthonormal basis is a (tight) frame with optimal lower and upper frame bound $A=B=1$. Any positive $A<1$ and any finite $B > 1$ gives a lower or upper frame bound, respectively.
		\item The union of any two orthonormal basis is a tight frame with optimal frame bounds $A=B=2$.
		\item The union of any orthonormal basis with $L$ arbitrary unit vectors is a frame with bounds $A=1$ and $B=L+1$.

		\medskip
		We give now examples in the finite dimensional Hilbert space $\Rd$.
		\item Let $\mathcal{H} = \Rd$ and let
		\begin{equation}
			\mathbf{F} = \{v_k = (v_{k,1}, \ldots, v_{k,d}) \mid k = 1, \ldots, K\},
		\end{equation}
		of $K$ column-vectors, $K \geq d$. Then, the so-called frame operator is given by
		\begin{equation}
			S_{\mathbf{F}}: \Rd \to \Rd \qquad f \mapsto S_{\mathbf{F}} f = \sum_{k=1}^K \langle f, v_k \rangle v_k = D_\mathbf{F} C_\mathbf{F} f,
		\end{equation}
		where $C_\mathbf{F}$ and $D_\mathbf{F}$ are the matrices (operators) given by
		\begin{equation}
			C_{\mathbf{F}} =
			\begin{pmatrix}
				v_{1,1} & \ldots & v_{1,d}\\
				& \ddots &\\
				v_{K,1} & \ldots & v_{K,d}
			\end{pmatrix}
			\quad \text{ and } \quad
			D_\mathbf{F} = C_\mathbf{F}^T.
		\end{equation}
		The matrix $C_\mathbf{F}$ is the coefficient matrix and its transpose (adjoint) is the synthesis matrix. Note that $S_\mathbf{F}$ is its own transpose (it is self-adjoint) and that the elements $f$ and $S_\mathbf{F} f$ may in general differ from one another.
		
		Furthermore, let $\sigma_1 \leq \ldots \leq \sigma_d$ denote the singular values of $C_\mathbf{F}$. Then, for any $f \in \Rd$, we have
		\begin{equation}
			\sigma_1^2 \norm{f}^2 \leq \norm{C_\mathbf{F} f}^2 = \langle S_\mathbf{F} f, f \rangle = \sum_{k=1}^K |\langle f, v_k \rangle|^2 \leq \sigma_d^2 \norm{f}^2.
		\end{equation}
	%		% UE
			\item For $N \in \N$, consider the following collection of vectors
			\begin{equation}
				\mathbf{F}_N = \{ \left(\cos(2 \pi k/N), \sin(2 \pi k/N)\right) \mid k = 0, \ldots N-1 \}
			\end{equation}
			
		Let $N = 1$. Then, $\mathbf{F}_1$ does not span $\R^2$ and cannot be a frame. The upper frame bound is $B=1$, whereas the lower frame bound vanishes, i.e., $A = 0$. For $N = 2$, we obtain a unit vector and its negative. Hence, the upper bound is $B=2$ but $\mathbf{F}_2$ does not span $\R^2$ and so the lower bound is $A=0$.
		
		Hence, let $N \geq 3$. Then, the coefficient operator is given by the matrix
		\begin{equation}
			C_{\mathbf{F}_N} =
			\begin{pmatrix}
				\cos(2 \pi 0/N) & \sin(2 \pi 0/N)\\
				\vdots & \vdots\\
				\cos(2 \pi (N-1)/N) & \sin(2 \pi (N-1)/N)
			\end{pmatrix}.
		\end{equation}
		We compute the Gram matrix
		\begin{equation}
			C_{\mathbf{F}_N}^T C_{\mathbf{F}_N} =
			\begin{pmatrix}
				\sum_{k=0}^{N-1} \cos(2 \pi k/N)^2 & \sum_{k=0}^{N-1} \cos(2 \pi k/N) \sin(2 \pi k/N)\\
				\sum_{k=0}^{N-1} \cos(2 \pi k/N) \sin(2 \pi k/N) &\sum_{k=0}^{N-1} \sin(2 \pi k/N)^2
			\end{pmatrix}.
		\end{equation}
		It is easily seen that
		\begin{equation}
			\sum_{k=0}^{N-1} \cos(2 \pi k/N) \sin(2 \pi k/N) = 0,
		\end{equation}
		as the parity of the sine terms makes everything cancel out. Next, we observe that
		\begin{equation}
			\sum_{k=0}^{N-1} \cos(2 \pi k/N)^2 + \sum_{k=0}^{N-1} \sin(2 \pi k/N)^2 = N.
		\end{equation}
		Hence, we only need to compute the value of the sum $\sum_{k=0}^{N-1} \cos(2 \pi k/N)^2$. We use the trigonometric identity $\cos(x)^2 = \frac{1}{2} + \frac{\cos(2x)}{2}$ and obtain
		\begin{equation}
			\sum_{k=0}^{N-1} \cos(2 \pi k/N)^2 =  \frac{1}{2} \sum_{k=0}^{N-1} \left(1 + \cos(2 \pi 2k/N)\right) = \frac{N}{2} +  \frac{1}{2} \sum_{k=0}^{N-1} \cos(2 \pi 2k/N).
		\end{equation}
		Now, we write
		\begin{align}
			\sum_{k=0}^{N-1} \cos(2 \pi 2k/N) & = \Re \left( \sum_{k=0}^{N-1} e^{2 \pi i 2k/N} \right)\\
			& = \frac{1 - \left(e^{4 \pi i/N}\right)^N}{1 - e^{4 \pi i/N}} = 0,
		\end{align}
		for $N \geq 3$, by the formula$\sum_{k=0}^{N-1} q^k = \frac{1-q^N}{1-q}$, $q \neq 1$.
		Therefore,
		\begin{equation}
			\sum_{k=0}^{N_1} \cos(2 \pi k/N)^2 = \sum_{k=0}^{N-1} \sin(2 \pi k/N)^2 = \frac{N}{2}.
		\end{equation}
		It follows that
		\begin{equation}
			C_{\mathbf{F}_N}^T C_{\mathbf{F}_N} =
			\begin{pmatrix}
				\frac{N}{2} & 0\\
				0 & \frac{N}{2}
			\end{pmatrix},
		\end{equation}
		and, so the squared singular values are $\sigma_1^2 = \sigma_2^2 = \frac{N}{2}$. It readily follows that the collections of functions $\mathbf{F}_N$, $N \geq 3$ yield tight frames with frame bounds $A = B = \frac{N}{2}$.
	\end{enumerate}
	\flushright{$\diamond$}
\end{example}

After these first, simple examples, we will now introduce the concept of unconditional convergence and some important operators.
\begin{definition}
	Let $\{e_\gamma \mid \gamma \in \Gamma\}$ be a countable set in a Banach space $\mathcal{B}$. The series $\sum_{\gamma \in \Gamma} e_\gamma$ is said to converge unconditionally to the element $f \in \mathcal{B}$ if for every $\varepsilon > 0$ there exists a finite set $F_0 \subset \Gamma$ such that
	\begin{equation}
		\norm{f - \sum_{\gamma \in F} e_\gamma}_\mathcal{B} < \varepsilon
		\qquad \text{ for all finite set } F \supset F_0 .
	\end{equation}
\end{definition}

Since the index set $\Gamma$ is countable, there exists a bijective map $\pi : \N \to \Gamma$ such that $\Gamma$ can be enumerated by $\N$. Then, the convergence of $\sum_{\gamma \in \Gamma} e_\gamma$ can be defined by the convergence of the partial sums $\sum_{k=1}^N e_{\pi(k)}$, i.e.,
\begin{equation}\label{eq_uncond_conv_N}
	f = \lim_{N \to \infty} \sum_{k=1}^N e_{\pi(k)}.
\end{equation}
This approach encounters the following problems. First, if $\Gamma$ is an unstructured index set, then there is no canonical enumeration $\pi$ and, thus, no natural sequence of partial sums. Second, in general it is not clear whether the limit in \eqref{eq_uncond_conv_N} is independent of the enumeration $\pi$. However, if the series converges unconditionally, then these problems cannot arise. This is shown by the following result.
\begin{proposition}
	Let $\{e_\gamma \mid \gamma \in \Gamma\}$ be a countable set in the Banach space $\mathcal{B}$. Then, the following are equivalent.
	\begin{enumerate}[(i)]
		\item $\sum_{\gamma \in \Gamma} e_\gamma$ converges unconditionally to $f \in \mathcal{B}$.
		\item For every enumeration $\pi : \N \to \Gamma$ the sequence of partial sums $\sum_{k=1}^N e_{\pi(k)}$ converges to $f \in \mathcal{B}$, i.e.,
		\begin{equation}
			\lim_{N \to \infty} \norm{f - \sum_{k=1}^N e_{\pi(k)}}_{\mathcal{B}} = 0 .
		\end{equation}
		In particular, the limit does not depend on the specific enumeration $\pi$.
	\end{enumerate}
\end{proposition}
\begin{proof}
	$(i) \Rightarrow (ii)$: Let $\pi : \N \to \Gamma$ be an enumeration of $\Gamma$ and let $\varepsilon > 0$. Since $\sum_{\gamma \in \Gamma} e_\gamma$ converges unconditionally, there is a finite set $F_0 \subset \Gamma$ such that
	\begin{equation}
		\norm{ f - \sum_{\gamma \in F} e_\gamma}_\mathcal{B} < \varepsilon, \qquad \text{ for all finite } F \supset F_0.
	\end{equation}
	Now, choose $N_0$ large enough such that $F_0 \subset \{\pi(1), \ldots, \pi(N_0) \}$. Then
	\begin{equation}
		\norm{ f - \sum_{k=1}^N e_{\pi(k)}}_\mathcal{B} < \varepsilon, \qquad \text{ for } N \geq N_0.
	\end{equation}
	
	$(ii) \Rightarrow (i)$: Assume that every rearrangement of $\sum_{\gamma \in \Gamma} e_\gamma$ converges to $f \in \mathcal{B}$, but not unconditionally. Then, there exists $\varepsilon > 0$ such that for every finite set $F \subset \Gamma$ there is $F' \supset F$ with
	\begin{equation}
		\norm{ f - \sum_{\gamma \in F'} e_\gamma}_\mathcal{B} \geq \varepsilon .
	\end{equation}
	Now, fix an enumeration $\pi : \N \to \Gamma$. Since $\sum_{k=1}^\infty e_{\pi(k)}$ converges, there is an index $N_0 \in \N$ such that
	\begin{equation}
		\norm{f - \sum_{k=1}^N e_{\pi(k)}}_\mathcal{B} < \frac{\varepsilon}{2}, \qquad \forall N \geq N_0 .
	\end{equation}
	By induction, we can therefore construct a sequence of finite sets $F_n \subset \Gamma$ of cardinality $N_n$ with the following properties;
	\begin{enumerate}[(a)]
		\item $F_n \subset F_{n+1}$ for $n \in \N$
		\item $\norm{ f - \sum_{\gamma \in F_{2n}} e_\gamma}_{\mathcal{B}} > \varepsilon$ for the sets with even index
		\item $F_{2n+1}$ is of the form $\{ \pi(1), \ldots , \pi(N_{2n+1})\}$ where $N_{2n+1} (\geq N_0)$ is chosen large enough so that $F_{2n+1} \supset F_{2n}$. Then
		\begin{equation}
			\norm{ f - \sum_{\gamma \in F_{2n+1}} e_\gamma}_\mathcal{B} = \norm{ f - \sum_{k=1}^{N_{2n+1}} e_{\pi(k)}}_\mathcal{B} < \frac{\varepsilon}{2} .
		\end{equation}
	\end{enumerate}
	Now, we define a new rearrangement $\sigma : \N \to \Gamma$ by enumerating the elements in the finite sets $F_1, F_2 \backslash F_1, \ldots, F_{n+1} \backslash F_n, \ldots$ consecutively. Then, we have
	\begin{align}
		\norm{  \sum_{k=N_{2n}+1}^{N_{2n+1}} e_{\sigma(k)} }_\mathcal{B}
		& = \norm{ \sum_{\gamma \in F_{2n+1}} e_\gamma - \sum_{\gamma \in F_{2n}} e_\gamma}_\mathcal{B}\\
		& \geq \norm{ f - \sum_{\gamma \in  F_{2n}} e_\gamma}_\mathcal{B} - \norm{ f - \sum_{\gamma \in  F_{2n+1}} e_\gamma}_\mathcal{B} > \varepsilon - \frac{\varepsilon}{2} = \frac{\varepsilon}{2}.
	\end{align}
	Therefore, $\sum_{k=1}^\infty e_{\sigma(k)}$ does not converge, which contradicts the assumption.
\end{proof}
If $\Gamma = \Z^d$, then $F_0$ can be taken to be the cube $\{k \in \Z^d \mid |k_j| \leq N, j=1, \ldots, d\}$ or the ball $\{k \in \Z^d \mid |k| \leq N \}$. Thus, unconditional convergence implies convergence of the rectangular and radial partial sums, but the convergence of the rectangular and radial partial sums does in general not imply unconditional convergence.
%
%\begin{example}
%	Since $\{e^{2 \pi i k x} \mid k \in \Z\}$ is an orthonormal basis for $L^2(\T)$, the Fourier series
%	\begin{equation}
%		f(x) = \sum_{k \in \Z} \widehat{f}(k) e^{2 \pi i k x}
%	\end{equation}
%	converges unconditionally in $L^2(\T)$. However, convergence in $L^p(\T)$, $p \neq 2$ is only conditional. The partial sums $f_N(x) = \sum_{|k| \leq N} \widehat{f}(k) e^{2 \pi i k x}$ converge to $f \in L^p(\T)$, but there are rearrangements which do not converge in $L^p(\T)$.
%	\flushright{$\diamond$}
%\end{example}

Often, it is necessary to interchange the action of a linear operator with summation. If a series converges unconditionally, then this interchange is always justified.
\begin{lemma}\label{lem_operator_uncond_conv}
	Let $A$ be a bounded operator between two Banach spaces $\mathcal{B}_1$ and $\mathcal{B}_2$. If $f = \sum_{\gamma \in \Gamma} e_\gamma$ converges unconditionally in $\mathcal{B}_1$, then $\sum_{\gamma \in \Gamma} A e_\gamma$ converges unconditionally in $\mathcal{B}_2$ and
	\begin{equation}
		\sum_{\gamma \in \Gamma} A e_\gamma = A \sum_{\gamma \in \Gamma} e_\gamma = A f.
	\end{equation}
\end{lemma}
\begin{proof}
	Given $\varepsilon > 0$, there is a finite set $F_0 \subset \Gamma$ such that
	\begin{equation}
		\norm{f - \sum_{\gamma \in F} e_\gamma}_{\mathcal{B}_1} < \frac{\varepsilon}{\norm{A}_{op}},
	\end{equation}
	for all $F \supset F_0$. Therefore
	\begin{align}
		\norm{ A f - \sum_{\gamma \in F} A e_\gamma}_{\mathcal{B}_2} = \norm{A (f - \sum_{\gamma \in F} e_\gamma)}_{\mathcal{B}_2} \leq \norm{A}_{op}  \norm{f - \sum_{\gamma \in F} e_\gamma}_{\mathcal{B}_1} < \varepsilon,
	\end{align}
	for $F \supset F_0$. Thus $\sum_{\gamma \in \Gamma} A e_\gamma$ converges unconditionally to the limit $A f$.
\end{proof}
The next lemma is especially useful in the area of Gabor analysis, as it shows that series over sets with product structure, i.e., $\Gamma = \Gamma_1 \times \Gamma_2$, can be seen as iterated series. The statement will only be given for the set $\Z^d \times \Z^d$.
\begin{lemma}
	Suppose that $\sum_{(k,l) \in \Z^{2d}} e_{k,l}$ converges unconditionally to $f \in \mathcal{B}$. Then the partial sum $s_{k,N} = \sum_{|l| \leq N} e_{k,l}$ converges to some element $g_k \in \mathcal{B}$ for each $k \in \Z^d$, and $f = \sum_{k \in \Z^d} g_k$ with unconditional convergence. Likewise, $\sum_{|k| \leq M} e_{k,l}$ converges to some element $h_l \in \mathcal{B}$ for each $l \in \Z^d$ and $f = \sum_{l \in \Z^d} h_l$.
\end{lemma}
\begin{proof}
	As $f = \sum_{k,l \in \Z^d} e_{k,l}$ converges unconditionally, there exist $M_0, N_0 \in \N$ such that
	\begin{equation}
		\norm{f - \sum_{|k| \leq M} \sum_{|l| \leq N} e_{k,l}}_\mathcal{B} < \varepsilon
	\end{equation}
	for every $M \geq M_0$ and $N \geq N_0$. If $F \subset \Z^d$ is an arbitrary finite subset and $N \geq N' \geq N_0$, then by the above estimate
	\begin{align}
		\norm{\sum_{k \in F}  \sum_{N' \leq |l| \leq N} e_{k,l}}_\mathcal{B}
		& = \norm{\sum_{\substack{k \in F \\ \text{or } |k| \leq M_0}} s_{k,N} - \sum_{\substack{k \in F \\ \text{or } |k| \leq M_0}} s_{k,N'}}_\mathcal{B} \\
		& \leq \norm{f - \sum_{\substack{k \in F \\ \text{or } |k| \leq M_0}} s_{k,N'}}_\mathcal{B}
		+ \norm{f - \sum_{\substack{k \in F \\ \text{or } |k| \leq M_0}} s_{k,N}}_\mathcal{B} \\
		& < 2 \varepsilon .
	\end{align}
	Consequently, the sequence $a_N = \sum_{k \in F} s_{k,N}$ is a Cauchy sequence and thus converges in $\mathcal{B}$. In particular, if $F = \{ k \}$, then the partial sums $s_{k,N}$ converge to some element $g_k = \sum_{l \in \Z^d} e_{k,l} \in \mathcal{B}$. Further,
	\begin{equation}
		\lim_{N \to \infty} \sum_{k \in F} s_{k,N} = \sum_{k \in F} g_k .
	\end{equation}
	Next, we show that $\sum_{k \in \Z^d} g_k$ converges unconditionally to $f = \sum_{k,l \in \Z^d} e_{k,l}$. We note that the last estimate (involving the $2 \varepsilon$) above holds uniformly for all finite sets $F \subset \Z^d$. If $F \supset\{ k \in \Z^d \mid |k| \leq M_0 \}$, then, by the above calculations, we have
	\begin{align}
		\norm{f - \sum_{k \in F} g_k}_\mathcal{B}
		& \leq \norm{f - \sum_{k \in F} s_{k,N}}_\mathcal{B}
		+ \norm{\sum_{k \in F} s_{k,N} - \sum_{k \in F} g_k}_\mathcal{B} \\
		& \leq \varepsilon + \limsup_{N \leq N'} \norm{(s_{k,N} - s_{k,N'})}_\mathcal{B} < 3 \varepsilon.
	\end{align}
	Hence, $f = \sum_{k \in \Z^d} g_k$ with unconditional convergence.
\end{proof}
We remark that the converse is not true in general. The convergence of an iterated sum does not necessarily imply the unconditional convergence of the double series.

We continue with some important operators and a study of their properties.
\begin{definition}
	Let $\mathcal{H}$ be a (separable) Hilbert space. For a sequence $(e_\gamma)_{\gamma \in \Gamma}$, the coefficient or analysis operator is given by
	\begin{equation}
		C: \mathcal{H} \to \ell^2(\Gamma), \qquad C f = (\langle f, e_\gamma \rangle)_{\gamma \in \Gamma} .
	\end{equation}
	The synthesis or reconstruction operator is defined for a (finite) sequence $c = (c_\gamma)_{\gamma \in \Gamma}$ by
	\begin{equation}
		D: \ell^2(\Gamma) \to \mathcal{H}, \qquad D c = \sum_{\gamma \in \Gamma} c_\gamma e_\gamma \in \mathcal{H}.
	\end{equation}
	For $f \in \mathcal{H}$, the frame operator on $\mathcal{H}$ is then given by
	\begin{equation}
		S: \mathcal{H} \to \mathcal{H}, \qquad S f = \sum_{\gamma \in \Gamma} \langle f, e_\gamma \rangle e_\gamma .
	\end{equation}
\end{definition}

We will now study the fundamental properties of these operators and start with the following result.
\begin{lemma}
	Let $\Gamma$ be a countable set without accumulation points. Consider the sequence $(e_\gamma)_{\gamma \in \Gamma}$ in the Hilbert space $\mathcal{H}$ and suppose that $\sum_{\gamma \in \Gamma} c_\gamma e_\gamma$ is convergent for all sequences $(c_\gamma)_{\gamma \in \Gamma} \in \ell^2(\Gamma)$. Then
	\begin{equation}
		D: \ell^2(\Gamma) \to \mathcal{H},
		\qquad
		D (c_\gamma)_{\gamma \in \Gamma} = \sum_{\gamma \in \Gamma} c_\gamma e_\gamma
	\end{equation}
	defines a bounded linear operator. The adjoint operator is given by
	\begin{equation}
		D^*: \mathcal{H} \to \ell^2(\Gamma),
		\qquad
		D^*f = (\langle f, e_\gamma \rangle)_{\gamma \in \Gamma}
	\end{equation}
	Furthermore,
	\begin{equation}\label{eq_Bessel_D}
		\sum_{\gamma \in \Gamma} |\langle f, e_\gamma\rangle|^2 \leq \norm{D}_{op}^2 \norm{f}_\mathcal{H}^2, \qquad \forall f \in \mathcal{H}.
	\end{equation}
\end{lemma}
\begin{proof}
	Consider the sequence of bounded linear operators
	\begin{equation}
		D_n : \ell^2(\Gamma) \to \mathcal{H},
		\qquad
		D_n(c_\gamma)_{\gamma \in \Gamma} = \sum_{|\gamma| \leq n} c_\gamma e_\gamma .
	\end{equation}
	Then $D_n \to D$ point-wise as $n \to \infty$. By the Banach-Steinhaus theorem (uniform boundedness principle) $D$ is bounded.
	
	Next, we want to find the adjoint operator $D^*$. Let $f \in \mathcal{H}$, $(c_\gamma)_{\gamma \in \Gamma} \in \ell^2(\Gamma)$, then
	\begin{equation}
		\langle f, D(c_\gamma) \rangle_\mathcal{H}
		= \langle f, \sum_{\gamma \in \Gamma} c_\gamma e_\gamma \rangle_\mathcal{H}
		= \sum_{\gamma \in \Gamma}\overline{c_\gamma} \langle f, e_\gamma\rangle_\mathcal{H}.
	\end{equation}
	We will now present a way to find $D^* f$. As, by assumption, the series $\sum_{\gamma \in \Gamma} \overline{c_\gamma} \langle f, e_\gamma \rangle_\mathcal{H}$ converges for all $(c_\gamma)_{\gamma \in \Gamma} \in \ell^2(\Gamma)$, this implies that $(\langle f, e_\gamma\rangle_\mathcal{H})_{\gamma \in \Gamma} \in \ell^2(\Gamma)$ as well (see, e.g., \cite[p.~145]{Heuser_FA}). Therefore, we can write
	\begin{equation}
		\langle f, D(c_\gamma) \rangle_\mathcal{H} = \langle (\langle f, e_\gamma \rangle_\mathcal{H}), (c_\gamma) \rangle_{\ell^2(\Gamma)}.
	\end{equation}
	Thus, we conclude that
	\begin{equation}
		D^* f = (\langle f, e_\gamma \rangle )_{\gamma \in \Gamma}.
	\end{equation}
	
	The adjoint of a bounded operator is itself bounded, and $\norm{D}_{op} = \norm{D^*}_{op}$. By assumption, we therefore have
	\begin{equation}
		\norm{D^* f}_\mathcal{H}^2 \leq \norm{D}_{op}^2 \norm{f}_\mathcal{H}^2, \qquad \forall f \in \mathcal{H}.
	\end{equation}
\end{proof}
\noindent
We note that we also showed that $D^* = C$ and, consequently, $C^* = D$.

Sequences $(e_\gamma)_{\gamma \in \Gamma}$ for which an inequality of type \eqref{eq_Bessel_D} holds, play a fundamental role in the area of time-frequency analysis, not least because they already fulfill one half of the frame inequality.
\begin{definition}
	A sequence $(e_\gamma)_{\gamma \in \Gamma}$ in a Hilbert space $\mathcal{H}$ is called a Bessel sequence if there exists a constant $B > 0$ such that
	\begin{equation}\label{eq_Bessel}
		\sum_{\gamma \in \Gamma} | \langle f, e_\gamma \rangle|^2 \leq B \norm{f}_\mathcal{H}^2, \qquad \forall f \in \mathcal{H}.
	\end{equation}
\end{definition}
Any number $B$ satisfying \eqref{eq_Bessel} is called a Bessel bound for $(e_\gamma)_{\gamma \in \Gamma}$. We will now express the Bessel condition in terms of the synthesis operator.
\begin{theorem}
	Let $(e_\gamma)_{\gamma \in \Gamma}$ be a sequence in a Hilbert space $\mathcal{H}$ and $B > 0$ be given. Then $(e_\gamma)_{\gamma \in \Gamma}$ is a Bessel sequence with Bessel bound $B$ if and only if
	\begin{equation}
		D: (c_\gamma)_{\gamma \in \Gamma} \mapsto \sum_{\gamma \in \Gamma} c_\gamma e_\gamma
	\end{equation}
	is a well-defined bounded operator from $\ell^2(\Gamma)$ into $\mathcal{H}$ and $\norm{D}_{op} \leq \sqrt{B}$.
\end{theorem}
\begin{proof}
	For the first part, assume that $(e_\gamma)_{\gamma \in \Gamma}$ is a Bessel sequence with Bessel bound $B$. Let $(c_\gamma)_{\gamma \in \Gamma} \in \ell^2(\Gamma)$. Now, we want to show that $D (c_\gamma)_{\gamma \in \Gamma}$ is well-defined, i.e., that $\sum_{\gamma \in \Gamma} c_\gamma e _\gamma$ is convergent\footnote{Note that, for the moment, ``well-defined" means convergent for spherical partial sums in this context. As we will see later, the convergence is actually unconditional.}. Consider $m,n \in \N$, $n > m$. Then
	\begin{align}
		\norm{\sum_{|\gamma| \leq n} c_\gamma e_\gamma - \sum_{|\gamma| \leq m} c_\gamma e_\gamma}_\mathcal{H} & = \norm{\sum_{|\gamma| = m+1}^n c_\gamma e_\gamma}\\
		& = \sup_{\norm{g}_\mathcal{H} = 1} \left| \langle \sum_{|\gamma| = m+1}^n c_\gamma e_\gamma, g \rangle \right|\\
		& \leq \sup_{\norm{g}_\mathcal{H} = 1} \sum_{|\gamma| = m+1}^n | c_\gamma \langle e_\gamma, g \rangle|\\
		& \leq \left( \sum_{|\gamma|=m+1}^n |c_\gamma|^2\right)^{1/2} \sup_{\norm{g}_\mathcal{H}=1} \left( \sum_{|\gamma|=m+1}^n |\langle e_\gamma, g \rangle|^2 \right)^{1/2}\\
		& \leq \sqrt{B} \left( \sum_{|\gamma|=m+1}^n |c_\gamma|^2\right)^{1/2} .
	\end{align}
	Since $(c_\gamma)_{\gamma \in \Gamma} \in \ell^2(\Gamma)$, we know that $(\sum_{|\gamma| \leq n} | c_\gamma|^2)_{n \in \N}$ is a Cauchy sequence in $\C$.	The above calculation shows that $(\sum_{|\gamma| \leq n} c_\gamma e_\gamma)_{n \in \N}$ is a Cauchy sequence in $\mathcal{H}$ and therefore convergent. Thus, $D (c_\gamma)_{\gamma \in \Gamma}$ is well-defined. Clearly, $D$ is linear. Also, as $\norm{D (c_\gamma)}_\mathcal{H} = \sup_{\norm{g}_\mathcal{H} = 1} | \langle D (c_\gamma), g \rangle|$, a calculation as above shows that $D$ is bounded and that $\norm{D}_{op} \leq \sqrt{B}$.
	
	For the opposite implication, suppose that $D$ is well-defined and that $\norm{D}_{op} \leq \sqrt{B}$. Then \eqref{eq_Bessel_D} shows that $(e_\gamma)_{\gamma \in \Gamma}$ is a Bessel sequence with bound $B$.
\end{proof}
\begin{corollary}
	If $(e_\gamma)_{\gamma \in \Gamma}$ is a sequence in $\mathcal{H}$ and $\sum_{\gamma \in \Gamma} c_\gamma e_\gamma$ is convergent for all $(c_\gamma)_{\gamma \in \Gamma} \in \ell^2(\Gamma)$, then $(e_\gamma)_{\gamma \in \Gamma}$ is a Bessel sequence.
\end{corollary}
As the Bessel condition \eqref{eq_Bessel} remains the same, regardless of how the elements $(e_\gamma)_{\gamma \in \Gamma}$ are ordered, we obtain the following result.
\begin{corollary}
	If $(e_\gamma)_{\gamma \in \Gamma}$ is a Bessel sequence in $\mathcal{H}$, then $\sum_{\gamma \in \Gamma} c_\gamma e_\gamma$ converges unconditionally for all $(c_\gamma)_{\gamma \in \Gamma} \in \ell^2(\Gamma)$.
\end{corollary}
The proof of the last corollary is a simple repetition of the proof of Lemma \ref{lem_operator_uncond_conv}. As a consequence, we obtain that, if the Gabor system
\begin{equation}
	\G(g, \alpha, \beta) = \{ M_{\beta l} T_{\alpha k} g \mid k,l \in \Z^d \}
\end{equation}
is a Gabor frame, then the Gabor series
\begin{equation}
	f = \sum_{k,l \in \Z^d} c_{k,l} M_{\beta l} T_{\alpha k} g
\end{equation}
converges unconditionally. In particular, the rectangular partial sums
\begin{equation}
	s_{K,L} = \sum_{\substack{k \in \Z^d \\ |k_j| \leq K}} \sum_{\substack{l \in \Z^d \\ |l_j| \leq L}} c_{k,l} M_{\beta l} T_{\alpha k} g
\end{equation}
as well as the radial partial sums
\begin{equation}
	\widetilde{s}_{N} = \sum_{\substack{k,l \in \Z^d \\ |k|^2 + |l|^2 \leq N}} c_{k,l} M_{\beta l} T_{\alpha k} g
\end{equation}
both converge to $f$.

We will now state some basic properties of the frame operator.
\begin{proposition}\label{pro_frame_properties}
	Suppose $\{e_\gamma \mid \gamma \in \Gamma\}$ is a frame. Let $S : \mathcal{H} \to \mathcal{H}$ be the associated frame operator, i.e.,
	\begin{equation}
		S f = \sum_{\gamma \in \Gamma} \langle f, e_\gamma \rangle e_\gamma .
	\end{equation}
	Then, $S$ is self-adjoint and it is a positive, invertible operator satisfying $A I_\mathcal{H} < S < B I_\mathcal{H}$ and $B^{-1} I_\mathcal{H} < S^{-1} < A^{-1} I_\mathcal{H}$.
	
	Furthermore, the optimal frame bounds are $B_* = \norm{S}_{op}$ and $A^* = \norm{S^{-1}}_{op}^{-1}$.
\end{proposition}
\begin{proof}
	Obviously, the frame operators is given by $S = D C  = C^* C = D D^*$ and consequently $S$ is self-adjoint. Since
	\begin{equation}
		\langle S f, f \rangle = \sum_{\gamma \in \Gamma} |\langle f, e_\gamma \rangle|^2,
	\end{equation}
	the operator inequality $A I \leq S \leq B I$ is just the frame inequality rewritten. $S$ is invertible on $\mathcal{H}$ because $A > 0$ by the assumption that $\{e_\gamma \mid \gamma \in \Gamma \}$ is a frame, which shows that $S$ is injective. From the frame inequality we also deduce that
	\begin{equation}
		B I - S \leq (B-A) I
		\quad \text{ and, also, } \quad
		0 \leq I - B^{-1} S \leq \frac{B-A}{B} I.
	\end{equation}
	This implies that
	\begin{equation}
		\norm{(I - B^{-1} S)}_{op} = \sup_{\norm{f}_\mathcal{H} = 1} | \langle (I - B^{-1}S)f,f \rangle| \leq \frac{B-A}{B} = 1- \frac{A}{B} < 1.
	\end{equation}
	Therefore, the inverse frame operator can be constructed by a Neumann series (see Appendix \ref{app_Neumann}). As inequalities are preserved under multiplication with positive commuting operators, we also have $A S^{-1} \leq S S^{-1} \leq B S^{-1}$.
	
	The statement on the optimal upper frame bound follows from the frame inequality \eqref{eq_frame} and the fact that the norm of a positive (hence self-adjoint) operator is given by $\norm{S}_{op} = \sup \{\langle S f , f \rangle \mid \norm{f} = 1 \}$. The statement for the optimal lower frame bound follows similarly.
\end{proof}
As a consequence of Proposition \ref{pro_frame_properties} we obtain a first reconstruction formula for $f$ from the frame coefficients.
\begin{corollary}
	Let $\{e_\gamma \mid \gamma \in \Gamma\}$ be a frame with bounds $0 < A \leq B < \infty$. Then $\{S^{-1} e_\gamma \mid \gamma \in \Gamma\}$ is a frame with frame bounds $0 < B^{-1} \leq A^{-1} < \infty$, the so-called canonical dual frame. Now, every $f \in \mathcal{H}$ has an expansion of the form
	\begin{equation}\label{eq_expansion_dual_coeff}
		f = \sum_{\gamma \in \Gamma} \langle f, S^{-1} e_\gamma \rangle e_\gamma
	\end{equation}
	and
	\begin{equation}\label{eq_expansion_dual_window}
		f = \sum_{\gamma \in \Gamma} \langle f, e_\gamma \rangle S^{-1} e_\gamma ,
	\end{equation}
	where both sums converge unconditionally in $\mathcal{H}$.
\end{corollary}
% UE
\begin{proof}
%To be done in the exercise session.
	First, as $S$ is self-adjoint, observe that
	\begin{equation}
		\sum_{\gamma \in \Gamma} | \langle f, S^{-1} e_\gamma \rangle|^2 = \sum_{\gamma \in \Gamma} | \langle S^{-1} f, e_\gamma \rangle|^2 = \langle S (S^{-1} f), S^{-1} f \rangle = \langle S^{-1} f, f \rangle.
	\end{equation}
	Therefore, Proposition \ref{pro_frame_properties} implies that
	\begin{equation}
		B^{-1} \norm{f}_\mathcal{H}^2 \leq \langle S^{-1} f, f \rangle = \sum_{\gamma \in \Gamma} |\langle f, S^{-1} e_\gamma \rangle|^2 \leq A^{-1} \norm{f}_\mathcal{H}^2 .
	\end{equation}
	Thus, the collection $\{S^{-1} e_\gamma \mid \gamma \in \Gamma \}$ is a frame with bounds $B^{-1}$ and $A^{-1}$.
	Using the factorization $I_\mathcal{H} = S^{-1} S = S S^{-1}$, we obtain the series expansions
	\begin{equation}
		f = S(S^{-1} f) = \sum_{\gamma \in \Gamma} \langle S^{-1} f, e_\gamma \rangle e_\gamma = \sum_{\gamma \in \Gamma} \langle f, S^{-1} e_\gamma \rangle e_\gamma
	\end{equation}
	and
	\begin{equation}
		f = S^{-1} S f = \sum_{\gamma \in \Gamma} \langle f, e_\gamma \rangle S^{-1} e_\gamma .
	\end{equation}
	Because both $(\langle f, e_\gamma\rangle)_{\gamma \in \Gamma}$ and $(\langle f, S^{-1} e_\gamma \rangle)_{\gamma \in \Gamma}$ are in $\ell^2(\Gamma)$, both series converge unconditionally.
\end{proof}

The two reconstruction formulas of $f$ should be compared to orthonormal expansions. On one hand, \eqref{eq_expansion_dual_coeff} provides a non-orthogonal expansion of $f$ with respect to the frame elements $e_\gamma$ with coefficients obtained from the inner products of $f$ with the canonical dual frame. On the other hand, \eqref{eq_expansion_dual_window} is a reconstruction of $f$ from the measurements with respect to the frame, i.e., from the frame coefficients, with the elements of the canonical dual frame as expanding functions. For orthonormal basis and tight frames these two aspects -- series expansion with respect to a set of vectors and reconstruction from inner products -- coincide. However, in contrast to orthonormal bases, the coefficients in a frame expansion of type \eqref{eq_expansion_dual_coeff} are in general not unique. The coefficients $\langle f, S^{-1} e_\gamma \rangle$ are canonical in the following sense.
\begin{proposition}
	If $\{ e_\gamma \mid \gamma \in \Gamma \}$ is a frame for $\mathcal{H}$ and
	\begin{equation}
		f = \sum_{\gamma \in \Gamma} c_\gamma e_\gamma
	\end{equation}
	for some coefficients $c \in \ell^2(\Gamma)$, then
	\begin{equation}
		\sum_{\gamma \in \Gamma} |c_\gamma|^2 \geq \sum_{\gamma \in \Gamma} | \langle f, S^{-1} e_\gamma \rangle |^2 ,
	\end{equation}
	with equality if and only if $c_\gamma = \langle f, S^{-1} e_\gamma \rangle$ for all $\gamma \in \Gamma$.
\end{proposition}
% UE
\begin{proof}
%To be done in the exercise session.
	Set $a_\gamma = \langle f, S^{-1} e_\gamma \rangle$. Then, $f = \sum_{\gamma \in \Gamma} a_\gamma e_\gamma$ and
	\begin{equation}
		\langle f, S^{-1} f \rangle = \sum_{\gamma \in \Gamma} a_\gamma \langle e_\gamma, S^{-1} f \rangle = \sum_{\gamma \in \Gamma} |a_\gamma|^2 .
	\end{equation}
	On the other hand,
	\begin{equation}
		\langle f, S^{-1} f \rangle = \sum_{\gamma \in \Gamma} c_\gamma \langle e_\gamma, S^{-1} f \rangle = \sum_{\gamma \in \Gamma} c_\gamma \overline{a_\gamma} = \langle c, a \rangle_{\ell^2}
	\end{equation}
	Therefore, we see that $\norm{a}_{\ell^2}^2 = \langle c, a \rangle_{\ell^2}$. Now, we compute
	\begin{align}
		\norm{c}_{\ell^2}^2 & = \norm{c - a + a}_{\ell^2}^2\\
		& = \norm{c - a}_{\ell^2}^2 + \norm{a}_{\ell^2}^2 + \langle c-a, a \rangle_{\ell^2} + \langle a, c - a \rangle_{\ell^2}\\
		& = \norm{c - a}_{\ell^2}^2 + \norm{a}_{\ell^2}^2 \geq \norm{a}_{\ell^2}^2 ,
	\end{align}
	with equality if and only if $c = a$.
\end{proof}

\begin{lemma}\label{lem_ONB_tight_frame}
	\begin{enumerate}[(a)]
		\item Let $\{ e_\gamma \mid \gamma \in \Gamma \}$ be a tight frame with frame bounds $A=B=1$ and $\norm{e_\gamma}_\mathcal{H} = 1$ for all $\gamma \in \Gamma$. Then $\{ e_\gamma \mid \gamma \in \Gamma \}$ is an orthonormal basis for $\mathcal{H}$.
		\item If $\{e_\gamma \mid \gamma \in \Gamma \}$ is a frame, then $\{ S^{-1/2} e_\gamma \mid e_\gamma \}$ is a tight frame with frame bounds $A=B=1$.
		\item Let $\{e_\gamma \mid \gamma \in \Gamma \}$ be a frame. Then the inverse frame operator is given by
		\begin{equation}
			S^{-1} f = \sum_{\gamma \in \Gamma} \langle f, S^{-1} e_\gamma \rangle S^{-1} e_\gamma .
		\end{equation}
		Thus, $S^{-1}$ is the frame operator associated to the canonical dual frame $\{ S^{-1} e_\gamma \mid \gamma \in \Gamma \}$.
	\end{enumerate}
\end{lemma}
\begin{proof}
	\textit{(a)}: By the frame inequality \eqref{eq_frame} we have
	\begin{equation}
		1 = \norm{e_{\gamma'}}_\mathcal{H}^2 = \sum_{\gamma \in \Gamma} | \langle e_{\gamma'}, e_\gamma \rangle|^2 = 1 + \sum_{\gamma \neq \gamma'} |\langle e_{\gamma'}, e_\gamma \rangle|^2 .
	\end{equation}
	Consequently, $\langle e_{\gamma'}, e_\gamma \rangle = \delta_{\gamma', \gamma}$.
	
	\textit{(b)}: %To be done in the exercise session.
	% UE
	Note that since $S$ is a positive operator, the operator $S^{-1/2}$ is well-defined and positive by the spectral theorem for bounded, self-adjoint operators. Writing $f$ as
	\begin{equation}
		f = S^{-1/2} S (S^{-1/2} f) = \sum_{\gamma \in \Gamma} \langle f, S^{-1/2} e_\gamma \rangle S^{-1/2} e_\gamma ,
	\end{equation}
	we obtain
	\begin{equation}
		\norm{f}_\mathcal{H}^2 = \langle f, f \rangle = \sum_{\gamma \in \Gamma} | \langle f, S^{-1/2} e_\gamma \rangle|^2 .
	\end{equation}
	
	\textit{(c)}: We note that
	\begin{equation}
		S^{-1} f = S^{-1} S (S^{-1} f) = \sum_{\gamma \in \Gamma} \langle f, S^{-1} e_\gamma \rangle S^{-1} e_\gamma .
	\end{equation}
\end{proof}
We note that the elements of the tight frame $\{ S^{-1/2} e_\gamma \mid \gamma \in \Gamma \}$ need not be normalized in general, therefore, it need not be an orthonormal basis.

\subsection{Gabor Systems}
We have now settled the theory of frames in an abstract setting. Now, we will return to the special collection of time-frequency shifts of a window function and sampling of the STFT. This will yield a discrete representation of a function $f \in \Lt$ where we may also put a physical interpretation on the coefficients, as desired.

We re-call the definition of a time-frequency shift.
\begin{definition}
	Let $\gamma = (x,\omega) \in \Rd \times \Rd$. A time-frequency shift by $\gamma$ is denoted by
	\begin{equation}
		\pi(\gamma) = M_\omega T_x .
	\end{equation}
\end{definition}
Next, we define Gabor systems and Gabor frames.
\begin{definition}
	Let $g \in \Lt$ (non-zero) and consider a countable set $\Gamma \subset \R^{2d}$. The collection
	\begin{equation}
		\G (g, \Gamma) = \{ \pi(\gamma) g \mid \gamma \in \Gamma\}
	\end{equation}
	is called a Gabor system. If $\G(g,\Gamma)$ is a frame for $\Lt$, i.e., the frame inequality
	\begin{equation}
		A \norm{f}_2^2 \leq \sum_{\gamma \in \Gamma} |\langle f, \pi(\gamma) g \rangle|^2 = \sum_{\gamma \in \Gamma} |V_g f(\gamma)|^2 \leq B \norm{f}_2^2, \qquad \forall f \in \Lt,
	\end{equation}
	is fulfilled for some constants $0 < A \leq B < \infty$, then it is called a Gabor frame.
\end{definition}
In short, we may as well write $\norm{V_g f(\gamma)}_{\ell^2(\Gamma)} \asymp \norm{f}_{L^2(\Rd)}$. So, the energy of the coefficients extracted from the measurements of the STFT is comparable to the energy of the signal.

We note that Gabor systems are sometimes (inaccurately) also called Weyl-Heisenberg systems\footnote{Weyl-Heisenberg systems are derived from  the so-called Weyl-Heisenberg operators, which yield symmetric time-frequency shifts. In general, the elements of a Gabor system and of a Weyl-Heisenberg system differ by (exactly determinable) phase factors. Nonetheless, both systems have the same properties, in particular, if one system is a frame, so is the other.}. The Gabor frame operator is given by
\begin{equation}
	S_{g,\Gamma} f = \sum_{\gamma \in \Gamma} \langle f, \pi(\gamma) g \rangle \ \pi(\gamma) g.
\end{equation}

\subsection{Gabor Systems over Lattices and Point Sets}

We have introduced Gabor systems for general (countable) point sets in $\Gamma \subset \R^{2d}$. However, there are comparably few results on Gabor systems and frames for general point sets compared to the case where the set possesses a group structure, i.e., when the set is a lattice.

\begin{definition}
	A (full-rank) lattice $\L$ in $\Rd$ is a discrete, co-compact subgroup of $\Rd$. That is, it can be represented by an invertible matrix $M \in GL(\R,d)$;
	\begin{equation}
		\L = \{ k_1 v_1 + \ldots + k_d v_d \mid v_1, \ldots , v_d \in \Rd, \ k_1, \ldots , k_d \in \Z \} = M \Z^d,
	\end{equation}
	where the matrix $M$ has columns $v_1, \ldots, v_d$;
	\begin{equation}
	 M = \begin{pmatrix}
		v_{1,1} & v_{2,1} & & v_{d,1}\\
		\vdots & \vdots & \cdots & \vdots\\
		v_{1,d} & v_{2,d} & & v_{d,d}
	 \end{pmatrix}
	\end{equation}
	In particular, the vectors $v_1 , \ldots , v_d$ constitute a basis for $\Rd$.
\end{definition}

We note that the matrix representing a lattice is not unique. This is owed to the fact that we may choose from countably many bases (so-called unimodular matrices) for $\Z^d$. However, there is a characteristic number, called the volume of the lattice, which is unique.
\begin{definition}
	The volume of the lattice $\L = M \Z^d \subset \Rd$ is given by
	\begin{equation}
		\vol(\L) = |\det(M)|.
	\end{equation}
	The reciprocal of the volume is called the density of the lattice
	\begin{equation}
		\delta(\L) = \frac{1}{\vol(\L)}.
	\end{equation}
\end{definition}
Let $\L$ be a lattice in $\R^{2d}$ and $g \in \Lt$. If $\G(g,\L) = \{ \pi(\l) g \mid \l \in \L \}$ is the resulting Gabor system, then we say that the system has redundancy $red(\G) = \delta(\L)$. This measures the grade of overcompleteness of the system. 

The notion of density is one of the most important concepts in Gabor analysis. A lot of (no-go) results for the characterization of Gabor frames are formulated by density results. The density of a lattice can also be generalized to arbitrary point sets by means of the Beurling density.
\begin{definition}
	Let $\Gamma \subset \Rd$ be a point set and let $\mathbf{B}_R(x)$ be the ball of radius $R$ centered at $x$. The lower and upper Beurling density of $\Gamma$ are given by
	\begin{equation}
		D^-(\Gamma) = \liminf_{R \to \infty} \frac{\min_{x \in \Rd} \# (\Gamma \cap \mathbf{B}_R(x))}{\vol(\mathbf{B}_R(x))}
	\end{equation}
	and
	\begin{equation}
		D^+(\Gamma) = \limsup_{R \to \infty} \frac{\max_{x \in \Rd} \# (\Gamma \cap \mathbf{B}_R(x))}{\vol(\mathbf{B}_R(x))},
	\end{equation}
	respectively.
\end{definition}
If the point set is a lattice $\L$, then we have
\begin{equation}
	\delta(\L) = D^-(\L) = D^+(\L).
\end{equation}
Also, we note that in the definition of the Beurling density, the balls $\mathbf{B}_R(x)$ could, e.g., be replaced by cubes without changing the values of the densities. The notion of Beurling density is particularly important for non-uniform Gabor systems, i.e., where the index set $\Gamma \subset \R^{2d}$ does not possess a group structure.

If $\Gamma \subset \R^{2d}$ is not a lattice, then we say that the Gabor system $\G(g, \Gamma) = \{\pi(\gamma) \mid \gamma \in \Gamma \}$ is a non-uniform Gabor system. The analysis of such systems is rather difficult as, in general, we cannot exploit any group structure.

We will now return to the study of lattices.
\begin{definition}\label{def_dual_lattice}
	Let $\L =M \Z^d \subset \Rd$ be a lattice. Then, its dual lattice is given by the set $\L^\perp$ of all vectors $v^\perp$ such that $v^\perp \cdot v \in \Z$ for all $v \in \L$.
\end{definition}
We can also find an explicit matrix representation for the dual lattice.
\begin{equation}
		\L^\perp = M^{-T} \Z^d .
	\end{equation}
It is easy to see that
\begin{equation}\label{def_adjoint_lattice}
	\vol(\L^\perp) = |\det(M^{-T})| = |\det(M^{-1})| = \frac{1}{\vol(\L)} = \delta(\L) .
\end{equation}
Further, we note that if $\l \in \L$ and $\l^\perp \in \L^\perp$, then $e^{2 \pi i \l^\perp \cdot \l} = 1$, by the definition of the dual lattice. Also, $(\L^\perp)^\perp = \L$. However, in time-frequency analysis it is often of greater use to work with the adjoint lattice instead of the dual lattice.
\begin{definition}
	Let $\L \subset \R^{2d}$ be a lattice and let $\L^\circ \subset \R^{2d}$ be its adjoint lattice. Then a point $\l^\circ = (x^\circ, \omega^\circ) \in \R^{2d}$ belongs to $\L^\circ$ if and only if
	\begin{equation}
		\pi(\l)\pi(\l^\circ) = \pi(\l^\circ) \pi(\l)
	\end{equation}
	for all $\l \in \L$.
\end{definition}
We note that the adjoint lattice is only defined for even dimensions. This results from the fact that we deal with lattices in the time-frequency plane and treat two variables simultaneously.

Having the group structure of the lattice at hand, the canonical dual frame and its frame operator are particularly nice to compute.
\begin{proposition}\label{pro_canonical_dual_Gabor}
	Let $\L \subset \R^{2d}$ be a lattice and let $\G(g,\L)$ be a frame for $\Lt$. Then, there exists a dual window $\widetilde{g} \in \Lt$ such that the canonical dual frame of $\G(g,\L)$ is $\G(\widetilde{g}, \L)$. Consequently, every $f \in \Lt$ possesses an expansion of the form
	\begin{align}
		f & = \sum_{\l \in \L} \langle f, \pi(\l) g \rangle \ \pi(\l) \widetilde{g} \\
		& = \sum_{\l \in \L} \langle f, \pi(\l) \widetilde{g} \rangle \ \pi(\l) g ,	
	\end{align}
	with unconditional convergence in $\Lt$. Further, the following norm equivalences hold:
	\begin{equation}
		A \norm{f}_2^2 \leq \sum_{\l \in \L} |V_g f(\l)|^2 \leq B \norm{f}_2^2,
	\end{equation}
	\begin{equation}
		B^{-1} \norm{f}_2^2 \leq \sum_{\l \in \L} |V_{\widetilde{g}} f(\l)|^2 \leq A^{-1} \norm{f}_2^2 ,
	\end{equation}
	for all $f \in \Lt$.
\end{proposition}
\begin{proof}
	We start by noting that the Gabor frame operator commutes with the time-frequency shifts $\pi(\l)$. Given $f \in \Lt$ and $\l' \in \L$, we compute
	\begin{align}
		\pi(\l')^{-1} S_{g,\L} \, \pi(\l') f
		& = \sum_{\l \in \L} \langle \pi(\l') f, \pi(\l) g \rangle \, \pi(\l')^{-1} \pi(\l) g\\
		& = \sum_{\l \in \L} \langle f, \pi(\l')^{-1} \pi(\l) g \rangle \, \pi(\l')^{-1} \pi(\l) g\\
		& = \sum_{\l \in \L} \langle f, \pi(\l-\l') g \rangle \, \pi(\l-\l') g\\
		& = S_{g,\L} f.
	\end{align}
	Note that, when merging $\pi(\l')^{-1}\pi(\l)$, the appearing phase factors cancel out as they appear as complex conjugates and that $\l-\l' \in \L$.
	
	Consequently, the inverse frame operator $S_{g, \L}^{-1}$ commutes with these time-frequency shifts as well. The canonical dual frame consists of functions of the form
	\begin{equation}
		S_{g,\L}^{-1} \pi(\l) g = \pi(\l) S_{g,\L}^{-1} g.
	\end{equation}
	We take $\widetilde{g} = S_{g,\L}^{-1} g$. The other assertions follow from the general theory for frames in a Hilbert space and have already been proven.
\end{proof}
The function
\begin{equation}
	\widetilde{g} = S_{g,\L}^{-1} g
\end{equation}
is called the canonical dual window to $g$. A simple consequence, following from the general theory and Proposition \ref{pro_canonical_dual_Gabor}, is the following.
\begin{corollary}
	Let $\L \subset \R^{2d}$ be a lattice and let $\G(g, \L)$ be a Gabor frame for $\Lt$ with canonical dual window $\widetilde{g} = S_{g,\L}^{-1} g$. Then, the inverse frame operator is given by
	\begin{equation}
		S_{g,\L}^{-1} f = \sum_{\l \in \L} \langle f, \pi(\l) \widetilde{g} \rangle \ \pi(\l) \widetilde{g} .
	\end{equation}
\end{corollary}

\section{The Symplectic and the Metaplectic Group}\label{sec_symp_meta}
In this section, we will define the symplectic and the metaplectic group and extract some of their features which are useful in time-frequency analysis. However, we will not go into depth and some statements will be given without proof as this would go beyond the scope of the course.
\subsection{The Symplectic Group}
Computing the commutator of two time-frequency shifts yields
\begin{equation}
	[\pi(\l'), \pi(\l)] = \left(1 - e^{2 \pi i (\omega \cdot x' - x \cdot \omega')}\right) \pi(\l') \pi(\l), \qquad \l = (x,\omega), \ \l' = (x', \omega').
\end{equation}
We note that if $\l \in \L$ and $\l^\circ \in \L^\circ$, where $\L$ and $\L^\circ$ are adjoint lattices to one-another, then the commutator vanishes. By introducing the skew-symmetric form
\begin{equation}
	\sigma(\l',\l) = \l' \cdot J \l,
\end{equation}
where
$J = \begin{pmatrix}
	0 & I\\
	-I & 0
\end{pmatrix}$ and $I$ is the identity matrix in $\R^{d \times d}$, we can write the commutator in the following way;
\begin{equation}
	[\pi(\l'), \pi(\l)] = \left(1 - e^{2 \pi i \sigma(\l', \l)}\right) \pi(\l') \pi(\l).
\end{equation}
The form $\sigma$ is called the standard symplectic form and $J$ is called the standard symplectic matrix. Note that $J^T = J^{-1}$ and $J^2 = -I$. Clearly, $\sigma(\l,\l) = 0$ and $\sigma(\l',\l) = - \sigma(\l, \l')$.

The notion of the symplectic form motivates the following definition.
\begin{definition}
	A matrix $S \in GL(\R,2d)$ is called symplectic, if it preserves the symplectic form $\sigma$, i.e.,
	\begin{equation}
		\sigma(S\l',S\l) = \sigma(\l',\l).
	\end{equation}
\end{definition}
From the above definition, we see that an equivalent definition of symplectic matrices is given by the characterization
\begin{equation}
	S^T J S = J.
\end{equation}
From the above characterization, it readily follows that $\det(S) = \pm 1$ if $S \in Sp(\R,2d)$, which in particular implies that $S$ is invertible.

We will now show that symplectic matrices actually form a group under matrix multiplication, denoted by $Sp(\R,2d)$\footnote{The notation for the symplectic group is not consistent in the literature. It may as well be denoted by $Sp(\R,d)$.}.
\begin{proposition}
	The set of symplectic matrices forms a group under matrix multiplication.
\end{proposition}
\begin{proof}
	Let $S$ and $S'$ be symplectic. Then
	\begin{equation}
		(S S')^T J S S' = S'^T (S^T J S) S' = S'^T J S' = J.
	\end{equation}
	So, the matrix $S S' \in Sp(\R,2d)$. Clearly, matrix multiplication is associative and the identity matrix is symplectic. So, we need to check that the inverse matrix $S^{-1} \in Sp(\R,2d)$. We have
	\begin{equation}
		S^T J S = J \; \Leftrightarrow \; S^T J = J S^{-1} \; \Leftrightarrow \; J = S^{-T} J S^{-1}.
	\end{equation}
\end{proof}
\begin{proposition}
	A matrix $S$ is symplectic if and only if $S^T$ is symplectic.
\end{proposition}
\begin{proof}
	Let $S$ be symplectic, then $S^{-1}$ is symplectic as well and we have
	\begin{align}
		S^{-T} J S^{-1} = J & \; \Leftrightarrow \; (S^{-T} J S^{-1})^{-1} = J^{-1} \; \Leftrightarrow \; S J^{-1} S^{T} = J^{-1}.
	\end{align}
	As $J^{-1} = -J$, we have
	\begin{equation}
		-S J S^T = -J \; \Leftrightarrow S J S^T = J
	\end{equation}
	Therefore, $S$ is symplectic if and only if $S^T$ is symplectic.
\end{proof}

It is common, and often useful, to write symplectic matrices as block matrices
\begin{equation}
	S =
	\begin{pmatrix}
		A & B\\
		C & D
	\end{pmatrix},
\end{equation}
where $A, B, C, D$ are $d \times d$ matrices. Using this notation we will now prove that $\det(S) = 1$ for $S \in Sp(\R, 2d)$ (see also \cite[Chap.~4]{Fol89}, \cite{MacMac03}, \cite{DufSal17}).
\begin{proposition}
	If $S \in Sp(\R,2d)$, then $\det(S) = 1$.
\end{proposition}
\begin{proof}
	Let $S \in Sp(\R,2d)$. Since $\det(J) = 1$, it is obvious from
	\begin{equation}
		S^T J S = J
	\end{equation}
	that $\det(S) = \pm 1$.
	
	We consider the matrix $S^T S + I$. Since $S^T S$ is symmetric and positive definite, the eigenvalues of $S^T S + I$ are real and greater than 1 because
	\begin{equation}
		\min_{|x|=1} x \cdot (S^T S + I) x = \min_{|x|=1} (x \cdot S^T S x + |x|^2) = \min_{|x|=1} (|Sx|^2+|x|^2) > 1.
	\end{equation}
	As $\det(S) \neq 0$, it is invertible and we have
	\begin{equation}\label{eq_symp_det1_proof}
		S^T S + I = S^T (S + S^{-T}) = S^T(S + J S J^{-1}).
	\end{equation}
	Next, we compute
	\begin{align}
		S + J S J^{-1} & =
		\begin{pmatrix}
			A & B\\
			C & D
		\end{pmatrix}
		+
		\begin{pmatrix}
			0 & I\\
			-I & 0
		\end{pmatrix}
		\begin{pmatrix}
			A & B\\
			C & D
		\end{pmatrix}
		\begin{pmatrix}
			0 & -I\\
			I & 0
		\end{pmatrix}
		=
		\begin{pmatrix}
			A & B\\
			C & D
		\end{pmatrix}
		+
		\begin{pmatrix}
			D & -C\\
			-B & A
		\end{pmatrix}\\
		& =
		\begin{pmatrix}
			A+D & B-C\\
			-B+C & A+D
		\end{pmatrix}
	\end{align}
	Setting $A+D = E$ and $B-C = F$ we make the unitary transform (polar factorization)
	\begin{equation}
		S + J S J^{-1} =
		\begin{pmatrix}
			E & F\\
			-F & E
		\end{pmatrix}
		=
		\frac{1}{\sqrt{2}}
		\begin{pmatrix}
			I & I\\
			i I & -i I
		\end{pmatrix}
		\begin{pmatrix}
			E+iF & 0\\
			0 & E-iF
		\end{pmatrix}
		\frac{1}{\sqrt{2}}
		\begin{pmatrix}
			I & -i I\\
			I & i I
		\end{pmatrix}
	\end{equation}
	Note that\footnote{Note that the determinant of a block diagonal matrix
	$S = \begin{pmatrix}
			A & B\\
			C & D
		\end{pmatrix}$ can be computed by the formula $\det(S) = \det(AD - BC)$ if all blocks have the same size and the blocks $C$ and $D$ commute \cite{Sil00}.}
		\begin{equation}
			\det\left(\frac{1}{\sqrt{2}}
			\begin{pmatrix}
				I & I\\
				i I & -i I
			\end{pmatrix}\right) = \left(\tfrac{1}{\sqrt{2}}\right)^{2d} \det(-2 i I) = 2^{-d} (-2 i)^{d} = (-i)^d
		\end{equation}
		and
		\begin{equation}
			\det\left(\frac{1}{\sqrt{2}}
			\begin{pmatrix}
				I & -i I\\
				I & i I
			\end{pmatrix}\right) = \left(\tfrac{1}{\sqrt{2}}\right)^{2d} \det(2 i I) = 2^{-d} (2 i)^{d} = i^d.
		\end{equation}
		Hence, these factors cancel to 1 when taking the determinant. Now, we plug this factorization into \eqref{eq_symp_det1_proof} and obtain
	\begin{align}
		0 < 1 < \det(S^T S + I) & = \det(S^T(S+JSJ^{-1}))\\
		& = \det(S) \det(E+iF) \det(E-iF)\\
		& = \det(S) \det(E+iF) \, \overline{\det(E+iF)}\\
		& = \det(S) |\det(E+iF)|^2
	\end{align}
	We see that none of the factors $\det(S)$ and $|\det(E+iF)|^2$ can be zero. So, we may divide by $|\det(E+iF)|^2 > 0$ and obtain
	\begin{equation}
		\det(S) > 0.
	\end{equation}
	Since $\det(S) \in \{-1,1\}$, it follows that $\det(S) = 1$.
\end{proof}
\begin{remark}
	We note that the proof given above is not the standard proof found in textbooks, but it does not require more than basic linear algebra. A very quick proof -- once one knows the Pfaffian $\text{Pf}(S)$ of a matrix $S$ and the necessary properties -- is as follows (see, e.g., \cite{MacMac03});
	\begin{equation}
		\text{Pf}(J) = \text{Pf}(S^T J S) = \det(S) \text{Pf}(J)
		\quad \Longrightarrow \quad
		\det(S) = 1.
	\end{equation}
	
	An equivalent proof is obtained via exterior 2-forms (see \cite{DufSal17}). For this, we note that, using the language of differential forms, we may write
	\begin{equation}
		\sigma = dx \wedge d\omega = \sum_{k=1}^d dx_k \wedge d\omega_k.
	\end{equation}
	Then
	\begin{equation}
		\frac{\sigma^d}{d!} = \frac{\sigma \wedge \ldots \wedge \sigma}{d!} = dx_1 \wedge d\omega_1 \wedge \ldots \wedge dx_d \wedge d\omega_d = \vol_{\R^{2d}}
	\end{equation}
	Yet, another quick proof is via the theory of Lie groups (see, e.g., \cite{Fol89}). Noting that $Sp(\R,2d)$ is connected and $I \in Sp(\R,2d)$ gives the result.
	\begin{flushright}
		$\diamond$
	\end{flushright}
\end{remark}

From the block notation, we also get the following characterization of symplectic matrices.
\begin{proposition}
	A matrix $S$ is symplectic if and only if the two sets of equivalent conditions are fulfilled.
	\begin{equation}\label{eq_cond_symp1}
		A B^T = B A^T, \ C D^T = D C^T
		\quad \text{ and } \quad
		A D^T - B C^T = I.
	\end{equation}
	\begin{equation}\label{eq_cond_symp2}
		A^T C = C^T A, \ B^T D = D^T B
		\quad \text{ and } \quad
		A^T D - C^T B = I,
	\end{equation}
	This yields the following formula for the inverse of a symplectic matrix
	\begin{equation}\label{eq_symp_inv}
		S^{-1} =
		\begin{pmatrix}
			D^T & - B^T\\
			-C^T & A^T
		\end{pmatrix} .
	\end{equation}
\end{proposition}
\begin{proof}
	Let
	$S =
	\begin{pmatrix}
		A & B\\
		C & D
	\end{pmatrix}$, then
	$S^T =
	\begin{pmatrix}
		A^T & C^T\\
		B^T & D^T
	\end{pmatrix}$. We use that $S^T \in Sp(\R,2d)$ and compute
	\begin{align}
		S J S^T & =
		\begin{pmatrix}
			A & B\\
			C & D
		\end{pmatrix}
		\begin{pmatrix}
			0 & I\\
			-I & 0
		\end{pmatrix}
		\begin{pmatrix}
			A^T & C^T\\
			B^T & D^T
		\end{pmatrix} =
		\begin{pmatrix}
			A & B\\
			C & D
		\end{pmatrix}
		\begin{pmatrix}
			B^T & D^T\\
			-A^T & -C^T
		\end{pmatrix}\\
		& =
		\begin{pmatrix}
			A B^T - B A^T & A D^T - B C^T\\
			C B^T - D A^T & C D^T - D C^T
		\end{pmatrix} =
		\begin{pmatrix}
			0 & I\\
			-I & 0
		\end{pmatrix}.
	\end{align}
	Note that
	\begin{equation}
		C B^T - D A^T = -I \; \Leftrightarrow \; (D A^T - C B^T)^T = I^T \; \Leftrightarrow \; A D^T - B C^T = I.
	\end{equation}
	We have proved \eqref{eq_cond_symp1}. For proving \eqref{eq_cond_symp2} we make a similar computation to obtain
	\begin{equation}
		S^T J S =
		\begin{pmatrix}
			A^T C - C^T A & A^T D - C^T B\\
			B^T C - D^T A & B^T D - D^T B
		\end{pmatrix} =
		\begin{pmatrix}
			0 & I\\
			-I & 0
		\end{pmatrix}.
	\end{equation}
\end{proof}
Note that in the case $d = 1$ we recover the well-known inversion formula for determinant 1 matrices;
\begin{equation}
	S =
	\begin{pmatrix}
		a & b\\
		c & d
	\end{pmatrix}
	\quad \text{ and } \quad
	S^{-1} =
	\begin{pmatrix}
		d & -b\\
		-c & a
	\end{pmatrix}.
\end{equation}
Also, we see that in the case $d = 1$, conditions \eqref{eq_cond_symp1} as well as \eqref{eq_cond_symp2} collapse to the fact that $\det(S) = ad - bc = 1$. In particular, $Sp(\R,2) = SL(\R,2)$. This is the only case where the symplectic group coincides with the special linear group. In general, the symplectic group $Sp(\R,2d)$ is a proper subgroup of $SL(\R,2d)$. To be more precise, the dimension of the group $Sp(\R, 2d)$ is $(2d+1)d$. This is due to the fact that the system of equations
\begin{equation}
	S^T J S = J
\end{equation}
is redundant. We note that $J^T = -J$, which means that $J$ is skew-symmetric. This implies that $S^T J S$ is skew-symmetric as well. The (vector) space of skew-symmetric matrices has dimension
$\begin{pmatrix}
	2d\\
	2
\end{pmatrix} = (2d-1)d$. Hence, the identity $S^T J S = J$ puts $(2d-1)d$ many constraints on the $(2d)^2$ variables of $S$ and leaves $(2d+1)d$ free variables.
\begin{example}
	Consider the following diagonal matrix
	\begin{equation}
		M = \begin{pmatrix}
			\alpha & 0 & 0 & 0\\
			0 & \frac{1}{\alpha	} & 0 & 0\\
			0 & 0 & \alpha & 0\\
			0 & 0 & 0 & \frac{1}{\alpha}
		\end{pmatrix}.
	\end{equation}
	Then, clearly, $\det(M) = 1$, but the blocks do not satisfy \eqref{eq_cond_symp1} as
	\begin{equation}
		A D^T - B C^T =
		\begin{pmatrix}
			\alpha^2 & 0\\
			0 & \frac{1}{\alpha^2}
		\end{pmatrix} \neq I, \quad \alpha \in \R, \, \alpha \neq \pm 1.
	\end{equation}
	This example is easily extended to higher dimensions and we see that $Sp(\R,2d) \subsetneq SL(\R, 2d)$, for $d > 1$, by example.
	\begin{flushright}
		$\diamond$
	\end{flushright}
\end{example}
Our next goal is to find generator matrices of the symplectic group. For this purpose, we define the following sets, which are actually subgroups of $Sp(\R,2n)$. We define the group of (symplectic) shearing and dilation matrices, respectively, as follows;
\begin{equation}
	U_P = \left\{
	\begin{pmatrix}
		I & P\\
		0 & I
	\end{pmatrix} \mid P = P^T \right\},
	\qquad
	V_Q = \left\{
	\begin{pmatrix}
		I & 0\\
		Q & I
	\end{pmatrix} \mid Q = Q^T \right\}
\end{equation}
\begin{equation}
	D_L = \left\{
	\begin{pmatrix}
		L & 0\\
		0 & L^{-T}
	\end{pmatrix} \mid \det(L) \neq 0 \right\}.
\end{equation}
\begin{proposition}\label{pro_decomposition_symplectic}
	Let the sets $U_P$, $V_Q$ and $D_L$ be defined as above. Then
	\begin{equation}
		V_Q D_L U_P = \left\{
		\begin{pmatrix}
			A & B\\
			C & D
		\end{pmatrix} \in Sp(\R,2d)
		 \mid \det(A) \neq 0 \right\}.
	\end{equation}
\end{proposition}
\begin{proof}
	By a direct computation we see that
	\begin{equation}
		\begin{pmatrix}
			I & 0\\
			Q & I
		\end{pmatrix}
		\begin{pmatrix}
			L & 0\\
			0 & L^{-T}
		\end{pmatrix}
		\begin{pmatrix}
			I & P\\
			0 & I
		\end{pmatrix} =
		\begin{pmatrix}
			L & L P\\
			Q L & Q L P + L^{-T}
		\end{pmatrix}
	\end{equation}
	From this calculation, we see that we may simply take $A = L$ and $\det(A) \neq 0$. It follows that $P = A^{-1} B$ and $Q = C A^{-1}$. Now, we must show that
	\begin{equation}
		P = P^T, \qquad Q = Q^T, \qquad D = Q L P + L^{-T} = C A^{-1} B + A^{-T}.
	\end{equation}
	However, this follows from \eqref{eq_cond_symp1} and \eqref{eq_cond_symp2}.
\end{proof}
We state the next result without proof. A proof can be found in, e.g., \cite[Chap.~4.1.]{Fol89} or \cite[Chap.~3]{Gos11}
\begin{theorem}
	The symplectic group is generated by $U_P \cup D_L \cup \{J\}$ or $V_Q \cup D_L \cup \{J\}$.
\end{theorem}
In \cite[Chap.~3]{Gos11} we also find that any $S \in Sp(\R,2n)$ can be factored into two symplectic matrices $S'$ and $S''$ with $\det(A') \neq 0$ and $\det(A'') \neq 0$. This factorization is however not unique.

We are now going to introduce symplectic lattices.
\begin{definition}
	A lattice $\L$ is called symplectic if its generating matrix is a scalar multiple of a symplectic matrix, i.e.,
	\begin{equation}
		\L = r S \Z^{2d}, \qquad S \in Sp(\R,2d)
	\end{equation}
	and $\Z^{2d}$ is equipped with the standard basis.
\end{definition}
\begin{example}
	The requirement that $\Z^{2d}$ is equipped with the standard basis is crucial. For example, consider the permutation matrix
	\begin{equation}
		P_{2,3} =
		\begin{pmatrix}
			1 & 0 & 0 & 0\\
			0 & 0 & 1 & 0\\
			0 & -1 & 0 & 0\\
			0 & 0 & 0 & 1
		\end{pmatrix}
		\quad \text{ and } \quad
		\det(P_{2,3})=1.
	\end{equation}
	Particularly, we make a symplectic coordinate change in the second and third coordinate. Then, as lattices,
	\begin{equation}
		\L_I = I \Z^{2d}
		\quad \text and \ \quad
		\L_{2,3} = P_{2,3} \Z^{2d}
	\end{equation}
	coincide, but $P_{2,3} \Z^{2d}$ is not a symplectic lattice! It is quickly verified that $P_{2,3}$ is indeed not a symplectic matrix, as $A D^T - B C^T = 0 \neq I$. The reason why we need to say that $P_{2,3} \Z^{2d}$ is not a symplectic lattice comes from the fact that we mixed the coordinates $(x_1,x_2,\omega_1,\omega_2)$ to $(x_1,-\omega_1,x_2,\omega_2)$. From the time-frequency analyst's viewpoint, this means that we take the following time-frequency shifts
	\begin{equation}
		\pi((x_1,-\omega_1);(x_2,\omega_2)) = M_{(x_2,\omega_2)} T_{(x_1,-\omega_1)}
	\end{equation}
	instead of
	\begin{equation}
		\pi((x_1,x_2);(\omega_1,\omega_2)) = M_{(\omega_1,\omega_2)} T_{(x_1,x_2)}
	\end{equation}
	Indeed, this has an effect on resulting Gabor systems and the frame properties, as we will see later on.
	\begin{flushright}
		$\diamond$
	\end{flushright}
\end{example}

We note that for any lattice we have $\L = - \L$. So, without loss of generality we may assume that $r > 0$ if $\L$ is symplectic. The volume of the symplectic lattice $\L$ is then given by $\vol(\L) =  r^{2d}$. We note the following facts. The adjoint lattice is given by
\begin{equation}
	\L^\circ = J \L^\perp.
\end{equation}
Let $\L = r S \Z^{2d}$ be a symplectic lattice with $S \in Sp(\R,2d)$. Then the adjoint lattice is simply given by
\begin{equation}
	\L^\circ = r^{-2} \L.
\end{equation}
This can be seen by explicitly writing
\begin{equation}
	\L^\circ = r^{-1} J S^{-T} \Z^{2d} = r^{-1} \underbrace{J S^{-T} J^{-1}}_{=S} \Z^{2d} = r^{-2} r S \Z^{2d} = r^{-2} \L .
\end{equation}
Noting that $r^{2d} = \vol(\L)$ we may as well write
\begin{equation}
	\L^\circ = \vol(\L)^{-1/d} \L
\end{equation}
in this case. In particular, any symplectic lattice $\L$ of volume $\vol(\L) = 1$ is its own adjoint, i.e., $\L = \L^\circ$. This implies that any 2-dimensional lattice of unit volume is its own adjoint since $Sp(\R,2) = SL(\R,2)$.

For working in the time-frequency plane (or phase space), it is useful to introduce a version of the Fourier transform, the so-called symplectic Fourier transform.
\begin{definition}
	Let $F \in \Lt[2d]$. Then the symplectic Fourier transform is given by
	\begin{equation}
		\F_\sigma F(z) = \iint_{\R^{2d}} F(z') e^{-2 \pi i \sigma(z',z)}, \, dz'
	\end{equation}
\end{definition}
It has the following properties.
\begin{proposition}
	The symplectic Fourier transform is unitary and involutive, i.e.,
	\begin{equation}
		\norm{\F_\sigma F}_2 = \norm{F}_2
		\quad \text{ and } \quad
		\F_\sigma \circ \F_\sigma = I .
	\end{equation}
\end{proposition}
\begin{proof}
	First, we note that the symplectic Fourier transform can be expressed by the usual Fourier transform
	\begin{equation}
		\F_\sigma F(z) = \F F(J z), \qquad z \in \R^{2d}.
	\end{equation}
	Next, we observe that $J^2 = -I$. The usual Fourier transform gives $\F (\F F)(z) = F(-z)$ and, thus
	\begin{equation}
		\F_\sigma (\F_\sigma F)(z) = \F(\F F)(-z) = F(z).
	\end{equation}
\end{proof}
The proposition tells us in particular that the symplectic Fourier transform is its own inverse, thus
\begin{align}
	\F_\sigma F(z) = \iint_{\R^{2d}} F(z') e^{-2 \pi i \sigma(z',z)} \, dz'\\
	F(z) = \iint_{\R^{2d}} \F_\sigma F(z') e^{-2 \pi i \sigma(z',z)} \, dz' .
\end{align}
Furthermore, the symplectic Fourier transform behaves well under the action of the symplectic group:
\begin{align}
	\F_\sigma F(Sz) & = \iint_{\R^{2d}} F(z') e^{-2 \pi i \sigma(z', Sz)} \, dz' \\
	& = \iint_{\R^{2d}} F(z') e^{-2 \pi i \sigma(S^{-1} z', S^{-1} S z)} \, dz' \\
	& = \iint_{\R^{2d}} F(S z'') e^{-2 \pi i \sigma(z'', z)} \, dz'' \\
	& = \F_\sigma( F(S \, .))(z),
\end{align}
where we used the fact that $S^{-1}$ is symplectic and the substitution $z'' = S^{-1} z'$.

From the proposition above, it also follows that the symplectic Fourier transform satisfies the following variant of the Plancherel formula;
\begin{equation}
	\langle \F_\sigma F, G \rangle = \langle F, \F_\sigma G \rangle .
\end{equation}
With this notation, we may also formulate a symplectic version of the Poisson summation formula. Let $\L$ be a lattice and let $\L^\circ$ be its adjoint lattice. For a continuous function $F$ satisfying the decay condition $|F(z)| \leq C (1+|z|)^{-2d-\varepsilon}$ and $|\F_\sigma F(z')| \leq C (1+|z'|)^{-2d-\varepsilon}$ with $\varepsilon > 0$, $C > 0$ the following equality holds point-wise.
\begin{equation}
	\sum_{\l \in \L} F(\l + z) = \vol(\L)^{-1} \sum_{\l^\circ \in \L^\circ} \F_\sigma F(\l^\circ) e^{2 \pi i \sigma(\l^\circ, z)} .
\end{equation}
Also, we note that, by using the symplectic Fourier transform we can write the relations between the STFT and the Rihaczek distribution as well as the relation between the ambiguity function and the Wigner distribution in the following way.
\begin{equation}
	V_g f(z) = \F_\sigma \left(R(f,g)\right)(z)
	\quad \text{ and } \quad
	R(f,g)(z) = \F_\sigma \left( V_g f \right)(z)
\end{equation}
and
\begin{equation}
	A(f,g)(z) = \F_\sigma \left(W(f,g)\right)(z)
	\quad \text{ and } \quad
	W(f,g)(z) = \F_\sigma \left( A(f,g) \right)(z) .
\end{equation}

\begin{example}
	\textit{Let $\G(g,\L)$ be a Gabor system in $\Lt[]$ with $\vol(\L) = 2$ and $g$ an odd function, i.e., $g(t) = - g(-t)$ such that $\G(g,\L)$ is a Bessel system and $Ag$ satisfies the decay conditions for the (symplectic) Poisson summation formula to hold point-wise\footnote{Later we will see that there is actually an appropriate function space.}. Then the Gabor system cannot be a frame.}
	
	The full proof requires the following preliminary result obtained from the theory on Toeplitz matrices and Laurent operators (see Appendix \ref{app_Toeplitz}). We will simply state the necessary result without proof (see \cite{Jan96} for the original result and \cite{Faulhuber_Note_2018} for the necessary extension).
	
	For Gabor systems over lattices of density $\delta(\L) \in 2\N$, the sharp lower frame bound is given by the minimum of
	\begin{equation}
		0 \leq A = \min_{z \in \R^2} \vol(\L)^{-1} \sum_{\l^\circ \in \L^\circ} Ag(\l^\circ)e^{2 \pi i \sigma(\l^\circ,z)}.
	\end{equation}
	Now, we re-call that we have the algebraic relation and the symplectic Fourier transform connection between the ambiguity function and the Wigner distribution. We have
	\begin{align}
		\sum_{\l^\circ \in \L^\circ} Ag(\l^\circ) & = \underbrace{\vol(\L^\circ)^{-1}}_{=2^{-1}} \sum_{\l \in \L} \F_\sigma(Ag)(\l) = \sum_{\l \in \L} 2^{-1} Wg(\l) = \sum_{\l \in \L} 2^{-1} 2 A_{-g} g(2 \l)\\ & = - \sum_{\l \in \L} Ag(2\l).
	\end{align}
	Now, since $\L$ is symplectic, we note that $\L^\circ = 2 \L$, or $2 \l \in \L^\circ$ for all $\l \in \L$. Therefore, we have
	\begin{equation}
		\sum_{\l^\circ \in \L^\circ} Ag(\l^\circ) = - \sum_{\l^\circ \in \L^\circ} Ag(\l^\circ) = 0.
	\end{equation}
	This shows that, in this case,
	\begin{equation}
		0 = A = \min_{z \in \R^2} \vol(\L)^{-1} \sum_{\l^\circ \in \L^\circ} Ag(\l^\circ)e^{2 \pi i \sigma(\l^\circ,z)}
	\end{equation}
	by choosing $z = 0$.
	
	The example extends easily to Gabor systems $\G(g,\L)$ in $\Lt$, with $g(t) = -g(-t)$ and symplectic lattices $\L$ with $\vol(\L)^{-1} = 2^{d}$ (so that $\L^\circ = 2 \L$). At the moment, it is an open problem to extend the result to non-symplectic lattices. Also, it is not know whether the result can be extended to periodic structures or irregular point sets with lower Beurling density $2^d$.
	\begin{flushright}
		$\diamond$
	\end{flushright}
\end{example}

\subsection{The Heisenberg Group}
Time-frequency shifts are the fundamental operators in time-frequency analysis. In this chapter we change the point of view and consider the collection of time-frequency shifts $\{M_\omega T_x \mid (x,\omega) \in \R^{2d} \}$ as an object of independent interest. The Heisenberg group and its representation theory emerge as the underlying structure. One may even say that time-frequency analysis can be seen as an aspect of the theory on the Heisenberg group.

Time-frequency shifts do not commute and this requires delicate book keeping of the appearing phase factors. A composition of two time-frequency shifts yields
\begin{equation}\label{eq_tf_composition}
	\pi(\l) \pi(\l') = e^{-2 \pi i \omega' \cdot x} \pi(\l + \l') .
\end{equation}
Often, this additional phase factor can simply be ignored, sometimes its appearance might be considered annoying. However, this phase factor is essential for the deeper structure of time-frequency shifts and it is the reason for the appearance of a non-commutative group in time-frequency analysis.

We introduce an additional coordinate in addition to $x$ and $\omega$. By \eqref{eq_tf_composition}, time-frequency shifts are not closed under composition. As suggested by \eqref{eq_tf_composition}, we add the torus $\R/\Z = \{e^{2 \pi i \tau} \mid \tau \in \R \}$ and try to find a group multiplication on $\R^{2d} \times \T$ which is consistent with \eqref{eq_tf_composition}. It will be advantageous to consider symmetric time-frequeny shifts of the form
\begin{equation}
	\rho(\l) = M_{\omega/2} T_x M_{\omega/2} = T_{x/2} M_\omega T_{x/2} = e^{-\pi i \omega \cdot x} \pi(\l).
\end{equation}
We note that we also have the following identities
\begin{equation}
	A(f,g)(\l) = \langle \pi(-\l/2) f, \pi(\l/2) g \rangle = \langle f , \rho(\l) g \rangle .
\end{equation}

We compute
\begin{equation}\label{eq_Hr_composition}
	e^{2 \pi i \tau} \rho(\l) e^{2 \pi i \tau'} \rho(\l') = e^{2 \pi i(\tau + \tau')} e^{-\pi i \sigma(\l,\l')} \rho(\l + \l').
\end{equation}
From \eqref{eq_Hr_composition} we deduce the following abstract group multiplication on $\R^{2d} \times \T$.
\begin{definition}
	The reduced Heisenberg group $\mathbf{H}_r$ is the locally compact space $\mathbf{H}_r = \R^{2d} \times \T$ with the multiplication
	\begin{equation}\label{eq_Hr_multiplication}
		(x,\omega,e^{2\pi i \tau}) \circ (x',\omega', e^{2 \pi i \tau'}) = (x+x', \omega+\omega', e^{2 \pi i (\tau+\tau')}e^{\pi i(x' \cdot \omega - x \cdot \omega')}) .
	\end{equation}
\end{definition}
Since unitary operators form a group and since unitary operators of the form $e^{2 \pi i \tau} \rho(x,\omega)$ are closed under composition by \eqref{eq_Hr_composition}, $\mathbf{H}_r$ is indeed a group. For an element in $\mathbf{H}_r$ we write $\mathbf{h}_r = (x,\omega,e^{2 \pi i \tau})$. It follows from \eqref{eq_Hr_multiplication} that the neutral element, denoted by $\mathbf{1}_r$, and the inverse of an element, denoted by $\mathbf{h}_r^{-1}$, are given by
\begin{equation}
	\mathbf{1}_r = (0,0,1)
	\quad \text{ and } \quad
	\mathbf{h}_r^{-1} = (-x, - \omega, e^{-2 \pi i \tau}),
\end{equation}
respectively.

As a topological object $\mathbf{H}_r$ is identical with $\R^{2d} \times \T$. It inherits the product metric from $\R^{2d} \times \T$, which means that a sequence $({\mathbf{h}_r}_n)_n = \left((x_n, \omega_n, e^{2 \pi i \tau_n})\right)_n$ in $\mathbf{H}_r$ converges to $\mathbf{h}_r \in \mathbf{H}_r$ if and only if
\begin{equation}
	\lim_{n \to \infty} (|x-x_n|, | \omega - \omega_n|, |\tau - \tau_n|) = (0,0,\Z)
	\quad \Leftrightarrow \quad
	\lim_{n \to \infty} |\mathbf{h}_r - {\mathbf{h}_r}_n| = (0,0,1) = \mathbf{1}_r.
\end{equation}

The algebraic structure of $\mathbf{H}_r$, however, is different from the algebraic structure of $\R^{2d} \times \T$, which is an Abelian group. But $\mathbf{H}_r$ with group multiplication \eqref{eq_Hr_multiplication} is non-Abelian. For instance, take the products
\begin{equation}
	(x,0,1) \circ (0,\omega,1) = (x, \omega, e^{- \pi i x \cdot \omega})
	\quad \text{ and } \quad
	(0,\omega,1) \circ (x,0,1) = (x, \omega, e^{\pi i x \cdot \omega}).
\end{equation}

There are several versions of ``the" Heisenberg group. In the theory of Lie groups the simply connected version is usually preferred. This leads to the following definition.
\begin{definition}
	The full Heisenberg group $\mathbf{H} = \mathbf{H}(d)$ is the Euclidean space $\R^{2d} \times \R$ under the group multiplication
	\begin{equation}\label{eq_H_multiplication}
		(x,\omega,\tau) \bullet (x',\omega',\tau') = (x+x',\omega+\omega',\tau+\tau' + \tfrac{1}{2}(x' \cdot \omega - x \cdot \omega'))
	\end{equation}
\end{definition}
A comparison of the multiplications \eqref{eq_Hr_multiplication} and \eqref{eq_H_multiplication} shows that $\mathbf{H}_r$ is a quotient of $\mathbf{H}$ with respect to the subgroup $\{0\} \times \{0\} \times \Z$. The mapping $(x,\omega,\tau) \mapsto (x, \omega, e^{2 \pi i \tau})$ is a homomorphism from $\mathbf{H}$ onto $\mathbf{H}_r$.

As every locally compact group, the Heisenberg group has a Haar measure, which is invariant under (left) group translations. For the full Heisenberg group $\mathbf{H}$, it is the Lebesgue measure on $\R^{2d+1}$ and for $\mathbf{H}_r$ it is the Lebesgue measure on $\R^{2d} \times \T$.
\begin{lemma}\label{lem_Haar_measure}
	The Lebesgue measure $d \mathbf{h} = d(x,\omega,\tau)$ is invariant under (left) translations of $\mathbf{H}$. This means that
	\begin{equation}
		\int_{\mathbf{H}} F(\mathbf{h}) \, d\mathbf{h}= \int_{\mathbf{H}} F(\mathbf{h}_0 \bullet \mathbf{h}) \, d\mathbf{h}
	\end{equation}
	for every $F \in L^1(\R^{2d+1})$ and $\mathbf{h}_0 \in \mathbf{H}$.
\end{lemma}
\begin{proof}
	By definition,
	\begin{equation}
		\int_{\mathbf{H}} F(\mathbf{h}) \, d\mathbf{h} = \int_{\Rd} \int_{\Rd} \int_{\R} F(x,\omega,\tau) \, d(x,\omega,\tau) .
	\end{equation}
	Since $F \in L^1(\R^{2d+1})$, every order of integration is allowed. Then
	\begin{align}
		\int_{\mathbf{H}} F(\mathbf{h}_0 \bullet \mathbf{h}) \, d\mathbf{h}
		& = \int_{\R^{2d}} \int_{\R} F\left((x_0,\omega_0,\tau_0) \bullet (x,\omega,\tau)\right) \,d\tau \, d(x,\omega)\\
		& = \int_{\R^{2d}} \int_{\R} F\left(x_0+x, \omega_0+\omega, \tau_0+\tau+\tfrac{1}{2}(x\cdot \omega_0 - x_0\cdot\omega)\right) \,d\tau \, d(x,\omega)\\
		& = \int_{\R^{2d}} \int_{\R} F(x,\omega,\tau) \, d\tau \, d(x,\omega) = \int_{\mathbf{H}} F(\mathbf{h}) \, d\mathbf{h} ,
	\end{align}
	as the Lebesgue measure is invariant under the appearing translations.
\end{proof}
Similarly, the Haar measure on $\mathbf{H}_r$ is
\begin{equation}
	\int_{\mathbf{H}_r} F(\mathbf{h}_r) \, d\mathbf{h}_r = \int_{\R^{2d}} \int_0^1 F(x, \omega, e^{2 \pi i \tau}) \, d\tau \, d(x,\omega) .
\end{equation}
With the invariant (left) Haar measure at hand, we can now define a convolution on $L^1(\mathbf{H})$ or $L^1(\mathbf{H}_r)$. We will now omit the notation for the group multiplication.
\begin{definition}
	For $F_1, F_2 \in L^1(\mathbf{H})$, the (left) convolution is defined as
	\begin{equation}\label{eq_H_convolution}
		(F_1 * F_2)(\mathbf{h}_0) = \int_{\mathbf{H}} F_1(\mathbf{h}) F_2(\mathbf{h}^{-1} \mathbf{h}_0) \, d\mathbf{h}.
	\end{equation}
\end{definition}
It is obvious how to come up with the definition of the convolution for $\mathbf{H}_r$. With this operation, the Heisenberg group becomes a Banach algebra and satisfies
\begin{equation}
	\norm{F_1 * F_2}_{L^1(\mathbf{H})} \leq \norm{F_1}_{L^1(\mathbf{H})}\norm{F_2}_{L^1(\mathbf{H})} .
\end{equation}
While $L^1(\mathbf{H})$ and $L^1(\R^{2d+1})$ coincide as Banach spaces, they differ as Banach algebras. $L^1(\R^{2d+1})$ with the usual convolution is commutative, while $L^1(\mathbf{H})$ is not commutative. We will now compare the two convolutions. In terms of coordinates, \eqref{eq_H_convolution} becomes
\begin{equation}
	(F_1 * F_2)(x_0,\omega_0,\tau_0) = \int_{\R^{2d}}\int_{\R} F_1(x,\omega,\tau)F_2(x_0-x,\omega_0-\omega,\tau_0-\tau+\tfrac{1}{2}(x \cdot \omega_0 - x_0 \cdot \omega)) \, d\tau \, d(x,\omega) .
\end{equation}

In time-frequency analysis, the variables $x$ and $\omega$ have a physical meaning, while the auxiliary variable $\tau$ is added to create a group structure. It is sometimes necessary to extend a function from the time-frequency plane to the Heisenberg group. Since $\T$ is compact, it is often more convenient to extend to the reduced Heisenberg group $\mathbf{H}_r$. The extension we use is given by
\begin{equation}
	F_r(x,\omega,e^{2 \pi i \tau}) = e^{-2 \pi i \tau} F(x,\omega).
\end{equation}
We note that $\norm{F}_{L^p(\R^{2d})} = \norm{F_r}_{L^p(\mathbf{H}_r)}$ and the convolution of the extension yields a new operation on $L^1(\R^{2d})$:
\begin{equation}
	F_r *_{\mathbf{H}_r} G_r(x, \omega) = (F_r * G_r)(x,\omega,e^{2 \pi i \tau})
\end{equation}
Let $F,G \in L^1(\R^{2d})$, then
\begin{align}
	& (F_r * G_r)(x,\omega,e^{2 \pi i \tau})\\
	= \, & \int_{\R^{2d}} \int_0^1 F_r(x',\omega',e^{2 \pi i \tau'}) G_r(x-x',\omega-\omega', e^{2 \pi i (\tau - \tau')} e^{\pi i (x' \cdot \omega - x \cdot \omega')} \, d\tau' \, d(x',\omega')\\
	= \, & e^{-2 \pi i \tau} \int_{\R^{2d}} F(x',\omega') G(x-x',\omega-\omega') e^{\pi i (x' \cdot \omega - x \cdot \omega')} \, d(x',\omega').
\end{align}
This leads to the following notation.
\begin{definition}
	For two functions $F, G \in L^1(\R^{2d})$, the twisted convolution is given by
	\begin{equation}\label{eq_def_twisted_conv}
		F \natural G(x,\omega) = \int_{\R^{2d}} F(x',\omega')G(x-x',\omega-\omega')e^{\pi i(x' \cdot \omega - x \cdot \omega')} \, d(x',\omega') .
	\end{equation}
\end{definition}
With this notation we get
\begin{equation}\label{eq_twisted_conv_Hr}
	F_r *_{\mathbf{H}_r} G_r = (F \natural G)_r .
\end{equation}
Most properties of the ordinary convolution carry over to the twisted convolution. However, as the twisted convolution is intimately connected to $\mathbf{H}_r$, it is non-commutative and enjoys some properties which might come as a surprise. For example, $L^2 \natural L^2 \subset L^2$ and $L^2$ is actually a Banach algebra under the twisted convolution.

We started our excursion to the Heisenberg group from the symmetric time-frequency shifts $\rho(\l)$. However, we could as well have started from time-frequency shifts $\pi(\l)$, which lead, again, to a new group law. To have our notation closer to the standard literature, however, we will not use $\pi(\l)$ but its reflected version $\widetilde{\pi}(\l) = T_x M_\omega$. The composition of two such time-frequency shifts yields
\begin{equation}
	\widetilde{\pi}(\l) \widetilde{\pi}(\l') = e^{2 \pi i x' \cdot \omega} \widetilde{\pi}(\l+\l').
\end{equation}
\begin{definition}
	The polarized Heisenberg group $\mathbf{H}^{pol}$ is $\R^{2d+1}$ under the group multiplication
	\begin{equation}\label{eq_Hpol_multiplication}
		(x,\omega,\tau) \odot (x',\omega',\tau') = (x+x',\omega+\omega', \tau+\tau'+x'\cdot\omega) .
	\end{equation}
\end{definition}
The quotient $\mathbf{H}^{pol}/(\{0\}\times\{0\}\times\Z)$ under the homomorphism $(x,\omega,\tau) \mapsto (x,\omega,e^{2 \pi i \tau})$ is then the reduced polarized Heisenberg group under the group multiplication
\begin{equation}
	(x,\omega,e^{2 \pi i \tau}) \circledcirc (x',\omega',e^{2 \pi i \tau'}) = (x+x',\omega+\omega',e^{2 \pi i(\tau+\tau'+x'\cdot\omega)}) .
\end{equation}
Note that the inverse element in $\mathbf{H}^{pol}$ is given by $(x,\omega,\tau)^{-1} = (-x,-\omega,-\tau+x\cdot \omega)$, whereas in $\mathbf{H}$ it was given by $(x,\omega,\tau)^{-1} = (-x,-\omega,-\tau)$.

We note that $\mathbf{H}^{pol}$ can be identified with a matrix group with elements of the form
\begin{equation}
	\begin{pmatrix}
		1 & \omega_1 & \ldots & \omega_d & \tau\\
		0 & 1 & \ldots & 0 & x_1\\
		\vdots & & \ddots & & \vdots\\
		0 & 0 & \ldots & 1 & x_d\\
		0 & 0 & \ldots & 0 & 1
	\end{pmatrix}
\end{equation}
From this identification, it is rather easy to note that \eqref{eq_Hpol_multiplication} indeed defines a group multiplication. The polarized version of the Heisenberg group with multiplication given by \eqref{eq_Hpol_multiplication} can be generalized to associate a Heisenberg-type group to any locally compact Abelian group \cite{Igu72}. On the other hand, the appearance of the symplectic form is special to the symmetric version of the Heisenberg group.

On $\Rd$, however, the choice of $\mathbf{H}$ versus $\mathbf{H}^{pol}$ is more a matter of taste and convenience, as the following lemma shows.
\begin{lemma}\label{lem_iso_H_Hpol}
	$\mathbf{H}$ and $\mathbf{H}^{pol}$ are isomorphic via the group isomorphism
	\begin{equation}
		\iota(x,\omega,\tau) = (x,\omega, \tau + \tfrac{1}{2} x \cdot \omega)
	\end{equation}
	from $\mathbf{H}$ onto $\mathbf{H}^{pol}$.
\end{lemma}
\begin{proof}
	We see that $\iota$ is a bijection on $\R^{2d+1}$. For $\mathbf{h}, \mathbf{h}' \in \mathbf{H}$ we compute
	\begin{align}
		\iota (\mathbf{h} \bullet \mathbf{h}') & = \iota\left(x+x',\omega+\omega',\tau+\tau'+ \tfrac{1}{2}(x' \cdot \omega - x \cdot \omega') \right)\\
		& = \left(x+x',\omega+\omega',\tau+\tau'+\tfrac{1}{2}(x' \cdot \omega - x \cdot \omega')+\tfrac{1}{2} (x+x')\cdot(\omega+\omega')\right) .
	\end{align}
	Now, we compute the multiplication of $\iota(\mathbf{h})$ and $\iota(\mathbf{h}')$ in $\mathbf{H}^{pol}$;
	\begin{align}
		\iota(\mathbf{h}) \odot \iota(\mathbf{h}') & = \left(x,\omega,\tau+\tfrac{1}{2}x \cdot \omega\right) \odot \left(x',\omega',\tau'+\tfrac{1}{2}x' \cdot \omega'\right)\\
		& = \left(x+x',\omega+\omega',\tau+\tau'+\tfrac{1}{2}(x \cdot \omega + x' \cdot \omega') + x' \cdot \omega \right)\\
		& = \iota(\mathbf{h} \bullet \mathbf{h}') .
	\end{align}
\end{proof}

\subsection{Representation Theory}
In the previous part, we have chosen a path contrary to the usual direction in mathematics. We deduced the Heisenberg group from the composition rules of time-frequency shifts. More often, however, algebraic structures are studied by means of their morphisms into matrices or operators. This study is called representation theory and is an own mathematical branch. We will only collect the most important notions and results. Also, we note that the concepts we are about to introduce make sense for arbitrary locally compact groups, but we will limit ourselves to the study of the Heisenberg group.

\begin{definition}\label{def_uni_rep}
	A (unitary) representation $(\mathcal{H}, \pi)$ of a locally compact group $\mathbf{H}$ is a strongly continuous homomorphism $\pi$ from $\mathbf{H}$ into the group of unitary operators $\mathcal{U}(\mathcal{H})$ on the Hilbert space $\mathcal{H}$. This means that
	\begin{enumerate}[(i)]
		\item $\pi( \mathbf{h}_1 \mathbf{h}_2) = \pi( \mathbf{h}_1) \pi( \mathbf{h}_2), \quad \mathbf{h}_1, \mathbf{h}_2 \in \mathbf{H}$.
		\item $\pi(\mathbf{h}^{-1}) = \pi(\mathbf{h})^{-1} = \pi(\mathbf{h})^*, \quad \mathbf{h} \in \mathbf{H}$.
		\item If $\lim_{k \to \infty} \mathbf{h}_k = \mathbf{h}$ and $f \in \mathcal{H}$, then $\lim_{k \to \infty} \pi(\mathbf{h}_k) f = \pi(\mathbf{h}) f$.
	\end{enumerate}
\end{definition}

\begin{definition}
	Two representations $(\mathcal{H}_1, \pi_1)$ and $(\mathcal{H}_2, \pi_2)$ are equivalent if there exists a unitary operator $U : \mathcal{H}_1 \to \mathcal{H}_2$ such that
	\begin{equation}
		U \, \pi_1(\mathbf{h}) \, U^{-1} = \pi_2(\mathbf{h}) \quad \forall \mathbf{h} \in \mathbf{H}.
	\end{equation}
	In this case, $U$ is called the intertwining operator.
\end{definition}
Often the attention is restricted to a special class of representations, namely irreducible representations.
\begin{definition}
	A representation $(\mathcal{H}, \pi)$ is called irreducible if $\{0\}$ and $\mathcal{H}$ are the only closed subspaces that are invariant \footnote{A subspace $V \subset \mathcal{H}$ is called invariant under an operator $\pi: \mathcal{H} \to \mathcal{H}$ if $\pi V \subset V$, i.e., $\pi v \in V$ for all $v \in V$.} under all operators $\pi(\mathbf{h})$, $\mathbf{h} \in \mathbf{H}$.
\end{definition}

The goal of representation theory is to understand all unitary representations of a locally compact group. Since in a technical sense every representation can be decomposed into irreducible representations, it is often sufficient to classify these. Our modest goal is to understand the basic aspects of the representation theory of the Heisenberg group and to find all its representations.

Actually, the heart of time-frequency analysis is one particular representation of $\mathbf{H}$.

\begin{example}[The Schrödinger Representation]
	The Schrödinger representation is a representation of the Heisenberg group by (symmetric) time-frequency shifts. It acts on $\Lt$ by means of the unitary operators
	\begin{equation}
		\rho(\mathbf{h}) = \rho(x, \omega, \tau) = e^{2 \pi i \tau} M_{\omega/2} T_x M_{\omega/2} = e^{2 \pi i \tau} e^{ \pi i x \cdot \omega} T_x M_\omega .
	\end{equation}
	\flushright{$\diamond$}
\end{example}
It follows from \eqref{eq_Hr_composition} and \eqref{eq_Hr_multiplication} that $\rho$ is indeed a unitary representation of $\mathbf{H}$. By using the dilations $\delta_\kappa(x,\omega,\tau) = (\kappa x, \omega, \kappa \tau)$ with $\kappa \in \R \backslash \{0\}$, we obtain a one-parameter family $\rho_\kappa$ of unitary representations as follows:
\begin{equation}
	\rho_\kappa(\mathbf{h}) = \rho \circ \delta_\kappa(\mathbf{h}) = \rho(\kappa x, \omega , \kappa \tau) =  e^{2 \pi i \kappa \tau} e^{\pi i \kappa x \cdot \omega} T_{\kappa x} M_{\omega} .
\end{equation}
With the isomorphism $\iota$ from Lemma \ref{lem_iso_H_Hpol} these representations can be easily transformed to the polarized Heisenberg group. For $(x, \omega, \tau) \in \mathbf{H}^{pol}$, the representation $\rho'_\kappa$ is given by
\begin{equation}
	\rho'_\kappa(x,\omega,\tau) = \rho_\kappa( \iota^{-1}(x,\omega,\tau)) = e^{2 \pi i \kappa \tau} T_{\kappa x} M_\omega,
\end{equation}
which is compatible with the composition of time-frequency shifts $\pi^\vee(x,\omega) = T_x M_\omega$.

As functions are usually easier to deal with than operators, one may consider taking inner products to go from a representation $(\mathcal{H}, \pi)$ to functions on the corresponding group.
\begin{definition}
	Let $(\mathcal{H}, \pi)$ be a representation of $\mathbf{H}$ and $f,g \in \mathcal{H}$. Then, the function $\mathbf{h} \mapsto \langle f, \pi(\mathbf{h}) g \rangle$ is called a representation coefficient of $\pi$.
\end{definition}

\begin{example}
	For the Schrödinger representation we obtain the following representation coefficients:
	\begin{equation}\label{eq_rep_coeff_Schrödinger}
		\langle f, \rho(x, \omega, \tau) g \rangle = e^{-2 \pi i \tau} \langle f, \rho(x,\omega) g \rangle = e^{-2 \pi i \tau} A(f,g)(x,\omega) = e^{-2 \pi i \tau} e^{\pi i x \cdot \omega} V_g f(x, \omega).
	\end{equation}
	\flushright{$\diamond$}
\end{example}
Thus, up to the phase factor $e^{-2 \pi i \tau} e^{\pi i x \cdot \omega}$, the representation coefficients of $\rho$ coincide with the STFT. Therefore, we may use our knowledge on the STFT to analyze the Schrödinger representation. On the other hand, abstract methods from representation theory are applicable to the STFT. The identity \eqref{eq_rep_coeff_Schrödinger} is the very reason why the Heisenberg group and its representation theory play such an important role in time-frequency analysis.

We will now state the main properties of the Schrödinger representation.
\begin{theorem}\label{thm_Schrödinger_irreducible}
	Each representation $\rho_\kappa$, $\kappa \in \R \backslash \{0\}$, is an irreducible, unitary representation of $\mathbf{H}$ with kernel $\{0\} \times \{0\} \times \tfrac{1}{\kappa} \Z$.
\end{theorem}
\begin{proof}
	% UE
	We verify the list of properties specified in Definition \ref{def_uni_rep}. Obviously, $\pi(\mathbf{h})$ is a unitary operator for each $\mathbf{h} \in \mathbf{H}$. It follows from the group law of $\mathbf{H}$ and the composition law of $\rho(x,\omega,\tau)$ that $\rho$ is a homomorphism from $\mathbf{H}$ into the group of unitary operators. Furthermore
	\begin{align}
		\rho(\mathbf{h}^{-1}) = \rho(-x,-\omega-\tau)
		= e^{-2 \pi i \tau} M_{-\omega/2} T_{-x} M_{-\omega/2}
		= \rho(\mathbf{h})^*.
	\end{align}
	
	\textit{Irreducibility}: Suppose $\mathcal{K} \neq \{0\}$ is an invariant closed subspace for $\rho$, i.e., $\rho \mathcal{K} \subset \mathcal{K}$. We have to show that $\mathcal{K} = \mathcal{H} = \Lt$.
	
	Fix a non-zero element $g \in \mathcal{K}$ and let $f \in \mathcal{K}^\perp$. Since $\rho(\mathbf{h}) g \in \mathcal{K}$ by the invariance of $\mathcal{K}$, we conclude that
	\begin{equation}
		0 = | \langle f, \rho(\mathbf{h}) g \rangle| = |A(f,g)(x,\omega)| = |V_gf(x,\omega)| \quad \forall (x,\omega) \in \R^{2d}.
	\end{equation}
	Since the STFT, and hence the ambiguity function, is one-to-one, we conclude that $f = 0$ and therefore $\mathcal{K}^\perp = \{0\}$.
	
	\textit{Strong continuity}: Suppose that $\mathbf{h}_n \to \mathbf{h}$ in $\mathbf{H}$ and $f \in \Lt$. Since $\norm{\rho(\mathbf{h}_n) f - \rho(\mathbf{h}) f}_2 = \norm{\rho(\mathbf{h}^{-1} \mathbf{h}_n) f - f}_2$, it suffices to assume that $\lim_{n \to \infty} \mathbf{h}_n = 0$ and to show that
	\begin{equation}
		\lim_{|x_n|+|\omega_n|+|\tau_n| \to 0} \norm{e^{2 \pi i \tau_n} M_{\omega_n/2} T_{x_n} M_{\omega_n/2}f - f}_2 = 0.
	\end{equation}
	This follows from the separate continuities
	\begin{equation}
		\lim_{|\tau_n| \to 0}\norm{e^{2 \pi i \tau_n} f - f}_2 = 0, \, \lim_{|x_n| \to 0} \norm{T_x f - f}_2 = 0 \, \text{ and } \, \lim_{|\omega_n| \to 0} \norm{M_\omega f - f}_2 = 0.
	\end{equation}
	
	\textit{Kernel of $\rho$}: Obviously $\rho(0,0,k) = I$ for $k \in \Z$. Conversely, if $\rho(x,\omega,\tau) = I$, then $e^{2 \pi i \tau} e^{\pi i x \cdot \omega} M_{\omega} f = T_{-x} f$ for all $f \in \Lt$. If $x \neq 0$, choose $f \neq 0$ with $\supp(f) \subset \{t \in \Rd \mid |t| \leq |x|/2 \}$. Then $\supp(M_\omega f) \cap \supp(T_{-x} f) = \emptyset$, which is a contradiction. Thus $x = 0$, which yields $e^{2 \pi i \tau} M_\omega f = f$. The same argument applied to $e^{2 \pi i \tau} T_\omega \widehat{f} = \widehat{f}$ yields $\omega = 0$. This implies that $\rho(\mathbf{h}) = I$ only for $\mathbf{h} = (0,0,k)$, $k \in \Z$.
	
	Since the dilation $\delta_\kappa(x,\omega,\tau) = (\kappa x, \omega, \kappa \tau)$ is an automorphism of $\mathbf{H}$ and since $\rho_\kappa = \rho \circ \delta_\kappa$, the representations $\rho_\kappa$ are irreducible with kernel $\delta_\kappa^{-1} (\{0\}\times\{0\}\times\Z) = \{0\}\times\{0\}\times\frac{1}{\kappa}\Z$.
\end{proof}

Since the kernel of $\rho$ is $\{0\} \times \{0\} \times \Z$, $\rho$ restricts to an irreducible, unitary representation of the reduced Heisenberg group $\mathbf{H}_r$, which we also call $\rho$, via
\begin{equation}
	\rho(x, \omega, e^{2 \pi i \tau}) = e^{2 \pi i \tau} e^{\pi i x \cdot \omega} T_x M_\omega .
\end{equation}
If we work only with the Schrödinger representation, it is sufficient to consider $\mathbf{H}_r$ instead of $\mathbf{H}$. In this case, the compactness of the center of $\mathbf{H}_r$ is advantageous in certain constructions.

For a deeper investigation of the representation of $\mathbf{H}$ we need another (standard) construction from representation theory, namely the integrated representation. This will set up a bijective correspondence between unitary representations of $\mathbf{H}$ and so-called non-degenerate $*$-representations of the Banach algebra $L^1(\mathbf{H})$.

\begin{definition}[Integrated Representation]
	Given a unitary representation $(\mathcal{H}, \pi)$ of $\mathbf{H}$ and given $F \in L^1(\mathbf{H})$, $\pi(F)$ denotes the operator
	\begin{equation}\label{eq_rep_operator}
		\pi(F) = \int_{\mathbf{H}} F(\mathbf{h}) \pi(\mathbf{h}) \, d \mathbf{h} .
	\end{equation}
	The operator-valued integral is defined in the weak sense by
	\begin{equation}
		\langle \pi(F) f, g \rangle = \int_{\mathbf{H}} F(\mathbf{h}) \langle \pi(\mathbf{h}) f, g \rangle \, d \mathbf{h},
	\end{equation}
	for all $f,g \in \mathcal{H}$.
\end{definition}
\begin{example}
	The irreducible unitary representations of $\Rd$ can be identified with the characters $\chi_\omega(x) = e^{2 \pi i x \cdot \omega}$, $x, \omega \in \Rd$. The integrated representation is then given by
	\begin{equation}
		\chi_\omega(f) = \int_{\Rd} f(x) e^{2 \pi i x \cdot \omega} \, dx = \widehat{f}(-\omega).
	\end{equation}
	The integrated representation can thus be seen as a substitute for the Fourier transform on non-Abelian groups.
	\flushright{$\diamond$}
\end{example}
The main properties of the integrated representation readily follow from its definition.
\begin{proposition}\label{pro_int_rep}
	\begin{enumerate}[(a)]
		\item\label{pro_int_rep_a} The operator $\pi(F)$ is bounded on $\mathcal{H}$ with operator norm
		\begin{equation}
			\norm{\pi(F)}_{op} \leq \norm{F}_1 .
		\end{equation}
		\item\label{pro_int_rep_b} The mapping $F \mapsto \pi(F)$ is an algebra homomorphism from $L^1(\mathbf{H})$ into the set of bounded operators on $\mathcal{H}$.
		\item\label{pro_int_rep_c} Let $F^*(\mathbf{h}) = \overline{F(\mathbf{h}^{-1})}$ denote the involution on $L^1(\mathbf{H})$. Then $\pi(F^*) = \pi(F)^*$.
	\end{enumerate}
\end{proposition}
\begin{proof}
	\begin{enumerate}[(a)]
		\item The sesquilinear form $(f,g) \mapsto \int_{\mathbf{H}} F(\mathbf{h}) \langle \pi(\mathbf{h}) f, g \rangle \, d\mathbf{h}$ satisfies the estimates
		\begin{equation}
			| \langle \pi(F) f, g \rangle| \leq \sup_{\mathbf{h} \in \mathbf{H}} |\langle \pi(\mathbf{h}) f, g \rangle| \int_\mathbf{H} |F(\mathbf{h})| \, d \mathbf{h} \leq \norm{F}_1 \norm{f}_\mathcal{H} \norm{g}_\mathcal{H} .
		\end{equation}
		Therefore, $\pi(\mathcal{H})$ defines a bounded linear operator on $\mathcal{H}$ \footnote{If $Q: \mathcal{H} \times \mathcal{H} \to \C$ is a bounded sesquilinear form on the Hilbert space $\mathcal{H}$, satisfying the estimate $|Q(f,g)| \leq C \norm{f}_\mathcal{H} \norm{g}_\mathcal{H},$ then there exists a unique bounded operator $A$ on $\mathcal{H}$ with norm $\norm{A}_{op} \leq C$ such that $Q(f,g) = \langle A f, g \rangle_{\mathcal{H}}.$} (see also \cite{Con_FA90}) with $\norm{\pi(F)}_{op} \leq \norm{F}_1$.
		\item Since $F \mapsto \pi(F)$ is obviously linear, we only need to check that $\pi(F_1 * F_2) = \pi(F1)\pi(F_2)$. The proof resembles the proof of the multiplicativity of the Fourier transform, that is, $\widehat{f_1 * f_2} = \widehat{f_1} \widehat{f_2}$. With $f,g \in \mathcal{H}$ we compute first that
		\begin{align}
			\langle \pi(F_1 * F_2)f, g \rangle
			& = \int_{\mathbf{H}} (F_1 * F_2) (\mathbf{h}) \, \langle \pi(\mathbf{h}) f, g \rangle \, d \mathbf{h}\\
			& = \int_\mathbf{H} \int_\mathbf{H} F_1(\mathbf{h}') F_2(\mathbf{h}'^{-1} \mathbf{h}) \langle \pi(\mathbf{h}) f, g \rangle \, d\mathbf{h}' \, d \mathbf{h}\\
			& = \int_\mathbf{H} F_1(\mathbf{h}') \int_\mathbf{H} F_2(\mathbf{h}) \langle \pi(\mathbf{h}' \mathbf{h}) f, g \rangle \, d\mathbf{h} \, d \mathbf{h}' ,
		\end{align}
		by using the substitution $\mathbf{h} \mapsto \mathbf{h}' \mathbf{h}$ and Lemma \ref{lem_Haar_measure}. Interchanging the order of integration is justified as $F_1, F_2 \in L^1(\mathbf{H})$ and the representation coefficient is bounded. On the other hand,
		\begin{align}
			\langle \pi(F_1) \pi(F_2) f, g \rangle
			& = \int_\mathbf{H} F_1(\mathbf{h}') \langle \pi(\mathbf{h}') \pi(F_2)f, g \rangle \, d\mathbf{h}'\\
			& = \int_\mathbf{H} F_1(\mathbf{h}') \int_\mathbf{H} F_2(\mathbf{h}) \langle \pi(\mathbf{h}) f, \pi(\mathbf{h}')^* g \rangle \, d\mathbf{h} \, d\mathbf{h}'\\
			& = \int_\mathbf{H} F_1(\mathbf{h}') \int_\mathbf{H} F_2(\mathbf{h}) \langle \pi(\mathbf{h}' \mathbf{h}) f, g \rangle \, d\mathbf{h} \, d \mathbf{h}'\\
			& = \langle \pi(F_1 * F_2) f, g \rangle.
		\end{align}
		Thus $\pi(F_1 * F_2) = \pi(f_1) \pi(F_2)$.
		\item We have $\langle \pi(F^*) f, g \rangle = \int_\mathbf{H} \overline{F(\mathbf{h}^{-1})} \langle \pi(\mathbf{h}), g \rangle \, d\mathbf{h}$, whereas
		\begin{align}
			\langle \pi(F)^* f, g \rangle
			& = \langle f, \pi(F) g \rangle = \overline{\langle \pi(F) g, f \rangle}\\
			& = \int_\mathbf{H} \overline{F(\mathbf{h})} \overline{\langle \pi(\mathbf{h}) g, f \rangle} \, d \mathbf{h}\\
			& = \int_\mathbf{H} \overline{F(\mathbf{h})} \langle \pi(\mathbf{h}^{-1}) f, g \rangle \, d \mathbf{h} \qquad \left(\pi(\mathbf{h})^* = \pi(\mathbf{h})^{-1} = \pi(\mathbf{h}^{-1})\right)\\
			& = \int_\mathbf{H} \overline{F(\mathbf{h}^{-1})} \langle \pi(\mathbf{h}) f, g \rangle \, d \mathbf{h}
		\end{align}
		since the integral is invariant under the substitution $\mathbf{h} = (x, \omega, \tau) \mapsto \mathbf{h}^{-1} = (-x, -\omega, - \tau)$ (because the Haar measure is the Lebesgue measure). Thus $\pi(F^*) = \pi(F)^*$.
	\end{enumerate}
\end{proof}
We consider this construction for the Schrödinger representation $\rho$. Since $\rho$ is a representation of $\mathbf{H}_r$, we may consider the integrated representation on $L^1(\mathbf{H}_r)$. It is convenient to omit the (trivial) third component from $\rho$ and to write
\begin{equation}
	\rho(x,\omega) = \rho(x,\omega,e^{2 \pi i 0})	
\end{equation}
and
\begin{equation}
	\rho(x,\omega,e^{2 \pi i \tau}) = e^{2 \pi i \tau} \rho(x, \omega) .
\end{equation}
After expanding $\Phi \in L^1(\mathbf{H}_r)$ into a Fourier series with respect to the third coordinate, that is,
\begin{equation}
	\Phi(x,\omega,\tau) = \sum_{k \in \Z} \Phi_k(x,\omega) e^{2 \pi i k \tau},
\end{equation}
(with appropriate convergence), the integrated representation becomes
\begin{align}
	\rho(\Phi) & = \iint_{\R^{2d}} \int_0^1 \sum_{k \in \Z} \Phi_k(x, \omega)e^{2 \pi i k \tau} e^{2 \pi i \tau} \rho(x,\omega) \, d\tau \, d(x,\omega)\\
	& = \iint_{\R^{2d}} \Phi_{-1}(x,\omega) \rho(x,\omega) \, d(x,\omega) .
\end{align}
Thus, only the component $\Phi_{-1} \in L^1(\R^{2d})$ contributes. In other words, the integrated representation only sees $L^1(\R^{2d})$ instead of $L^1(\mathbf{H}_r)$. Conversely, if we (again) extend $F \in L^1(\R^{2d})$ to $L^1(\mathbf{H}_r)$ by $F_r(x,\omega,\tau) = e^{-2 \pi i \tau} F(x,\omega)$, we obtain
\begin{align}
	\rho(F_r) & = \iint_{\R^{2d}} \int_0^1 F(x,\omega) e^{-2 \pi i \tau} \rho(x,\omega) e^{2 \pi i \tau} \, d\tau \, d(x,\omega)\\
	& = \iint_{\R^{2d}} F(x,\omega) \rho(x,\omega) \, d(x, \omega) .
\end{align}
The third coordinate $\tau$ does not appear, and we may write $\rho(F) = \rho(F_r)$ without ambiguity.

We prove that the main properties of the Schrödinger representation extended to $L^1(\R^{2d})$ in a slightly more general form.
\begin{theorem}\label{thm_BanachAlgebra_twisted_conv}
	Let $(\mathcal{H}, \pi)$ be a unitary representation of $\mathbf{H}$ such that $\pi(0,0,\tau) = e^{2 \pi i \tau} I_{\mathcal{H}}$. For $F \in L^1(\R^{2d})$ write
	\begin{equation}
		\pi(F) = \pi(F_r) = \iint_{\R^{2d}} F(x,\omega) \pi(x,\omega,0) \, d(x,\omega) .
	\end{equation}
	Then
	\begin{enumerate}[(i)]
		\item\label{thm_BanachAlgebra_twisted_conv_a} $\pi( F \natural G) = \pi(F)\pi(G)$ for $F,G \in L^1(\R^{2d})$ and
		\item\label{thm_BanachAlgebra_twisted_conv_b} $\pi$ is one-to-one on $L^1(\R^{2d})$ .
	\end{enumerate}
\end{theorem}
\begin{proof}
	\begin{enumerate}[(i)]
		\item This follows from \eqref{eq_twisted_conv_Hr} and Proposition \ref{pro_int_rep} (b):
		\begin{equation}
			\pi(F \natural G) = \pi((F \natural G)_r) = \pi(F_r *_{\mathbf{H}_r} G_r) = \pi(F_r) \pi(G_r) = \pi(F) \pi(G) .
		\end{equation}
		\item Suppose that $\pi(F) = 0$. Then, for all $f,g \in \mathcal{H}$ and all $\xi, \eta \in \Rd$, we have
		\begin{align}
			0 & = \langle \pi(F) \pi(\xi, \eta) f, \pi(\xi, \eta) g \rangle\\
			& = \iint_{\R^{2d}} F(x, \omega) \langle \pi(\xi,\eta)^{-1} \pi(x,\omega) \pi(\xi,\eta) f, g \rangle \, d(x,\omega) .
		\end{align}
	\end{enumerate}
	Since $\pi(\xi,\eta)^{-1} \pi(x,\omega) \pi(\xi,\eta) = e^{-2 \pi i(x \cdot \eta - \xi \cdot \omega)} \pi(x,\omega)$, we therefore obtain for all $\xi, \eta$ that
	\begin{equation}
		\iint_{\R^{2d}} F(x, \omega) \langle \pi(x,\omega) f, g \rangle e^{-2 \pi i(x \cdot \eta - \xi \cdot \omega)} \, d(x,\omega) = 0.
	\end{equation}
	This expression, however, is just the Fourier transform of $F(x,\omega) \langle \pi(x,\omega) f, g \rangle$ at $(\eta, - \xi)$ (one could also say it is the symplectic Fourier transform of this expression at $(-\xi, -\eta)$). Since it vanishes for all of $\R^{2d}$, we conclude from the Fourier inversion theorem, that
	\begin{equation}
		F(x,\omega) \langle \pi(x,\omega) f, g \rangle = 0, \quad \text{for almost all } (x,\omega)
	\end{equation}
	and this holds for all $f,g \in \mathcal{H}$. Thus, $F(x,\omega) = 0$ (almost everywhere) and we have proved that the extended representation is one-to-one on $L^1(\R^{2d})$.
\end{proof}
One also says that $\pi$ extends to a faithful representation of the Banach algebra $L^1(\R^{2d})$ under twisted convolution.

\subsubsection{The Stone-von Neumann Theorem}
There is a classification of all irreducible unitary representations of the Heisenberg group, given by the Stone-von Neumann theorem. Its motivation lies in the commutation relations of the quantum mechanical operators. We will follow Gröchenig \cite{Gro01}, giving von Neumann's original proof, which (once again) highlights the fundamental role of Gaussians.
\begin{theorem}[Stone-von Neumann]\label{thm_SvN}
	Every irreducible representation of $\mathbf{H}$ is equivalent to exactly one of the following representations:
	\begin{enumerate}[(i)]
		\item $\chi_{a,b} (x, \omega, \tau) = e^{2 \pi i (a \cdot x + b \cdot \omega)}$ acting on $\C$, for some $(a,b) \in \R^{2d}$, or
		\item $\rho_\kappa$ for some $\kappa \in \R \backslash \{0\}$, acting on $\Lt$.
	\end{enumerate}
\end{theorem}
Since $\{0\} \times \{0\} \times \R \subset \ker \chi_{a,b}$, the characters $\chi_{a,b}$ are simply the characters of the quotient group $\mathbf{H}/(\{0\}\times\{0\}\times\R) \cong \R^{2d}$.

The proof of the Stone-von Neumann theorem requires some preparation. We will first state and prove Schur's lemma, which gives a criterion for the irreducibility of a representation. Then, we will study some statements on Gaussians in the language of representations and twisted convolutions.

\begin{lemma}[Schur's Lemma]\label{lem_Schur}
	Let $(\mathcal{H}, \pi)$ be a unitary representation of $\mathbf{H}$. Then the following are equivalent.
	\begin{enumerate}[(a)]
		\item $\pi$ is irreducible.
		\item For every $g \in \mathcal{H} \backslash \{0\}$ the subspace spanned by the finite linear combinations of $\pi(\mathbf{h}) g$, $\mathbf{h} \in \mathbf{H}$, is dense in $\mathcal{H}$.
		\item If a bounded operator $S :  \mathcal{H}  \to \mathcal{H}$ satisfies $\pi(\mathbf{h}) S = S \pi(\mathbf{h})$ for all $\mathbf{h} \in \mathbf{H}$, then $S = c I_\mathcal{H}$ for some $c \in \C$.
	\end{enumerate}
\end{lemma}
\begin{proof}
	$(a) \Longleftrightarrow (b)$: For $g \in \mathcal{H}$, $g \neq 0$, let $\mathcal{H}_g = \text{span}\{ \pi(\mathbf{h}) g \mid \mathbf{h} \in \mathbf{H} \}$. If
	\begin{equation}
		f = \sum_{k=1}^n c_k \pi(\mathbf{h_k}) g \in \mathcal{H}_g,
	\end{equation}	
	then $\pi(\mathbf{h}) f = \sum_{k=1}^n c_k \pi(\mathbf{h h}_k) g\in \mathcal{H}_g$ as well. Thus, $\mathcal{H}_g$ is invariant under $\pi$. Since all $\pi(\mathbf{h})$, $\mathbf{h} \in \mathbf{H}$, are unitary, hence bounded, operators, the closure $\overline{\mathcal{H}_g}$ is also invariant. If $\pi$ is irreducible, then $\overline{\mathcal{H}_g} = \mathcal{H}$, as claimed. 
	
	Conversely, for any $g \in \mathcal{H} \backslash \{0\}$ we have that $\overline{\mathcal{H}_g} = \mathcal{H}$ by assumption and these are the smallest (non-trivial) subspaces of $\mathcal{H}$ invariant under $\pi$. Hence, only $\{0\}$ and $\mathcal{H}$ are invariant, so $\pi$ is irreducible.
	
	\medskip
	
	$(c) \Longrightarrow (a)$: Suppose that $\mathcal{K}$ is a closed invariant subspace for $\pi$ and let $P$ be the orthogonal projection onto $\mathcal{K}$. Since $\langle f, \pi(\mathbf{h}) g\rangle = \langle \pi(\mathbf{h}^{-1}) f, g \rangle = 0$ for all $f \in \mathcal{K}$, $g \in \mathcal{K}^\perp$, and $\mathbf{h} \in \mathbf{H}$, we conclude that $\mathcal{K}^\perp$ is also invariant under $\pi$. Therefore, $P \pi(\mathbf{h})(I_\mathcal{H} - P) f = 0$ for all $f \in \mathcal{H}$, and consequently
	\begin{equation}
		P \pi(\mathbf{h}) f = P \pi(\mathbf{h}) P f = \pi(\mathbf{h}) P f, \quad \forall \mathbf{h} \in \mathbf{H}.
	\end{equation}
	Therefore, by assumption $P = c I_\mathcal{H}$. But since $P = P^2$, we must have either $c=0$ and $\mathcal{K} = \{0\}$ or $c=1$ and $\mathcal{K} = \mathcal{H}$. Thus, $\pi$ is irreducible.
	
	\medskip
	
	$(a) \Longrightarrow (c)$: Suppose that $\pi$ is irreducible and that $S \pi(\mathbf{h}) = \pi(\mathbf{h}) S$ for all $\mathbf{h} \in \mathbf{H}$. Since
	\begin{equation}
		S^* \pi(\mathbf{h}) = \left( \pi(\mathbf{h}^{-1}) S \right)^* = \left( S \pi(\mathbf{h}^{-1}) \right)^* = \pi(\mathbf{h}) S^*,
	\end{equation}
	$S^* S$ also commutes with all $\pi(\mathbf{h})$ and we may assume, without loss of generality, that $S$ is self-adjoint. If $S \neq c I_\mathcal{H}$, then the spectral theorem for (bounded) self-adjoint operators implies the existence of an orthogonal projection $P$, $P \neq 0$, $P \neq I_\mathcal{H}$, that commutes with $S$ and with all $\pi(\mathbf{h})$, $\mathbf{h} \in \mathbf{H}$. Then $P \mathcal{H}$ is invariant under $\pi$ and $P \mathcal{H}$ is not trivial. This contradicts the irreducibility of $\pi$.
\end{proof}
As announced, we are going to state more properties about Gaussians.
\begin{lemma}\label{lem_Gauss_2d}
	Let $\varphi(t) = 2^{d/4} e^{-\pi t^2}$ be the normalized standard Gaussian in $\Lt$. We define
	\begin{equation}
		\Phi^{\xi, \eta}(x, \omega) = \langle \rho(\xi,\eta) \varphi, \rho(x,\omega) \varphi \rangle
	\end{equation}
	and set $\Phi = \Phi^{0,0}$. Then
	\begin{enumerate}[(a)]
		\item\label{lem_Gauss_2d_a}
		\begin{equation}
			\Phi(x,\omega) = e^{-\tfrac{\pi}{2} (x^2+\omega^2)}
			\quad \text{ and } \quad
			\Phi^{\xi,\eta}(x,\omega) = e^{\pi i (x \cdot \eta - \xi \cdot \omega)} \Phi(x-\xi, \omega - \eta)
		\end{equation}
		\item\label{lem_Gauss_2d_b}
		\begin{equation}
			\Phi \natural \Phi^{\xi,\eta} = e^{- \tfrac{\pi}{2}(\xi^2 + \eta^2)} \Phi .
		\end{equation}
		\item\label{lem_Gauss_2d_c} If $(\mathcal{H}, \pi)$ is a unitary representation of the Heisenberg group such that $\pi(0,0,\tau) = e^{2 \pi i \tau} I_\mathcal{H}$, then
		\begin{equation}
			\pi(\xi,\eta,0)\pi(\Phi) = \pi(\Phi^{\xi, \eta}) .
		\end{equation}
	\end{enumerate}
\end{lemma}
\begin{proof}
	\begin{enumerate}[(a)]
		\item By a direct computation and using the fact that the Gaussian is invariant under the Fourier transform, we see (once again) that
		\begin{equation}
			\langle \varphi, \rho(x,\omega) \varphi \rangle = A \varphi = \Phi(x,\omega).
		\end{equation}
		Then, a direct computation shows
		\begin{align}
			\Phi^{\xi,\eta}(x,\omega) & = \langle \rho(\xi,\eta) \varphi, \rho(x,\omega) \varphi \rangle = \langle M_{\frac{\eta}{2}} T_\xi M_{\frac{\eta}{2}} \varphi, M_{\frac{\omega}{2}} T_x M_{\frac{\omega}{2}} \varphi \rangle \\
			& = \langle \varphi, M_{-\frac{\eta}{2}} T_\xi M_{-\frac{\eta}{2}} M_{\frac{\omega}{2}} T_x M_{\frac{\omega}{2}} \varphi \rangle \\
			& = e^{-\pi i \eta \cdot \xi} e^{\pi i \eta \cdot x} \langle \varphi, T_{-\xi} M_{\frac{\omega-\eta}{2}} T_x M_{\frac{\omega-\eta}{2}} \varphi \rangle \\
			& = e^{-\pi i \eta \cdot \xi} e^{\pi i \eta \cdot x} e^{-\pi i \xi \cdot(\omega-\eta)} \langle \varphi, M_{\frac{\omega-\eta}{2}} T_{x-\xi} M_{\frac{\omega-\eta}{2}} \varphi \rangle \\
			& = e^{\pi i (x \cdot \eta - \xi \cdot \omega)} \Phi(x-\xi, \omega-\eta).
		\end{align}
		
		\item According to the definition of the twisted convolution \eqref{eq_def_twisted_conv}, we have
		\begin{equation}\label{eq_Gauss_twisted}
			\Phi \natural \Phi^{\xi,\eta}(x,\omega) = \iint_{\R^{2d}} \langle \varphi, \rho(x',\omega') \varphi \rangle \langle \rho(\xi,\eta) \varphi, \rho(x-x',\omega-\omega')\varphi \rangle e^{\pi i (x\cdot \omega' - x' \cdot \omega)} \, d(x',\omega') .
		\end{equation}
		By \eqref{eq_rep_coeff_Schrödinger}, $\langle \varphi, \rho(x',\omega')\varphi \rangle = e^{\pi i x' \cdot \omega'} V_\varphi \varphi(x',\omega')$ and, by a similar calculation
		\begin{align}
			\langle \rho(\xi,\eta) \varphi, \rho(x-x',\omega-\omega') \varphi \rangle
			& = \langle \rho(\xi,\eta) \varphi, e^{\pi i(x \cdot \omega' - x' \cdot \omega)} \rho(-x',-\omega') \rho(x,\omega) \varphi \rangle \quad \text{by } \eqref{eq_Hr_composition}\\
			& = e^{-\pi i(x \cdot \omega' - x' \cdot \omega)} \langle \rho(x', \omega') \rho(\xi,\eta) \varphi, \rho(x,\omega) \varphi \rangle\\
			& = e^{-\pi i(x \cdot \omega' - x' \cdot \omega)} e^{- \pi i x' \cdot \omega'} \overline{V_{\rho(\xi,\eta) \varphi} (\rho(x,\omega)\varphi) (x',\omega')} .
		\end{align}
		Substituting these expressions into \eqref{eq_Gauss_twisted}, all the complex exponentials cancel and the twisted convolution becomes an inner product of two STFTs. We will use the orthogonality relations \eqref{eq_OR} to finish the proof.
		\begin{align}
			\Phi \natural \Phi^{\xi,\eta}(x,\omega)
			& = \iint_{\R^{2d}} V_\varphi \varphi(x',\omega') \overline{V_{\rho(\xi,\eta) \varphi} (\rho(x,\omega)\varphi) (x',\omega')} \, d(x',\omega')\\
			& = \langle \varphi, \rho(x,\omega) \varphi \rangle \overline{\langle \varphi, \rho(\xi,\eta) \varphi \rangle}\\
			& = \Phi(\xi,\eta) \Phi(x,\omega)\\
			& = e^{-\tfrac{\pi}{2}(\xi^2+\eta^2)} \Phi(x,\omega) .
		\end{align}
		\item
		\begin{align}
			\pi(\xi,\eta,0)\pi(\Phi)
			& = \iint_{\R^{2d}} \Phi(x,\omega) \pi(\xi,\eta,0) \pi(x,\omega,0) \, d(x,\omega)\\
			& = \iint_{\R^{2d}} e^{\pi i (x \cdot \eta - \xi \cdot \omega)} \Phi(x,\omega) \pi(x+\xi,\omega+\eta,0) \, d(x,\omega)\\
			& = \iint_{\R^{2d}} e^{\pi i (x \cdot \eta - \xi \cdot \omega)} \Phi(x-\xi,\omega-\eta) \pi(x,\omega) \, d(x,\omega)\\
			& = \pi(\Phi^{\xi,\eta}),
		\end{align}
		by (a).
	\end{enumerate}
\end{proof}

Note in particular that the above lemma shows that $\Phi \natural \Phi = \Phi$. This means that the twisted convolution operator $F \mapsto F \natural \Phi$ is a projection. For a given representation $(\mathcal{H}, \pi)$ of $\mathbf{H}$, this observation will help us to define a subspace of $\mathcal{H}$ which corresponds to the Gaussians in $\Lt$.

We can now prove the Stone-von Neumann theorem.
\begin{proof}(of the Stone-von Neumann theorem \ref{thm_SvN})
	Let $(\mathcal{H}, \pi)$ be an irreducible unitary representation of $\mathbf{H}$. Since
	\begin{equation}
		\pi(0,0,\tau)\pi(x,\omega,\tau') = \pi(x,\omega,\tau+\tau') = \pi(x,\omega,\tau')\pi(0,0,\tau),
	\end{equation}
	Schur's lemma implies that $\pi(0,0,\tau) = \chi(\tau) I_\mathcal{H}$, for some $\chi(\tau) \in \C$.
	
	As $\tau \mapsto \pi(0,0,\tau)$ is a homomorphism of $\R$ into the group of unitary operators we must have $|\chi(\tau)| = 1$ and $\chi(\tau_1 + \tau_2) = \chi(\tau_1)\chi(\tau_2)$. Therefore, $\chi$ is of the form $\chi(\tau) = e^{2 \pi i\kappa \tau}$ for some $\kappa \in \R$. We distinguish two cases.
	
	\textbf{Case I: $\kappa = 0$}. We see that $\pi(x,\omega,\tau)$ no longer depends on $\tau$. Thus, $\pi$ induces an irreducible representation $\widetilde{\pi}(x,\omega) = \pi(x,\omega,\tau)$ of the Abelian group $\R^{2d}$. Now, all operators $\widetilde{\pi}(x,\omega)$ commute and by Schur's lemma \ref{lem_Schur} $\widetilde{\pi}(x,\omega) = \chi(x,\omega) I_\mathcal{H}$. Therefore, $\chi$ must be a character of the form $\chi_{a,b}(x,\omega) = e^{2 \pi i (a \cdot x + b \cdot \omega)}$. Since every subspace is invariant under the action of the identity operator, the irreducibility forces $\mathcal{H}$ to be one-dimensional.
	
	\textbf{Case II: $\kappa \neq 0$}. By considering $\pi( \delta_\kappa^{-1} (x,\omega,\tau)) = \pi(\tfrac{x}{\kappa},\omega,\tfrac{\tau}{\kappa})$ we may assume, without loss of generality, that $\kappa = 1$. We note that this representation satisfies the assumptions of Theorem \ref{thm_BanachAlgebra_twisted_conv} and Proposition \ref{lem_Gauss_2d}\eqref{lem_Gauss_2d_c}.
	
	We look at the integrated representation of $\pi$ and its action on the Gaussians $\Phi(x,\omega) = \langle \varphi, \rho(x,\omega) \varphi \rangle$ and $\Phi^{\xi, \eta}(x,\omega) = \langle \rho(\xi, \eta) \varphi, \rho(x,\omega) \varphi \rangle$, which are representation coefficients of the Schrödinger representation. Theorem \ref{thm_BanachAlgebra_twisted_conv}\eqref{thm_BanachAlgebra_twisted_conv_b} and Proposition \ref{pro_int_rep}\eqref{pro_int_rep_c} imply that $\pi(\Phi) \neq 0$ and  that $\pi(\Phi) = \pi(\Phi)^*$. Now, Lemma \ref{lem_Gauss_2d} furnishes the fundamental identity
	\begin{equation}\label{eq_proof_SvN_Gaussian}
		\pi(\Phi) \pi(\xi,\eta,0) \pi(\Phi) = \pi(\Phi) \pi(\Phi^{\xi,\eta}) = \pi(\Phi \natural \Phi^{\xi,\eta}) = e^{- \tfrac{\pi}{2} (\xi^2+\eta^2)} \pi(\Phi).
	\end{equation}
	Thus, a Gaussian appears outside of $\pi(\Phi)$ almost by magic. Further, for $(\xi, \eta) = (0,0)$, we have $\pi(\Phi)^2 = \pi(\Phi)$ and so $\pi(\Phi)$ is a non-zero, orthogonal projection.
	
	Now, choose a normalized vector $g \in \pi(\Phi) \mathcal{H}$. Then $g = \pi(\Phi) g$ and $\norm{g}_\mathcal{H}^2 = \langle g, \pi(\Phi) g \rangle$. Using Lemma \ref{lem_Gauss_2d} and \eqref{eq_proof_SvN_Gaussian} (several times), we obtain
	\begin{align}\label{eq_proof_SvN_g_Gauss}
		\langle \pi(\xi,\eta,0) g, \pi(x,\omega,0) g \rangle
		& = \langle \pi(\Phi) g, \pi(-\xi,-\eta,0) \pi(x,\omega,0) \pi(\Phi) g \rangle\\
		& = \langle \pi(\Phi) g, \underbrace{\pi(x-\xi,\omega-\eta, \tfrac{1}{2}(\xi \cdot \omega - x \cdot \eta))}_{\pi(0,0,\frac{1}{2}(\xi \cdot \omega - x \cdot \eta))\pi(x-\xi,\omega-\eta,0)} \pi(\Phi) g \rangle\\
		& = e^{\pi i (x \cdot \eta - \xi \cdot \omega)} \langle g, \pi(\Phi) \pi(x-\xi,\omega-\eta,0) \pi(\Phi) g \rangle\\
		& = e^{\pi i (x \cdot \eta - \xi \cdot \omega)} e^{- \tfrac{\pi}{2} \left( (x-\xi)^2 + (\omega- \eta)^2 \right)} \langle g, \pi(\Phi) g \rangle\\
		& = \Phi^{\xi,\eta}(x,\omega)\\
		& = \langle \rho(\xi,\eta) \varphi, \rho(x,\omega) \varphi \rangle .
	\end{align}
	
	Now, we define an operator $U$ in the following way.
	\begin{equation}
		U \Big( \sum_{k=1}^n c_k \pi(x_k, \omega_k, 0) g \Big) = \sum_{k=1}^n c_k \rho(x_k, \omega_k) \varphi .
	\end{equation}
	Then, $U$ is defined on the subspace $\mathcal{H}_g \subset \mathcal{H}$ spanned by finite linear combinations of $\pi(x,\omega,0) g$. Since $\pi$ is irreducible by assumption, $\mathcal{H}_g$ is dense in $\mathcal{H}$ by Schur's lemma \ref{lem_Schur}. Also, by Theorem \ref{thm_Schrödinger_irreducible} the Schrödinger representation is irreducible, hence, $U$ has dense range in $\Lt$. We will now show, using \eqref{eq_proof_SvN_g_Gauss}, that $U$ is an isometry on $\mathcal{H}_g$, since
	\begin{align}
		\norm{U \Big( \sum_{k=1}^n c_k \pi(x_k, \omega_k, 0) g \Big)}_2^2
		& = \norm{\sum_{k=1}^n c_k \rho(x_k, \omega_k) \varphi}_2^2\\
		& = \sum_{k=1}^n \sum_{l=1}^n c_k \overline{c_l} \langle \rho(x_k,\omega_k) \varphi, \rho(x_l, \omega_l) \varphi \rangle\\
		& = \sum_{k=1}^n \sum_{l=1}^n c_k \overline{c_l} \langle \pi(x_k,\omega_k,0) g, \pi(x_l, \omega_l,0) g \rangle\\
		& = \norm{\sum_{k=1}^n c_k \pi(x_k, \omega_k, 0) g}_\mathcal{H}^2
	\end{align}
	Consequently, $U$ extends to a unitary operator from $\mathcal{H}$ onto $\Lt$.
	
	Finally, we verify that $U$ intertwines the representations $\pi$ and $\rho$. Let
	\begin{equation}
		f = \sum_{k=1}^n c_k \pi(x_k,\omega_k,0) g \in \mathcal{H}_g.
	\end{equation}
	Then
	\begin{align}
		U \pi(x,\omega,\tau) f & = U \Big( \sum_{k=1}^n c_k e^{2 \pi i \tau} \pi(x,\omega,0) \pi(x_k,\omega_k,0) g \Big)\\
		& = \sum_{k=1}^n c_k e^{2 \pi i \tau} \rho(x,\omega) \rho(x_k,\omega_k) \varphi\\
		& = \rho(x,\omega,\tau) U f .
	\end{align}
	By continuity, the above calculations extend from $\mathcal{H}_g$ to $\mathcal{H}$. Thus, we have shown that any irreducible representation $\pi$ of $\mathbf{H}$ with $\pi(x,\omega,\tau) = e^{2 \pi i \tau} \pi(x,\omega,0)$ is unitarily equivalent to the Schrödinger representation $\rho$.
\end{proof}

A slight variation in the above proof yields the following generalization of the Stone-von Neumann theorem.
\begin{theorem}\label{thm_SvN_general}
	If $(\mathcal{H}, \pi)$ is a unitary (irreducible) representation of $\mathbf{H}$, such that $\pi(x,\omega,\tau) = e^{2 \pi i \tau} \pi(x,\omega, 0)$, then $\pi$ is a (finite or infinite) direct sum of representations equivalent to $\rho$.
\end{theorem}
\begin{proof}
	Choose an orthonormal basis $\{g_j \mid j \in J\}$ of the subspace $\pi(\Phi) \mathcal{H}$. Then, in \eqref{eq_proof_SvN_g_Gauss} replacing $g$ by $g_j$ and $g_k$ shows that $\mathcal{H}_{g_j}$ is orthogonal to $\mathcal{H}_{g_k}$ for $j \neq k$. Furthermore, the same proof as above shows that $\mathcal{H}_{g_j}$ is equivalent to $\rho$.
	
	We need to show that $\mathcal{H} = \oplus_{j \in J} \mathcal{H}_{g_j}$. If $\mathcal{K} = \left( \oplus_{j \in J} \mathcal{H}_{g_j} \right)^\perp \neq \{0\}$, then $\pi$ restricted to $\mathcal{K}$ satisfies the hypotheses of the theorem , and $\pi(\Phi) \mathcal{K} \neq \{0\}$. But this contradicts the choice of the orthonormal basis $\{g_j\}$, thus $\mathcal{K} = \{0\}$.
\end{proof}

As a consequence of the Stone-von Neumann theorem we obtain the following result.
\begin{corollary}
	Every irreducible unitary representation of $\mathbf{H}_r$ is equivalent to exactly one of the following representations
	\begin{enumerate}[(i)]
		\item $\chi_{a,b}(x,\omega,e^{2 \pi i \tau}) = \chi_{a,b}(x,\omega)$ acting on $\C$, for some $(a,b) \in \R^{2d}$, or
		\item $\rho_k(x,\omega,e^{2 \pi i \tau}) = e^{2 \pi i k \tau} e^{\pi i k x \cdot \omega} T_{k x} M_\omega = e^{2 \pi i k \tau} M_{\omega/2} T_{k x} M_{\omega_2}$ for some $k \in \Z\backslash\{0\}$.
	\end{enumerate}
\end{corollary}
\begin{proof}
	Every unitary representation of $\mathbf{H}_r$ extends to a unitary representation of $\mathbf{H}$. Therefore, the irreducible unitary representations of $\mathbf{H}_r$ occur among the list of $\chi_{a,b}$ and $\pi_\kappa$, $\kappa \in \R\backslash\{0\}$. Conversely, a representation $\pi$ of $\mathbf{H}$ yields a representation of $\mathbf{H}_r \cong \mathbf{H}/(\{0\}\times\{0\}\times\Z)$, if and only if $\ker(\pi) \supset \{0\}\times\{0\}\times\Z$. Among the irreducible representations, these are exactly $\chi_{a,b}$ and $\rho_k$.
\end{proof}

\begin{remark}
	We note that in the proof of the Stone-von Neumann theorem we have (implicitly) used Plancherel's theorem at several points.
	
	Now, let
	\begin{equation}
		\pi(x,\omega,\tau) = \rho(\omega,-x,\tau) = e^{2 \pi i \tau} e^{\pi i x \cdot \omega} M_{-x} T_\omega.
	\end{equation}
	Then, $\pi$ acting on $\Lt$ yields an irreducible unitary representation of $\mathbf{H}$ with $\pi(0,0,\tau) = e^{2 \pi i \tau} I$. Therefore, it is unitarily equivalent to the Schrödinger representation $\rho$ and there exists a unitary operator $U$ (on $\Lt$), such that
	\begin{equation}
		U^{-1} \rho(x,\omega,\tau) U = \pi(x, \omega, \tau).
	\end{equation}
	We compute the effect of the Fourier transform $\F$ on $\rho$;
	\begin{align}
		\F \rho(x,\omega,\tau) & = e^{2 \pi i \tau} e^{-\pi i x \cdot \omega} \F M_\omega T_x = e^{2 \pi i \tau} e^{-\pi i x \cdot \omega} T_\omega M_{-x} \F\\
		& = e^{2 \pi i \tau} e^{ \pi i x \cdot \omega} M_{-x} T_\omega \F = \pi(x,\omega,\tau) \F.
	\end{align}
	Thus, $U \F \rho(\mathbf{h}) = U \pi(\mathbf{h}) \F = \rho(\mathbf{h}) U \F$. By Schur's Lemma \ref{lem_Schur} , $U \F = c I$, or $\F = c U^{-1}$, i.e., $\F$ is a multiple of a unitary operator. So we have to compute a concrete Fourier transform (any one) in order to determine $c$. We choose the example of the Gaussian $g(t) = e^{-\pi t^2}$ (see again \ref{thm_FT_Gauss}) and get
	\begin{equation}
		\F g(t) = g(t).
	\end{equation}
	Hence, $c = 1$ and we see that $\F$ is indeed unitary, i.e., we obtain Plancherel's theorem.
	
	At this point, it should be mentioned that we used a slightly more general form of Schur's lemma, as given in \cite[p.~823]{How80}: If you have an intertwining operator between irreducible representations, which is densely defined and its adjoint is densely defined (the operator is said to be closable), it will be unitary up to a scalar.
	
	In a similar fashion we can derive the inversion formula for the Fourier transform. Recall that $f^\vee(t) = f(-t)$. Then
	\begin{align}
		\left(\rho(x,\omega,\tau) f\right)^\vee (t) = e^{2 \pi i \tau} e^{\pi i x \cdot \omega} T_{-x} M_{-\omega} f(-t) = \rho(-x,-\omega,\tau) f^\vee(t).
	\end{align}
	So $^\vee \rho(x,\omega,\tau) = \rho(-x-\omega, \tau)^\vee$, if we denote the reflection operator by $^\vee$ as well. Composition with $\F^2$ yields
	\begin{equation}
		^\vee \F^2 \rho(x,\omega,\tau) = \, ^\vee \F \rho(\omega,-x,\tau) \F = \, ^\vee \rho(-x,-\omega,\tau) \F^2 = \rho(x,\omega,\tau)^\vee \F^2.
	\end{equation}
	Again, Schur's Lemma \ref{lem_Schur} gives $^\vee \F^2 = c I$ and again $c = 1$. Written out, the identity $^\vee \F^2 = I$ is
	\begin{equation}
		\int_{\Rd} \F f(\omega) e^{2 \pi i x \cdot \omega} \, d\omega = f(x).
	\end{equation}
	\begin{flushright}
		$\diamond$
	\end{flushright}
\end{remark}

\subsection{The Metaplectic Group}
Now, we are going to present a group of unitary operators, the so-called metaplectic operators. We will see that these operators have generators which are closely related to the generators of the symplectic group. In fact, the metaplectic group $Mp(\R,2d)$ is a (reducible) representation of the (connected) two-fold cover of the symplectic group $Sp(\R,2d)$. This is the fastest way of introducing (and defining) the metaplectic group. Equivalently, one can define it by saying that the sequence of group homomorphisms
\begin{equation}
	0 \to \Z_2 \to Mp(\R,2d) \to Sp(\R,2d) \to 0
\end{equation}
is exact. This means that the image of each homomorphism is the kernel of the next one.

%UE
However, we are going to use a more constructive approach to the metaplectic group. Since every lattice in the time-frequency plane is of the form $\L = M \Z^{2d}$ for some $M \in GL(\R,2d)$, it is first necessary to understand which of the automorphisms $z \mapsto M z$ on $\R^{2d}$, $z=(x,\omega)$, extend to automorphisms $\mathbf{a}_M(z,\tau) = (M z, \tau)$ of $\mathbf{H}$. These are exactly the automorphisms where the matrix is symplectic, because
\begin{align}
	\mathbf{a}_M(z,\tau) \mathbf{a}_M(z',\tau') = (Mz,\tau)(Mz',\tau') = (M(z+z'),\tau + \tau' + \tfrac{1}{2} \sigma(M z, M z'))
\end{align}
and
\begin{equation}
	\mathbf{a}_M((z,\tau)(z',\tau')) = \mathbf{a}_M(z+z',\tau+\tau'+\tfrac{1}{2} \sigma(z,z')) = (M(z+z'),\tau+\tau'+\tfrac{1}{2} \sigma(z,z')).
\end{equation}

By composing the Schrödinger representation with an automorphism $\mathbf{a}_S$, $S \in Sp(\R,2d)$, we obtain a new representation $\rho_S = \rho \circ \mathbf{a}_S$ defined by
\begin{equation}
	\rho_S(z,\tau) = \rho(S z, \tau).
\end{equation}
This representation has the following properties. First, $\rho_S$ is irreducible because $\rho(S z)$ still runs through all time-frequency shifts and, second, $\rho_S(0,0,\tau) = \rho(0,0,\tau) = e^{2 \pi i \tau} I_{L^2}$.

By the Stone-von Neumann theorem the representations $\rho$ and $\rho_S$ are equivalent and there exists a unitary operator $\mu(S)$ such that
\begin{equation}\label{eq_unitary_equiv}
	\rho(S z, \tau) = \rho_S(z, \tau) = \mu(S) \rho(z, \tau) \mu(S)^{-1} .
\end{equation}

We will start with computing how the Schrödinger representation interacts with 3 important unitary operators.
\begin{example}
	\begin{enumerate}[(a)]
		\item We start with the Fourier transform operator $\F$ and how it interacts with the symmetric time-frequency shifts $\rho(x,\omega) = \rho(x,\omega,0)$.
		\begin{align}
			\F \rho(x,\omega) & = e^{\pi i x \cdot \omega} \F T_x M_\omega = e^{\pi i x \cdot \omega} M_{-x} T_\omega \F\\
			& = M_{-x/2} T_\omega M_{-x/2} \F = \rho(J (x,\omega)) \F
		\end{align}
		Thus, the representation $\rho(J(x,\omega),\tau)$ is unitarily equivalent to the Schrödinger representation with the Fourier transform as intertwining operator.
		
		\item Next, we consider the dilation operator $\mathcal{D}_L f(t) = |\det(L)|^{-1/2} f(L^{-1} t)$, with $\det(L) \neq 0$. This time, we use the standard time-frequency shifts $\pi(x,\omega) = M_\omega T_x$, instead. This simplifies the notation in the following computation
		\begin{align}
			\mathcal{D}_L M_\omega T_x f(t) & = |\det(L)|^{-1/2} e^{2 \pi i \omega \cdot (L^{-1} t)} f(L^{-1} t - x)\\
			& = |\det(L)|^{-1/2} e^{2 \pi i (L^{-T} \omega) \cdot t} f( L^{-1} (t - Lx))\\
			& = M_{L^{-T} \omega} T_{L x} \mathcal{D}_L f(t).
		\end{align}
		Thus, we obtain that $\rho(x,\omega,\tau)$ and $\rho(D_L(x,\omega),\tau)$, $D_L =
		\begin{pmatrix}
			L & 0\\
			0 & L^{-T}
		\end{pmatrix}$, are unitarily equivalent.
		\item Finally, consider the linear chirp operator $\mathcal{V}_Qf(t) = e^{\pi i Qt \cdot t} f(t)$, $Q^T = Q$. We compute
		\begin{align}
			\mathcal{V}_Q \, \rho(x,\omega) \, \mathcal{V}_Q^{-1} f(t)
			& =  e^{\pi i Q t \cdot t} e^{\pi i x \cdot \omega} e^{2 \pi i \omega \cdot(t-x)} e^{- \pi i Q(t-x) \cdot (t-x)} f(t-x)\\
			& = e^{\pi i (\omega + Q x) \cdot x} e^{2 \pi i (\omega + Qx) \cdot (t-x)} f(t-x)\\
			& =\rho(x, Qx + \omega, \tau) f(t) = \rho(V_Q (x,\omega), \tau) f(t),
		\end{align}
		where $V_Q =
		\begin{pmatrix}
			I & 0\\
			Q & I
		\end{pmatrix}$.
	\end{enumerate}
\end{example}
In the examples, we have identified the intertwining operators of the Schrödinger representation for the generator matrices $J$, $D_L$ and $V_Q$ of the symplectic group. These are, up to phase factors, the Fourier transform $\F$, the dilation operator $\mathcal{D}_L$ and the linear chirp $\mathcal{V}_Q$, respectively. We note that \eqref{eq_unitary_equiv} actually defines a whole class of unitary operators $\{c \mu(S) \mid |c| = 1\}$. The choice of the phase factor is not important as long as we consider a single operator $\mu(S_1)$, but it becomes crucial for their composition. Indeed, for $S_1, S_2 \in Sp(\R,2d)$ and a particular choice of $\mu(S_1)$ and $\mu(S_2)$, and $\mu(S_1 S_2)$, the repeated application of \eqref{eq_unitary_equiv} yields
\begin{equation}
	\mu(S_1 S_2)^{-1} \mu(S_1) \mu(S_2) \, \rho(x,\omega) \, \mu(S_2)^{-1} \mu(S_1)^{-1} \mu(S_1 S_2) = \rho((S_1 S_2)^{-1} S_1 S_2 (x,\omega)) = \rho(x,\omega).
\end{equation}
Schur's lemma, however, only implies that $\mu(S_1 S_2) = c \mu(S_1) \mu(S_2)$, for some $c \in \C$, $|c|=1$, but we cannot assert that $c=1$. We might hope to adjust the phase factor for each $\mu(S)$ so that $S \mapsto \mu(S)$ becomes a homomorphism of $Sp(\R,2d)$. In fact, the phase factors of the operators $\mu(S)$ can be chosen such that either $\mu(S_1 S_2) = \mu(S_1)\mu(S_2)$ or $\mu(S_1 S_2) = -\mu(S_1)\mu(S_2)$, but this is the only freedom we have. Assuming we made the correct choice for the phase factors (up to $\pm1$), $\mu$ is then called the metaplectic representation of $Sp(\R,2d)$, which is a double-valued unitary representation of $Sp(\R,2d)$.

The fastidious way (mentioned at the beginning of the section) to deal with the double-valuedness is to pass to the double covering group of $Sp(\R,2d)$, called the metaplectic group denoted by $Mp(\R,2d)$.\footnote{$Mp(\R,2d)$ might as well be denoted by $Mp(\R,d)$ or $Sp_2(\R,2d)$ or $Sp_2(\R,d)$.} 
As promised, we will however choose a more constructive approach.

After these preliminaries, we will define a group of unitary operators, namely quadratic Fourier transforms, closely connected to the symplectic matrices. Let
\begin{equation}
	S =
	\begin{pmatrix}
		A & B\\
		C & D
	\end{pmatrix} \in Sp(\R,2d)
	\quad \text{ and } \det(A) \neq 0.
\end{equation}
Recall that, by Proposition \ref{pro_decomposition_symplectic}, a symplectic matrix with $\det(A) \neq 0$ can be decomposed into the building blocks $V_Q$, $D_L$ and $J$. In particular, the matrix
\begin{equation}
	V_P D_L J V_{-Q} J^{-1} =
	\begin{pmatrix}
		L & LQ\\
		PL & PLQ + L^{-T}
	\end{pmatrix}
\end{equation}
is symplectic, $P = P^T$, $Q = Q^T$ and $\det(L) \neq 0$. Next, we note that
\begin{equation}
	S J^{-1} =
	\begin{pmatrix}
		B & -A\\
		D & -C
	\end{pmatrix},
\end{equation}
We might as well consider the decomposition $S_W = V_P D_L J V_Q =
\begin{pmatrix}
	LQ & L\\
	PLQ - L^{-T} & PL
\end{pmatrix}$. Then, this decomposes matrices of the form $S =
\begin{pmatrix}
	A & B\\
	C & D
\end{pmatrix}$ with $\det(B) \neq 0$. There is an associated quadratic form $W(x,x') = \frac{1}{2} Px^2 - L^{-1} x \cdot x' + Q x'^2$, called the generating function\footnote{The term ``generating function" comes from classical mechanics and need not concern us any further.} of the matrix $S_W$%UE(compare also with the exercise session)
. With this notion we define the quadratic Fourier transform.

\begin{definition}
	For $W(x,x') = \frac{1}{2} Px^2 - L^{-1} x \cdot x' + \frac{1}{2} Q x'^2$, the operator
	\begin{equation}\label{eq_quadratic_FT}
		\widehat{S}_{W,m} f(x) = (-i)^{d/2} i^m \sqrt{|\det(L^{-1})|}\int_{\Rd} f(x') e^{2 \pi i W(x,x')} \, dx',
	\end{equation}
	is called the quadratic Fourier transform. The integer $m \in \{0,1,2,3\}$ is called the Maslov index\footnote{Of course, if $m$ is an appropriate choice of the Maslov index, then $m+2$ is an equally good choice, but $\widehat{S}_{W,m} = - \widehat{S}_{W,m+2}$. Thus \eqref{eq_quadratic_FT} associates two operators to each quadratic form $W$.} and is chosen such that
	\begin{equation}
		m \pi \equiv \arg (\det(L^{-1})) \mod 2 \pi .
	\end{equation}
\end{definition}
We will see that $\mu(S_W) = \widehat{S}_{W,m}$. First, we start by assigning the appropriate phase factors (up to $\pm1$) to the representations of the generator matrices. We start with the standard symplectic matrix. In this case, the quadratic form becomes $W(x,x') = - x \cdot x'$ and we have
\begin{equation}
	\widehat{J} f(x) = \mu(J) f(x) = (-i)^{d/2} \int_{\Rd} f(x') e^{-2 \pi i x \cdot x'} \, dx' = (-i)^{d/2} \F f(x).
\end{equation}
We note that
\begin{equation}
	\widehat{J}^{-1} = \mu(J^{-1}) = i^{d/2} \F^{-1}.
\end{equation}
The metaplectic operator of the matrix $D_L$ is given by
\begin{equation}
	\widehat{D}_{L,m} f(x) = \mu(D_L) f(x) = i^m \sqrt{|\det(L^{-1})|} \, f(L^{-1}x) = i^m \mathcal{D}_L f(x), \quad \det(L) \neq 0.
\end{equation}
Finally, the metaplectic operator of the matrix $V_Q$ is given by
\begin{equation}
	\widehat{V}_Q f(x) = \mu(V_Q) f(x) = e^{\pi i Q x^2} f(x), \quad Q = Q^T.
\end{equation}
We have the following factorization result for the quadratic Fourier transforms, similar to the factorization of symplectic matrices with $\det(B) \neq 0$	 (or $\det(A) \neq 0$).
\begin{proposition}
	Let $W(x,x') = \frac{1}{2} P x^2 - L^{-1} x \cdot x' + \frac{1}{2} Q x'^2$.
	\begin{enumerate}[(i)]
		\item We have the factorization
		\begin{equation}
			\widehat{S}_{W,m} = \widehat{V}_P \widehat{D}_{L,m} \widehat{J} \widehat{V}_Q .
		\end{equation}
		\item The operators $\widehat{S}_{W,m}$ are unitary (extend to unitary operators) on $\Lt$. The inverse operator of $\widehat{S}_{W,m}$ is given by
		\begin{equation}
			\widehat{S}_{W,m}^{-1} = \widehat{S}_{W^*,m^*}
			\quad \text{ and } \quad
			W^*(x,x') = -W(x',x), m^* = d-m.
		\end{equation}
	\end{enumerate}
\end{proposition}
\begin{proof}
	% UE
%	To be done in the exercise session.
	\begin{enumerate}[(i)]
		\item By definition we have
		\begin{equation}
			\widehat{J} f(x) = (-i)^{d/2} \int_{\Rd} f(x') e^{-2 \pi i x \cdot x'} \, dx'.
		\end{equation}
		We note that
		\begin{equation}
			\widehat{D}_{L,m} \widehat{J} f(x) = (-i)^{d/2} i^m \sqrt{|\det(L^{-1})|} \int_{\Rd} f(x') e^{-2 \pi i L^{-1} x \cdot x'} \, dx'.
		\end{equation}
		Applying the chirps yields
		\begin{align}
			\widehat{S}_{W,m} f(x) & = \widehat{V}_P \widehat{D}_{L,m} \widehat{J} \widehat{V}_Q f(x)\\
			& = (-i)^{d/2} i^m \sqrt{|\det(L^{-1})|} \int_{\Rd} f(x') e^{2 \pi i \left( \tfrac{1}{2} P x^2 - L^{-1} x \cdot x' + \tfrac{1}{2} Q x'^2\right)} \, dx' .
		\end{align}
		\item The operators $\widehat{V}_Q$ and $\widehat{D}_{L,m}$ are unitary and, like $\F$, the operator $\widehat{J}$ extends to a unitary operator on $\Lt$.
		
		It is easy to see that
		\begin{equation}
			\widehat{V}_Q^{-1} = \widehat{V}_{-Q}
			\quad \text{ and } \quad
			\widehat{D}_{L,m}^{-1} = \widehat{D}_{L^{-1},-m}.
		\end{equation}
		We compute
		\begin{align}
			\widehat{J}^{-1} \widehat{D}_{L^{-1},-m} & = i^{d/2} i^{-m} \sqrt{|\det(L)|} \int_{\Rd} f(L x') e^{2 \pi i x \cdot x'} \, dx'\\
			& = (-i)^{d/2} i^{d-m} \sqrt{|\det(L^{-1})} \int_{\Rd} f(x') e^{2 \pi i L^{-T} x \cdot x'} \, dx'\\
			& = \widehat{D}_{-L^{-T},d-m} \widehat{J} f(x).
		\end{align}
		From this we obtain the result by noting that
		\begin{align}
			\widehat{S}_{W,m}^{-1} f(x) & = \widehat{V}_Q^{-1} \widehat{J}^{-1} \widehat{D}_{L,m}^{-1} \widehat{V}_P^{-1} f(x) = \widehat{V}_{-Q} \widehat{D}_{-L^{-T},d-m} \widehat{J} \widehat{V}_{-P} f(x)\\
			& = (-i)^{d/2} i^{d-m} \sqrt{|\det(L^{-1})|} \int_{\Rd} f(x') e^{2 \pi i \left( -\tfrac{1}{2} Q x^2 + x \cdot L^{-1} x' - \tfrac{1}{2} P x'^2\right)} \, dx' .
		\end{align}
	\end{enumerate}
\end{proof}
From the proposition above, it follows that the operators $\widehat{S}_{W,m}$ are unitary and are thus a subset of $\mathcal{U}(\Lt)$. By adding the identity operator to this subset, we see by the Proposition that this subset is closed under inversion and, thus forms a subgroup. Adding the identity operators is necessary as it cannot be represented by an operator $\widehat{S}_{W,m}$ because $\det(B) = 0$ for the identity matrix.

\begin{definition}
	The group generated by the quadratic Fourier transforms $\widehat{S}_{W,m}$ and the identity operators is called the metaplectic group, denoted by $Mp(\R,2d)$. The elements of $Mp(\R,2d)$ are called metaplectic operators.
\end{definition}

Every $\widehat{S} \in Mp(\R,2d)$ is thus, by definition a product $\widehat{S}_{W_1,m_1} \ldots \widehat{S}_{W_k, m_k}$ of quadratic Fourier transforms associated to symplectic matrices with $\det(B) \neq 0$. Similarly to the result we had on symplectic matrices, we get a characterization of all metaplectic operators, which we state without proof and refer to \cite[Chap.~7]{Gos11}.
\begin{theorem}
	Every $\widehat{S} \in Mp(\R,2d)$ can be written as a product of two quadratic Fourier transforms $\widehat{S}_{W,m}$ and $\widehat{S}_{W',m'}$.
\end{theorem}
We note that this factorization is not unique. As a corollary we obtain the following decomposition result.
\begin{corollary}
	The metaplectic group is generated by the operators $\widehat{V}_Q$, $\widehat{D}_{L,m}$ and $\widehat{J}$.
\end{corollary}
We end our study of the metaplectic group with the natural projection map.
\begin{theorem}
	The mapping $\widehat{S}_{W,m} \mapsto S_{W}$, which to a quadratic Fourier transform
	\begin{equation}
		\widehat{S}_{W,m} f(x) = (-i)^{d/2} i^m \sqrt{|\det(L^{-1})|} \int_{\Rd} f(x) e^{2 \pi i W(x,x')} \, dx'
	\end{equation}
	associates a symplectic matrix with $\det(B) \neq 0$ and generating function $W$, extends into a surjective group homomorphism
	\begin{equation}
		\pi^{Mp}: Mp(\R,2d) \to Sp(\R,2d),
	\end{equation}
	i.e.,
	\begin{equation}
		\pi^{Mp}(\widehat{S}_1 \widehat{S}_2) = \pi^{Mp}(\widehat{S}_1) \pi^{Mp}(\widehat{S}_2).
	\end{equation}
	The kernel of $\pi^{Mp}$ is
	\begin{equation}
		\ker(\pi^{Mp}) = \{ \pm I \}.
	\end{equation}
	Hence, $\pi^{Mp} : Mp(\R,2d) \to Sp(\R, 2d)$ is a twofold covering of the symplectic group.
\end{theorem}

To end this whole chapter, we will consider some consequences for the STFT as well as for Gabor systems.
\begin{lemma}
	Let $f,g \in \Lt$ and $S \in Sp(\R,2d)$. Set $(x',\omega') = S (x,\omega)$. Then
	\begin{equation}
		V_g f(S (x,\omega)) = e^{\pi i (x \cdot \omega - x' \cdot \omega')} V_{\widehat{S}^{-1} g} (\widehat{S}^{-1} f) (x,\omega),
	\end{equation}
	where $\widehat{S} = \mu(S)$.
\end{lemma}
\begin{proof}
	The proof is by direct computation.
	\begin{align}
		V_gf(S(x,\omega)) & = e^{-\pi i x' \cdot \omega'} \langle f, \rho(S (x,\omega)) g \rangle\\
		& = e^{-\pi i x' \cdot \omega'} \langle f, \widehat{S} \rho(x,\omega) \widehat{S}^{-1} g \rangle\\
		& = e^{\pi i (x \cdot \omega - x' \cdot \omega')} \langle \widehat{S}^{-1} f, \pi(x,\omega) \widehat{S}^{-1} g \rangle\\
		& = e^{\pi i (x \cdot \omega - x' \cdot \omega')} V_{\widehat{S}^{-1} g} (\widehat{S}^{-1} f) (x,\omega).
	\end{align}
\end{proof}

This leads to the following theorem.
\begin{theorem}
	The Gabor system $\G(g,\L)$ is a frame if and only if $\G(\widehat{S} g, S \L)$ is a frame. In this case, both systems possess the same frame bounds.
\end{theorem}
\begin{proof}
	% UE
%	To be done in the exercise session.	
	We compute
	\begin{align}
		\sum_{\l \in S \L} | \langle f, \rho(\l) \widehat{S} g \rangle|^2
		& = \sum_{\l \in \L} | \langle f, \rho(S \l) \widehat{S} g \rangle|^2\\
		& = \sum_{\l \in \L} | \langle f, \widehat{S} \rho(\l) \widehat{S}^{-1} \widehat{S} g \rangle|^2\\
		& = \sum_{\l \in \L} | \langle \widehat{S}^{-1} f, \rho(\l) g \rangle|^2 .
	\end{align}
	Now, $\G(g,\L)$ is a frame, i.e.,
	\begin{equation}
		A \norm{\widehat{S}^{-1} f}_2^2 \leq \sum_{\l \in \L} | \langle \widehat{S}^{-1} f, \rho(\l) g \rangle|^2 \leq B \norm{\widehat{S}^{-1} f}_2^2, \qquad \forall f \in \Lt.
	\end{equation}
	But, we may exchange the middle expression in the frame inequality to get
	\begin{equation}
		A \norm{\widehat{S}^{-1} f}_2^2 \leq \sum_{\l \in S \L} | \langle f, \rho(\l) \widehat{S} g \rangle|^2 \leq B \norm{\widehat{S}^{-1} f}_2^2, \qquad \forall f \in \Lt.
	\end{equation}
	Since $\norm{\widehat{S}^{-1} f}_2^2 = \norm{f}_2^2$ the result follows.
\end{proof}
We also note the following property for the Gabor frame operator.
\begin{proposition}
	Let $\L \subset \R^{2d}$, $S \in Sp(\R,2d)$ and $\widehat{S} = \mu(S)$. Let $S_{g,\L}$ be the frame operator of the Gabor frame $\G(g,\L)$. Then
	\begin{equation}
		S_{\widehat{S}g, S \L} = \widehat{S} S_{g,\L} \widehat{S}^{-1}
	\end{equation}
\end{proposition}
\begin{proof}
	We expand the Gabor frame operator
	\begin{align}
		S_{\widehat{S}g, S \L} f & = \sum_{\l \in S \L} \langle f, \rho(\l) \widehat{S} g\rangle \rho(\l) \widehat{S} g\\
		& = \sum_{\l \in \L} \langle f, \rho(S\l) \widehat{S} g \rangle \rho(S \l) \widehat{S} g\\
		& = \sum_{\l \in \L} \langle  f, \widehat{S} \rho(\l) \widehat{S}^{-1} \widehat{S} g \rangle \widehat{S} \rho(\l) \widehat{S}^{-1} \widehat{S} g\\
		& = \sum_{\l \in \L} \langle \widehat{S}^{-1} f , \rho(\l) g \rangle \widehat{S} \rho(\l) g\\
		& = \widehat{S} S_{g,\L} \widehat{S}^{-1} f
	\end{align}
	As we assumed that $\G(g,\L)$ is a frame, the interchange of summation and the action of the operator $\widehat{S}$ from the second to last line to the last line is justified by Lemma \ref{lem_operator_uncond_conv}, as the sum converges unconditionally.
\end{proof}

\section{Function Spaces}
We will now introduce function spaces commonly used in time-frequency analysis. For this purpose, we first introduce spaces of smooth functions and some related notation before passing to a crash course on distribution theory. Besides \cite{Gro01}, the book by Bényi and Okoudjou \cite{BenOko_ModulationSpaces} serves as a resource for this chapter.

For $x \in \Rd$ and a multi-index $\alpha =(\alpha_1, \ldots , \alpha_d) \in \N_0^d$ we write $|\alpha| = \sum_{_k=1}^d \alpha_k$, $\alpha! = \prod_{k=1}^d \alpha_k!$, $x^\alpha = \prod_{k=1}^d x_k^{\alpha_k}$ and $\partial^\alpha$ for the higher order partial derivative operator $\prod_{k=1}^d \partial_{x_k}^{\alpha_k}$, where $\partial_{x_k}^{\alpha_k} = \frac{\partial^{\alpha_k}}{\partial_{x_k}^{\alpha_k}}$. Furthermore, we write $\alpha \leq \beta$ if $\alpha_k \leq \beta_k$ for all $k = 1, \ldots , d$.

Let $\Omega$ be an open subset of $\Rd$ and $k \in \N_0$.
\begin{definition}
	The collection of all functions $f: \Omega \to \C$ such that $\partial^\alpha f$ is continuous for all multi-indices $\alpha$ with $|\alpha| \leq k$ is denoted by $C^k(\Omega)$ and
	\begin{equation}
		C^\infty(\Omega) = \bigcap_{k=0}^\infty C^k(\Omega).
	\end{equation}
	$C^\infty_c$ denotes the collection of all functions $f \in C^\infty(\Omega)$ with compact support in $\Omega$.
\end{definition}
We note that for two functions $f,g \in C^{|\alpha|}(\Omega)$ the product rule holds;
\begin{equation}
	\partial^\alpha (f g) = \sum_{\beta \leq \alpha} \frac{\alpha!}{\beta!(\alpha-\beta)!} (\partial^\beta f)(\partial^{\alpha-\beta}g).
\end{equation}
We note that these spaces are examples of Fr\'echet spaces, which are complete metric spaces, that are not normable. For example, the topology on $C^\infty(\Omega)$ is that of uniform convergence of a function and all its derivatives on compact sets; more precisely, the topology of $C^\infty(\Omega)$ is given by seminorms
\begin{equation}
	p_{l,m}(f) = \sup \{ |\partial^\alpha f(x)| \mid x \in \mathbf{K}_l, |\alpha| \leq m \},
\end{equation}
where $\Omega = \cup_{l=1}^\infty \mathbf{K}_l$, $\mathbf{K}_l$ compact and $\mathbf{K}_l \subset \mathbf{K}_{l+1}^\circ$. The notation $\mathbf{K}^\circ$ stands for the interior of the set $\mathbf{K}$.

Note now that $C^\infty_c(\Omega)$ is a non-trivial subset of $C^\infty(\Omega)$, e.g., a properly scaled and translated version of the function $\psi(x) = e^{\frac{1}{|x|^2-1}}$, $|x| \leq 1$ and $\psi(x) = 0$ for $|x| > 1$ is in $C^\infty_c(\Omega)$. Inspired by the seminorms that determine the topology of $C^\infty(\Omega)$ it is tempting to impose on $C^\infty_c(\Rd)$ the topology given by the family of norms
\begin{equation}
	\norm{f}_m = \sup\{|\partial^\alpha f(x)| \mid x \in \Omega, |\alpha| \leq m\}.
\end{equation}
However, with this topology $C^\infty_c(\Omega)$ is incomplete.

Thus, one prefers to give this space a topology that makes it complete, although not metrizable; more precisely, the convergence $f_k \to 0$ as $k \to \infty$ in $C^\infty_c(\Omega)$ means that there exists a compact subset $\mathbf{K} \subset \Omega$ such that for all $k$ the supports of $f_k$ are contained in $\mathbf{K}$ and for each multi-index $\alpha$, $\partial^\alpha f_k$ converges uniformly to 0.

The properties of the convolution imply that if $f \in L^p(\Rd)$, $1 \leq p < \infty$ and $\varphi \in C^\infty_c(\Rd)$, then $f * \varphi \in C^\infty_c(\Rd)$, which ultimately yields that $C^\infty_c(\Rd)$ is dense in $L^p(\Rd)$.

We will now recall the definition of the Schwartz space $\mathcal{S}(\Rd)$ of rapidly decreasing functions.
\begin{definition}
	A function $f$ belongs to $\mathcal{S}(\Rd)$ if $f \in C^\infty(\Rd)$ and
	\begin{equation}
		\norm{f}_{N,\alpha} = \sup \{(1+|x|)^N |\partial^\alpha f(x)| \mid x \in \Rd \} < \infty,
	\end{equation}
	for all non-negative integers $N$ and multi-indices $\alpha$.
\end{definition}

Equivalently, the complete topology of $\mathcal{S}(\Rd)$ can be given through the family of seminorms
\begin{equation}
	p_{\alpha, \beta}(f) = \sup \{|x^\alpha \partial^\beta f(x)| \mid x \in \Rd \}.
\end{equation}
We note that for $f,g \in \mathcal{S}(\Rd)$ we also have $fg \in \mathcal{S}(\Rd)$ and $f*g \in \mathcal{S}(\Rd)$. Furthermore we have the (strict) inclusions
\begin{equation}
	C^\infty_c(\Rd) \subset \mathcal{S}(\Rd) \subset C^\infty(\Rd).
\end{equation}

In distribution theory, the space $C^\infty_c$ is called the space of test functions and usually denoted by $\mathcal{D}$. Its dual space is denoted by $\mathcal{D}'$ and is called the space of distributions. The dual space of $\mathcal{S}$ is denoted by $\mathcal{S}'$ is called the space of tempered distributions.

A (conjugate-)linear functional $\ell$ belongs to the space of tempered distributions $\mathcal{S}'$ if and only if there exist $M,N \in \N$ such that
\begin{equation}
	|\ell(f)| \lesssim \sum_{|\alpha| \leq M} \sum_{|\beta| \leq N} p_{\alpha,\beta}(f), \quad \forall f \in \mathcal{S}.
\end{equation}
It is common to write $\langle \ell, f \rangle$ instead of $\ell(f)$ \footnote{Note that we have complex conjugation in the second argument of the bracket $\langle \, . \,  , \, . \, \rangle$, so our functionals are actually anti-linear.}.

The prototype example of a tempered distribution is the Dirac delta $\delta_0$. For $f \in \mathcal{S}$, we defined $\delta_0(f) = \langle \delta_0, f \rangle = \overline{f(0)}$. We note that
\begin{equation}
	|\langle \delta_x, f \rangle| \leq \norm{f}_{0,0},
\end{equation}
which shows that $\delta_x$ is indeed an element of $\mathcal{S}'$.

Translations and modulations of a (tempered) distribution $\ell$ are given by
\begin{equation}
	\langle T_x \ell, f \rangle = \langle \ell, T_{-x}f \rangle
	\quad \text{ and } \quad
	\langle M_\omega \ell, f \rangle = \langle \ell, M_{-\omega} f \rangle .
\end{equation}
The derivative of a (tempered) distribution is given by
\begin{equation}
	\langle \partial^\alpha \ell, f \rangle = (-1)^\alpha \langle \ell, \partial^\alpha f \rangle .
\end{equation}
\begin{example}
	\begin{enumerate}[(i)]	
		\item The ``principle value distribution" $\text{p.v.} \frac{1}{x}$ is defined by
		\begin{equation}
			\langle \text{p.v.} \tfrac{1}{x}, f \rangle = \lim_{\varepsilon \to 0} \int_{|x| > \varepsilon} \frac{\overline{f(x)}}{x} \, dx.
		\end{equation}
		Splitting the integral over the intervals $(-\infty,-\varepsilon)$ and $(\varepsilon, \infty)$ and using integration by parts we get
		\begin{equation}
			\langle \text{p.v.}\tfrac{1}{x}, f \rangle = - \lim_{\varepsilon \to 0} \int_{|x| > \varepsilon} \log|x| \, \overline{f'(x)} \, dx + (\overline{f(-\varepsilon)} - \overline{f(\varepsilon)}) \log(\varepsilon).
		\end{equation}
		But $|(\overline{f(-\varepsilon)} - \overline{f(\varepsilon)}) \log(\varepsilon)| \leq 2 \norm{f'}_\infty \varepsilon |\log(\varepsilon)| \to 0$ as $\varepsilon \to 0$. Hence, we are left with
		\begin{equation}
			- \lim_{\varepsilon \to 0} \int_{|x| > \varepsilon} \log|x| \overline{f'(x)} \, dx = - \int_\R \log|x| \overline{f'(x)} \, dx,
		\end{equation}
		since $\log|x|$ is indeed integrable on $[-\varepsilon, \varepsilon]$. Therefore, as
		\begin{equation}
			- \int_\R \log|x| \overline{f'(x)} \, dx = \langle \log(|x|), -f' \rangle = \langle \log(|x|)',f\rangle,
		\end{equation}
		we see that, in the distributional sense,
		\begin{equation}
			(\log(|x|))' = \text{p.v.} \tfrac{1}{x}.
		\end{equation}
		
		\item For a locally integrable function $g$, we can construct a linear functional $\ell_g$ on $\mathcal{D}$ by
		\begin{equation}
			\langle \ell_g, f \rangle := \int_{\Rd} g(x) \overline{f(x)} \, dx .
		\end{equation}
		If $\supp(f) \subset \mathbf{K}$ compact, we have $|\langle \ell_g, f \rangle| \leq \norm{f}_0 \,  \norm{g \indicator_{\mathbf{K}}}_{L^1}$.
		
		Consider the sign function $\text{sign}(x)$, then $\text{sign}'(x) = 2 \delta_0$. For $\supp(f) \subset (-a, a)$, we compute
		\begin{equation}
			\langle \ell_{\text{sign}}', f \rangle = - \langle \ell_{\text{sign}}, f' \rangle = \int_{-a}^0 \overline{f'(x)} \, dx - \int_0^a \overline{f'(x)} \, dx = 2 \overline{f(0)}.
		\end{equation}
	\end{enumerate}
	\flushright{$\diamond$}
\end{example}

\subsection{Modulation Spaces}
In 1979, Hans Feichtinger presented the function space $S_0$ at a workshop at the University of Vienna, which now also goes under the name $M^1$ or $M^{1,1}$ (see also the original article of Feichtinger \cite{Fei81} and for a modern treatise see the article by Jakobsen \cite{Jakobsen_S0_2018}). Mainly during the 1980s, a whole new family of function spaces was then discovered and studied by Feichtinger, which we now call (weighted) modulation spaces. While we will only consider them on $\Rd$, Feichtinger introduced them for locally compact Abelian groups. At the center of the definition stands the idea to impose a natural norm condition on the convolution $M_\omega g * f$ of a modulated window function $g \in \mathcal{S}$ and a tempered distribution $f \in \mathcal{S}'$. By observing that this convolution is (basically) just the STFT of $f$ with window $g$, it becomes convenient to impose appropriate decay and integrability conditions on the STFT. This is how we will approach (unweighted) modulation space theory.

First, we introduce the following mixed-norm space.
\begin{definition}
	A function $F$ is said to belong to the function space $L^{p,q}(\R^{2d})$ if and only if
	\begin{equation}
		\norm{F}_{L^{p,q}} = \left( \int_{\Rd} \left( \int_{\Rd} |F(x,\omega)|^p \, dx \right)^{q/p} \, d\omega \right)^{1/q}.
	\end{equation}
\end{definition}

Now, we give the (classical) definition of a modulation space.
\begin{definition}
	Let $1 \leq p,q \leq \infty$ and fix $g \in \mathcal{S}(\Rd)$ (non-zero). The modulation space $M^{p,q}$ is the set of all tempered distributions $f \in \mathcal{S}'(\Rd)$ for which the $L^{p,q}(\R^{2d})$ mixed norm of the STFT is finite, i.e.,
	\begin{equation}
		\norm{f}_{M^{p,q}} = \left( \int_{\Rd} \left( \int_{\Rd} |V_gf(x,\omega)|^p \, dx \right)^{q/p} \, d\omega \right)^{1/q} < \infty
	\end{equation}
	For $p = \infty$ and/or $q = \infty$, the usual adjustment of using the essential supremum is made. If $p = q$, we simply write $M^p$ instead of $M^{p,p}$.
\end{definition}
We note that this definition clearly depends on the choice of $g$. It will be one goal to show that the modulation spaces are actually independent of the particular choice of the window $g \in \mathcal{S}(\Rd)$. However, until this result has been derived, we will pick the standard Gaussian $g_0(t) = 2^{d/4} e^{- \pi t^2}$ as the particular window to define modulation spaces. Then, we want to find their basic properties: determine whether they are Banach spaces, characterize equivalent norms, find dual spaces and inclusion properties.

\subsubsection{Time-Frequency Analysis of Distributions}
In order to answer the above questions on modulation spaces, we start with an auxiliary lemma, which we state without proof (since it follows from direct computation and the fact that $M_\omega$ and $t^\beta$ as well as $T_x$ and $\partial^\alpha$ commute). It will allow us to interchange the differential and multiplication operators with time-frequency shifts.
\begin{lemma}\label{lem_aux_Schwartz}
	If $g \in \mathcal{S}(\Rd)$, then
	\begin{equation}
		\partial^\alpha t^\beta (M_\omega T_x  g(t)) = \sum_{\gamma_1 \leq \alpha} \sum_{\gamma_2 \leq \beta} 
		\begin{pmatrix}
			\alpha \\ \gamma_1
		\end{pmatrix}
		\begin{pmatrix}
			\beta \\ \gamma_2
		\end{pmatrix}
		x^{\gamma_2} (2 \pi i \omega)^{\gamma_1} M_\omega T_x ( \partial^{\alpha-\gamma_1} t^{\beta-\gamma_2} g(t)),
	\end{equation}
	for all $(x,\omega) \in \R^{2d}$ and multi-indices $\alpha, \beta$.
\end{lemma}
The following result shows the continuity properties of the STFT for the Schwartz space and the space of tempered distributions. We recall that continuity of $f \in \mathcal{S}'(\Rd)$ means that there exists a constants $C > 0$ and integers $M,N > 0$ such that\footnote{Recall that $X$ is the position operator, which multiplies a function with its argument and $D^\alpha$ is the differentiation operator $D^\alpha: f(t) \mapsto \partial^\alpha f(t)$.}
\begin{equation}\label{eq_cont_S'}
	|\langle f, g \rangle| \leq C \sum_{|\alpha| \leq M} \sum_{|\beta| \leq N} \norm{D^\alpha X^\beta g}_\infty,
\end{equation}
for all $g \in \mathcal{S}(\Rd)$.
\begin{corollary}
	The operator-valued map $(x,\omega) \mapsto M_\omega T_x$ is strongly continuous on $\mathcal{S}(\Rd)$ and weak$^*$-continuous on $\mathcal{S}'(\Rd)$.
\end{corollary}
\begin{proof}
	For $g \in \mathcal{S}(\Rd)$ we have to show that
	\begin{equation}
		\lim_{|x|,|\omega| \to 0} \norm{D^\alpha X^\beta (M_\omega T_x g - g)}_\infty = 0,
	\end{equation}
	for all $\alpha$ and $\beta$. According to Lemma \ref{lem_aux_Schwartz} we have
	\begin{align}
		\norm{D^\alpha X^\beta (M_\omega T_x g - g)}_\infty \leq & \, \norm{M_\omega T_x(D^\alpha X^\beta g) - D^\alpha X^\beta g}_\infty\\
		& + \mathop{\sum_{\gamma_1 \leq \alpha} \sum_{\gamma_2 \leq \beta}}_{(\gamma_1,\gamma_2) \neq 0} \begin{pmatrix}
			\alpha \\ \gamma_1
		\end{pmatrix}
		\begin{pmatrix}
			\beta \\ \gamma_2
		\end{pmatrix}
		|x^{\gamma_2} (2\pi i \omega)^{\gamma_1}| \norm{D^{\alpha-\gamma_1} X^{\beta-\gamma_2} g}_\infty
	\end{align}
	The convergence of the first term, as $|x|, |\omega| \to 0$, follows immediately if $g \in C^\infty_c(\Rd)$, which is dense in $\mathcal{S}(\Rd)$. The general result follows by a density argument (3-$\varepsilon$ argument). The terms in $\mathop{\sum \sum}_{\gamma \neq 0}$ converge to 0 because $\gamma_1 \neq 0$ or $\gamma_2 \neq 0$ in all terms.
	
	If $f \in \mathcal{S}'(\Rd)$, $g \in \mathcal{S}(\Rd)$, then
	\begin{equation}
		\lim_{|x|,|\omega| \to 0} \langle M_\omega T_x f, g \rangle = \lim_{|x|,|\omega| \to 0} \langle f, T_{-x} M_{-\omega} g \rangle = \langle f, g \rangle,
	\end{equation}
	which shows the weak$^*$-continuity of $M_\omega T_x$ on $\mathcal{S}'(\Rd)$.
\end{proof}
Note that the last statement implies that the STFT of a tempered distribution is a continuous function on the time-frequency plane.
\begin{theorem}
	Let $g \in \mathcal{S}(\Rd)$ (non-zero) and $f \in \mathcal{S}'(\Rd)$. Then the STFT $V_gf$ is continuous and there are constants $C>0$, $N\geq 0$, such that
	\begin{equation}
		|V_g f(x,\omega)| \leq C (1+|x|+|\omega|)^N,
	\end{equation}
	for all $(x, \omega) \in \R^{2d}$.
\end{theorem}
\begin{proof}
	The continuity estimate \eqref{eq_cont_S'} yields
	\begin{equation}
		|\langle f, M_\omega T_x g \rangle| \leq C \sum_{|\alpha| \leq M_1} \sum_{|\beta| \leq M_2} \norm{D^\alpha X^\beta (M_\omega T_x g)}_\infty.
	\end{equation}
	Using Lemma \ref{lem_aux_Schwartz} we obtain
	\begin{align}
		|V_g f(x,\omega)| & \leq C \sum_{\stackrel{|\alpha| \leq M_1}{|\beta| \leq M_2}} \sum_{\stackrel{\gamma_1 \leq \alpha}{\gamma_2 \leq \beta}}
		\begin{pmatrix}
			\alpha \\ \gamma_1
		\end{pmatrix}
		\begin{pmatrix}
			\beta \\ \gamma_2
		\end{pmatrix}
		|x^{\gamma_2}| \, |(2 \pi i \omega)^{\gamma_1}| \, \norm{D^{\alpha-\gamma_1} X^{\beta-\gamma_2} g}_\infty\\
		& \leq C \max_{\stackrel{|\alpha|\leq M_1}{|\beta| \leq M_2}} \norm{D^\alpha X^\beta g}_\infty \sum_{\stackrel{|\alpha| \leq M_1}{|\beta| \leq M_2}} \sum_{\stackrel{\gamma_1 \leq \alpha}{\gamma_2 \leq \beta}}
		\begin{pmatrix}
			\alpha \\ \gamma_1
		\end{pmatrix}
		\begin{pmatrix}
			\beta \\ \gamma_2
		\end{pmatrix}
		|x^{\gamma_2}| \, |(2 \pi i \omega)^{\gamma_1}|.
	\end{align}
	Since $g \in \mathcal{S}(\Rd)$, the right-hand side is a polynomial of degree $M_1$ in the variables $|x_k|$ and of degree $M_2$ in the variables $|\omega_k|$. By choosing $N = \max(M_1,M_2)$ any such polynomial is bounded by $C(1+|x|+|\omega|)^N$ for a suitable constant $C$. The continuity was already shown in the previous result.
\end{proof}
In contrast, consider a linear functional on $\mathcal{S}(\Rd)$ which comes from a (measurable) function $f$ on $\Rd$ and defines a tempered distribution by $\langle \ell_f, g \rangle = \int_{\Rd} f(t) \overline{g(t)} \, dt$. Then $f$ does not necessarily possess polynomial growth. As an example consider the function
\begin{equation}
	f(t) = e^t \cos(e^t).
\end{equation}
Since $f$ is the derivative of the bounded function $\sin(e^t) \in L^\infty(\R) \subset \mathcal{S}'(\R)$, we have $f \in \mathcal{S}'(\R)$.

For a characterization of the Schwartz class we need a new technique to create Schwartz functions. This will be provided by the following result, which then leads to a characterization of $\mathcal{S}(\Rd)$.
\begin{proposition}\label{pro_aux_Schwartz}
	Fix $g \in \mathcal{S}$ (non-zero) and assume $F(x,\omega)$ has rapid decay on $\R^{2d}$, that is, for all $n \geq 0$ there is a constant $C_n > 0$ such that
	\begin{equation}
		|F(x,\omega)| \leq C_n (1+|x|+|\omega|)^{-n} .
	\end{equation}
	Then, the integral
	\begin{equation}
		f(t) = \iint_{\R^{2d}} F(x,\omega) M_\omega T_x g(t) \, d(x,\omega)
	\end{equation}
	defines a function in $\mathcal{S}(\Rd)$.
\end{proposition}
\begin{proof}
	The integral under consideration is absolutely convergent in the variable $x$ and $\omega$. Differentiating under the integral with respect to the parameter $t$ is permitted as long as the resulting integral is absolutely convergent. By our assumptions, this is true and stays true even if we replace $F$ by $F \cdot P$, where $P$ is an arbitrary polynomial. By Lemma \ref{lem_aux_Schwartz} we have
	\begin{equation}
		\partial^\alpha t^\beta f(t) =
		\sum_{\gamma_1 \leq \alpha} \sum_{\gamma_2 \leq \beta} 
		\begin{pmatrix}
			\alpha \\ \gamma_1
		\end{pmatrix}
		\begin{pmatrix}
			\beta \\ \gamma_2
		\end{pmatrix}
		\iint_{\R^{2d}} F(x,\omega) \, x^{\gamma_2} (2 \pi i \omega)^{\gamma_1} M_\omega T_x ( \partial^{\alpha-\gamma_1} t^{\beta-\gamma_2} g(t)) \, d(x, \omega).
	\end{equation}
	Set $C = \max \{ \norm{M_\omega T_x \partial^{\gamma_1} t^{\gamma_2} g}_\infty \mid \gamma_1 \leq \alpha, \gamma_2 \leq \beta \}$, which is finite since $g \in \mathcal{S}(\Rd)$, and take the supremum in the above integral with respect to $t$. Then
	\begin{equation}\label{eq_estimate_DX}
		\norm{D^\alpha X^\beta f}_\infty \leq C \iint_{\R^{2d}} |F(x,\omega)| P(x,\omega) \, d(x,\omega) < \infty,
	\end{equation}
	where $P$ is a polynomial in $|x_k|$ and $|\omega_k|$ depending only on $\alpha$ and $\beta$. In fact, $P$ is explicitly given by
	\begin{align}
		P(x, \omega) & = \sum_{\gamma_1 \leq \alpha} \sum_{\gamma_2 \leq \beta}
		\begin{pmatrix}
			\alpha \\ \gamma_1
		\end{pmatrix}
		\begin{pmatrix}
			\beta \\ \gamma_2
		\end{pmatrix}
		|x^{\gamma_2}| \, |(2 \pi i \omega)^{\gamma_1}| \\
		& = \prod_{k=1}^d (1+|x_k|)^{\beta_k} (1 + 2 \pi |\omega_k|)^{\alpha_k}.
	\end{align}
	Thus, $f \in \mathcal{S}(\Rd)$ because $F$ is rapidly decreasing.
\end{proof}
After these preparations we can now establish a characterization of $\mathcal{S}(\Rd)$ via its STFT.
\begin{theorem}\label{thm_characterization_S}
	Fix $g \in \mathcal{S}(\Rd)$ (non-zero). Then, for $f \in \mathcal{S}'(\Rd)$ the following are equivalent.
	\begin{enumerate}[(i)]
		\item $f \in \mathcal{S}(\Rd)$.
		\item $V_g f \in \mathcal{S}(\R^{2d})$.
		\item For all $n \geq 0$, there is $C_n > 0$ such that
		\begin{equation}
			|V_g f(x,\omega)| \leq C_n(1+|x|+|\omega|)^{-n}, \quad \forall (x,\omega) \in \R^{2d}.
		\end{equation}
	\end{enumerate}
\end{theorem}
\begin{proof}
	$(i) \Rightarrow (ii)$: We use the factorization $V_gf = \mathcal{F}_2 \mathcal{T}_a(f \otimes \overline{g})$ of the STFT, where $\F_2$ is the Fourier transform with respect to the second variable and $\mathcal{T}_a F(x,t) = F(t,t-x)$ is the asymmetric coordinate change. If $f,g \in \mathcal{S}(\Rd)$, then $f \otimes \overline{g} \in \mathcal{S}(\R^{2d})$. Since $\mathcal{S}(\R^{2d})$ is invariant under both operators $\mathcal{T}_a$ and $\F_2$, it follows that $V_gf \in \mathcal{S}(\R^{2d})$.
	
	$(ii) \Rightarrow (iii)$ is obvious.
	
	$(iii) \Rightarrow (i)$: Set
	\begin{equation}
		\widetilde{f} = \norm{g}_2^{-2} \iint_{\R^{2d}} V_g f(x, \omega) M_\omega T_x g \, d(x,\omega). 
	\end{equation}
	Proposition \ref{pro_aux_Schwartz} implies that $\widetilde{f} \in \mathcal{S}(\Rd)$. At the same time the inversion formula for the STFT shows that $\widetilde{f} = f$. Thus, $f \in \mathcal{S}(\Rd)$.
\end{proof}

\begin{corollary}\label{cor_seminorms_S}
	If $g \in \mathcal{S}(\Rd)$, then the collection of seminorms
	\begin{equation}
		\norm{V_g f}_{L^\infty_{(1+|z|)^s}} = \sup_{z \in \R^{2d}} (1+|z|)^s |V_g f(z)|, \quad s \geq 0,
	\end{equation}
	forms an equivalent collection of seminorms for $\mathcal{S}(\Rd)$.
\end{corollary}
\begin{proof}
	Set
	\begin{equation}
		\widetilde{\mathcal{S}}(\Rd) = \{ f \in \Lt \mid \sup_{z \in \R^{2d}} (1 + |z|)^s |V_g f(z)| < \infty, \; \forall s \geq 0\}.
	\end{equation}
	Theorem \ref{thm_characterization_S} implies that $\widetilde{\mathcal{S}}(\Rd) = \mathcal{S}(\Rd)$ as an equality between two sets. It remains to show that they have the same topology. To verify this, we apply \eqref{eq_estimate_DX} to the inversion formula
	\begin{equation}
		f = \norm{g}_2^{-2} \iint_{\R^{2d}} V_g f(x,\omega) M_\omega T_x g \, d(x,\omega).
	\end{equation}
	Then, we majorize the appearing polynomial $P$ by $(1+|z|)^n$ for sufficiently large $n$. So
	\begin{align}
		\norm{D^\alpha X^\beta f}_\infty & \leq C \iint_{\R^{2d}} |V_gf(z)| \, P(z) \, dz\\
		& \leq \widetilde{C} \iint_{\R^{2d}} |V_gf(z)| \, (1+|z|)^n \, dz\\
		& \leq \sup_{z \in \R^{2d}} \left(|V_gf(z)| \, (1+|z|)^{n+2d+1}\right) \, \widetilde{C} \iint_{\R^{2d}} (1+|z|)^{-2d-1} \, dz < \infty.
	\end{align}
	This estimate shows that the identity operator $I: \widetilde{\mathcal{S}}(\Rd) \to \mathcal{S}(\Rd)$ is continuous. By the open mapping theorem it is therefore an isomorphism, and thus the two topologies coincide on $\mathcal{S}(\Rd)$.
\end{proof}

\begin{corollary}
	Assume that $g, \widetilde{g} \in \mathcal{S}(\Rd)$ (both non-zero).
	\begin{enumerate}[(a)]
		\item If $|F(x,\omega)| \leq C (1+|x|+|\omega|)^N$ for some constants $C > 0$, $N \geq 0$, then the integral
		\begin{equation}
			f = \iint_{\R^{2d}} F(x,\omega) M_\omega T_x g \, d(x,\omega),
		\end{equation}
		defines a tempered distribution f $\in \mathcal{S}'(\Rd)$ in the sense that for all $\varphi \in \mathcal{S}(\Rd)$
		\begin{equation}
			\langle f, \varphi \rangle = \iint_{\R^{2d}} F(x,\omega) \langle M_\omega T_x g, \varphi \rangle \, d(x,\omega).
		\end{equation}
		
		\item In particular, if $F = V_g f$ for some $f \in \mathcal{S}'(\Rd)$, then
		\begin{equation}\label{eq_inv_STFT_S}
			f = \frac{1}{\overline{\langle g, \widetilde{g} \rangle}} \iint_{\R^{2d}} V_g f(x,\omega) M_\omega T_x \widetilde{g} \, d(x,\omega).
		\end{equation}
	\end{enumerate}
\end{corollary}
\begin{proof}
	\begin{enumerate}[(a)]
		\item The integral converges absolutely because $|F(z)| = \mathcal{O}(|z|^N)$ with $N$ fixed and $|V_g \varphi(z)| = \mathcal{O}(|z|^{-n})$ for all $n \geq 0$, whenever $\varphi \in \mathcal{S}(\Rd)$. Therefore, we find that
		\begin{equation}
			|\langle f, \varphi \rangle| \leq C \sup_{z \in \R^{2d}} \left(|V_g \varphi (z)|(1+|z|)^{N+2d+1}\right) \iint_{\R^{2d}} (1+|z|)^{-2d-1} \, dz.
		\end{equation}
		By Corollary \ref{cor_seminorms_S} we have that $f$ defines a continuous linear functional on $\mathcal{S}(\Rd)$, i.e., $f \in \mathcal{S}'(\Rd)$.
		
		\item By (a), the integral in \eqref{eq_inv_STFT_S} defines a tempered distribution $\widetilde{f}$ by
		\begin{equation}
			\langle \widetilde{f}, \varphi \rangle = \frac{1}{\overline{\langle g, \widetilde{g} \rangle}} \iint_{\R^{2d}} V_g f(x,\omega) \langle M_\omega T_x \widetilde{g}, \varphi \rangle \, d(x,\omega).
		\end{equation}
		However, since
		\begin{equation}
			\varphi = \frac{1}{\langle g, \widetilde{g} \rangle} \iint_{\R^{2d}} V_{\widetilde{g}} \varphi(x,\omega) M_\omega T_x g \, d(x,\omega)
		\end{equation}
		holds in $\mathcal{S}(\Rd)$, we also have
		\begin{equation}
			\langle f, \varphi \rangle = \frac{1}{\overline{\langle g, \widetilde{g} \rangle}} 
			\iint_{\R^{2d}} \langle f, M_\omega T_x g \rangle \overline{V_{\widetilde{g}} \varphi(x,\omega)} \, d(x,\omega) = \langle \widetilde{f}, \varphi \rangle,
		\end{equation}
		and so $\widetilde{f} = f$. This proves the inversion formula on $\mathcal{S}'(\Rd)$.
	\end{enumerate}
\end{proof}

\subsubsection{Properties of Modulation Spaces}
After the excursion into time-frequency analysis on the Schwartz space and for tempered distributions, we return to our original object of study, the modulation spaces. We re-call that for a tempered distribution $f$ belonging to a modulation space $M^{p,q}(\Rd)$ is equivalent to the fact that the $L^{p,q}(\R^{2d})$ mixed-norm of the STFT with a fixed Schwartz window $g$ is finite, i.e.,
\begin{equation}
	\norm{f}_{M^{p,q}(\Rd)} = \norm{V_g f}_{L^{p,q}(\R^{2d})} < \infty.
\end{equation}
We will start working with the standard Gaussian $g_0(t) = 2^{d/4} e^{-\pi t^2}$ as our window of choice. The first result that we collect on modulation spaces is the following.
\begin{proposition}
	The modulation space $M^2(\Rd)$ is the Hilbert space $\Lt$, i.e.,
	\begin{equation}
		M^2(\Rd) = \Lt .
	\end{equation}
\end{proposition}
\begin{proof}
	By assumption, the norm
	\begin{equation}
		\norm{f}_{M^2}^2 = \iint_{\R^{2d}} |V_{g_0} f(x,\omega)|^2 \, d(x,\omega)
	\end{equation}
	is finite. It follows that $V_{g_0} f(x,\omega) = \F (f T_x \overline{g_0})(\omega) \in L^2(\Rd, d\omega)$ for almost all $x \in \Rd$. By Plancherel's theorem we have
	\begin{equation}
		\int_{\Rd} |\F(f T_x \overline{g_0})(\omega)|^2 \, d\omega = \int_{\Rd} |f(t)|^2 |g_0(t-x)|^2 \, dt.
	\end{equation}
	Consequently,
	\begin{align}
		\norm{f}_{M^2}^2 & = \int_{\Rd} \left( \int_{\Rd} |f(t)|^2 |g_0(t-x)|^2 \, dt \right) \, dx\\
		& = \iint_{\R^{2d}} |f(t)|^2 |g_0(s)|^2 \, d(t,s)\\
		& = \left( \int_{\Rd} |f(t)|^2 \, dt \right) \left( \int_{\Rd} |g_0(s)|^2 \, ds \right)\\
		& = \norm{f}_2^2 \norm{g_0}_2^2 = \norm{f}_2^2.
	\end{align}
\end{proof}

Since we worked in $\Lt$ throughout most of the course anyways, we are now of course interested in modulation spaces different from $M^2(\Rd)$.

We will now (informally) define the adjoint operator of $V_g$. This will then lead to an inversion formula for the STFT for modulation spaces.
\begin{definition}
	Given a non-zero window $g$ and a function $F$ on $\R^{2d}$, let
	\begin{equation}
		V^*_g F = \iint_{\R^{2d}} F(x,\omega) M_\omega T_x g \, d(x,\omega).
	\end{equation}
	The integral is to be interpreted in the weak sense by
	\begin{align}
		\langle V^*_g F, f \rangle
		& = \iint_{\R^{2d}} F(x,\omega) \langle M_\omega T_x g, f \rangle \, d(x, \omega)\\
		& = \iint_{\R^{2d}} F(x,\omega) \overline{V_g f(x,\omega)} \, d(x,\omega)\\
		& = \langle F, V_g f \rangle .
	\end{align}
\end{definition}
The following proposition states under which conditions $V^*_g$ is a well-defined operator.
\begin{proposition}\label{pro_inv_STFT_S'}
	Let $g, \widetilde{g} \in \mathcal{S}(\Rd)$. Then
	\begin{enumerate}[(a)]
		\item $V_g^*$ maps $L^{p,q}$ into $M^{p,q}$ and satisfies
		\begin{equation}
			\norm{V_g^* F}_{M^{p,q}} \leq C \norm{V_{g_0} g}_{L^1} \norm{F}_{L^{p,q}}.
		\end{equation}
		
		\item In particular, if $F = V_{\widetilde{g}} f$, then the inversion formula
		\begin{equation}
			f = \frac{1}{\langle g, \widetilde{g} \rangle} \iint_{\R^{2d}} V_{\widetilde{g}} f(x,\omega) M_\omega T_x g \, d(x,\omega)
		\end{equation}
		holds on $M^{p,q}(\Rd)$. In short, $I_{M^{p,q}} = \langle g, \widetilde{g} \rangle^{-1} V_g^* V_{\widetilde{g}}$.
		
		\item The definition of $M^{p,q}(\Rd)$ is independent of the choice of the window $g \in \mathcal{S}(\Rd)$. Different windows yield equivalent norms.
	\end{enumerate}
\end{proposition}
\begin{proof}
	\begin{enumerate}[(a)]
		\item We first note\footnote{We refer to the textbook of Gröchenig \cite[Lemma 11.1.2]{Gro01} and remark that the proof is the same as for $L^p(\Rd)$ spaces.} that the Hölder inequality holds for mixed-norm spaces in the following form. If $F \in L^{p,q}(\R^{2d})$ and $G \in L^{p',q'}(\R^{2d})$, with $\frac{1}{p} + \frac{1}{p'} = 1$ and $\frac{1}{q} + \frac{1}{q'} = 1$, then $F G \in L^1(\R^{2d})$ and
		\begin{equation}
			\left| \iint_{\R^{2d}} F(z) \overline{G(z)} \, dz \right| \leq \norm{F}_{L^{p,q}} \norm{G}_{L^{p',q'}} .
		\end{equation}
		We will now show that $V_g^* F \in \mathcal{S}'(\Rd)$. Let $\varphi \in \mathcal{S}(\Rd)$, $F \in L^{p,q}(\R^{2d})$, then we estimate
		\begin{align}
			|\langle V_g^* F, \varphi \rangle| & = |\langle F, V_g \varphi \rangle|\\
			& \leq \norm{F}_{L^{p,q}} \norm{V_g \varphi}_{L^{p',q'}}\\
			& \leq \norm{F}_{L^{p,q}} \norm{(1+|z|)^n V_g \varphi}_\infty \norm{(1+|z|)^{-n}}_{L^{p',q'}}.
		\end{align}
		This expression is finite for sufficiently large $n$. Using the equivalence of seminorms (Corollary \ref{cor_seminorms_S}), it follows that the expression
		\begin{equation}
			V^*_g F = \iint_{\R^{2d}} F(x,\omega) M_\omega T_x g \, d(x,\omega)
		\end{equation}
		is well-defined as a tempered distribution. This has as a consequence that $V_g^* F$ has a continuous STFT, which is explicitly given by\footnote{From the second to the third line we use the covariance principle \ref{pro_covar} for the function $$\overline{V_g(M_\eta T_\xi \widetilde{g}) (x,\omega)} = \overline{\langle M_\eta T_\xi \widetilde{g}, M_\omega T_x g \rangle} = \langle M_\omega T_x g, M_\eta T_\xi \widetilde{g} \rangle = V_{\widetilde{g}} (M_\omega T_x g)(\xi,\eta).$$}
		\begin{align}
			V_{\widetilde{g}} (V_g^* F) (\xi, \eta) & = \langle V_g^*F, M_\eta T_\xi \widetilde{g} \rangle\\
			& = \iint_{\R^{2d}}F(x,\omega) \overline{V_g (M_\eta T_\xi \widetilde{g})(x,\omega)} \, d(x,\omega)\\
			& = \iint_{\R^{2d}}F(x,\omega) V_{\widetilde{g}} g(\xi-x,\eta-\omega) e^{-2 \pi i x \cdot (\eta-\omega)} \, d(x,\omega).
		\end{align}
		By taking absolute values in this identity, we obtain the pointwise estimate
		\begin{equation}\label{eq_VgVg*F}
			| V_{\widetilde{g}} (V_g^* F)(\xi,\eta)| \leq (|F| * |V_{\widetilde{g}} g|)(\xi,\eta)
		\end{equation}
		The convolution relation $L^1 * L^p \subset L^p$ extends to mixed-norm spaces, i.e., $L^1 * L^{p,q} \subset L^{p,q}$ and for $F \in L^1(\R^{2d})$ and $G \in L^{p,q}(\R^{2d})$ we have $\norm{F*G}_{L^{p,q}} \leq \norm{F}_{L^1} \norm{G}_{L^{p,q}}$. Applied to the above estimate, this yields
		\begin{equation}\label{eq_equiv_norms}
			\norm{V_{\widetilde{g}} (V_g^* F)}_{L^{p,q}} \leq C \norm{F}_{L^{p,q}} \norm{V_{\widetilde{g}} g}_{L^1}.
		\end{equation}
		Since, $g, \widetilde{g} \in \mathcal{S}(\Rd)$, $V_{\widetilde{g}} g$ decays rapidly and the right-hand side is finite.
		
		To measure the modulation space norm of $V_g^* F$, we take the fixed (canonical) window $g_0$ and obtain
		\begin{equation}
			\norm{V_g^*F}_{M^{p,q}} = \norm{V_{g_0} (V_g^*F)}_{L^{p,q}} \leq C \norm{F}_{L^{p,q}} \norm{V_{g_0} g}_{L^1} < \infty.
		\end{equation}
		
		\item If $V_{\widetilde{g}} f \in L^{p,q}(\R^{2d})$, then
		\begin{equation}
			\widetilde{f} = \frac{1}{\langle g, \widetilde{g} \rangle} V_g^* V_{\widetilde{g}} f \in M^{p,q}
		\end{equation}
		by the above proof. The equality $\widetilde{f} = f$ follows from the inversion formula for $\mathcal{S}'(\Rd)$.
		
		\item The proof of (a) also implies the equivalence of norms. Using \eqref{eq_equiv_norms} with $g = \widetilde{g} \in \mathcal{S}(\Rd)$ (we may also assume $\norm{g}_2 = 1$), we obtain
		\begin{equation}
			\norm{f}_{M^{p,q}} = \norm{V_{g_0} f}_{L^{p,q}} = \norm{V_{g_0} (V_g^* V_g f)}_{L^{p,q}} \leq C \norm{V_{g_0} g}_{L^1} \norm{V_g f}_{L^{p,q}}.
		\end{equation}
		Interchanging the roles of $g$ and $g_0$, we have
		\begin{equation}
			\norm{V_g f}_{L^{p,q}} \leq C \norm{V_g g_0}_{L^1} \norm{V_{g_0} f}_{L^{p,q}} = C_1 \norm{f}_{M^{p,q}}.
		\end{equation}
		Thus, $f \in M^{p,q}$ if and only if $V_g f \in L^{p,q}$ for some and hence all $g \in \mathcal{S}(\Rd)$ and the norms $\norm{V_{g_0}f}_{L^{p,q}}$ and $\norm{V_g f}_{L^{p,q}}$ are equivalent on $M^{p,q}(\Rd)$.
	\end{enumerate}
\end{proof}

\begin{proposition}
	If $1 \leq p,q < \infty$, then $\mathcal{S}(\Rd)$ is a dense subspace of $M^{p,q}(\Rd)$.
\end{proposition}
\begin{proof}
	We estimate the $M^{p,q}$-norm of $f$ by
	\begin{equation}
	\norm{f}_{M^{p,q}} = \norm{V_{g_0} f}_{L^{p,q}} \leq \norm{(1+|z|)^n \, V_{g_0} f}_\infty \norm{(1+|z|)^{-n}}_{L^{p,q}}.
	\end{equation}
	If $f \in \mathcal{S}(\Rd)$, then the expression is finite, hence $\mathcal{S}(\Rd) \subset M^{p,q}(\Rd)$.
	
	Next, we choose an exhausting sequence of compact sets $K_n \subset \R^{2d}$, e.g., $K_n = \{ z \in \R^{2d} \mid |z| \leq n\}$, and a window $g \in \mathcal{S}(\Rd)$, $\norm{g}_2 = 1$. We set $F_n = V_g f \, \indicator_{K_n}$ and
	\begin{equation}
		f_n = V^*_g F_n = \iint_{\R^{2d}} F_n(x, \omega) M_\omega T_x g \, d(x,\omega) = \iint_{K_n} V_g f(x, \omega) M_\omega T_x g \, d(x,\omega).
	\end{equation}
	Since $F_n$ decays rapidly, we see that $f_n \in \mathcal{S}(\Rd)$ by Proposition \ref{pro_aux_Schwartz}. Furthermore, by Proposition \ref{pro_inv_STFT_S'} we obtain that
	\begin{equation}
		\norm{f - f_n}_{M^{p,q}} = \norm{V_g^*(V_gf - F_n)}_{M^{p,q}} \leq C \norm{ V_g f - F_n}_{L^{p,q}}.
	\end{equation}
	If $p,q < \infty$, then $\norm{f - f_n}_{M^{p,q}} \to 0$, thus $\mathcal{S}(\Rd)$ is dense in $M^{p,q}(\Rd)$.
\end{proof}

\begin{theorem}
	\begin{enumerate}[(a)]
		\item $M^{p,q}(\Rd)$ is a Banach space for $1 \leq p,q \leq \infty$.
		\item $M^{p,q}(\Rd)$ is invariant under time-frequency shifts and $\norm{M_\omega T_x f}_{M^{p,q}} \leq C \norm{f}_{M^{p,q}}$.
		\item If $p = q$, then $M^{p,q}(\Rd)$ is invariant under the Fourier transform.
	\end{enumerate}
\end{theorem}
\begin{proof}
	\begin{enumerate}[(a)]
		\item Let $V = \{ F \in L^{p,q}(\R^{2d}) \mid F = V_{g_0} f \}$. Then $V$ is a subspace of $L^{p,q}(\R^{2d})$ which is isometrically isomorphic to $M^{p,q}(\Rd)$ by definition. Therefore, $M^{p,q}(\Rd)$ inherits its linear structure from $L^{p,q}(\R^{2d})$. We only need to show the completeness of $M^{p,q}(\Rd)$ or, equivalently, that $V$ is a closed subspace of $L^{p,q}(\R^{2d})$.
		
		Let $(f_n)_{n \in \N}$ be a Cauchy sequence in $M^{p,q}(\Rd)$. Then $(V_{g_0} f_n)_{n \in \N}$ is a Cauchy sequence in $L^{p,q}(\R^{2d})$ and so there exists a function $F \in L^{p,q}(\R^{2d})$ such that
		\begin{equation}
			\lim_{n \to \infty} \norm{V_{g_0} f_n - F}_{L^{p,q}} = 0.
		\end{equation}
		To see that $F$ is of the form $V_{g_0} f$, define
		\begin{equation}
			f = \frac{1}{\norm{g_0}_2^2} V_{g_0}^* F = \frac{1}{\norm{g_0}_2^2} \iint_{\R^{2d}} F(x,\omega) M_\omega T_x g_0 \, d(x,\omega).
		\end{equation}
		By Proposition \ref{pro_inv_STFT_S'} we have $f \in M^{p,q}(\Rd)$ and $f_n = \norm{g_0}_2^{-2} V_{g_0}^* V_{g_0} f_n$. Thus, we have the estimate
		\begin{equation}
			\norm{f - f_n}_{M^{p,q}} = \norm{g_0}_2^{-2} \norm{ V_{g_0}^*(F - V_{g_0} f)}_{M^{p,q}} \leq C \norm{F - V_{g_0} f_n}_{L^{p,q}}.
		\end{equation}
		Therefore, the limit of the Cauchy sequence $f_n$ is $f$ and thus $M^{p,q}(\Rd)$ is complete.
		
		\item The invariance properties follow from the translation invariance of $L^{p,q}(\R^{2d})$. In particular, $|V_g (M_\omega T_x f)| = |T_{(x,\omega)} V_g f|$. Thus, we have
		\begin{equation}
			\norm{M_\omega T_x f}_{M^{p,q}} = \norm{T_{(x,\omega)} V_{g_0} f}_{L^{p,q}} \leq C \norm{V_{g_0} f}_{L^{p,q}} = C \norm{f}_{M^{p,q}}.
		\end{equation}
		
		\item For the invariance under the Fourier transform we use the fundamental identity \eqref{eq_fitf} and the norm equivalence in Proposition \ref{pro_inv_STFT_S'}. We obtain that
		\begin{align}
			\norm{\widehat{f}}_{M^{p}}^p & = \norm{V_{g_0} \widehat{f}}_p^p \leq C \norm{V_{\widehat{g_0}} \widehat{f}}_p^p\\
			& = C \iint_{\R^{2d}} |V_{\widehat{g_0}} \widehat{f}(x,\omega)|^p \, d(x,\omega)\\
			& = C \iint_{\R^{2d}} |V_{g_0} f(-\omega, x)|^p \, d(x,\omega)\\
			& = C \iint_{\R^{2d}} |V_{g_0} f(x,\omega)|^p \, d(x,\omega) \leq C' \norm{f}_{M^p}^p .
		\end{align}
	\end{enumerate}
\end{proof}

We remark that, for $1 \leq  p,q < \infty$, the dual space of $L^{p,q}$ is $L^{p',q'}$, with $\frac{1}{p} + \frac{1}{p'} = 1$ and $\frac{1}{q} + \frac{1}{q'} = 1$. The same holds for the modulation spaces $M^{p,q}$ and the duality may be expressed by
\begin{equation}
	\langle f ,h \rangle = \iint_{\R^{2d}} V_{g_0} f(z) \overline{V_{g_0} h(z)} \, dz,
\end{equation}
$f \in M^{p,q}$, $h \in M^{p',q'}$.

\subsection{Feichtinger's Algebra}
In this section we collect some facts about Feichtinger's algebra which is often denoted by $S_0(\Rd)$. Its definition is often given as follows.
\begin{definition}
	Let $g \in \mathcal{S}(\Rd)$ be fixed. Then $S_0(\Rd)$ is the set of all tempered distributions such that
	\begin{equation}
		\norm{f}_{S_0} = \norm{V_g f}_{L^{1,1}} = \iint_{\R^{2d}} |V_gf(x,\omega)| \, d(x,\omega) < \infty.
	\end{equation}
\end{definition}
Thus $S_0(\Rd) = M^1(\Rd)$. We note that there are several equivalent definitions of $S_0(\Rd)$ (see \cite{Jakobsen_S0_2018}). We will state an equivalent definition after the following result.
\begin{lemma}
	Assume $f,g \in \Lt$. If $V_g f \in L^1(\R^{2d})$, then, both, $f \in S_0(\Rd)$ and $g \in S_0(\Rd)$.
\end{lemma}
\begin{proof}
	The assumptions imply that $V_g f$ is a well-defined continuous function on $\R^{2d}$. We will fix a window $\widetilde{g} \in \mathcal{S}(\Rd)$ such that $\langle f, \widetilde{g} \rangle \neq 0$ and $\langle g, \widetilde{g} \rangle \neq 0$. Using the inversion formula for $\mathcal{S}'(\Rd)$
	\begin{equation}
		f = \frac{1}{\langle \widetilde{g}, g \rangle} \iint_{\R^{2d}} V_g f(x,\omega) M_\omega T_x \widetilde{g} \, d(x,\omega)
	\end{equation}
	we obtain
	\begin{equation}
		V_{\widetilde{g}} f = \frac{1}{\langle \widetilde{g}, g \rangle} V_{\widetilde{g}} V^*_{\widetilde{g}} (V_g f).
	\end{equation}
	By equation \eqref{eq_VgVg*F} we have
	\begin{equation}
		|V_{\widetilde{g}} V^*_{\widetilde{g}} (V_g f)(x,\omega)| \leq (|V_g f| * |V_{\widetilde{g}} \widetilde{g}|)(x,\omega).
	\end{equation}
	Hence,
	\begin{equation}
		|V_{\widetilde{g}} f (x,\omega)| \leq \frac{1}{|\langle \widetilde{g}, g \rangle|} (|V_g f| * |V_{\widetilde{g}} \widetilde{g}|)(x,\omega).
	\end{equation}
	By using Young's convolution inequality, we obtain
	\begin{equation}
		\norm{f}_{M^1(\Rd)} = \norm{V_{\widetilde{g}}f}_{L^1(\R^{2d})} \leq \frac{1}{| \langle \widetilde{g}, g \rangle|} \norm{V_g f}_{L^1(\R^{2d})} \norm{V_{\widetilde{g}} \widetilde{g}}_{L^1(\R^{2d})} < \infty.
	\end{equation}
	Thus $f \in M^1(\Rd)$. The result for $g$ follows by interchanging the roles of $f$ and $g$.
\end{proof}
This leads to the following definition of $S_0(\Rd)$, which is also quite natural from the time-frequency point of view.
\begin{definition}
	Feichtinger's algebra $S_0(\Rd)$ is the set of all $f \in \Lt$ such that
	\begin{equation}
		\iint_{\R^{2d}} |V_f f(x,\omega)| \, d(x,\omega) < \infty .
	\end{equation}
\end{definition}

Also, $S_0(\Rd)$ is invariant under the action of the metaplectic group.
\begin{proposition}
	If $S \in Sp(\R,2d)$, then $\mu(S) = \widehat{S} \in Mp(\R,2d)$ is an isomorphism from $S_0(\Rd)$ to $S_0(\Rd)$.
\end{proposition}

Feichtinger's algebra also possesses a remarkable minimality property. In fact, the search for the smallest Banach space that is invariant under time-frequency shifts led H.~Feichtinger to the discovery of $M^1 = S_0$.

In Banach space theory there is a standard procedure to construct minimal spaces which are to contain certain ``atoms". For time-frequency shifts this procedure works as follows. Given a non-zero function $g \in M^1(\Rd)$, define $\mathcal{M}(\Rd)$ to be the vector space of all (non-uniform) Gabor expansions
\begin{equation}\label{eq_Gabor_expansion_M1}
	f = \sum_{\gamma \in \Gamma} c_\gamma \pi(\gamma) g,
\end{equation}
where $\{ \pi(\gamma) = M_{\gamma_2} T_{\gamma_1} \mid \gamma \in \Gamma \subset \R^{2d}\}$ is an arbitrary countable set and
\begin{equation}
	\sum_{\gamma \in \Gamma} |c_\gamma| < \infty.
\end{equation}
The norm on $\mathcal{M}(\Rd)$ is
\begin{equation}
	\norm{f}_{\mathcal{M}} = \inf \sum_{\gamma \in \Gamma} |c_\gamma|,
\end{equation}
where the infimum is taken over all representations of \eqref{eq_Gabor_expansion_M1}. Then, $\mathcal{M}(\Rd)$ is indeed a Banach space. Moreover, it turns out that $M^1(\Rd) = S_0(\Rd) = \mathcal{M}(\Rd)$. Consequently, every $f \in M^1(\Rd)$ has an expansion of type \eqref{eq_Gabor_expansion_M1} and $\norm{f}_{\mathcal{M}}$ is an equivalent norm on Feichtinger's algebra.

Lastly, we state the minimality property of $M^1$. We start with the following result.
\begin{theorem}
	Let $B$ be a Banach space of tempered distributions with the following properties:
	\begin{enumerate}[(i)]
		\item $B$ is invariant und time-frequency shifts, and
		\begin{equation}
			\norm{M_\omega T_x f}_B \leq \norm{f}_B,
		\end{equation}
		for all $f \in B$.
		\item $M^1 \cap B \neq \{0\}$.
	\end{enumerate}
	Then $M^1$ is embedded in $B$.
\end{theorem}
We also have the following chain of inclusions;
\begin{equation}
	M^1 \subset M^{p,q} \subset M^\infty, \quad 1 \leq p,q \leq \infty.
\end{equation}

The minimal space $M^1$ can be seen as a space of test functions with $M^\infty$ as the associated space of distributions. Indeed, weighted versions of the modulation space $M^1_{v_s}$, $v_s = (1+|z|)^s$ with $\norm{f}_{M^1_{v_s}} = \norm{(1+|z|)^s f}_{M^1}$ are the building blocks of $\mathcal{S}(\Rd)$ in the sense that $\mathcal{S} = \bigcap_{s \geq 0} M^1_{v_s}$. Furthermore, the tempered distributions are obtained from $\mathcal{S}' = \bigcup_{s \geq 0} M^\infty_{1/v_s}$. According to the above statements, any modulation space $M^{p,q}$ may be analyzed within the pair $(M^1, M^\infty)$ instead of $(\mathcal{S}, \mathcal{S}')$. Also, in contrast to $\mathcal{S}(\Rd)$ and $\mathcal{S}'(\Rd)$, which are Fr\'e{}chet spaces, the modulation spaces $M^1$ and $M^\infty$ are Banach spaces. They have a much simpler mathematical structure and are easier to use.

\subsection{Wiener Amalgam Spaces}
The Wiener space comes up in the treatment of periodic functions and the Poisson summation formula. We start with the definition of the Wiener space, which is named after N.~Wiener. For this, we introduce the notation $\mathcal{Q}_\alpha = [0,\alpha]^d$, which is the cube of with side length $\alpha$. If $\alpha = 1$, we simply write $\mathcal{Q}$
\begin{definition}
	A function $g \in L^\infty(\Rd)$ belongs to the Wiener space $W = W(\Rd)$ if
	\begin{equation}
		\norm{g}_W = \sum_{k \in \Z^d} \esssup_{x \in \mathcal{Q}} |g(x+k)| < \infty.
	\end{equation}
	The subspace of continuous functions is denoted by $W_0(\Rd)$.
\end{definition}

Observe that the norm can also be written as
\begin{equation}
	\norm{g}_W = \sum_{k \in \Z^d} \norm{g \, T_k \indicator_\mathcal{Q}}_\infty .
\end{equation}
Functions in $W$ are locally bounded and globally in $\ell^1$. Hence, the norm mixes (``amalgamates") a local property of a function, namely boundedness, with a global property. This space was first introduced by N.~Wiener to study Tauberian theorems\footnote{Tauberian theorems play an important role in analytic number theory and give conditions under which a divergent series can be summed to yield a meaningful result. Similar methods have earlier been considered by N.~H.~Abel. N.~Wiener also used Fourier analysis methods, in particular the insight that asymptotic (global; decay) properties are take to local (smoothness) properties.} \cite{Wie32}. We note that the Wiener space contains all bounded functions with compact support and is therefore dense in any $L^p(\Rd)$-space for $1 \leq p < \infty$.

More generally, Wiener amalgam spaces are defined as follows.
\begin{definition}
	We say a measurable function $f$ belongs to the Wiener amalgam space $W(L^p, \ell^q)(\Rd)$ if
	\begin{equation}
		\norm{f}_{W(L^p,\ell^q)} = \left(\sum_{k \in \Z^d} \norm{f \, T_k \indicator_\mathcal{Q}}_p^q\right)^{1/q} < \infty .
	\end{equation}
\end{definition}
Since $\mathcal{Q}$ can be covered by a finite number of cubes $\mathcal{Q}_\alpha$, we can take any size of the cube (and replace the translations by $T_{\alpha k}$) to obtain an equivalent norm. Hence, the Wiener amalgam spaces are independent of the choice of $\alpha$. Also, the integer lattice  $\Z^d$ can be replaced by an arbitrary lattice with a similar argument.

Furthermore, we note that the characteristic function $\indicator_\mathcal{Q}$ can be replaced by a bounded function $\phi$ with compact support and/or the $\ell^q$-norm may as well be replaced by an $L^q$-norm;
\begin{equation}
	\norm{f}_{W(L^p,L^q)} = \left( \int_{\Rd} \left( \int_{\Rd} | (f \, T_x \overline{\phi})(\omega)|^p \, d\omega \right)^{q/p} \, dx \right)^{1/q}.
\end{equation}
Again, we get an equivalent norm, i.e.,
\begin{equation}
	\norm{f}_{W(L^p,L^q)} \asymp \norm{f}_{W(L^p,\ell^q)}.
\end{equation}
So, we may simply write $W^{p,q}$ without ambiguity, as long as we are only interested in general properties of Wiener amalgam spaces and not specific norms of specific functions. For more details we refer to \cite{Hei03}.

We note the following connection between Wiener amalgam spaces and modulation spaces. Let $\phi \in C^\infty_c(\Rd) \subset \mathcal{S}(\Rd)$, then
\begin{align}
	\norm{f}_{W(\F L^p,L^q)} & = \left( \int_{\Rd} \left( \int_{\Rd} | \F (f \, T_x \overline{\phi})(\omega)|^p \, d\omega \right)^{q/p} \, dx \right)^{1/q}\\
	& = \left( \int_{\Rd} \left( \int_{\Rd} |V_\phi f(x, \omega)|^p \, d\omega \right)^{q/p} \, dx \right)^{1/q}\\
	& = \left( \int_{\Rd} \left( \int_{\Rd} |V_{\widehat{\phi}} \widehat{f}(\omega,-x)|^p \, d\omega \right)^{q/p} \, dx \right)^{1/q}\\
	& = \norm{\widehat{f}}_{M^{p,q}}.
\end{align}
In shorter notation, we have $\F M^{p,q} = W(\F L^p, L^q)$.

In view of the Poisson summation formula, which was our motivation for introducing the Wiener space, we note that $W(\Rd) = W^{\infty,1}(\Rd)$. The space $W \cap \F W$ is in a sense the largest natural space on which the Poisson summation formula holds point-wise. If both, $f$ and $\widehat{f}$ are in $W(\Rd) \subset L^1(\Rd)$, then they are also continuous and both series
\begin{equation}
	\sum_{k \in \Z^d} f(x+k)
	\quad \text{ and } \quad
	\sum_{k \in \Z^d} \widehat{f}(k) e^{2 \pi i k \cdot x}
\end{equation}
converge absolutely by the definition of $W(\Rd)$. Moreover, the proof of Proposition \ref{pro_PSF_Zd} shows that the sums are indeed equal for all $x \in \Rd$, which means that:

\medskip

\textit{The Poisson summation formula holds pointwise with absolute convergence in both series whenever $f$ and $\widehat{f} \in W(\Rd)$.}

\medskip

The above statement holds if we replace $W(\Rd)$ by Feichtinger's algebra $S_0(\Rd) = M^1(\Rd)$, as is shown by the following embedding result for $M^1(\Rd)$.
\begin{proposition}
	If $f \in M^1(\Rd)$, then $f \in W(\Rd)$ and $\widehat{f} \in W(\Rd)$.
\end{proposition}
\begin{proof}
	We choose a window $g \in C^\infty_c(\Rd) \subset \mathcal{S}(\Rd)$ such that $0 \leq g(x) \leq 1$ (particularly, $g$ is real-valued) for all $x \in \Rd$ and $g(x) = 1$ for $x \in [-1,1]^d$. Then $\indicator_{\mathcal{Q}}(t) \leq T_x g(t)$ for all $x \in \mathcal{Q} = [0,1]^d$ and all $t \in \Rd$, and
	\begin{equation}
		\norm{f \, T_k \indicator_{\mathcal{Q}}}_\infty \leq \norm{f \, T_{k+x} \overline{g}}_\infty.
	\end{equation}
	The converse of the Lemma of Riemann-Lebesgue in the form $\norm{h}_\infty \leq \norm{\widehat{h}}_1$ provides the connection to the STFT of $f$ with window $g$, as for $x \in Q$,
	\begin{equation}
		\norm{f \, T_{k+x} \overline{g}}_\infty \leq \norm{\F(f \, T_{k+x}\overline{g})}_1 = \int_{\Rd} |V_g f(k+x, \omega)| \, d\omega.
	\end{equation}
	Summation over $k$ yields
	\begin{align}
		\sum_{k \in \Z^d} \norm{f \, T_k \indicator_Q}_\infty
		\leq \sum_{k \in \Z^d} \int_{\mathcal{Q}} \left( \int_{\Rd} |V_g f(k+x,\omega)| \, d\omega \right) \, dx
		= \iint_{\R^{2d}} |V_g f(x,\omega)| \, d(x,\omega).
	\end{align}
	Thus, we have shown that
	\begin{equation}
		\norm{f}_{W(\Rd)} \leq \norm{V_g f}_{L^1(\R^{2d})} = \norm{f}_{M^1}
	\end{equation}
	For the result involving $\widehat{f}$  we argue similarly, but define $g$ such that $\widehat{g}$ has the necessary properties. Then, the fundamental identity of time-frequency analysis \eqref{eq_fitf} yields the result.
\end{proof}

\subsection{The WKNS Sampling Theorem Re-Visited}
The following result is given in \cite[Chap.~10]{FeiJak20} and is a generalization of the WKNS sampling theorem for band-limited functions. We state it only for the case of band-limited functions with $\supp(\widehat{f}) \subset [-\frac{1}{2}, \frac{1}{2}]$, but the generalizations for other band-limits follow as in Section \ref{sec_WNKS}. Before we can state the result, we need to define a nice class of band-limited functions.
\begin{definition}
	For an interval $I \subsetneq [-\frac{1}{2}, \frac{1}{2}]$ we set
	\begin{equation}
		B_I^1(\R) = \{ f \in W(\R) \mid \supp(\widehat{f)} \subset I \}
	\end{equation}
\end{definition}
As remarked in \cite[Chap.~10]{FeiJak20}, it can be shown that
\begin{equation}
	B_I^1(\R) = \{f \in S_0(\R) \mid \supp(\widehat{f}) \subset I \} = \{f \in L^1(\R) \mid \supp(\widehat{f}) \subset I \}.
\end{equation}
The generalization of the classical sampling theorem, now with nice localization of the building blocks, reads as follows.
\begin{theorem}
	Let $f \in \Lt[]$ such that
	\begin{equation}
		f(t) = \int_\R \widehat{f}(\omega) e^{2 \pi i \omega t} \, dt,
	\end{equation}
	with $\supp(\widehat{f}) = I \subsetneq [-\frac{1}{2}, \frac{1}{2}]$. Let $g \in S_0(\R)$, $\norm{g}_2 = 1$, such that $\widehat{g}(\omega) = 1$ for $\omega \in I$ and $\supp(\widehat{g}) \subset [-\frac{1}{2}, \frac{1}{2}]$. Then, we have
	\begin{equation}
		f(t) = \sum_{k \in \Z} f(k) g(t-k), \quad \forall t \in \R
		\quad \text{ and } \quad
		\forall f \in B_I^1(\R),
	\end{equation}
	with absolute convergence in $(S_0(\R), \norm{.}_{S_0})$, $(W(\R), \norm{.}_W)$, and $(C_0(\R), \norm{.}_\infty)$.
\end{theorem}
For more information on the Shannon sampling theorem with ``nice" windows and a time-frequency point-of-view, the interested reader is also referred to \cite{StrTan05}.

\section{The Zak Transform and Density Principles}
\subsection{The Zak Transform}
We are now going to study the Zak transform, which is often used in time-frequency analysis to study properties of the Gabor frame operator. It is a version of the Poisson summation formula and a popular tool in engineering for the analysis of Gabor frames.

We note that Zak transform is often defined with a (positive) parameter and that there are several versions of the Zak transform, which differ by normalization. We name the expository article by Janssen \cite{Jan88}or the textbook of Gröchenig \cite{Gro01} as references. Here, we will use the Zak transform without parameter (actually, with parameter $\alpha = 1$).

\begin{definition}
	The Zak transform of a function $f$ on $\Rd$ is the function on $\R^{2d}$ defined by
	\begin{equation}
		Z f(x, \omega) = \sum_{k \in \Z^d} f(x+k) e^{-2 \pi i k \cdot \omega} .
	\end{equation}
\end{definition}
This definition is compatible with most versions of the Zak transform with parameter $\alpha$ (which normally refers to a (non-unitary) dilation of the first or both arguments) for the case $\alpha = 1$.

The Zak transform was first used by Gel'fand for a problem in differential equations \cite{Gel50}. The transform has been rediscovered several times and is now named after J.~Zak, who used it for problems in solid state physics, e.g., in \cite{Zak67}.

We will first show that the Zak transform inherits continuity and decay properties from $f$.
\begin{lemma}\label{lem_Zak_prop}
	\begin{enumerate}[(a)]
		\item If $f \in L^1(\Rd)$, then $Z f \in L^1(\mathcal{Q} \times \mathcal{Q})$.
		\item If $f \in W(\Rd)$, then $Z f \in L^\infty (\R^{2d})$.
		\item If $f \in W_0(\Rd)$, then $Z f$ is continuous on $\R^{2d}$.
		\item If $f \in \Lt$, then $Z f$ is defined almost everywhere and $Z f(x,\omega) \in L^2(\mathcal{Q}, d\omega)$ for almost all $x \in \Rd$.
	\end{enumerate}
\end{lemma}
\begin{proof}
	\begin{enumerate}[(a)]
		\item By the periodization trick we have
		\begin{equation}
			\int_\mathcal{Q} |Z f(x,\omega)| \, dx \leq \int_\mathcal{Q} \sum_{k \in \Z^d} |f(x+k)| \, dx = \norm{f}_1,
		\end{equation}
		for all $\omega \in \Rd$, and therefore
		\begin{equation}
			\int_\mathcal{Q} \left( \int_\mathcal{Q} |Zf(x,\omega) \, dx \right) \, d \omega \leq \norm{f}_1.
		\end{equation}
		
		\item We have
		\begin{equation}
			|Zf(x,\omega)| \leq \sum_{k \in \Z^d} |f(x+k)| \leq 2^d \norm{f}_W,
		\end{equation}
		where the factor $2^d$ stems from the fact that any interval of length one contains at most 2 distinct points of minimum distance 1 between each other. Hence, any translated cube $\mathcal{Q}+k$ contains at most $2^d$ points of the form $(x+k)$, $k \in \Z^d$.
		
		\item Given $\varepsilon > 0$, there is $N > 0$ such that
		\begin{equation}
			\sum_{|k| > N} \norm{f \, T_k \indicator_\mathcal{Q}}_\infty < \frac{\varepsilon}{4}.
		\end{equation}
		Then, the main term $\sum_{|k| \leq N} f(x+k) e^{-2 \pi i k \cdot \omega}$ is uniformly continuous on compact sets of $\R^{2d}$, and there exists a $\delta > 0$, such that
		\begin{equation}
			\left| \sum_{|k| \leq N} f(x'+k) e^{-2 \pi i k \cdot \omega'} - \sum_{|k| \leq N} f(x+k) e^{-2 \pi i k \cdot \omega} \right| < \frac{\varepsilon}{2},
		\end{equation}
		for $|x'-x|+|\omega'-\omega| < \delta$. Consequently $|Zf(x',\omega') - Zf(x,\omega)| < \varepsilon$ if $|x'-x|+|\omega'-\omega| < \delta$.
		
		\item The periodization of the $L^2$-norm of $f$ yields
		\begin{equation}
			\int_\mathcal{Q} \sum_{k \in \Z^d} |f(x+k)|^2 \, dx = \int_{\Rd} |f(x)|^2 \, dx < \infty.
		\end{equation}
		This implies that the sequence $\left(f(x+k)\right)_{k \in \Z^d} \in \ell^2(\Z^d)$ for almost all $x \in \Rd$. Consequently, $Zf(x, .)$ is a Fourier series with square-summable coefficients and therefore in $L^2(\mathcal{Q}, d\omega)$. In particular, $Zf$ is defined almost everywhere.
	\end{enumerate}
\end{proof}

We collect a few more properties of the Zak transform. The interpretation of the equalities depends on the function space from which $f$ comes. First, we note its quasi-periodicity, that is, for $l \in \Z^d$
\begin{equation}
	Z f(x, \omega+l) = Z f(x,\omega)
\end{equation}
and
\begin{equation}
	Z f(x+l,\omega) = e^{2 \pi i l \cdot \omega} Z f(x,\omega).
\end{equation}
Thus, $Z f$ is completely determined by its values on the (unit) cube $[0,1]^d \times [0,1]^d = \mathcal{Q} \times \mathcal{Q} \subset \R^{2d}$. It also behaves nicely with time-frequency shifts.
\begin{equation}\label{eq_covariance_Zak}
	Z (M_\eta T_\xi f)(x,\omega) = e^{2 \pi i \eta \cdot x} Z f(x-\xi, \omega - \eta),
\end{equation}
as can be shown by a direct computation. Equation \eqref{eq_covariance_Zak} resembles the covariance principle of the STFT and is one reason why the Zak transform is considered a joint time-frequency representation. For the special choice $\xi = k$, $\eta = l$ with $k,l \in \Z^d$, we obtain
\begin{align}\label{eq_Zak_Tk_Ml}
	Z(M_l T_k f)(x,\omega) & = e^{2 \pi i l \cdot x} Zf(x-k,\omega-l)\\
	& = e^{2 \pi i l \cdot x} e^{-2 \pi i k \cdot \omega} Zf(x,\omega).
\end{align}
Note that $T_k$ and $M_l$ commute since $(k,l) \in \Z^{2d}$, which is self-dual. Thus the operators $T_k$ and $M_l$ are mapped to multiplication operators by the Zak transform.

\begin{example}
	\begin{enumerate}[(i)]
		\item Let $d = 1$ and consider the characteristic function $\indicator_{[0,1]}$. Then\footnote{The floor function is $\lfloor x \rfloor = \max \{k \in \Z \mid k \leq x \}$.}
		\begin{equation}
			Z f(x,\omega) = \sum_{k \in \Z} \indicator_{[0,1]}(x+k) e^{-2 \pi i k \omega} = e^{2 \pi i \lfloor x \rfloor \omega} .
		\end{equation}
		In particular, $Z \indicator_{[0,1]}(x,\omega) = 1$ for almost all $(x,\omega) \in [0,1] \times [0,1]$ and $|Z \indicator_{[0,1]}|^2 \equiv 1$. Similarly, if $d \geq 1$, we have $Z \indicator_\mathcal{Q}(x,\omega) = 1$, for almost all $(x,\omega)  \in \mathcal{Q} \times \mathcal{Q}$.
		
		\item For $\varphi_s (t) = e^{-\pi s t^2}$, $s > 0$, we have
		\begin{equation}
			Z \varphi_s (x,-\omega) = \sum_{k \in \Z} e^{-\pi s (x+k)^2} e^{2 \pi i k \omega}.
		\end{equation}
		This has close connections to the Jacobi theta functions in analytic number theory. For example
		\begin{equation}
			\vartheta_3(\tau; z) = \sum_{k \in \Z} e^{\pi i \tau k^2} e^{2 \pi i k z}, \quad (\tau, z) \in \mathbb{H} \times \C.
		\end{equation}
	\end{enumerate}
	Therefore, the restrictions $\tau = i s$, $s > 0$ and $z = \omega \in \R$ yield
	\begin{equation}
		\vartheta_3(i s; \omega) = Z\varphi_s(0,-\omega).
	\end{equation}
	\begin{flushright}
		$\diamond$
	\end{flushright}
\end{example}

We also have an inversion formula for the Zak transform. Let $f \in L^1(\Rd)$, then $Zf(x,\omega) \in L^1(\mathcal{Q}, d\omega)$ for almost all $x$. Therefore, the following computations are justified.
\begin{align}
	\int_{\mathcal{Q}} Zf(x,\omega) \, d \omega
	& = \int_\mathcal{Q} \left( \sum_{k \in \Z^d} f(x+k) e^{-2 \pi i k \cdot \omega} \right) \, d \omega\\
	& = \sum_{k \in \Z^d} \int_\mathcal{Q} f(x+k) e^{-2 \pi i k \cdot \omega} \, d\omega
	= \sum_{k \in \Z^d} f(x+k) \int_\mathcal{Q} e^{-2 \pi i k \cdot \omega} \, d\omega\\
	& = f(x),
\end{align}
for almost all $x \in \Rd$. Also, if $f \in L^1(\Rd)$, then $Zf(x,\omega) \in L^1(\mathcal{Q}, dx)$ for almost all $\omega \in \Rd$. Thus, a similar computation shows that
\begin{align}
	\int_\mathcal{Q} Z f(x, \omega) e^{-2 \pi i x \cdot \omega} \, dx
	& = \int_\mathcal{Q} \left( \sum_{k \in \Z^d} f(x+k) e^{-2 \pi i (k+x) \cdot \omega} \right) \, d x \\
	& = \sum_{k \in \Z^d} \int_\mathcal{Q} f(x+k) e^{-2 \pi i (k+x) \cdot \omega} \, d x
	= \int_{\Rd} f(x) e^{-2 \pi i x \cdot \omega} \, dx\\
	& = \widehat{f}(\omega).
\end{align}

The next result is the Poisson summation formula in disguise and relates the Zak transform of a function with the Zak transform of its Fourier transform.
\begin{proposition}
	Let $f \in W(\Rd)$ and $\widehat{f} \in W(\Rd)$ (so $f$ and $\widehat{f}$ are actually continuous). Then
	\begin{equation}
		Zf(x,\omega) = e^{2 \pi i x \cdot \omega} Z\widehat{f}(\omega,-x),
	\end{equation}
	for all $(x,\omega) \in \R^{2d}$.
\end{proposition}
\begin{proof}
	We apply the Poisson summation formula to the function $g(t) = M_{-\omega} T_{-x} f(t)$.
	\begin{align}
		Zf(x,\omega) & = \sum_{k \in \Z^d} g(k)
		= \sum_{k \in \Z^d} \widehat{g}(k)
		= \sum_{k \in \Z^d} T_{-\omega} M_x \widehat{f}(k)
		= \sum_{k \in \Z^d} \widehat{f}(k+\omega) e^{2 \pi i x \cdot (k+\omega)}\\
		& = e^{2 \pi i x \cdot \omega} Z \widehat{f}(\omega,-x).
	\end{align}
\end{proof}

The Zak transform can also be seen as a unitary operator from $\Lt$ to $L^2(\mathcal{Q} \times \mathcal{Q})$, which is shown in the following Plancherel-like result.
\begin{theorem}\label{thm_Zak_unitary}
	If $f \in W(\Rd)$, then
	\begin{equation}
		\iint_{\mathcal{Q} \times \mathcal{Q}} |Zf(x,\omega)|^2 \, d(x,\omega) = \norm{f}_2^2.
	\end{equation}
	Consequently, $Z$ extends to a unitary operator from $\Lt$ onto $L^2(\mathcal{Q} \times \mathcal{Q})$.
\end{theorem}
\begin{proof}
	For fixed $x$, the Zak transform is a Fourier series with coefficients $f(x+k)$, $k \in \Z^d$. If $f \in \Lt$, then these coefficients are in $\ell^2(\Z^d)$ for almost all $x$. Hence, Plancherel's theorem for Fourier series implies
	\begin{equation}
		\int_\mathcal{Q} |Zf(x,\omega)|^2 \, d\omega = \sum_{k \in \Z^d} |f(x+k)|^2.
	\end{equation}
	For the $x$-integration, the periodization trick yields
	\begin{equation}
		\int_\mathcal{Q} \left( \int_\mathcal{Q} |Zf(x,\omega)|^2 \, d\omega \right) \, dx = \sum_{k \in \Z^d} \int_\mathcal{Q} |f(x+k)|^2 \, dx = \norm{f}_2^2.
	\end{equation}
	Thus, $Z$ is an isometry on the subspace $W(\Rd)$, which is dense in $L^2(\Rd)$. By a density argument, $Z$ therefore extends to an isometry on $\Lt$.
	
	Next, note that $Z \indicator_\mathcal{Q}(x,\omega) = 1$ for almost all $(x,\omega) \in \mathcal{Q} \times \mathcal{Q}$ and by \eqref{eq_Zak_Tk_Ml}
	\begin{equation}
		Z(M_l T_k \indicator_\mathcal{Q})(x,\omega) = e^{2 \pi i (l \cdot x - k \cdot \omega)},
	\end{equation}
	for $k,l \in \Z^d$ and $(x,\omega) \in \mathcal{Q} \times \mathcal{Q}$. This means that the Zak transform $Z$ maps the orthonormal basis $\mathcal{G}(\indicator_\mathcal{Q}, \Z^{2d})$ of $\Lt$ onto the orthonormal basis $\{ e^{2 \pi i (l \cdot x - k \cdot \omega)} \mid k,l \in \Z^d\}$ of $L^2(\mathcal{Q} \times \mathcal{Q})$. Therefore, $Z$ is surjective and, hence, a unitary operator.
\end{proof}

We are now going to study the Zak transform and its relationship to Gabor frames. We start with a result on the Gabor frame operator.
\begin{theorem}
	Let $g \in \Lt$ and consider the lattice $\Z^{2d}$ and write $S_g$ for the associated Gabor frame operator. Then\footnote{We may use different windows $g$ and $\widetilde{g}$ for the analysis and synthesis process. In this case the proof shows that $Z(S_{g, \widetilde{g}} f) = \overline{Z g} Z\widetilde{g} Zf$.}
	\begin{equation}
		Z( S_g f) = |Zg|^2 Zf.
	\end{equation}
	Thus, the Gabor frame operator is equivalent to the multiplication operator $Z S_g Z^{-1}$ with multiplier $|Zg|^2$ on $L^2(\mathcal{Q} \times \mathcal{Q})$.
\end{theorem}
\begin{proof}
	We have
	\begin{equation}
		Z(M_l T_k g)(x,\omega) = e^{2 \pi i (l \cdot x - k \cdot \omega)} Zg(x,\omega).
	\end{equation}
	Hence,
	\begin{align}
		Z(S_g f)(x,\omega) & = \sum_{k,l \in \Z^d} \langle f, M_l T_k g \rangle \, Z (M_l T_k g)(x,\omega)\\
		& = \sum_{k,l \in \Z^d} \langle Z f, Z(M_l T_k g) \rangle Z(M_l T_k g)(x,\omega)\\
		& = \left(\sum_{k,l \in \Z^d} \iint_{\mathcal{Q} \times \mathcal{Q}} Zf(\xi,\eta) \overline{Z g(\xi,\eta)} e^{-2 \pi i (l \cdot \xi - k \cdot \eta)} \, d(\xi,\eta) \, e^{2 \pi i (l \cdot x - k \cdot \omega)}\right) Zg(x,\omega).
	\end{align}
	Since $\{e_{k,l}(x,\omega) = e^{2 \pi i (l \cdot x - k \cdot \omega)} \mid k,l \in \Z^d \}$ is an orthonormal basis for $L^2(\mathcal{Q} \times \mathcal{Q})$, the expression in the brackets is just the orthonormal expansion of
	\begin{equation}
		Zf \overline{Zg} = \sum_{k,l} \langle Zf \overline{Zg}, e_{k,l} \rangle \, e_{k,l}
	\end{equation}
	in $L^2(\mathcal{Q} \times \mathcal{Q})$. Therefore, the result follows.
\end{proof}
\begin{corollary}\label{cor_Zak_bounds}
	Let $g \in \Lt$ and consider the Gabor system $\G(g, \Z^{2d})$.
	\begin{enumerate}[(a)]
		\item $S_g$ is bounded on $\Lt$ if and only if $|Zg|^2 \in L^\infty(\R^{2d})$.
		\item $\G(g, \Z^{2d})$ is a frame if and only if
		\begin{equation}
			0 < a \leq |Zg(x,\omega)| \leq b < \infty
		\end{equation}
		for almost all $(x,\omega) \in \R^{2d}$. In this case the optimal frame bounds are
		\begin{align}
			A & = \essinf_{(x,\omega) \in \mathcal{Q} \times \mathcal{Q}} |Zg(x,\omega)|^2,\\
			B & = \esssup_{(x,\omega) \in \mathcal{Q} \times \mathcal{Q}} |Zg(x,\omega)|^2.
		\end{align}
		\item $\G(g, \Z^{2d})$ is an orthonormal basis for $\Lt$ if and only if $|Zg(x,\omega)|^2 = 1$ for almost all $(x,\omega) \in \R^{2d}$.
	\end{enumerate}
\end{corollary}
\begin{proof}
	We note that a multiplication operator $f \mapsto m f$ is bounded if and only if $m \in L^\infty$, and is invertible if and only if $m^{-1} \in L^\infty$ (see, e.g., \cite{Con_FA90}).
	Therefore $(a)$ and $(b)$ follow immediately, since the conjugation $Z S_g Z^{-1}$ preserves the spectrum.
	
	As for $(c)$, if $\G(g, \Z^{2d})$ is an orthonormal basis, then $S_g = I_{L^2}$, and consequently $Z S_g Z^{-1}$ amounts to multiplication by $|Zg|^2 = 1$. Conversely, if $|Zg|^2=1$ a.e., then by $(b)$, $\G(g, \Z^{2d})$ is a tight frame with bounds $A=B=1$. By Theorem \ref{thm_Zak_unitary}
	\begin{equation}
		\norm{g}_2^2 = \iint_{\mathcal{Q} \times \mathcal{Q}} |Zg(x,\omega)|^2 \, d(x,\omega) = 1.
	\end{equation}
	Now, since $\norm{M_l T_k g}_2^2 = \norm{g}_2^2 = 1$, the Gabor system $\G(g,\Z^{2d})$ is an orthonormal basis by Lemma \ref{lem_ONB_tight_frame}.
\end{proof}

\subsection{The Balian-Low Theorem}
The Balian-Low Theorem can be seen as a ``no-go result" in the sense that it tells us that an orthonormal basis for $\Lt$ cannot come from a Gabor system with a nice window. Now, there are many versions of Balian-Low type results and they can also be seen as uncertainty principles. We will state two versions of the Balian-Low theorem, one for functions in the Hilbert space $\Lt[]$ ($d=1$), which is the original version, and one for the Wiener space $W(\Rd)$. 
\begin{theorem}[Balian-Low, Hilbert space version]
	If $\G(g, \Z^2)$ is an orthonormal basis for $\Lt[]$, then either $x \, g(x) \notin \Lt[]$ or $g'(x) \notin \Lt[]$.\footnote{From the time-frequency analysis point of view, it would be more appropriate to say \textit{either $x \, g(x) \notin \Lt[]$ or $\omega \, \widehat{g}(\omega) \notin \Lt[]$}.}
\end{theorem}
\begin{proof}
	The proof relies on the commutation relations for the position and momentum operators, which are defined as
	\begin{equation}
		X g(x) = x g(x)
		\quad \text{ and } \quad
		P g(x) = \frac{1}{2 \pi i} \, g'(x).
	\end{equation}
	Recall that $X$ and $P$ are self-adjoint and that
	\begin{equation}
		(PX - XP)g = \frac{1}{2 \pi i}	g,
	\end{equation}
	for $g \in \text{dom}(XP) \cap \text{dom}(PX)$. Furthermore, we have $\F(Pg) = X \F g$.
	
	We prove the theorem by contradiction. Suppose $\G(g,\Z \times \Z)$ is an orthonormal basis of $\Lt[]$, in particular $g \neq 0$, and suppose that $X g \in \Lt[]$ and $P g \in \Lt[]$. Then, using the orthonormal expansion of $Xg$, we have
	\begin{equation}
		\langle X g, P g \rangle = \sum_{k,l \in \Z} \langle X g, M_l T_k g \rangle \langle M_l T_k g, P g \rangle.
	\end{equation}
	We will now re-write the inner products in the series expansion. Since
	\begin{equation}
		X M_l T_k g(x) = (k+x-k)e^{2 \pi i l \cdot x} g(x-k) = k M_l T_k g(x) + M_l T_k X g(x),
	\end{equation}
	we have
	\begin{align}
		\langle X g, M_l T_k g \rangle & = \langle g, X M_l T_k g \rangle\\
		& = k \langle g, M_l T_k g \rangle + \langle g, M_l T_k X g \rangle\\
		& = 0 + \langle T_{-k} M_{-l} g, X g \rangle,
	\end{align}
	where we used the fact that $\langle g, M_l T_k g \rangle = 0$ because we assume $\G(g, \Z \times \Z)$ to be an orthonormal basis. Similarly,
	\begin{align}
		\langle P M_l T_k g, g \rangle & = \langle X T_l M_{-k} \widehat{g}, \widehat{g} \rangle\\
		& = l \langle T_l M_{-k} \widehat{g}, \widehat{g} \rangle + \langle T_l M_{-k} X \widehat{g}, \widehat{g} \rangle\\
		& =\langle M_l T_k P g, g \rangle.
	\end{align}
	Combining the above results, we obtain
	\begin{equation}
		\langle X g, P g \rangle = \sum_{k,l \in \Z} \langle P g, T_{-k} M_{-l} g \rangle \langle T_{-k} M_{-l} g, X g \rangle = \langle P g, X g \rangle.
	\end{equation}
	If we knew that $g \in \text{dom}(P X) \cap \text{dom}(X P)$, then we could rewrite this identity as
	\begin{equation}
		0 = \langle (P X - X P) g, g \rangle = \frac{1}{2 \pi i} \norm{g}_2^2.
	\end{equation}
	We choose a sequence $(g_n)_{n \in \N}$ in $C^\infty_c(\R)$, such that $\norm{g_n - g}_2 \to 0$, $\norm{X g_n - X g}_2 \to 0$ and $\norm{P g_n - P g}_2 \to 0$ (such a sequence exists). Then
	\begin{equation}
		\lim_{n \to \infty} \left( \langle X g_n, P g_n \rangle - \langle P g_n, X g_n \rangle \right) = \langle X g, P g \rangle - \langle P g, X g \rangle = 0.
	\end{equation}
	On the other hand, since $g_n \in \mathcal{S}(\R) \subset \text{dom}(P X) \cap \text{dom}(X P)$, this limit is also
	\begin{equation}
		\lim_{n \to \infty} \langle (P X - X P) g_n, g_n \rangle = \frac{1}{2 \pi i} \lim_{n \to \infty} \norm{g_n}_2^2 = \frac{1}{2 \pi i} \norm{g}_2^2.
	\end{equation}
	Thus, $g = 0$, contradicting the assumption that $\G(g, \Z \times \Z)$ is an orthonormal basis.
\end{proof}

\begin{remark}
	While the classical uncertainty principle provides a lower bound on the deviations in time and frequency, i.e.,
	\begin{equation}
		\norm{X g}_2 \norm{P g}_2 = \norm{X g}_2 \norm{X \widehat{g}}_2 \geq \frac{1}{4\pi} \norm{g}_2^2,
	\end{equation}
	the Balian-Low theorem (BLT) implies that a window of an orthonormal Gabor basis possesses the maximal uncertainty
	\begin{equation}
		\norm{X g}_2 \norm{P g}_2 = \infty.
	\end{equation}
	For the Hilbert space $\Lt$, this result holds for the conjugate variables $(x_k,\omega_k)$, $k = 1, \ldots , d$.
	
	There are also more general versions for (symplectic) lattices and more general index sets. For more details on this we refer, e.g., to \cite{GroHanHeiKut02}.
	\begin{flushright}
		$\diamond$
	\end{flushright}
\end{remark}

We will now state the Wiener amalgam version of the BLT.
\begin{theorem}[BLT, Wiener amalgam version]
	If $\G(g, \Z^{2d})$ is a frame for $\Lt$, then both $g \notin W_0(\Rd)$ and $\widehat{g} \notin W_0(\Rd)$
\end{theorem}

To prove the amalgam version of the BLT we show a quite special property of the Zak transform.
\begin{lemma}
	If $Zf$ is continuous on $\R^{2d}$, then $Zf$ has a zero in $\mathcal{Q} \times \mathcal{Q}$.
\end{lemma}
\begin{proof}
	Assume that $Zf(x,\omega) \neq 0$ for all $(x, \omega) \in \R^{2d}$. Since $Zf$ is continuous and $\R^{2d}$ is simply connected, there exists a continuous function $\varphi(x,\omega)$, such that
	\begin{equation}
		Zf(x,\omega) = |Zf(x,\omega)| e^{2 \pi i \varphi(x,\omega)}.
	\end{equation}
	We note that in this particular case a continuous logarithm exists. The quasi-periodicity
	\begin{equation}
		Zf(x+k,\omega+l) = e^{2 \pi i k \cdot \omega} Zf(x,\omega)
	\end{equation}
	implies the existence of a function $\kappa: \Z^{2d} \to \Z$, such that
	\begin{equation}
		\varphi(x+k,\omega+l) = \varphi(x,\omega) + k \cdot \omega + \kappa(k,l).
	\end{equation}
	We compute $\varphi(k,l) = \varphi(0,0) + \kappa(k,l)$ in two ways:
	\begin{align}
		\varphi(k,l) & = \varphi(0,l) + k \cdot l + \kappa(k,0)\\
		& = \left( \varphi(0,0) + \kappa(0,l) \right) + k \cdot l + \kappa(k,0)
	\end{align}
	and
	\begin{equation}
		\varphi(k,l) = \varphi(k,0) + \kappa(0,l) = \left(\varphi(0,0) + \kappa(k,0)\right) + \kappa(0,l).
	\end{equation}
	The two resulting expressions for $\kappa$ yield the contradiction
	\begin{equation}
		\kappa(k,l) = \kappa(0,l) + \kappa(k,0) + k \cdot l = \kappa(0,l) + \kappa(k,0).
	\end{equation}
	This contradiction shows that the initial assumption $Zf(x,\omega) \neq 0$ cannot hold.
\end{proof}
\begin{proof}(\textit{BLT, Wiener amalgam version}).
	If the window $g \in W_0(\Rd)$, then $Zg$ is continuous by Lemma \ref{lem_Zak_prop}. Hence, $Zg$ has a zero in $\mathcal{Q} \times \mathcal{Q}$ by the above lemma. Therefore, by Corollary \ref{cor_Zak_bounds} the lower frame bound vanishes and $\G(g,\Z^{2d})$ cannot be a frame. Hence, if $\G(g,\Z^{2d})$ is a frame, then $g \notin W_0(\Rd)$.

	We note that $\G(g, \Z^{2d})$ is a frame if and only if $\G(\widehat{g}, \Z^{2d})$ is a frame. Therefore, $\widehat{g} \notin W_0(\Rd)$ as well.
\end{proof}

At this point, it should be mentioned that neither of the presented versions of the BLT implies the other \cite[Chap.~8.4, Rem.~3]{Gro01}. We note that there are similar results for other function spaces as well (re-call also metatheorems \ref{mthm_A}, \ref{mthm_B}, \ref{mthm_C}). The common theme is the impossibility of constructing an orthonormal Gabor basis from a nice function with a nice Fourier transform. This impossibility of the existence of ONBs of the form $\G(g, \Z \times \Z)$ with (mild) decay and smoothness conditions on $g$ is the very reason why over-complete Gabor systems are studied and used in applications.

\subsection{Density of Gabor Frames}
The Balian-Low theorem states that we cannot obtain an orthonormal basis from a Gabor system where the window $g$ has good time-frequency concentration. It is accompanied by the density principle, which should be seen as an uncertainty principle for Gabor systems. In order to prove the density principle, we will introduce necessary results (partially) without proof. The results are part of a general duality theory for Gabor systems $\G(g,\L)$ and $\G(g,\L^\circ)$. We refer to the treatise \cite{GroKop19} (see also \cite{Kop17}).

The \textit{Fundamental Identity of Gabor Analysis} (FIGA) is a main tool to relate a Gabor system $\G(g,\L)$ to its adjoint system $\G(g,\L^\circ)$. This is of interest for the characterization of Gabor frames. So far, we have seen the frame inequality, which gives the condition on a set to be a frame. On the other hand, there are Riesz sequences, which are defined in a similar, but contrary manner. A sequence (or set) $\{f_\gamma  \mid \gamma \in \Gamma \}$ in a Hilbert space $\mathcal{H}$ is a Riesz sequence if and only if there exist positive constants $0 < A \leq B < \infty$ such that
\begin{equation}\label{eq_Riesz}
	A \norm{c}_{\ell^2}^2 \leq \norm{ \sum_{\gamma \in \Gamma} c_\gamma f_\gamma}_\mathcal{H}^2 \leq B \norm{c}_{\ell^2}^2, \qquad \forall c =(c_\gamma)_{\gamma \in \Gamma} \in \ell^2(\Gamma).
\end{equation}
A Riesz sequence is called Riesz basis if it is complete in $\mathcal{H}$, i.e., its linear span is dense in $\mathcal{H}$.

The completeness of a Riesz basis is an additional assumption and does not follow from the norm equivalence \eqref{eq_Riesz}: Take an orthonormal basis for $\mathcal{H}$ and remove one element. This family still satisfies \eqref{eq_Riesz}, but its linear span is clearly not dense in $\mathcal{H}$. In particular, a Riesz sequence need not be a frame since the latter is always (over)complete.

Conversely, a frame need not be a Riesz sequence: The union of two orthonormal bases is a frame, but there exists a non-trivial linear combination of zero since this system is linearly dependent. This contradicts the lower bound in \eqref{eq_Riesz}.

We will not further elaborate on the general theory of Riesz sequences and frames, but state ``the duality" between the Gabor systems $\G(g,\L)$ and $\G(g, \L^\circ)$ in terms of frames and Riesz sequences.

The first important result is the following.
\begin{theorem}[Bessel Duality]
	Let $g \in \Lt$ and $\L \subset \R^{2d}$ be a lattice. Then $\G(g,\L)$ is a Bessel sequence if and only if $\G(g, \L^\circ)$ is a Bessel sequence.
\end{theorem}

The structures of $\G(g,\L)$ and its dual system $\G(g,\L^\circ)$ are intimately related: A Gabor frame on one side corresponds to a Riesz sequence on the other. Furthermore, dual windows of $\G(g,\L)$ are characterized by a biorthogonality condition along the adjoint lattice.

We introduce a more general notion of the frame operator. The frame operator with analysis window $g$ and synthesis window $\widetilde{g}$ and index set $\Gamma \subset \R^{2d}$ is given by
\begin{equation}
	S_{g,\widetilde{g},\Gamma} f = \sum_{\gamma \in \Gamma} \langle f, \pi(\gamma) g \rangle \pi(\gamma) \widetilde{g}.
\end{equation}
\begin{definition}
	Let $g \in \Lt$ and let $\G(g,\L)$ be a Bessel sequence. We call $\widetilde{g} \in \Lt$ a dual window to $\G(g,\L)$, if $\G(\widetilde{g},\L)$ is a Bessel sequence and the reconstruction property $S_{g,\widetilde{g},\L} = S_{\widetilde{g}, g , \L} = I$ is satisfied.
\end{definition}
Now, if $\G(g, \L)$ is a Gabor frame, then there exists a dual window $\widetilde{g}$, such that $S_{g,\widetilde{g},\L} = S_{\widetilde{g}, g , \L} = I$. We state the result without proof and refer to \cite{Gro01} or \cite{Gro14}.
\begin{theorem}[Wexler-Raz Biorthogonality relations]
	Let $g \in \Lt$ and $\L \subset \R^{2d}$ be a lattice. Then, the following are equivalent.
	\begin{enumerate}[(i)]
		\item $\G(g,\L)$ is a frame for $\Lt$.
		\item $\G(g,\L^\circ)$ is a Bessel sequence and there exists a dual window $\widetilde{g} \in \Lt$ such that $\G(\widetilde{g}, \L^\circ)$ is a Bessel sequence and
		\begin{equation}\label{eq_Wex-Raz}
			\langle \widetilde{g}, \pi(\l^\circ) g \rangle = \vol(\L) \delta_{0,\l^\circ}, \quad \forall \l^\circ \in \L^\circ.
		\end{equation}
	\end{enumerate}
\end{theorem}
We remark that the proof relies on the Fundamental Identity of Gabor Analysis, which we will state (and prove) after the following lemma, which we state without proof (see \cite[Chap.~2]{Kop17} for the details).
\begin{lemma}\label{lem_Vgf_Vgh_M1}
	Let $g, \widetilde{g} \in M^1(\Rd)$, $f \in M^p(\Rd)$ and $h \in M^q(\Rd)$ with $\frac{1}{p} + \frac{1}{q} = 1$. Then, the product of the STFTs $\overline{V_g f} \, V_{\widetilde{g}} h$ is in $M^1(\R^{2d})$.
\end{lemma}

\begin{theorem}[Fundamental Identity of Gabor Analysis]\label{thm_FIGA}
	Let $f, h \in \Lt$, $g, \widetilde{g} \in M^1(\Rd)$ and $\L \subset\R^{2d}$ a lattice. Then
	\begin{equation}
		\sum_{\l \in \L} V_{g} f(\l) \, \overline{V_{\widetilde{g}} h(\l)}
		= \vol(\L)^{-1} \sum_{\l^\circ \in \L^\circ} V_{g} \widetilde{g}(\l^\circ) \, \overline{ V_{f} h(\l^\circ)},
	\end{equation}
	with absolute convergence on both sides.
\end{theorem}
\begin{proof}
	The technical assumptions for the Poisson summation formula to hold point-wise are fulfilled by Lemma \ref{lem_Vgf_Vgh_M1}. We will use the symplectic Fourier transform and the symplectic version of the Poisson summation formula to prove the result.
	\begin{align}\label{eq_FIGA_proof}
		\sum_{\l \in \L} V_{g} f(\l) \, \overline{V_{\widetilde{g}} h(\l)}
		& =	\vol(\L)^{-1} \sum_{\l^\circ \in \L^\circ} \F_\sigma \left(V_{g} f \, \overline{V_{\widetilde{g}} h}\right) (\l^\circ)
	\end{align}
	We write the symplectic Fourier transform explicitly as
	\begin{align}
		\F_\sigma \left(V_{g} f \, \overline{V_{\widetilde{g}} h}\right) (z)
		& =  \iint_{\R^{2d}} V_g f (z') \, \overline{V_{\widetilde{g}} h(z')} e^{-2 \pi i \sigma(z',z)} \, dz'.
	\end{align}
	Now, from the commutation relations \eqref{eq_comm_rel}, we obtain
	\begin{align}
		\pi(z')^{-1} \pi(z) \pi(z') = e^{2 \pi i \sigma(z',z)} \pi(z)
		\quad \text{ and } \quad
		V_{\pi(z')g} (\pi(z') f)(z) & = \langle f, \pi(z')^{-1} \pi(z) \pi(z') g \rangle\\
		& = e^{-2 \pi i \sigma(z',z)} V_g f(z).
	\end{align}
	Therefore,
	\begin{align}
		\F_\sigma \left(V_{g} f \, \overline{V_{\widetilde{g}} h}\right) (z)
		& =  \iint_{\R^{2d}} V_g f (z') \, \overline{V_{\widetilde{g}} h(z')} e^{-2 \pi i \sigma(z',z)} \, dz'\\
		& = \iint_{\R^{2d}} V_g f(z') \, \overline{V_{\pi(-z) \widetilde{g}} (\pi(-z) h) (z')} \, dz'\\
		& = \langle f, \pi(-z) h \rangle \overline{\langle g, \pi (-z) \widetilde{g} \rangle},
	\end{align}
	where we used the orthognoality relation \eqref{eq_OR} to obtain the last line in the above equation. Plugging this into \eqref{eq_FIGA_proof}, we get
	\begin{align}
		\sum_{\l \in \L} V_{g} f(\l) \, \overline{V_{\widetilde{g}} h(\l)}
		& =	\vol(\L)^{-1} \sum_{\l^\circ \in \L^\circ} \F_\sigma \left(V_{g} f \, \overline{V_{\widetilde{g}} h}\right) (\l^\circ)\\
		& = \vol(\L)^{-1} \sum_{\l^\circ \in \L^\circ} V_h f(-\l^\circ) \overline{V_{\widetilde{g}} g (-\l^\circ)}\\
		& =  \vol(\L)^{-1} \sum_{\l^\circ \in \L^\circ} V_g \widetilde{g}(\l^\circ) \overline{V_f h(\l^\circ)}.
	\end{align}
\end{proof}

\begin{remark}
	The frame inequality combined with Theorem \ref{thm_FIGA} yields the following bounds on the sharp upper frame bound.

	\smallskip
	
	\textit{Let $g \in M^1(\Rd)$, $\norm{g}_2 = 1$ and $\L \subset \R^{2d}$ a lattice. Let $\G(g,\L)$ be a Gabor system. Then $\G(g, \L)$ is actually a Bessel sequence with sharp Bessel bound $B$ and we have}
	\begin{equation}\label{eq_bounds}
		\norm{(V_g g(\l))}_{\ell^2(\L)} = \sum_{\l \in \L} |V_g g(\l)|^2 \leq B \leq \vol(\L)^{-1} \sum_{\l^\circ \in \L^\circ} |V_g g(\L^\circ)|= \vol(\L)^{-1} \norm{(V_g g(\l^\circ))}_{\ell^1(\L^\circ)}.
	\end{equation}

	\smallskip
	
	The lower bound follows immediately from the frame inequality by setting $f = g$. The upper frame bound follows by using the FIGA with $h = f$ and $\widetilde{g} = g$;
	\begin{align}
		\sum_{\l \in \L} |V_g f(\l)|^2 & = \sum_{\l \in \L} V_g f(\l) \overline{V_g f(\l)}\\
		& = \vol(\L)^{-1} \sum_{\l^\circ \in \L^\circ} V_g g(\L^\circ) \overline{V_f f (\l^\circ)}\\
		& \leq \vol(\L)^{-1} \sum_{\l^\circ \in \L^\circ} |V_g g(\L^\circ)| |Af (\l^\circ)| \\
		& \leq \left(\vol(\L)^{-1} \sum_{\l^\circ \in \L^\circ} |V_g g(\l^\circ)|\right) \norm{f}_2^2 .
	\end{align}
	Here, we used the fact that $|V_f f(x,\omega)| = |Af(x,\omega)| \leq Af(0,0) = \norm{f}_2^2$ from Lemma \ref{lem_ambi_max}. Setting
	\begin{equation}
		\widetilde{B} = \vol(\L)^{-1} \sum_{\l^\circ \in \L^\circ} |V_g g(\l^\circ)|
	\end{equation}
	we see that
	\begin{equation}
		\sum_{\l \in \L} |V_g f(\l)|^2 \leq \widetilde{B} \norm{f}, \quad \forall f \in \Lt.
	\end{equation}
	Hence, $\widetilde{B}$ is a Bessel bound (not necessarily the sharpest one) for the Gabor system $\G(g, \L^\circ)$ and we have $B \leq \widetilde{B}$, which finishes the proof of \eqref{eq_bounds}.
	
	For Gabor systems of the form $\G(g, \alpha \Z \times \beta \Z)$ with $g \in \mathcal{S}(\R)$ the result was proven by Tolimieri and Orr \cite{TolOrr95}. Therefore, $\widetilde{B}$ might also be referred to as the Tolimieri-Orr bound.
	\begin{flushright}
		$\diamond$
	\end{flushright}
\end{remark}

The Wexler-Raz biorthogonality relations yield a characterization of all dual windows in the lattice case.
\begin{corollary}\label{cor_Wex-Raz}
	Suppose $g, \widetilde{g} \in \Lt$ and $\L \subset \R^{2d}$ a lattice such that $\G(g,\L)$ and $\G(\widetilde{g},\L)$ are Bessel sequences. Then $\widetilde{g}$ is a dual window to $\G(g,\L)$ if and only if the Wexler-Raz biorthogonality relations \eqref{eq_Wex-Raz} are satisfied.
\end{corollary}

We can now state and proof the density theorem for Gabor systems.
\begin{theorem}[Density Theorem]
	Let $g \in \Lt$ and $\L \subset \R^{2d}$ be a lattice. If $\G(g,\L)$ is a frame, then $\vol(\L) \leq 1$.
\end{theorem}
\begin{proof}
	Let $\G(g,\L)$ be a frame and write $S_\G$ for the associated Gabor frame operator. Since in this case the frame operator is a positive invertible operator, the operator $S_\G^{-1/2}$ is a well-defined positive operator on $\Lt$. In particular it is self-adjoint.
	
	By Corollary \ref{cor_Wex-Raz}, the canonical dual window $g^\circ = S^{-1}_\G g$ fulfills the Wexler-Raz biorthogonality relations \eqref{eq_Wex-Raz}. This yields
	\begin{equation}
		\vol(\L) = \langle S^{-1}_\G g, g \rangle = \langle S^{-1/2}_\G g, S^{-1/2}_\G g \rangle = \norm{S^{-1/2}_\G g}_2^2.
	\end{equation}
	By Lemma \ref{lem_ONB_tight_frame} $\G(S^{-1/2}_\G g, \L)$ is a tight frame with frame bounds $A = B = 1$. Therefore
	\begin{align}
		\norm{S_\G^{-1/2} g}_2^4 & = |\langle S^{-1/2}_\G g, S^{-1/2}_\G g \rangle |^2\\
		& \leq \sum_{\l \in \L} |\langle S^{-1/2}_\G g, S^{-1/2}_\G \pi(\l) g \rangle |^2 = \norm{S^{-1/2} g}_2^2.
	\end{align}
	Hence, we obtain
	\begin{equation}
		\vol(\L) = \norm{S_\G^{-1/2} g}_2^2 \leq 1.
	\end{equation}
\end{proof}

Note that while the density principle gives a necessary condition (which might be far from sufficiency), the BLT states that this necessary density condition is never sharp for ``nice" windows.
\begin{corollary}
	If $\G(g, \L)$ is a frame for $\Lt$ and if either $g \in W_0(\Rd)$ or $\widehat{g} \in W_0(\Rd)$, then $\vol(\L) < 1$.
\end{corollary}
In particular, if we consider the standard Gaussian window $g_0(t) = 2^{d/4} e^{-\pi t^2}$, then we see that $\G(g_0, \Z^{2d})$ is not a frame for $\Lt$. This holds more generally for any set $\Gamma \subset \R^{2d}$ with lower Beurling density $D^-(\Gamma) = 1$.

\section{The Bargmann Transform}\label{sec_BT}
For many reasons the Gaussian function is one of the most popular windows in time-frequency analysis and it is of special interest to study the STFT and Gabor systems with a Gaussian window.

Recall that the $L^2$-normalized standard Gaussian on $\Rd$ is given by
\begin{equation}
	g_0(t) = 2^{d/4} e^{- \pi t^2}.
\end{equation}
Then
\begin{align}\label{eq_STFT_BT_motivation}
	V_{g_0} f(x,\omega) & = 2^{d/4} \int_{\Rd} f(t) e^{-\pi(t-x)^2} e^{-2 \pi i \omega \cdot t} \, dt\\
	& = 2^{d/4} \int_{\Rd} f(t) e^{-\pi t^2} e^{2 \pi x \cdot t} e^{-\pi x^2} e^{-2 \pi i \omega \cdot t} \, dt\\
	& = 2^{d/4} e^{-\pi i x \cdot \omega} e^{-\frac{\pi}{2}(x^2+\omega^2)} \int_{\Rd} f(t) e^{-\pi t^2} e^{2 \pi t \cdot (x - i \omega)} e^{-\frac{\pi}{2}(x - i \omega)^2} \, dt.
\end{align}
Let us convert $(x,\omega) \in \R^{2d}$ into a complex vector $z = x+ i \omega \in \C^d$. Recall that, for complex vectors, $w \cdot z = w^T z$ does not yield the inner product on $\C^d$, so $z^2 = (x+iy) \cdot (x+iy)$, whereas $|z|^2 = \overline{z} \cdot z = (x-iy) \cdot (x+iy)$. Then, comparing with the formula above, the following definition appears quite naturally.
\begin{definition}
	The Bargmann transform of a function $f$ on $\Rd$ is the function $Bf$ on $\C^d$ defined by
	\begin{equation}\label{eq_BT}
		Bf(z) = 2^{d/4} \int_{\Rd} f(t) e^{2 \pi t \cdot z - \pi t^2 - \frac{\pi}{2} z^2} \, dt.
	\end{equation}
	The (Bargmann-)Fock space $\F^2(\C^d)$ is the Hilbert space of all entire functions $F$ on $\C^d$ for which the norm
	\begin{equation}
		\norm{F}_{\F^2}^2 = \int_{\C^d} |F(z)|^2 e^{-\pi |z|^2} \, dz
	\end{equation}
	is finite. The inner product on $\F^2(\C^d)$ is given by
	\begin{equation}
		\langle F, G \rangle_{\F^2} = \int_{\C^d} F(z) \overline{G(z)} e^{-\pi |z|^2} \, dz.
	\end{equation}
\end{definition}
By means of the small calculation in \eqref{eq_STFT_BT_motivation}, we may re-write several of the statements about the STFT by means of the Bargmann transform.

\begin{proposition}\label{pro_BT}
	\begin{enumerate}[(a)]
		\item If $f$ is a function on $\Rd$ with polynomial growth, then its Bargmann transform is an entire function on $\C^d$. If we write $z = x + i \omega$, then
		\begin{equation}\label{eq_STFT_BT}
			V_{g_0}f(x,-\omega) = e^{\pi i x \cdot \omega} e^{-\frac{\pi}{2} |z|^2} Bf(z).
		\end{equation}
		\item If $f \in \Lt$, then
		\begin{equation}
			\norm{f}_2^2 = \int_{\C^d} |Bf(z)|^2 e^{-\pi |z|^2} \, dz = \norm{B f}_{\F^2}^2.
		\end{equation}
		Thus, $B$ is an isometry from $\Lt$ into $\F^2(\C^d)$.
	\end{enumerate}
\end{proposition}
\begin{proof}
	\begin{enumerate}[(a)]
		\item If $|f(t)| = \mathcal{O}(|t|^N)$, then the integral in \eqref{eq_BT} converges absolutely for every $z \in \C^d$ and uniformly over compact sets in $\C^d$. Therefore, one can differentiate under the integral (with respect to $z$) and, hence, $Bf(z)$ is an entire function.
		
		Equation \eqref{eq_STFT_BT} is just \eqref{eq_STFT_BT_motivation} rewritten in another notation.
		
		\item The statement follows form the fact that
		\begin{equation}
			\norm{V_g f}_2 = \norm{f}_2 \norm{g}_2
		\end{equation}
		and $\norm{g_0}_2 = 1$.
	\end{enumerate}
\end{proof}
The next aim is to show that $B$ is a unitary mapping from $\Lt$ onto $\F^2(\C^d)$. In view of the above proposition, we only need to show that the range of $B$ is dense in $\F^2(\C^d)$. For this purpose we will carry out a study of entire functions of several complex variables. Due to the use of the multi-index notation, the difference to the theory in one variable will almost not be visible.
\begin{theorem}
	\begin{enumerate}[(a)]
		\item The collection of monomials of the form
		\begin{equation}
			e_\alpha(z) = \left(\frac{\pi^{|\alpha|}}{\alpha!}\right)^{1/2} z^\alpha = \prod_{k =1}^d \left(\frac{\pi^{\alpha_k}}{\alpha_k!}\right)^{1/2} z_k^{\alpha_k},
		\end{equation}
		$\alpha = (a_1, \ldots , \alpha_d)$, $\alpha_k \in \N_0$, forms an orthonormal basis for $\F^2(\C^d)$.
		
		\item $\F^2(\C^d)$ is a reproducing kernel Hilbert space, that is
		\begin{equation}
			|F(z)| \leq \norm{F}_{\F^2} \, e^{\frac{\pi}{2}|z|^2}, \quad \forall z \in \C^d.
		\end{equation}
		The reproducing kernel is $K_w(z) = e^{\pi \overline{w} \cdot z}$, i.e.,
		\begin{equation}
			F(w) = \langle F, K_w \rangle_{\F^2}.
		\end{equation}
	\end{enumerate}
\end{theorem}
\begin{proof}
	\begin{enumerate}[(a)]
		\item We write each variable in polar coordinates; $z_k = r_k e^{i \theta_k}$. We start with computing the inner product of $z^\alpha$ with $z^\beta$ restricted to the poly-disc
		\begin{equation}
			P_R = \{z \in \C^d \mid |z_k| \leq R, \, k=1, \ldots d\}:
		\end{equation}
		\begin{align}
			\langle z^\alpha, z^\beta \rangle_{\F^2} & = \int_{P_R} z^\alpha \overline{z^\beta} e^{-\pi |z|^2}\\
			& =\prod_{k=1}^d \int_{|z_k| \leq R} z_k^{\alpha_k} \overline{z_k}^{\beta_k} e^{-\pi|z_k|^2} \, dz_k\\
			& = \prod_{k=1}^d \int_0^R \int_0^{2\pi} r_k^{\alpha_k + \beta_k + 1} e^{i(\alpha_k - \beta_k) \theta_k} e^{-\pi r_k^2} \, d\theta_k \, dr_k
		\end{align}
		If $\alpha \neq \beta$, then this integral equals 0 for all $R > 0$ and therefore
		\begin{equation}
			\langle z^\alpha, z^\beta\rangle_{\F^2} = \lim_{R \to \infty} \int_{P_R} z^\alpha \overline{z^\beta} e^{-\pi|z^2|} \, dz = 0.
		\end{equation}
		On the other hand, if $\alpha = \beta$ then
		\begin{equation}
			\int_{P_R} |z^\alpha|^2 e^{-\pi|z|^2} \, dz = \prod_{k=1}^d \left(2 \pi \int_0^R r^{2 \alpha_k + 1} e^{-\pi r_k^2} \, dr_k \right) = \mu_{\alpha,R}
		\end{equation}
		Consequently $\{\mu^{-1/2}_{\alpha,R} z^\alpha \mid \alpha \in \N_0^d\}$ is an orthonormal system in $L^2(P_R, e^{-\pi|z|^2} \, dz)$. For $R = \infty$, by making the change of variables $s = \pi r^2$ ($dr = \frac{ds}{2 \pi r}$), we can continue as follows:
		\begin{equation}
			\mu_{\alpha, \infty} = \prod_{k=1}^d \left(\int_0^\infty \left(\frac{s}{\pi}\right)^{\alpha_k} e^{-s} \, ds \right) = \prod_{k=1}^d \frac{\Gamma(\alpha_k)}{\pi^{\alpha_k}} = \prod_{k=1}^d \frac{\alpha_k!}{\pi^{\alpha_k}} = \frac{\alpha!}{\pi^{|\alpha|}}.
		\end{equation}
		In particular, we see that $\{e_\alpha \mid \alpha \in \N_0^d\}$ is an orthonormal system in $\F^2(\C^d)$.
		
		To prove completeness of the system in $\F^2(\C^d)$, we start from the power series expansion of $F$ in $\F^2(\C^d)$, which has the form
		\begin{equation}
			F(z) = \sum_{\alpha \in \N_0^d} c_\alpha z^\alpha.
		\end{equation}
		Suppose that $\langle F , e_\beta \rangle_{\F^2} = 0$ for all $\beta \in \N_0^d$. Then
		\begin{equation}
			\langle F, e_\beta \rangle_{\F^2} = \lim_{R \to \infty} \left( \frac{\pi^{|\beta|}}{\beta!} \right)^{1/2} \int_{P_R} \left(\sum_{\alpha \geq 0} c_\alpha z^\alpha \right) \overline{z^\beta} e^{-\pi |z|^2} \, dz.
		\end{equation}
		Since the power series expansion in the integral converges uniformly on compact sets, we can interchange the order of integration and summation to obtain
		\begin{equation}
			\int_{P_R} F(z) \overline{z^\beta} e^{-\pi|z|^2} \, dz  = \sum_{\alpha \geq 0} c_\alpha \int_{P_R} z^\alpha \overline{z^\beta} e^{-\pi |z|^2} \, dz = c_\beta \, \mu_{\beta,R}.
		\end{equation}
		Thus, by assumption
		\begin{equation}
			\langle F, e_\beta \rangle_{\F^2} = \left( \frac{\pi^{|\beta|}}{\beta!} \right)^{1/2} c_\beta \lim_{\R \to \infty} \mu_{\beta,R} = 0.
		\end{equation}
		This implies that $c_\beta = 0$ for all $\beta$ and thus $F \equiv 0$. As $\{e_\alpha\}$ is an orthonormal system in $\F^2(\C^d)$, it now follows that it is an orthonormal basis for $\F^2(\C^d)$.
		
		\item Since $F(z) = \sum_{\alpha \geq 0} \langle F, e_\alpha \rangle_{\F^2} \, e_\alpha(z)$, we obtain by the Cauchy-Schwarz inequality that
		\begin{equation}
			|F(z)| \leq \left( \sum_{\alpha \geq 0} | \langle F, e_\alpha \rangle_{\F^2}|^2 \right)^{1/2} \left( \sum_{\alpha \geq 0} \frac{\pi^\alpha}{\alpha!} |z^\alpha|^2 \right)^{1/2} = \norm{F}_{\F^2} \, e^{\frac{\pi}{2} |z|^2}.
		\end{equation}
		Thus, point evaluations are continuous linear functionals on the Hilbert space $\F^2(\C^d)$. By the Riesz representation theorem it follows that for each $w \in \C^d$ there is a function $K_w \in \F^2(\C^d)$ such that
		\begin{equation}
			F(w) = \langle F, K_w \rangle_{\F^2}
		\end{equation}
		Expanding $K_w$ with respect to the orthonormal basis $e_\alpha$ and using the above equation, we obtain $K_w$ explicitly as
		\begin{align}
			K_w(z) & = \sum_{\alpha \geq 0} \langle K_w, e_\alpha \rangle_{\F^2} \, e_\alpha(z)\\
			& = \sum_{\alpha \geq 0} \overline{e_\alpha(w)} \, e_\alpha(z)\\
			& = \sum_{\alpha \geq 0} \frac{\pi^{|\alpha|}}{\alpha!} \overline{w^\alpha} \, z^\alpha\\
			& = e^{\pi \overline{w} \cdot z}.
		\end{align}
	\end{enumerate}
\end{proof}
Equipped with the reproducing kernel $K_w$ for $\F^2(\C^d)$, we can prove that the Bargmann transform maps onto $\F^2(\C^d)$.
\begin{theorem}\label{thm_BT_unitary}
	The Bargmann transform is a unitary operator from $\Lt$ onto $\F^2(\C^d)$.
\end{theorem}
\begin{proof}
	In Proposition \ref{pro_BT} we have proved that the Bargmann transform is an isometry. Thus its range is a closed subspace of $\F^2(\C^d)$. Therefore, if we show that $B(\Lt)$ is dense in $\F^2(\C^d)$, then it follows that $B(\Lt) = \F^2(\C^d)$, so $B$ is surjective, which proves the claim.
	
	Using Example \ref{ex_STFT_g0} together with the orthogonality relations \ref{thm_ortho} we obtain
	\begin{equation}
		V_{g_0} (M_\eta T_\xi g_0)(x,\omega) = e^{-\pi i (x+\xi) \cdot (\omega - \eta)} e^{-\frac{\pi}{2} ((x-\xi)^2 + (\omega-\eta)^2)}.
	\end{equation}
	On the other hand, writing $z=x+i\omega$ and $w=\xi+i\eta$, we obtain
	\begin{align}
		B(M_{-\eta} T_\xi g_0)(z) & = e^{-\pi i x \cdot \omega} e^{\frac{\pi}{2}|z|^2} V_{g_0}(M_{-\eta} T_\xi g_0)(x,-\omega)\\
		& = e^{-\pi i x \cdot \omega} e^{\frac{\pi}{2}(x^2+\omega^2)} e^{-\pi i (x+\xi) \cdot (-\omega + \eta)} e^{-\frac{\pi}{2} ((x-\xi)^2 + (-\omega+\eta)^2)}\\
		& = e^{-\pi i \xi \cdot \eta} e^{-\frac{\pi}{2}(\xi^2+\eta^2)} e^{\pi (x \cdot \xi+ \omega \cdot \eta + i(\xi \cdot \omega - x \cdot \eta))}\\
		& = e^{-\pi i \xi \cdot \eta} e^{-\frac{\pi}{2} |w|^2} e^{\pi \overline{w} \cdot z}.	
	\end{align}
	
	We can rewrite this in short as
	\begin{equation}
		B(M_{-\eta} T_\xi g_0)(z) = e^{-\pi i \xi \cdot \eta} e^{-\frac{\pi}{2}|w|^2} K_w(z).
	\end{equation}
	This shows that the reproducing kernel of $\F^2(\C^d)$ is in the range of $B$. Now suppose that for some $F \in \F^2(\C^d)$ we have $\langle F, B f \rangle_{\F^2} = 0$ for all $f \in \Lt$. In particular, by the above equation we have for all $w \in \C^d$ that
	\begin{align}
		0 = \langle F, B(M_{-\eta} T_\xi g_0) \rangle_{\F^2}
		=  e^{\pi i \xi \cdot \eta} e^{-\frac{\pi}{2} |w|^2} \langle F, K_w \rangle_{\F^2}
		=  e^{\pi i \xi \cdot \eta} e^{-\frac{\pi}{2} |w|^2} F(w).
	\end{align}
	Therefore, $F \equiv 0$ and therefore the range of $B$ is dense in $\Lt$.
\end{proof}
Note that the Bargmann transform sends the collection of time-frequency shifted Gaussians $\{ M_{-\eta} T_\xi g_0 \}$ to the normalized reproducing kernel (up to a phase factor).
\begin{align}
	\norm{B(M_{-\eta} T_{\xi} g_0)}_{\F^2}^2 & = \norm{e^{-\frac{\pi}{2}|w|^2} K_w(z)}_{\F^2}^2 = \int_{\C^d} |e^{-\frac{\pi}{2}|w|^2} K_w(z)|^2 e^{-\pi |z|^2} \, dz\\
	& = \int_{\C^d} e^{-\pi |w|^2} e^{\pi (\overline{w} \cdot z + w \cdot \overline{z})} e^{-\pi |z|^2} \, dz \\
	& = \int_{\C^d} e^{-\pi|w-z|^2} \, dz = 1, \quad \forall w \in \C^d .
\end{align}

\NO{
\hrule
\medskip
The material after this point has not been presented in the lecture course. The only exception is Theorem \ref{thm_Lyu_Seip_Wall} (ii).
\medskip
\hrule
\medskip}

\NO{Since the Bargmann transform is a unitary operator, the pre-image of the orthonormal basis $\{e_\alpha \}$ consisting of the functions $h_\alpha = B^{-1} e_\alpha \in \Lt$ is an orthonormal basis for $\Lt$. The functions $h_\alpha$ are the Hermite functions. Even without an explicit formula we can derive their most important property for Fourier analysis.
\begin{proposition}
	The Hermite functions are eigenfunctions of the Fourier transform; specifically, for all $\alpha \geq 0$ we have
	\begin{equation}
		\F h_\alpha = (-i)^{|\alpha|} h_\alpha.
	\end{equation}
\end{proposition}
\begin{proof}
	For the proof we combine the fundamental identity \eqref{eq_fitf} with Proposition \ref{pro_BT}.
	\\
	Writing $z = x+i\omega \in \C^d$, we have
	\begin{equation}
		V_{g_0}h_\alpha (x,-\omega) = e^{\pi i x \cdot \omega} B h_\alpha(z) e^{-\frac{\pi}{2}|z|^2} = e^{\pi i x \cdot \omega} e^{-\frac{\pi}{2} |z|^2} e_\alpha.
	\end{equation}
	On the other hand, using \eqref{eq_fitf} and the property $\F g_0 = g_0$, we obtain
	\begin{align}
		V_{g_0} \widehat{h_\alpha}(x,-\omega) & = V_{\widehat{g_0}} \widehat{h_\alpha}(x,-\omega)\\
		& = e^{2 \pi i x \cdot \omega} V_{g_0} h_\alpha(\omega,x)\\
		& = e^{2 \pi i x \cdot \omega} e^{-\pi i x \cdot \omega} e^{-\frac{\pi}{2}|z|^2} B h_\alpha(\omega-ix)\\
		& = e^{\pi i x \cdot \omega} e^{-\frac{\pi}{2} |z|^2} \left( \frac{\pi^{|\alpha|}}{\alpha!} \right)^{1/2} (\omega-ix)^\alpha\\
		& = e^{\pi i x \cdot \omega} e^{-\frac{\pi}{2}|z|^2} \left( \frac{\pi^{|\alpha|}}{\alpha!} \right)^{1/2} (-i z)^\alpha\\
		& = (-i)^{|\alpha|} e^{\pi i x \cdot \omega} e^{-\frac{\pi}{2}|z|^2} e_\alpha(z).
	\end{align}
	Since $V_{g_0}$ is one-to-one, it follows that $\F h_\alpha = (-i)^{|\alpha|} h_\alpha$.
\end{proof}
}

\NO{
We note that if $f \in \Lt$ is a finite linear combination of Hermite functions, then $\norm{f}_2 = \norm{\widehat{f}}_2$. However, this is not an independent proof of Plancherel's theorem because the argument is circular. We have used Plancherel's theorem to show that the Bargmann transform is an isometry. In order to derive the result we also used that $\norm{V_gf}_2^2 = \norm{f}_2^2 \norm{g}_2^2$, which is a corollary of the orthogonality relations \eqref{eq_OR} and in order to derive these, we used Plancherel's theorem. The insight of this statement is that the unitarity of the Fourier transform and of the Bargmann transform are equivalent. The results in this section show how special the Gaussian function is, because the STFT with a Gaussian window is, except for a weighting factor, an entire function. Hence, we can switch to complex analysis to prove statements about $V_{g_0}$. This is the main reason why much more results are known for the STFT with a Gaussian window than with other windows. The fact that complex analysis in several variables also differs greatly from the one-dimensional theory is the main reason why, despite knowing much about STFTs with one-dimensional Gaussian windows, higher dimensional Gaussian Gabor Systems still keep many mysteries.
\\
Getting back to unitary representations of the Heisenberg Group, we may now write the Heisenberg group as $\mathbf{H} = \C^d \times \R$. For $z = x+i\omega \in \C^d$ and $w = \xi+i\eta$, we compute $B \rho(z,\tau) f$ for $f \in \Lt$. We start by writing $Bf(w)$ as a representation coefficient. Recall that $\langle f, \rho(x,\omega,\tau) g \rangle = e^{-2 \pi i \tau} e^{\pi i x \cdot \omega} V_g f(x,\omega)$. Then
\begin{align}
	Bf(w) & =e^{-\pi i \xi \cdot \eta} e^{\frac{\pi}{2} |w|^2} V_{g_0} f(\xi, -\eta)\\
	& =e^{-\pi i \xi \cdot \eta} e^{\frac{\pi}{2}	|w|^2} e^{-\pi i \xi \cdot (-\eta)} \langle f, \rho(\xi,-\eta,0) g_0 \rangle\\
	& = e^{\frac{\pi}{2}|w|^2} \langle f, \rho(\overline{w}, 0) g_0 \rangle.
\end{align}
Consequently,
\begin{align}
	B(\rho(z,\tau) f)(w) & = e^{\frac{\pi}{2}|w|^2} \langle \rho(z,\tau) f, \rho(\overline{w},0) g_0 \rangle\\
	& = e^{\frac{\pi}{2}|w|^2} \langle f, \rho(-z,-\tau) \rho(\overline{w},0) g_0 \rangle\\
	& = e^{2 \pi i \tau} e^{\frac{\pi}{2}|w|^2} e^{-\pi i \Im (-z \cdot \overline{\overline{w}})} \langle f, \rho(\overline{w}-z,0) g_0 \rangle\\
	& = e^{2 \pi i \tau} e^{\pi i \Im (w \cdot z)} e^{\frac{\pi}{2} |w|^2} e^{-\frac{\pi}{2}|w-\overline{z}|^2} Bf(w-\overline{z})\\
	& =e^{2 \pi i \tau} e^{\pi w \cdot z} e^{-\frac{\pi}{2}|z|^2} B(f(w-\overline{z})).
\end{align}
Here, we used the fact that $|w|^2-|w-\overline{z}|^2 = 2 \Re(\overline{w} \cdot z) - |z|^2$ \footnote{\NO{Also, note that the standard symplectic form is $\sigma(z,w) = \sigma((x,\omega),(\xi,\eta)) = \Im(z \cdot \overline{\omega})$.}}. Setting
\begin{equation}
	\beta(z, \tau) F(w) = e^{2 \pi i \tau} e^{\pi w \cdot z} e^{-\frac{\pi}{2}|z|^2} F(w-\overline{z}),
\end{equation}
we have derived the Bargmann representation of the Heisenberg group. Moreover, we have shown that it is unitarily equivalent to the Schrödinger representation via the Bargmann transform, i.e.,
\begin{equation}
	\beta(z,\tau) = \beta(x,\omega,\tau) = B \rho(x,\omega,\tau)B^{-1}.
\end{equation}
}

\NO{
\section{The Frame Set of a Window Function}\label{sec:frame_set}
One main issue in Gabor analysis is to determine all Gabor systems $\G(g, \Gamma)$ which form a Gabor frame. This question, however, is far too general to be answered completely. At the moment, the best we can hope for is to determine all sets $\Gamma \subset \R^{2d}$ such that for a fixed window $g$ the system $\G(g,\Gamma)$ is a Gabor frame. Even this task is extremely challenging and no general method is known to solve this problem. It is a case-by-case study and the tools to be used may differ largely for different (classes of) windows. In order to have a more systematic approach to this, still, very general problem, the notion of the frame set has been introduced (see e.g., \cite{Gro14}).
\\
\indent
However, we may distinguish between several types of frame sets and we are now going to introduce the list which was set up in \cite{Faulhuber_LyuNes_2019}. This generalizes the ideas from \cite{Gro14}.
\\
\indent
For a fixed dimension $d$ and a fixed window $g$ we define:
\begin{itemize}
	\item The $\alpha$-frame set
	\begin{equation}
		\mathfrak{F}_\alpha (g) = \{ \alpha \Z^{2d} \mid G(g, \alpha \Z^{2d}) \text{ is a frame}\}
	\end{equation}
	\item The separable or $(\alpha,\beta)$-frame set
	\begin{equation}
		\mathfrak{F}_{(\alpha,\beta)} (g) = \{ \alpha \Z^d \times \beta \Z^d \mid G(g, \alpha \Z^d \times \beta \Z^d) \text{ is a frame}\}
	\end{equation}
	For $d > 1$, we allow $\alpha$ and $\beta$ to be multi-indices ($\alpha \Z^d = \alpha_1 \Z \times \ldots \times \alpha_d \Z$, same for $\beta$).
	\item The lattice or $\L$-frame set
	\begin{equation}
		\mathfrak{F}_\L (g) = \{ \L \subset \R^{2d} \text{ lattice} \mid \G(g,\L) \text{ is a frame}\}
	\end{equation}
	\item The frame set
	\begin{equation}
		\mathfrak{F}(g) = \{ \Gamma \subset \R^{2d} \mid \G(g,\Gamma) \text{ is a frame}\}
	\end{equation}
\end{itemize}
We note that we have the following chain of inclusions
\begin{equation}
	\mathfrak{F}_\alpha (g) \subset \mathfrak{F}_{(\alpha,\beta)} (g) \subset \mathfrak{F}_\L (g) \subset \mathfrak{F} (g).
\end{equation}
We remark that in \cite{Faulhuber_LyuNes_2019} also the symplectic frame set has been introduced
\begin{equation}
	\mathfrak{F}_\sigma (g) = \{ \L_\sigma \subset \R^{2d} \text{ symplectic lattice} \mid \G(g, \L_\sigma) \text{ is a frame}\}.
\end{equation}
The chain of inclusion given in \cite{Faulhuber_LyuNes_2019} is however wrong. It is only true that
\begin{equation}
	\mathfrak{F}_\alpha (g) \subset \mathfrak{F}_\sigma (g) \subset \mathfrak{F}_\L (g) \subset \mathfrak{F} (g),
\end{equation}
but, in general, $\L_{(\alpha, \beta)}$ need not be in $\mathfrak{F}_\sigma (g)$ even if $\G(g,\L_{(\alpha, \beta)})$ is a frame. However, if $\alpha, \beta \in \R_+$ (and not multi-indices), then $\mathfrak{F}_{(\alpha, \beta)}(g) \subset \mathfrak{F}_\sigma (g)$.
\\
\indent
It is of current interest to better understand Gabor systems with lattices which are not symplectic. This is due to the fact that it is unknown how to classify those lattices and how to build the proper metaplectic-like operators. This problem does not occur in dimension $d=1$ as any 2-dimensional lattice is symplectic.
\\
\indent
In the case $d = 1$, some frame sets are known, even though there has been little progress until the work of Gröchenig and Stöckler \cite{GroechenigStoeckler_TotallyPositive_2013}. The only frame set of a (family of) function(s) where the entire frame set is known, is the Gaussian, which is why we will treat it in a separate section. But before, we will give a list of functions (ignoring the Gaussian) where at least the separable frame set is known.
\begin{itemize}
	\item The one-sided exponential $e^{-t} \indicator_{\R_+}(t) \; (\notin M^1(\R))$ \cite{Jan96} (Janssen, 1996):
	\begin{equation}
		\mathfrak{F}_{(\alpha,\beta)}(e^{-t} \indicator_{\R_+}(t)) = \{ \alpha \Z \times \beta \Z \mid \alpha \beta \leq 1 \}
	\end{equation}
	\item The hyperbolic secant $\tfrac{1}{\cosh(\pi t)} \; (\in M^1(\R))$ \cite{JanssenStrohmer_Secant_2002} (Janssen and Strohmer, 2002):
	\begin{equation}
		\mathfrak{F}_{(\alpha,\beta)}(\cosh(\pi t)^{-1}) = \{\alpha \Z \times \beta \Z \mid \alpha \beta < 1 \}
	\end{equation}
	\item The two-sided exponential $e^{-|t|} \; (\in M^1(\R))$ \cite{Janssen_CriticalDensity_2003} (Janssen, 2003):
	\begin{equation}
		\mathfrak{F}_{(\alpha,\beta)}(e^{-|t|}) = \{\alpha \Z \times \beta \Z \mid \alpha \beta < 1 \}
	\end{equation}
\end{itemize}
We say that the above frame sets are \textit{full}, as the necessary condition implied by the corresponding Balian-Low theorem is also sufficient. We note that if $\G(g, \alpha \Z \times \beta \Z)$ is a frame, then so is $\G( \F g, \beta \Z \times \alpha \Z)$ by the results on the symplectic and the metaplectic group in Section \ref{sec_symp_meta} \footnote{Also re-call the Fundamental Identity of Time-Frequency Analysis \ref{pro_fitf}, which is often interpreted in the way that the Fourier transform rotates the time-frequency plane by $90^\circ$.}. Therefore, the functions
\begin{itemize}
	\item $\F (e^{-t} \indicator_{\R_+}(t))(\omega) = \frac{1}{1 + 2 \pi i \omega}$
	\item $\F (e^{-|t|})(\omega) = \frac{2}{1 + 4 \pi^2 \omega^2} = \frac{2}{(1 + 2 \pi i \omega)(1 - 2 \pi i \omega)}$
\end{itemize}
possess full frame sets as well. The hyperbolic secant is a fixed-point of the Fourier transform and, hence, its Fourier transform does not provide a new window function with a full frame set. Gröchenig and Stöckler \cite{GroechenigStoeckler_TotallyPositive_2013} realized that the above mentioned window functions are so-called totally positive functions of finite type.
\begin{definition}
	A non-constant, measurable function $g: \R \to \R$ is called totally positive, if for any $N \in \N$ and any two sets of increasing real numbers
	\begin{equation}
		x_1 < x_2 < \ldots < x_N \qquad \textnormal{ and } \qquad y_1 < y_2 < \ldots < y_N,
	\end{equation}
	the matrix
	\begin{equation}
		M_g = \left( g(x_j-y_k) \right)_{j,k=1}^N
	\end{equation}
	is positive semidefinite, i.e.,
	\begin{equation}
		x \cdot M_g x \geq 0, \qquad \forall x \in \R^N.
	\end{equation}
\end{definition}
There is a nice characterization of totally positive functions by means of their Fourier transform due to I.~Schoenberg \cite{Schoe47, Schoe51}\footnote{Schoenberg's characterization was actually in terms of the two-sided Laplace transform, but the characterization is basically the same.}, which might as well serve as their definition.
\begin{theorem}[Schoenberg]
	A function $g: \R \to \R$ is totally positive if and only if its Fourier transform can be written as
	\begin{equation}
		\widehat{g} (\omega) = C \, e^{- \gamma \omega^2} e^{-2 \pi i \delta \omega} \prod_{k = 1}^\infty \frac{e^{2 \pi i \delta_k \omega}}{1 + 2 \pi i \delta_k \omega},
	\end{equation}
	with $C > 0$, $\gamma \geq 0$, $\delta, \delta_k \in \R$ and $0< \gamma + \sum_{k = 1}^\infty \delta_k^2 < \infty$.
\end{theorem}
As we are interested in Gabor systems with totally positive functions, we may as well assume that $\delta = 0$, as it corresponds to a time-shift of the function, which has no effect on the properties of the Gabor system. Hence, our functions of interest can be characterized by
\begin{equation}
	\widehat{g} (\omega) = C \, e^{- \gamma \omega^2} \prod_{k = 1}^\infty \frac{e^{2 \pi i \delta_k \omega}}{1 + 2 \pi i \delta_k \omega}.
\end{equation}
If $\gamma = 0$ and only finitely many $\delta_k$ are different from zero, then we call $g$ a totally positive function of finite type;
\begin{equation}
	\widehat{g} (\omega) = C \, \prod_{k = 1}^M \frac{e^{2 \pi i \delta_k \omega}}{1 + 2 \pi i \delta_k \omega}, \qquad M \in \N.
\end{equation}
If $\gamma \neq 0$ and, still, only finitely many $\delta_k$ are different from zero, we call $g$ a totally positive function of finite Gaussian type;
\begin{equation}
	\widehat{g} (\omega) = C \, e^{- \gamma \omega^2} \prod_{k = 1}^M \frac{e^{2 \pi i \delta_k \omega}}{1 + 2 \pi i \delta_k \omega}, \qquad M \in \N.
\end{equation}
The following result is due to Gröchenig and Stöckler \cite{GroechenigStoeckler_TotallyPositive_2013}.
\begin{theorem}[Gröchenig and Stöckler]
	Let $g$ be a totally positive function of finite type $M \geq 2$, then
	\begin{equation}
		\mathfrak{F}_{(\alpha, \beta)}(g) = \{ \alpha \Z \times \beta \Z \mid \alpha \beta < 1\}.
	\end{equation}
\end{theorem}
We note that any totally positive function of finite type $M \geq 2$ is in the modulation space $M^1(\R)$ (i.e., in Feichtinger's algebra $S_0(\R)$). Hence, the frame set of totally positive functions of finite type is full.
\\
\indent
Now, one might get the impression that the Balian-Low theorem is the only restriction for separable Gabor systems. However, there is another function for which the separable frame set is known, namely the characteristic function of an interval. The first parts of the frame set where described by Janssen and became known as Janssen's tie \cite{Janssen_Tie_2003}. The gaps where finally filled by Dai and Sun \cite{DaiSun_ABC_2016}. Their work describes in 17 different cases and on more than 100 pages for which parameters the Gabor system with the characteristic function of an interval and a separable lattice is a frame.
}

\NO{
\subsection{The Frame Set of a Gaussian}
As we have seen, for some window classes, it is known that the necessary condition imposed by the Balian-Low theorem is also sufficient for the Gabor system $\G(g, \alpha\Z \times \beta \Z)$ to be a frame. If we ask for general point sets $\Gamma$, or even only lattices $\L$, then there is precisely one window (or actually one class of windows) for which a full characterization of Gabor frames is known, namely the Gaussian window $g_0$ (and its metaplectic deformations as well as scalar multiples). In dimensions $d > 1$, hardly anything is known about (general) characterizations of Gabor frames.
\\
We will exploit the properties of the Bargmann transform to formulate the question on determining the frame set of a 1-dimensional (standard) Gaussian as a sampling problem in the Bargmann-Fock space $\F^2(\C)$. We consider relatively separated (also called uniformly discrete) sets $\Gamma \subset \C$, that is $\inf_{k \neq l} |z_k - z_l| > 0$ for all $z_k, z_l \in \Gamma$. The set $\Gamma$ is called a set of sampling for the Bargmann-Fock space $\F^2(\C)$ if there exist positive constants $0< A \leq B < \infty$ such that
\begin{equation}
	A \norm{F}_{\F^2}^2 \leq \sum_{z \in \Gamma} |F(z)|^2 e^{-\pi |z|^2} \leq B \norm{F}_{\F^2}^2, \quad \forall F \in \F^2(\C).
\end{equation}
The condition that $\Gamma \subset \C$ is relatively separated already ensures that the upper frame bound is finite. Since the Bargmann transform $B$ maps $\Lt[]$ unitarily to $\F^2(\C)$ by Theorem \ref{thm_BT_unitary}, we see that $\Gamma \subset \C$ is a set of sampling for $\F^2(\C)$ if and only if the real version of $\Gamma \subset \R^2$ yields a Gabor frame for the standard Gaussian $g_0$ \footnote{Again, we identify $\C$ and $\R^2$ in the natural way; $\C \ni z = x + i y \, \leftrightarrow (x,y) \in \R^2$.}.
}

\subsubsection{The Theorem of Lyubarskii and Seip and Wallstén}
The following theorem is so important that it deserves to have an own section. It still is one of the most celebrated results in Gabor analysis and was proven independently by Lyubarskii \cite{Lyu92} and Seip and Wallsten \cite{Sei92_1}, \cite{SeiWal92} in the year 1992.
\begin{theorem}\label{thm_Lyu_Seip_Wall}
	\begin{enumerate}[(i)]
		\NO{\item A discrete set $\Gamma \subset \C$ is a set of sampling for $\F^2(\C)$ if and only if $\Gamma$ can be expressed as a finite union of relatively separated sets and if it contains a relatively separated subset $\Gamma_0$ with $D^-(\Gamma_0) > 1$.}
		\item The Gabor system $\G(g_0, \Gamma)$ is a frame for $\Lt[]$ if and only if $\Gamma \subset \R^2$ is relatively separated and its lower Beurling density satisfies $D^-(\Gamma) > 1$.
	\end{enumerate}
\end{theorem}
\NO{
We note that the two statements are equivalent by the Bargmann transform. The proof relies, as pointed out above, on the Bargmann transform, which allows to switch to complex analysis (in one variable) and the sampling problem in the Bargmann-Fock space. In $\F^2(\C)$, entire functions with prescribed zeros can be constructed explicitly by means of (a version of) the Weierstrass sigma function
\begin{equation}
	\sigma (z) = z \prod_{\gamma \in \Gamma \backslash 0} \left(1- \frac{z}{\gamma}\right) e^{\frac{z}{\gamma} + \frac{z^2}{2 \gamma^2}}.
\end{equation}
The function $\tau(z) = \frac{\sigma(z)}{z}$ then takes the role of an interpolating function \footnote{\NO{Note the similarity of $\tau(z)$ to the role of sinc function $\frac{\sin(\pi x)}{\pi x}$ in the WNKS sampling theorem; $$\sin(\pi z) = \pi z \prod_{k \in \Z \slash \{0\}} \left(1 - \frac{z}{k}\right) e^{\frac{ z}{k}}, \quad z \in \C.$$}}. Then, from the lower Beurling density of the zero set, growth estimates for $\sigma$ are obtained. This leads to the result that $\Gamma$ is a set of sampling if and only if its lower Beurling density is (strictly) greater than 1.
}

\NO{
\subsection{The Frame Set in Higher Dimensions}
}
\NO{
As already mentioned, only few results on Gabor frames for $\Lt$, $d > 1$, are known. We state here the main result on tensor products of frames from \cite{Bou08} (see also Appendix \ref{app_Tensor}).
\begin{theorem}
	Let $\Gamma_1$ and $\Gamma_2$ be two countable sets. The sequence $(e_{\gamma_k})_{\gamma_k \in \Gamma_k}$ is a frame for the Hilbert space $\mathcal{H}_k$, $k=1,2$, if and only if $(e_{\gamma_1} \otimes e_{\gamma_2})_{(\gamma_1, \gamma_2) \in \Gamma_1 \times \Gamma_2}$ is a frame for $\mathcal{H}_1 \otimes \mathcal{H}_2$.
\end{theorem}
One consequence of the above theorem for higher dimensional Gabor systems is that the necessary density condition can no longer be sufficient.
\begin{example}
	Consider the Hilbert space $\Lt[2] \cong \Lt[] \otimes \Lt[]$ and the corresponding standard Gaussian\footnote{\NO{We are a bit sloppy with our notation here as $g_0(t) = 2^{d/4} e^{- \pi |t|^2}$ denotes the standard Gaussian in any dimension.}}
	\begin{equation}
		g_0(t_1,t_2) = 2^{1/2} e^{-\pi(t_1^2 + t_2^2)} = g_0(t_1) \otimes g_0(t_2).
	\end{equation}
		and consider the two Gabor systems $\G_1(g_0, \Z^2)$ and $\G_2(g_0,\alpha \Z^2)$ with $\alpha < 1$. Then, the Gabor system $\G = \G_1 \otimes \G_2$ is a (standard) Gaussian Gabor system in $\Lt[2]$ with arbitrarily high density, but it can never be a Gabor frame for $\Lt[2]$ because $\G_1$ is not a Gabor frame for its respective Hilbert space $\Lt[]$.
	\flushright{$\diamond$}
\end{example}
It is of course possible to construct a wealth of Gabor systems (not only Gaussian) with arbitrarily high density which fail to be frames. Also, it is straight forward to extend the method to tensor product Hilbert spaces of the form $\mathcal{H}_1 \otimes \ldots \otimes \mathcal{H}_N$.
}

\newpage
\begin{appendices}
\section{Appendices}
\subsection{More Uncertainty Principles}
This is a collection of some more uncertainty principles without proof. The interested reader is referred to \cite{Gro03_FeiStr} and \cite{Gro01}.
\begin{theorem}[Lieb's inequalities]
	Assume $f,g \in \Lt$. Then
	\begin{equation}
		\iint_{\R^{2d}} |V_g f(x,\omega)|^p \, d(x,\omega)
		\begin{cases}
			\leq \left(\tfrac{2}{p}\right)^d (\norm{f}_2 \norm{g_2})^p & 2 \leq p < \infty\\
			\geq \left(\tfrac{2}{p}\right)^d (\norm{f}_2 \norm{g_2})^p & 1 \leq p \leq 2\\
		\end{cases}
	\end{equation}
\end{theorem}
These inequalities are quite deep results and their proof requires the sharp constants in Hölder's inequality. A consequence of Lieb's inequalities is the following result.
\begin{theorem}[Gröchenig]
	Suppose that $\norm{f}_ = \norm{g}_2 = 1$. If $U \subset \R^{2d}$ and $\varepsilon \geq 0$ are chosen such that
	\begin{equation}
		\iint_U |V_g f(x,\omega)|^2 \, d(x,\omega) \geq (1-\varepsilon).
	\end{equation}
	Then, we have
	\begin{equation}
		|U| \geq \left( \tfrac{p}{2} \right)^{\frac{2d}{p-2}} (1-\varepsilon)^{\frac{p}{p-2}},
		\qquad
		\forall p > 2.
	\end{equation}
	In particular,
	\begin{equation}
		|U| \geq \sup_{p > 2} \left( \tfrac{p}{2} \right)^{\frac{2d}{p-2}} (1-\varepsilon)^{\frac{p}{p-2}} \geq 2^d (1- \varepsilon)^2,
		\qquad (p=4).
	\end{equation}
\end{theorem}

The classical Heisenberg-Pauli-Weyl uncertainty principle is minimized by dilated, time-frequency shifted Gauss functions. As the standard Gaussian is a fixed point of the Fourier transform, one may therefore suspect that Gaussian decay is the fastest possible simultaneous decay for a pair of Fourier transforms $(f, \widehat{f})$. This guess is answered affirmatively by Hardy's Theorem.
\begin{theorem}[Hardy]
	Let $f \in \Lt$ and assume that
	\begin{equation}
		f(x) = \mathcal{O}(e^{-a \pi x^2})
		\quad \text{ and } \quad
		\widehat{f}(\omega) = \mathcal{O}(e^{-b \pi \omega^2}),
	\end{equation}
	for some $a,b, > 0$. Then, three cases occur;
	\begin{enumerate}[(i)]
		\item If $a b = 1$, then $f(x) = c e^{-a \pi x^2}$.
		\item If $a b > 1$, then $f \equiv 0$.
		\item If $a b < 1$, then any finite linear combinations of Hermite functions satisfies these decay conditions.
	\end{enumerate}
\end{theorem}

Once again, this uncertainty principle can be rephrased in terms of the Rihaczek distribution.

\textit{Assume that $|Rf(x,\omega)| = \mathcal{O}(e^{-\pi (a x^2 + b \omega)^2})$. If $a b = 1$, then $f(x) = C e^{-a \pi x^2}$. If $a b > 1$, then $f \equiv 0$.}

According to Metatheorem \ref{mthm_C} we may conjecture a Hardy-like uncertainty principle for other time-frequency representations as well.
\begin{theorem}[Gröchenig, Zimmermann]
	Assume that
	\begin{equation}
		|V_g f(x,\omega)| = \mathcal{O}(e^{-\frac{\pi}{2} (a x^2 + b \omega^2)}),
	\end{equation}
	for some constants $a,b > 0$. Then three cases can occur;
	\begin{enumerate}[(i)]
		\item If $a b = 1$ and $V_g f \not\equiv 0$, then, both, $f$ and $g$ are scalar multiples of a time-frequency shift of a Gaussian of the form $e^{-a \pi x^2}$.
		\item If $a b > 1$, then either $f \equiv 0$ or $g \equiv 0$ (or both).
		\item If $a b < 1$, then the decay condition is satisfied whenever $f$ and $g$ are finite linear combinations of Hermite functions.
	\end{enumerate}
\end{theorem}

\subsection{Hilbert-Schmidt Operators}
Let $\mathcal{H}$ be a separable Hilbert space. A bounded operator $K: \mathcal{H} \to \mathcal{H}$ is called a Hilbert-Schmidt operator if
\begin{equation}
	\sum_{n \in \N} \norm{K e_n}_\mathcal{H}^2 < \infty ,
\end{equation}
for some orthonormal basis $\{e_n \mid n \in \N\}$ of $\mathcal{H}$. The Hilbert-Schmidt norm of $K$ is given by
\begin{equation}
	\norm{K}_{H.S.} = \left( \sum_{n \in \N} \norm{K e_n}_\mathcal{H}^2 \right)^{1/2}.
\end{equation}
This quantity is then independent of the choice of basis. The Hilbert-Schmidt norm dominates the operator norm. Let $f = \sum_{n \in \N} c_n e_n$, where $c_n = c_n(f)$, $\{e_n \mid n \in \N\}$ an orthonormal basis and $\norm{f}_\mathcal{H} = 1 = \sum_{n \in \N} |c_n|^2$, then
\begin{equation}
	\norm{K}_{op}^2 = \sup_{\stackrel{f \in \mathcal{H}}{\norm{f}_\mathcal{H}=1}} \norm{K f}_\mathcal{H}^2 \leq \sum_{n \in \N} |c_n|^2 \norm{K e_n}_\mathcal{H}^2 \leq \sum_{n \in \N} |c_n|^2 \sum_{n \in \N} \norm{K e_n}_\mathcal{H}^2 = \norm{K}_{H.S.}^2 .
\end{equation}
If $\mathcal{H} = \Lt$ and $K$ is an integral operator with integral kernel $k$, that is
\begin{equation}
	K f(x) = \int_{\Rd} k(x,y) f(y) \, dy ,
\end{equation}
then $K$ is Hilbert-Schmidt if and only if $k \in \Lt[2d]$. In this case
\begin{equation}
	\norm{K}_{H.S.} = \norm{k}_2 .
\end{equation}

\subsection{Neumann Series}\label{app_Neumann}
If a linear operator $U$ on a Banach space $\mathcal{B}$ is close enough to the identity operator $I$, then it is invertible. To be more precise, if $\norm{U-I}_{op} < 1$, then U is invertible on $\mathcal{B}$ and the inverse operator is given by the Neumann series
\begin{equation}
	U^{-1} = \sum_{k = 0}^\infty (I-U)^k.
\end{equation}
Furthermore,
\begin{equation}
	\norm{U^{-1}}_{op} \leq \frac{1}{1- \norm{I-U}_{op}}
\end{equation}
The convergence of the series should be understood in the sense of the operator norm
\begin{equation}
	\norm{U^{-1} - \sum_{k=1}^n (I-U)^k}_{op} \to 0, \quad n \to \infty.
\end{equation}
For details we refer to \cite[Thm.~2.2.3]{Christensen_2016}, \cite[App.~A.3]{Gro01} or \cite[p.~48]{Heuser_FA}.

\subsection{Toeplitz Matrices and Laurent Operators}\label{app_Toeplitz}
For details on this section and further reading we refer to \cite{BoettcherGrudsky_Toeplitz_2012}. A Toeplitz matrix $M$ is a (not necessarily square) matrix, which is constant on the diagonals, i.e.,
\begin{equation}
	M =
	\begin{pmatrix}
		m_0 & m_{-1} & m_{-2} & \ldots & \ldots & m_{-k}\\
		m_1 & m_0 & m_{-1} & m_{-2} & \ldots & m_{-(k-1)}\\
		\vdots & \ddots & \ddots & \ddots & \ddots & \vdots\\
		m_l & \ldots & \ldots & \ldots & \ldots & \cdot
	\end{pmatrix}
\end{equation}
A Laurent operator is a bi-infinite (square) Toeplitz matrix acting on $\ell^2(\Z)$;
\begin{equation}
	L =
	\begin{pmatrix}
		\ddots & \ddots & \ddots & \ddots & \ddots & \ddots &\ddots\\
		\ddots & m_1 & m_0 & m_{-1} & m_{-2} & m_{-3} & \ddots\\
		\ddots & m_2 & m_1 & m_0 & m_{-1} & m_{-2} & \ddots\\
		\ddots & m_3 & m_2 & m_1 & m_0 & m_{-1} & \ddots\\
		\ddots & \ddots & \ddots & \ddots & \ddots & \ddots & \ddots
	\end{pmatrix}
\end{equation}
The following theorem summaries the main results from \cite[Chap.~1]{BoettcherGrudsky_Toeplitz_2012}
\begin{theorem}
	Let $m(t) \in C(\T)$ with Fourier expansion
	\begin{equation}
		m(t) = \sum_{k \in \Z} m_k e^{2 \pi i k t},
	\end{equation}
	and let $L(m)$ be the associated Laurent operator. Then
	\begin{equation}
		\norm{L(m)}_{op} = \norm{m}_\infty
		\quad \text{ and } \quad
		\norm{L(m)^{-1}}_{op} \leq \norm{L(m^{-1})}_{op} = \norm{m^{-1}}_\infty
	\end{equation}
\end{theorem}
If $m$ is ``nice" enough, then the bound for the inverse frame operator is sharp as well.

\subsection{Proof of Hudson's Theorem}
We will now prove that $W(f,f) = W f > 0$ if and only if $f$ is a Gaussian function of the form
\begin{equation}
	f(t) = e^{-\pi t \cdot A t + 2 \pi b \cdot t + c},
\end{equation}
where $A \in GL(\C,d)$, $\Re(A) > 0$ (i.e., positive definite) and $A^* = \overline{A}^T = A$, $b \in \C^d$, $c \in \C$.

For the sufficiency, we note that any Gaussian function $f$ can be written as a scalar multiple of a  finite composition of metaplectic operators applied to the standard Gaussian $g_0(t) = 2^{d/4} e^{-\pi t^2}$ and an appropriate time-frequency shift. By the symplectic covariance principle for the Wigner transform, i.e., $W(f,g)(S z) = W(\widehat{S}^{1} f, \widehat{S}^{-1} g)(z)$, $S \in Sp(\R,2d)$ and $\widehat{S} = \mu(S)$, it is therefore sufficient to establish the sufficiency for the standard Gaussian. As
\begin{equation}
	W g_0(x,\omega) = 2^d e^{-2 \pi (x^2+\omega^2)},
\end{equation}
this part is done.

For the necessity part, we make use of the Bargmann transform and use some complex analysis. Assume that $f \in \Lt$ (non-zero) and $W f \geq 0$. We take the inner product of $Wf$ with the Wigner distribution of $g_0$ and note that, by using the covarince principle in Proposition \ref{pro_Wigner},
\begin{equation}
	W(M_{-\omega} T_x g_0) > 0, \quad \forall (x,\omega) \in \R^{2d}.
\end{equation}
Therefore,
\begin{equation}\label{eq_WfWg0}
	\langle W f, W(M_{-\omega} T_x g_0) \rangle = \iint_{\R^{2d}} Wf(\xi, \eta) W(M_{-\omega} T_x g_0)(\xi, \eta) \, d(\xi, \eta) > 0, \quad \forall (x, \omega) \in \R^{2d}.
\end{equation}
Now, we apply Moyal's formula \eqref{eq_Moyal} and Proposition \ref{pro_BT} for the Bargmann transform to identify the inner product $\langle Wf, W(M_{-\omega} T_x g_0) \rangle$ as the Bargmann transform of $f$. Writing $z = x + i \omega \in \C^d$, we obtain
\begin{align}
	\langle W f, W(M_{-\omega} T_x g_0) \rangle
	& = | \langle f, M_{-\omega} T_x g_0 \rangle|^2\\
	& = |V_{g_0} f(x,-\omega)|^2\\
	& = |Bf(z)|^2 e^{-\pi|z|^2}.
\end{align}
Since the entire function $Bf$ does not vanish by \eqref{eq_WfWg0}, there exists an entire function $q(z)$, such that
\begin{equation}
	B f(z) = e^{q(z)}.
\end{equation}
Furthermore, since
\begin{equation}
	|Bf(z)|e^{-\frac{\pi}{2}|z|^2} \leq \norm{V_{g_0} f}_\infty \leq \norm{f}_2 \norm{g_0}_2 = \norm{f}_2,
\end{equation}
$Bf$ satisfies the growth estimate
\begin{equation}
	|Bf(z)| \leq \norm{f}_2 e^{\frac{\pi}{2} |z|^2}.
\end{equation}
By taking the logarithm, we obtain the estimate
\begin{equation}
	|\Re \left( q(z) \right)| \leq c + \frac{\pi}{2} |z|^2.
\end{equation}
It follows by Caratheodory's inequality of complex analysis (see Appendix \ref{app_Caratheodory}) that $q$ itself satisfies
\begin{equation}
	|q(z)| \leq  C_1 + C_2 |z|^2.
\end{equation}
Therefore, $q$ must be a quadratic polynomial of the form
\begin{equation}
	q(z) = \pi A' z^2 + 2 \pi b' \cdot z + c'.
\end{equation}
The restriction of $B f$ to vectors in $i \Rd$ is
\begin{equation}
	Bf(i \omega) e^{-\frac{\pi}{2} \omega^2} = e^{q(i \omega) - \frac{\pi}{2} \omega^2} = e^{- \pi (\frac{1}{2} + A') \omega^2 + 2 \pi i b' \cdot \omega + c'},
\end{equation}
which is a Gaussian function. Next, we express the restriction of $V_{g_0} f$ to $\{0\} \times \R$ in two different ways. On one hand,
\begin{equation}
	V_{g_0} f(0, -\omega) = \langle f, M_{-\omega} g_0 \rangle = \F(f \, g_0)(-\omega).
\end{equation}
On the other hand, by using Proposition \ref{pro_BT} once more, we see that
\begin{equation}
	V_{g_0} f(0, - \omega) = Bf(i \omega) e^{- \frac{\pi}{2} \omega^2}.
\end{equation}
Therefore, $\F(f \, g_0)$ is in $\Lt$ and is a Gaussian function. It follows, by using the operator $\mu(J) = (-i)^{d/2} \F$, that the function $f \, g_0$ is itself a Gaussian function. This implies that $f$ itself is a Gaussian function and since $f \in \Lt$, $\Re(A) > 0$ is necessary.
\begin{flushright}
	$\square$
\end{flushright}

\subsection{Caratheodory's Inequality}\label{app_Caratheodory}
This inequality provides size estimates for an analytic function $f$ when size estimates are known only for the real part of $f$. As a special case we have the following theorem for entire functions of one variable \cite[p.~3]{Boa54}.
\begin{theorem}
	If $f$ is entire on $\C$ and $\Re f(z) \leq A_\varepsilon |z|^{\alpha + \varepsilon}$ for all $z \in \C$, then $f$ is a polynomial of degree at most $\alpha$.
\end{theorem}
\begin{corollary}[Caratheodory's inequality]
	If $f(z) = f(z_1, \ldots, z_d)$ is entire on $\C^d$ and
	\begin{equation}
		\Re f(z) \leq D + E |z|^2,
	\end{equation}
	then $f$ is a quadratic polynomial of the form $f(z) = A z^2 + b \cdot z + c$ for some $d \times d$ matrix over $\C$, $b \in \C^d$ and $c \in \C$.
\end{corollary}

\subsection{Tensor Hilbert Spaces}\label{app_Tensor}
We repeat a few of the details on tensor Hilbert spaces given in \cite{Bou08}. Let $\mathcal{H}_1$ and $\mathcal{H}_2$ be two Hilbert spaces with inner product $\langle . , . \rangle_k$ and norm $\norm{.}_k$, $k, = 1,2$, respectively. By $\mathcal{H}_1 \otimes_a \mathcal{H}_2$ we denote the algebraic tensor product of the two spaces. We can define an inner product on $\mathcal{H}_1 \otimes_a \mathcal{H}_2$ in the following way;
\begin{equation}
	\langle f_1 \otimes f_2, g_1 \otimes g_2 \rangle = \langle f_1,g_1\rangle_1 \langle f_2, g_2 \rangle_2, \quad \forall f_1,g_1 \in \mathcal{H}_1, \; \forall f_2, g_2 \in \mathcal{H}_2.
\end{equation}
The completion under this inner product is a Hilbert space, which we denote by $\mathcal{H}_1 \otimes \mathcal{H}_2$. This is called the topological tensor product of the Hilbert spaces $\mathcal{H}_1$ and $\mathcal{H}_2$.

\end{appendices}

\newpage
%\bibliographystyle{plain}
%\bibliography{../../../mybib}

\begin{thebibliography}{10}

\bibitem{BenOko_ModulationSpaces}
Árpád Bényi and Kasso~A. Okoudjou.
\newblock {\em {Modulation Spaces – With Applications to Pseudodifferential
  Operators and Nonlinear Schrödinger Equations}}.
\newblock {Applied and Numerical Harmonic Analysis}. Birkäuser, 2020.

\bibitem{Boa54}
Ralph~P. {Boas Jr.}
\newblock {\em {Entire Functions}}.
\newblock {Academic Press}, 1954.

\bibitem{Bou08}
Abdelkrim Bourouihiya.
\newblock {The tensor product of frames}.
\newblock {\em Sampling Theory in Signal \& Image Processing}, 7(1):65–76,
  2008.

\bibitem{BoettcherGrudsky_Toeplitz_2012}
Albrecht Böttcher and Sergei~M. Grudsky.
\newblock {\em {Toeplitz matrices, asymptotic linear algebra, and functional
  analysis}}.
\newblock Birkhäuser, 2012.

\bibitem{Christensen_2016}
Ole Christensen.
\newblock {\em {An Introduction to Frames and Riesz Bases}}.
\newblock {Applied and Numerical Harmonic Analysis}. Birkhäuser, 2. edition,
  2016.

\bibitem{Con_FA90}
John~B. Conway.
\newblock {\em {A Course in Functional Analysis}}.
\newblock {Graduate Texts in Mathematics}. Springer, second edition, 1990.

\bibitem{DaiSun_ABC_2016}
Xin-Rong Dai and Qiyu Sun.
\newblock {\em {The abc-Problem for Gabor systems}}, volume 244 of {\em
  {Memoirs of the American Mathematical Society}}.
\newblock American Mathematical Society, 2016.

\bibitem{Faulhuber_Note_2018}
Markus Faulhuber.
\newblock {A short note on the frame set of odd functions}.
\newblock {\em Bulletin of the Australian Mathematical Society},
  98(3):481–493, December 2018.

\bibitem{Faulhuber_LyuNes_2019}
Markus Faulhuber.
\newblock {On the Parity under Metaplectic Operators and an Extension of a
  Result of Lyubarskii and Nes}.
\newblock {\em Results in Mathematics}, 75, 2020.

\bibitem{Fei81}
Hans~G. Feichtinger.
\newblock {On a new Segal algebra}.
\newblock {\em Monatshefte für Mathematik}, 92(4): 269–289, 1981.

\bibitem{FeiJak20}
Hans~G. Feichtinger and Mads Jakobsen.
\newblock {Distribution Theory by Riemann Integrals}.
\newblock In Siddiqi~A. {Manchanda P.}, Lozi~R., editors, {\em {Mathematical
  Modelling, Optimization, Analytic and Numerical Solutions}}, page 33–76.
  Springer, 2020.

\bibitem{FeiStr98}
Hans~G. Feichtinger and Thomas Strohmer.
\newblock {\em {Gabor Analysis and Algorithms: Theory and Applications}}.
\newblock Birkhäuser Boston, Boston, MA, 1998.

\bibitem{Fol89}
Gerald~B. Folland.
\newblock {\em {Harmonic analysis in phase space}}.
\newblock Number 122 in {Annals of Mathematics Studies}. Princeton University
  Press, 1989.

\bibitem{Gab46}
Dennis Gabor.
\newblock {Theory of communication}.
\newblock {\em Journal of the Institution of Electrical Engineers},
  93(26):429–457, 1946.

\bibitem{Gel50}
Izrail~M. Gel'fand.
\newblock {Eigenfunction expansion for equations with periodic coefficients}.
\newblock {\em {Doklady Akademii Nauk SSSR}}, 1950.

\bibitem{Gos11}
Maurice A.~de Gosson.
\newblock {\em {Symplectic Methods in Harmonic Analysis and in Mathematical
  Physics}}, volume~7 of {\em {Pseudo-Differential Operators. Theory and
  Applications}}.
\newblock Birkhäuser/Springer Basel AG, Basel, 2011.

\bibitem{Gosson_Wigner_2017}
Maurice A.~de Gosson.
\newblock {\em {The Wigner Transform}}.
\newblock World Scientific, Singapore, 2017.

\bibitem{Gro_Poisson_1996}
Karlheinz Gröchenig.
\newblock {An uncertainty principle related to the Poisson summation formula}.
\newblock {\em Studia Mathematica}, 121(1):87–104, 1996.

\bibitem{Gro01}
Karlheinz Gröchenig.
\newblock {\em {Foundations of Time-Frequency Analysis}}.
\newblock {Applied and Numerical Harmonic Analysis}. Birkhäuser, Boston, MA, 2001.

\bibitem{Gro03_FeiStr}
Karlheinz Gröchenig.
\newblock {Uncertainty Principles for Time-Frequency Representations}.
\newblock In Hans~G. Feichtinger and Thomas Strohmer, editors, {\em {{A}dvances in {G}abor {A}nalysis}}, page 11–30. Birkhäuser, Boston, MA, 2003.

\bibitem{Gro14}
Karlheinz Gröchenig.
\newblock {The Mystery of Gabor Frames}.
\newblock {\em Journal of Fourier Analysis and Applications}, 20(4):865–895, 2014.

\bibitem{GroHanHeiKut02}
Karlheinz Gröchenig, Deguang Han, Christopher Heil, and Gitta Kutyniok.
\newblock {The Balian-Low Theorem for (Symplectic) Lattices in Higher Dimensions}.
\newblock {\em Applied and Computational Harmonic Analysis}, 13:169–176, 2002.

\bibitem{GroKop19}
K.\ Gr\"ochenig and S.\ Koppensteiner.
\newblock {Gabor Frames: Characterizations and Coarse Structure}.
\newblock In A.\ Aldroubi, C.\ Cabrelli, S.\ Jaffard, and U.\ Molter, editors, {\em New Trends in Applied Harmonic Analysis, Volume 2}, {Applied and Numerical Harmonic Analysis}, pp.\ 93--120. Springer, 2019.

\bibitem{GroJamMal19}
Karlheinz Gröchenig, Philippe Jaming, and Eugenia Malinnikova.
\newblock {Zeros of the Wigner Distribution and the Short-Time Fourier
  Transform}.
\newblock {\em Revista Matemática Compultense}, 33:723–744, 2020.

\bibitem{GroechenigStoeckler_TotallyPositive_2013}
Karlheinz Gröchenig and Joachim Stöckler.
\newblock {Gabor frames and totally positive functions}.
\newblock {\em Duke Mathematical Journal}, 162(6):1003–1031, 04 2013.

\bibitem{Hei03}
Christopher Heil.
\newblock {An introduction to weighted Wiener amalgams}.
\newblock In M.~Krishna, R.~Radha, and S.~Thangavelu, editors, {\em {Wavelets
  and their Applications}}, page 183–216. Allied Publishers, New Delhi, 2003.

\bibitem{Heuser_FA}
Harro Heuser.
\newblock {\em {Functional Analysis}}.
\newblock Wiley, NY, 1982.

\bibitem{How80}
Roger Howe.
\newblock {On the Role of the Heisenberg Group in Harmonic Analysis}.
\newblock {\em Bulletin of the American Mathematical Society (New Series)},
  3(2):821–843, 1980.

\bibitem{Hud74}
Robin Hudson.
\newblock {When is the Wigner quasi-probability density non-negative?}
\newblock {\em Reports on Mathematical Physics}, 6(2):249–252, 1974.

\bibitem{Igu72}
Jun-ichi Igusa.
\newblock {\em {Theta Functions}}.
\newblock {Grundleheren der mathematischen Wissenschaft}. Springer, 1972.

\bibitem{Jakobsen_S0_2018}
Mads~S. Jakobsen.
\newblock {On a (No Longer) New Segal Algebra: A Review of the Feichtinger
  Algebra}.
\newblock {\em Journal of Fourier Analysis and Applications},
  24(6):1579–1660, 2018.

\bibitem{Jan88}
Augustus J.~E.~M. Janssen.
\newblock {The Zak transform: a signal transform for sampled time-continuous
  signals}.
\newblock {\em Philips Journal of Research}, 43:23–69, 1988.

\bibitem{Jan96}
Augustus J.~E.~M. Janssen.
\newblock {Some Weyl-Heisenberg frame bound calculations}.
\newblock {\em Indagationes Mathematicae}, 7(2):165–183, 1996.

\bibitem{Jan98}
Augustus J. E.~M. Janssen.
\newblock {Proof of a Conjecture on Supports if Wigner Distributions}.
\newblock {\em Journal of Fourier Analysis and Applications}, 4(6):723–726,
  1998.

\bibitem{Janssen_CriticalDensity_2003}
Augustus J.~E.~M. Janssen.
\newblock {On Generating Tight Gabor Frames at Critical Density}.
\newblock {\em Journal of Fourier Analysis and Applications}, 9(2):175–214,
  2003.

\bibitem{Janssen_Tie_2003}
Augustus J.~E.~M. Janssen.
\newblock {Zak Transforms with Few Zeros and the Tie}.
\newblock In Hans~G. Feichtinger and Thomas Strohmer, editors, {\em {Advances
  in Gabor Analysis}}, page 31–70. Birkhäuser Boston, Boston, MA, 2003.

\bibitem{JanssenStrohmer_Secant_2002}
Augustus J.~E.~M. Janssen and Thomas Strohmer.
\newblock {Hyperbolic Secants Yield Gabor Frames}.
\newblock {\em Applied and Computational Harmonic Analysis}, 12(2):259–267,
  2002.

\bibitem{Kop17}
Sarah Koppensteiner.
\newblock {Characterization of Gabor Frames}.
\newblock Master's thesis, University of Vienna, 2017.

\bibitem{Kot33}
Vladimir~A. Kotelnikov.
\newblock {On the Transmission Capacity of the 'Ether' and Wire in
  Electrocommunications (1933) translated by V. E. Katsnelson}.
\newblock In John~J. Benedetto and Paulo J. S.~G. Ferreira, editors, {\em
  {Modern Sampling Theorey: Mathematics and Applications}}. Birkhäsuer, 2001.

\bibitem{Lie90}
Elliott Lieb.
\newblock {Integral bounds for radar ambiguity functions and Wigner
  distributions}.
\newblock {\em Journal of Mathematical Physics}, 31(3):594–599, 1990.

\bibitem{Lyu92}
Yurii Lyubarskii.
\newblock {Frames in the Bargmann space of entire functions}.
\newblock In {\em {Entire and Subharmonic Functions}}, page 167–180. American
  Mathematical Society, Providence, RI, 1992.

\bibitem{MacMac03}
Stephen~D. Mackey and Niloufer Mackey.
\newblock {On the determinant of Symplectic Matrices}, 2003.

\bibitem{DufSal17}
Dusa McDuff and Dietmar Salamon.
\newblock {\em {Introduction to Symplectic Topology}}.
\newblock Oxford University Press, 3rd edition, 2017.

\bibitem{Neumann_Quantenmechanik_1932}
John~von Neumann.
\newblock {\em {Mathematische Grundlagen der Quantenmechanik}}.
\newblock Springer, Berlin, 1932.

\bibitem{Schoe47}
Isaac~J. Schoenberg.
\newblock {On totally positive functions, Laplace integrals and entire
  functions of the Laguerre-Pólya-Schur type}.
\newblock {\em Proceedings of the National Academy of Sciences of the USA},
  33(1):11–17, 1947.

\bibitem{Schoe51}
Isaac~J. Schoenberg.
\newblock {On Polya frequency functions}.
\newblock {\em Journal d'Analyse Mathematique}, 1(1):331–374, 1951.

\bibitem{Sei92_1}
Kristian Seip.
\newblock {Density theorems for sampling and interpolation in the Bargmann --
  Fock space I}.
\newblock {\em Journal für die reine und angewandte Mathematik (Crelles
  Journal)}, 429:91--106, 1992.

\bibitem{SeiWal92}
Kristian Seip and Robert Wallstén.
\newblock {Density theorems for sampling and interpolation in the Bargmann --
  Fock space II}.
\newblock {\em Journal für die reine und angewandte Mathematik (Crelles
  Journal)}, 429:107--114, 1992.

\bibitem{Sil00}
John~R. Silvester.
\newblock {Determinants of Block Matrices}.
\newblock {\em The Mathematical Gazette}, 84(501):460–467, 2000.

\bibitem{StrTan05}
Thomas Strohmer and Jared Tanner.
\newblock {Implementations of Shannon's sampling theorem, a time-frequency
  approach}.
\newblock {\em Sampling Theory in Signal and Image Processing}, 4(1):1–17,
  2005.

\bibitem{TolOrr95}
Richard Tolimieri and Richard~S. Orr.
\newblock {Poisson summation, the ambiguity function, and the theory of
  Weyl-Heisenberg frames.}
\newblock {\em Journal of Fourier Analysis and Applications}, 1(3):233–247,
  1995.

\bibitem{Wie32}
Norbert Wiener.
\newblock {Tauberian Theorems}.
\newblock {\em Annals of Mathematics}, 33(1):1–100, 1932.

\bibitem{Wig32}
Eugine~P. Wigner.
\newblock {On the Quantum Correction For Thermodynamic Equilibrium}.
\newblock {\em Physical Review}, 40(5):749–759, 1932.

\bibitem{Zak67}
Joshua Zak.
\newblock {Finite Translations in Solid-State Physics}.
\newblock {\em Physical Review Letters}, 19(24):1385–1387, December 1967.

\bibitem{Zay93}
Ahmed~I. Zayed.
\newblock {\em {Advances in Shannon's Sampling Theory}}.
\newblock CRC Press, 1993.

\end{thebibliography}

\end{document}